\documentclass[12pt]{amsart}
\usepackage{amsmath,amssymb,amsthm,amscd}
\usepackage[all]{xy}

\usepackage{tikz}
\usepackage{graphicx, comment, wrapfig, floatrow, fullpage}

\ifx\pdfoutput\undefined

\else
\usepackage[pdftex,colorlinks=true,linkcolor=blue,urlcolor=blue]{hyperref}
\fi

\title{Bratteli diagrams, translation flows and their $C^{*}$-algebras}
\author{Ian F. Putnam}
\address{Department of Mathematics and Statistics, University of Victoria}
\email{ifputnam@uvic.ca}
\author{Rodrigo Trevi\~{n}o}
\address{Department of Mathematics, The University of Maryland}
\email{rodrigo@trevino.cat}

\newcommand{\Z}{\mathbb{Z}}

\newcommand{\R}{\mathbb{R}}
\newcommand{\C}{\mathbb{C}}

\newtheorem{defn}{Definition}[section]
\newtheorem{thm}[defn]{Theorem}

\newtheorem{prop}[defn]{Proposition}
\newtheorem{cor}[defn]{Corollary}
\newtheorem{lemma}[defn]{Lemma}

\newtheorem{rmk}[defn]{Remark}
\newtheorem{hyp}[defn]{Hypothesis}

\def\tr{\mathrm{Tr}}

\begin{document}
\begin{abstract}
In \cite{LT}, Kathryn Lindsey and the second author constructed
a translation surface from a bi-infinite Bratteli diagram. We continue
an investigation into these surfaces. The construction 
given in \cite{LT} was essentially combinatorial. 
Here, we provide explicit links between the path space of
the Bratteli diagram and the surface, including various 
intermediate topological spaces. This allows us to relate 
the $C^{*}$-algebras associated with  tail equivalence
on the Bratteli diagram and the  
foliation of the surface, under some mild hypotheses. This also 
allows
us to relate the K-theory of the $C^{*}$-algebras involved. We
 also treat the case of finite genus surfaces in some detail, where the 
 process of Rauzy-Veech induction (and its inverse) provide 
 an explicit construction of the Bratteli diagrams involved.
 \end{abstract}
\maketitle

\tableofcontents


\section{Introduction}
\label{intro}
There has been considerable interest over many years in the dynamics of 
foliations and flows on translation surfaces or flat surfaces. We refer the reader to
\cite{viana:notes} for a broader discussion. 

In \cite{LT}, Kathryn Lindsey and the second author introduced a construction
of translation surfaces based on combinatorial data. The main point
of the construction was that, while giving an alternate view
of the finite genus case, it also provided  a very general method of 
construction of  surfaces of infinite genus. In addition, it was shown that 
the dynamical behavior on these infinite genus surfaces was much broader than
the finite genus case.

The combinatorial data needed for the construction is a variation 
of a Bratteli diagram. A Bratteli diagram is a locally finite, but infinite
directed graph. They first appeared in Ola Bratteli's seminal work 
on inductive limits of finite dimensional $C^{*}$-algebras, or AF-algebras 
\cite{Bra:AF}. 
Bratteli used the diagrams to encode combinatorial data  on maps between
 direct sums of matrix algebras. Later, Renault  \cite{Ren:gpd} showed that 
 the diagrams could also be used to construct topological groupoids
 (equivalence relations) and the $C^{*}$-algebras constructed from 
 such examples coincided with those considered by Bratteli. More
 specifically, one considers the topological space of infinite paths
 in the diagram along with the equivalence relation known as tail equivalence:
 two paths are tail equivalent if they are equal beyond some fixed point.
 
More recently, Bratteli diagrams also been used extensively in 
dynamical systems, initiated by the work of 
Vershik \cite{vershik:approx, vershik:adic} and subsequently, Herman, Putnam and Skau \cite{HPS:BV}. In particular, 
this involved introducing the notion of an ordered Bratteli diagram.

Bratteli diagrams were first used in the context of the dynamics
of translation surfaces by A. Bufetov \cite{bufetov:limit}.
This was expanded upon  by K. Lindsey and the second author \cite{LT}.
Their innovation was to 
consider a bi-infinite Bratteli diagram, where the vertex and edge sets 
are indexed by the integers, rather than the positive
 integers as is usually the case. They also assume 
 a pair of orders on the edge set
  the first compares edges having the 
 same range or terminus and the second compares edges having 
 the same source or origin.
 
 The construction of the surface was then given in \cite{LT} in 
 a combinatorial manner: 
 finite paths gave open rectangular components for the surface and the 
 terminus, origin and order data provides rules for attaching them.
 One also sees that the leaves of the horizontal and vertical folitations
 correspond to right and left tail equivalence in the diagram.
From a dynamical standpoint that is quite satisfactory, but it leaves
open the question: if we think about the  AF-algebra of the diagram
and the foliation $C^{*}$-algebra, how exactly are they related?
The main goal of this paper is to address this question.

On the one hand, we have a very satisfactory description of the AF-algebra 
as given by Renault, by looking at the path space of the diagram, tail 
equivalence on it, and the standard construction of a groupoid 
$C^{*}$-algebra.  What is missing on the other side is a description of the 
surface itself in terms of the infinite path space of the diagram. 
At first glance, this seems a tall order because the former
is a locally Euclidean space while the latter is totally disconnected.
A good clue that these are not so far apart is provided by a 
very familiar, but under-appreciated, notion: decimal expansion. 
This is already a familiar idea in dynamics through the use 
of Markov partitions to code hyperbolic systems.
Let us take some time to describe this simple idea more clearly as
it is essentially the basis for the remainder of the paper. 

Everyone is familiar with the fact that every real number
has a decimal expansion which is (almost) unique.
A more mathematically precise view of decimal expansion
is as a map from $\prod_{1}^{\infty} \{ 0, 1, 2, \ldots, 9 \}$ 
to $[0,1]$ (simply ignoring the integer part). It is surjective and 
each point in the image has  a unique pre-image, except a countable
subset: rational numbers of the form $k/10^{l}, k \in \Z, l \geq 1$.

This becomes more interesting if the first space is given the product
topology. The map is then continuous, but the two spaces are remarkably 
different: the first is totally disconnected, while the second is 
connected.

Another viewpoint is to realize that the first space can be endowed with 
lexicographic order. The order topology coincides with the 
product topology and the map is order preserving. In fact, more is true:
if we note, for example, that $.19999 \cdots$ and $.2000 \cdots$ are both 
decimal expansions of $\frac{1}{5}$, the latter is precisely
the successor of the former in the lexicographic order. In fact, two points
are identified by the map if and only if one is the successor of the other.

Bratteli diagrams offer a vast generalization of this idea.
A Bratteli diagram, $\mathcal{B}$, consists of a sequence  of finite non-empty 
vertex sets $V_{n}, n \geq 0$ (we assume $\#V_{0} =1$ for convenience)
and edge sets $E_{n}, n \geq 1$: each edge $e$ in $E_{n}$ has a 
source $s(e)$ in $V_{n-1}$ and a range $r(e)$ in $V_{n}$.
We can then consider the space of infinite paths, denoted $X_{\mathcal{B}}$. 
It has  natural topology making it compact and totally disconnected.
We add two pieces of data: a state $\nu$ (see Definition \ref{BD:75}) and a
partial  
order on the edge sets where two edges, $e, f$, are comparable if and
only if $s(e) = s(f)$. Such items always exist. 
The path space $X_{\mathcal{B}}$ then becomes linearly ordered
by lexicographic order. In addition, the state provides a measure 
on this space in a natural way (\ref{path:120}). We can then define explicitly a map
from $X_{\mathcal{B}}$ to a closed interval which is order-preserving and identifies
two points if and only if one is the successor  of the other. 
 We leave the details
to Lemma \ref{order:20}.
Usual decimal 
expansion can be seen in the case
 $\#V_{n}=1, \#E_{n}=10$, for all $ n \geq 1$.

This is appealing, though not terribly deep. The Lindsey-Trevi\~{n}o 
starting point is to consider a bi-infinite Bratteli diagram
where the vertex and edge sets are indexed by the integers rather than
the natural numbers. We drop the condition that $\#V_{0} =1$.
In addition, we require two orders on the edge sets, 
one based on $s$ (as before) and the other on $r$ and two states, 
$\nu_{s}, \nu_{r}$. Our path space $X_{\mathcal{B}}$ now consists 
of bi-infinite paths. Basically, our surface is now obtained as 
a quotient of $X_{\mathcal{B}}$ by identifying successor/predecessors
in both orders. That is overly simplistic and we need
to make some subtle alterations. But let us leave that aside for the moment
and describe this space, locally. If we fix a finite path $p$ in the diagram 
going from vertex $v$ in $V_{m}$ to $w$ in $V_{n}$, $m < n$, we can look
at the set of all bi-infinite paths which agree with $p$ between $m$ and $n$.
This is a clopen set. But it is clear that such a path consists 
of three parts, from $-\infty$ to $s(p)$, then $p$, then 
from $r(p)$ to $\infty$. The first and third parts are clearly 
independent and  lie in the path spaces of two subdiagrams (although the 
first is oriented the wrong direction). Applying the map we described 
earlier using the $r$-data to the first and the $s$-data to the third,
we obtain a map to a closed rectangle in the plane which descends to 
a local homeomorphism on our quotient space. These maps can be used to
define an atlas for the quotient space which satisfy the condition making it
a translation surface. Moreover, if two points are right-tail equivalent
then they lie on the same horizontal line, while two points that 
are left-tail equivalent lie on the same vertical line. 
So our quotient map from $X_{\mathcal{B}}$ to the surface
maps right-tail equivalence  to the horizontal foliation
and left-tail equivalence to the vertical foliation. This provides
the links between the AF-algebras and the foliation algebras which is our main goal.

In section \ref{BD}, we describe basics of Bratteli diagrams. 
In particular, we have the 
classic version, the 
bi-infinite version and ordered versions of both. This includes
some basic concepts such as a simple diagram (\ref{BD:25}) (some telescope 
has full edge connections) and finite rank (\ref{BD:29}), which means that
there is a uniform bound on the cardinality of the vertex sets.
The third section
describes the path space of a Bratteli diagram, both classic and bi-infinite versions.
In the fourth section, we describe the consequences for the infinite path space of orders
on a Bratteli diagram. This includes a complete description
of the analogues of decimal expansion as discussed above.

As we indicated above, our basic idea is to begin with a 
bi-infinite ordered Bratteli diagram, $\mathcal{B}$, and 
take a quotient of the path space
of a bi-infinite Bratteli diagram, $X_{\mathcal{B}}$.
However, there are some bad points in this space that
need to be dealt with, just as flat surfaces in genus greater than one
necessarily  have
singularities. These fall into two types.
The first are those in which every path is maximal in the $\leq_{s}$-order 
or  maximal in the $\leq_{r}$-order  or minimal in the 
$\leq_{s}$-order 
or  minimal in the $\leq_{r}$-order. We refer to these as
 extremal (\ref{singular:50})
and,
if the diagram is finite rank, it is a finite set. More subtly, there
is a second type of point, which we call singular. We have two 
(partially defined) operations: taking the successor in the $\leq_{s}$-order
and taking the successor in the $\leq_{r}$-order. There may be 
points where their compositions are defined, in either order, but fail
to yield the same result. This is our ordered Bratteli
 diagram's way of telling
us the common point they represent
 in the quotient space will fail to have a
flat neighourhood. These points, which we denote by $\Sigma_{\mathcal{B}}$, 
must be removed (\ref{singular:50}). The set is at worst countable and
its union with the extremal points is closed. We now restrict our attention to the 
open 
compliment of this, which we denote by $Y_{\mathcal{B}}$ (\ref{surface:10}). 

In section \ref{surface}, we introduce our surface, $S_{\mathcal{B}}$. This is done by
identifying points  of $Y_{\mathcal{B}}$ with their $\leq_{s}$-successors
and their $\leq_{r}$-successors. Of course, this means that 
there are two intermediary spaces where only one of the two
identifications is done. The main work of this section is to explicitly
describe an atlas for the space $S_{\mathcal{B}}$ whose transition maps
are translations. That is, we show $S_{\mathcal{B}}$ is a flat surface.
It is worth noting that $S_{\mathcal{B}}$ depends only on 
the ordered Bratteli diagram, but the atlas also depends on the given state.

In section \ref{groupoids}, we pass from the various spaces of the previous section, 
to groupoids associated with them. While we use the term groupoid, these are
really simply equivalence relations. For the bi-infinite path space
$X_{\mathcal{B}}$ or its open subset $Y_{\mathcal{B}}$, we have 
right and left tail equivalence. For the surface, $S_{\mathcal{B}}$, we have
horizontal and vertical foliations.
The process of constructing
a $C^{*}$-algebra from a groupoid is technical; in particular, the groupoid 
requires its own topology. We describe all of these in quite concrete terms.
Finally, our maps between the various spaces of section 6 all induce
maps at the level of equivalence relations and we describe their properties.
Indeed, one of the quotient maps from $Y_{\mathcal{B}}$ 
does not respect tail equivalence in general and
 we are forced to make a small modification
of it  in Definition \ref{gpds:25}.

We turn to the $C^{*}$-algebras in section \ref{Cstar}. We explicitly show how 
the $C^{*}$-algebras of tail equivalence can be written as inductive limits
of a nested sequence of finite-dimensional subalgebras. In the case of 
one of 
the intermediate subalgebras, we also have an inductive system
 \ref{IntCstar:70}
and \ref{IntCstar:90} of subalgebras which are 'subhomogeneous'. 
That is, they involve only continuous functions from 
certain spaces into matrices.  The same also holds for the foliation algebra.

In section \ref{Fred}, we construct a very natural Fredholm module for our 
AF-algebra. The notion of a Fredholm module for $C^{*}$-algebras had its origins
in the seminal work of Brown, Douglas and Fillmore on extensions
of $C^{*}$-algebras but also from index theory through ideas of 
Atiyah and Kasparov, among many others. There are many good references but we mention
the three books by Blackadar  \cite{Bla:K}, Higson
and Roe \cite{HR:KHom} and Connes \cite{Con:NCG}. The prototype here 
is the $C^{*}$-algebra of continuous functions on a 
smooth manifold together with an elliptic
differential operator (or a bounded version of it). The algebra and operator 
interact in a special way.
In our situation, 
our Fredholm module provides a purely algebraic way of describing 
our quotient spaces (see Theorem \ref{Fred:60}). This description, in turn, 
is critical to some K-theory computations of the next section.

We describe the K-theory of the various $C^{*}$-algebras involved
in section \ref{K}, beginning with the AF-algebra. The computation
of the K-theory of an AF-algebra from a Bratteli diagram goes
back to Elliott's seminal paper \cite{Ell:AF}, but 
we give a treatment in some detail
for those readers for whom this is new. We also 
compute the K-theory of one intermediate $C^{*}$-algebra in
generality in  
Theorem \ref{K:30}. In many specific situations of interest, this
$C^{*}$-algebra has $K_{1}$ equal to the integers, while 
its inclusion in the AF-algebra induces an order isomorphism 
on $K_{0}$ (see Theorem \ref{K:38}).
We go on to compute the K-theory of the foliation algebra 
in Theorems \ref{K:40} and \ref{K:50}. One interesting conclusion
of these computations is that, when the Bratteli diagram has finite rank, the
$K_{0}$ group of the AF-algebra does also, in the sense of group theory. 
However, if
that group is not finitely-generated, then our surface has 
infinite genus (Corollary \ref{K:60}).

We end the paper with two sections in which we apply the tools 
developed so far: in section \ref{Cha}, we work out the $K$-theory 
of the horizontal foliation of Chamanara's surface. Chamanara's 
surface is perhaps the best known flat surface of infinite genus
 and finite area. In particular, we show
 how one can explicitly construct representatives of certain
  $K_0$ classes coming from the surface. We also show that the set of singular points
  $\Sigma_\mathcal{B}$ is non-empty and give an explicit identification of this set.

Section \ref{finitegenus} deals with flat surfaces of finite genus. Starting
 with basic definitions of flat surfaces, we review Veech's construction
  of zippered rectangles and Rauzy-Veech (RV) induction, which is a 
  procedure used in the renormalization of the vertical foliation of
   a flat surface. We follow this by developing an analogous induction 
   procedure for the horizontal foliation, which we call RH induction. 
   This is formally the inverse of RV induction, but we motivate it 
   geometrically and develop it independently of RV induction. As far
    as we know a lot of this has not been published before, although 
    many items appear in the recent work of Berk \cite{berk2021backward}.
     The reason we focus on the induction for the horizontal foliation 
     is that it turns out to give an ordered bi-infinite Bratteli diagram 
     which is simpler to analyse. We show that the set $\Sigma_\mathcal{B}$ of singular points
     for the typical surface is empty, implying that the singularities which
     appear in the infinite genus case appear precisely because these the set of
     singular $\Sigma_\mathcal{B}$ is non-empty. We compute the $K$-theory of the foliation 
     algebra of the horizontal foliation of the typical flat surface of 
     finite genus. We also show that the order structure on the $K_0$ groups
    is defined by the Schwartzman asymptotic cycle.

\textbf{A reader's guide:} This paper is long and contains ideas and concepts from the two 
fields that it aims to connect, operator algebras and Teichm\"uller dynamics.
Our recommendation for reading the paper is:
\begin{itemize}
\item \textbf{For the reader from operator algebras} we recommend a sequential reading of the paper. Such reader will be familiar with Bratteli diagrams, tail-equivalence relations, constructions derived from such relations, and $K$-theoretic computations. 
\item \textbf{For the reader from Teichm\"uller dynamics} we recommend first reading Sections \ref{BD}-\ref{surface} to understand the surface construction. The reader can then jump to more familiar territory either in the first half of Section \ref{Cha}, where we explain how Chamanara's surface obtained through our langauge, or Section \ref{finitegenus}, our treatment of finite-genus flat surfaces. Sections \ref{Cha} and \ref{finitegenus} end with $K$-theory computations of the foliation algebras, which will not fully make sense until the reader goes through the algebra constructions and $K$-theory computations of sections \ref{groupoids}-\ref{K}.
\end{itemize}

\vspace{.2in}
\textbf{Acknowledgements:} R.T. was partially supported by NSF grant 1665100 and Simons Collaboration Grant 712227. I.F.P. was supported by a Discovery Grant from
NSERC (Canada).

\section{Bratteli diagrams: ordered and bi-infinite}
\label{BD}
In this section, we discuss the notion of a Bratteli diagram.
This is a fairly well-known combinatorial object, but we 
will need to discuss a bi-infinite variation and also the 
notion of orders on both types.

\begin{defn}
\label{BD:5}
A \emph{ Bratteli diagram} is two sequences 
$V_{n}, n \geq 0 , E_{n}, n \geq 1,$ 
of pairwise disjoint, finite, nonempty sets
along with surjective maps $r: E_{n} \rightarrow V_{n}$ and 
$s:E_{n} \rightarrow V_{n-1}$ for $n \geq 1$. 
We also assume that $V_{0}$ consists of a single element we write as $v_{0}$.
 We write $V$ for the union of the $V_{n}$ and 
$E$ for the union of the $E_{n}$. We also write
$\mathcal{B} = (V, E, r, s)$.
\end{defn}

\begin{defn}
\label{BD:10}
A \emph{bi-infinite Bratteli diagram} is two sequences 
$V_{n}, E_{n}, n \in \Z,$ of pairwise disjoint, finite, nonempty sets
(dropping the requirement that $\# V_{0} =1$)
along with surjective maps $r: E_{n} \rightarrow V_{n}$ and 
$s:E_{n} \rightarrow V_{n-1}$. We write $V$ for the union of the $V_{n}$ and 
$E$ for the union of the $E_{n}$. We also write
$\mathcal{B} = (V, E, r, s)$.
\end{defn}

A standard convention when drawing  in drawing Bratteli diagrams 
is to draw them vertically, with $v_{0}$ at the top of the diagram
and level $V_{n+1}$ drawn below $V_{n}$. Here, we prefer to
draw them horizontally. That is, $V_{n+1}$ lies to the right of 
$V_{n}$, as shown below. For ordinary Bratteli diagrams, this change
is rather minor, but it seems helpful when considering 
bi-infinite ones, not to have to imagine 
the diagram extending off the top of 
page.

$$
\begin{tikzpicture}
  \filldraw (1,0) circle (4pt);
    \filldraw (1,3) circle (4pt);
     \draw (1,4) node {$V_{-1}$};
    
  \filldraw (4,0) circle (4pt);
    \filldraw (4,3) circle (4pt);
      \draw (4,4) node {$V_{0}$};
  \filldraw (7,0) circle (4pt);
    \filldraw (7,3) circle (4pt);
      \draw (7,4) node {$V_{1}$};
    \filldraw (10,0) circle (4pt);
    \filldraw (10,3) circle (4pt);
      \draw (10,4) node {$V_{2}$};
       \draw (0,0) -- (1,0);
    \draw (1,0) -- (4,0);
   \draw (4,0) -- (7,0);
     \draw (7,0) -- (10,0);
   \draw (10,0) -- (11,0);
     \draw (0,3) -- (1,3);
    \draw (1,3) -- (4,3);
   \draw (4,3) -- (7,3);
     \draw (7,3) -- (10,3);
   \draw (10,3) -- (11,3);
     \draw (0,1) -- (1,0);
    \draw (1,0) -- (4,3);
   \draw (4,3) -- (7,0);
     \draw (7,0) -- (10,3);
   \draw (10,3) -- (11,2);
   
 \end{tikzpicture}
$$


\begin{defn}
\label{BD:20}
If $\mathcal{B}$ is a bi-infinite Bratteli diagram, for every
pair of integers $m < n$, we let $E_{m,n}$ be the set of all 
paths from $V_{m}$ to $V_{n}$: that is, it consists of 
$p = (p_{i})_{m < i \leq n}$ with $p_{i}$ in $E_{i}$, $ m < i \leq n$, 
and $r(p_{i}) = s(p_{i+1})$, for $m < i < n$. We define 
$r:E_{m,n} \rightarrow V_{n}$ by 
$r(p) = r(p_{n})$ and $s:E_{m,n} \rightarrow V_{m}$ by 
$s(p) = s(p_{m+1})$. We make the same definition if 
$\mathcal{B}$ is a Bratteli diagram, restricting  to 
$0 \leq m < n$.
\end{defn}

We note the fairly standard notion of simplicity of a Bratteli diagram and
its obvious extension to the bi-infinite case.

\begin{defn}
\label{BD:25}
\begin{enumerate}
\item
A Bratteli diagram $\mathcal{B}$ is \emph{simple} if and 
only if, for every $m \geq 1$, there is $n > m $ such that 
for every vertex $v$ in $V_{m}$ and $w$ in $V_{n}$, there is a path
$p$ in $E_{m,n}$ with $s(p)=v, r(p)=w$.
\item 
A bi-infinite  Bratteli diagram $\mathcal{B}$ is
 $\mathcal{B}$ is \emph{simple} if,  
for every integer $m$, there are integers $l < m< n$ such that
there is a path from every vertex of $V_{l}$ to every vertex of
$V_{m}$ and there is a path from 
every vertex of $V_{m}$ to every vertex of
$V_{n}$.
\end{enumerate}
\end{defn}

We also introduce the following notion which will be 
convenient for much of what follows.

\begin{defn}
\label{BD:29}
A Bratteli diagram (or bi-infinite Bratteli diagram) is
\emph{finite rank} if there is a constant $K$ such that 
$\# V_{n} \leq K$, for all $n \geq 0$ (or all 
$n \in \Z$, respectively).
\end{defn}

We next discuss the notion of a state on a Bratteli diagram, 
and its analogue for the bi-infinite case.
We add as a small remark that it is usual to begin with a Bratteli diagram
and consider the set of all possible states on it. For 
our applications later, we will usually think of a Bratteli
diagram, together with a fixed state, as our data.

\begin{defn}
\label{BD:75}
\begin{enumerate}
\item
Let $\mathcal{B}$ be  a  Bratteli diagram. A \emph{state} on 
$\mathcal{B}$ is a function 
$\nu_{s}: V \rightarrow [0, \infty)$ which is not identically zero,
satisfying
\[
\nu_{s}(v) = \sum_{s(e) =v} \nu_{s}(r(e)), 
\]
for all $v$ in $V$.
We say that the state is \emph{normalized} if 
$\nu_{s}(v_{0}) = 1$ and \emph{faithful} if $\nu_{s}(v) > 0$, 
for all $v$ in $V$.
We let $S(\mathcal{B})$ be the set of all states on $\mathcal{B}$ and 
$S_{1}(\mathcal{B})$ denote the set of normalized states.
\item
Let $\mathcal{B}$ be  a  bi-infinite Bratteli diagram. A \emph{state} on 
$\mathcal{B}$ is a pair of functions 
$\nu_{r}, \nu_{s}: V \rightarrow [0, \infty)$, neither identically zero,
 satisfying
\begin{eqnarray*}
\nu_{r}(v) & =  & \sum_{r(e) =v} \nu_{r}(s(e)), \\
\nu_{s}(v) & =  & \sum_{s(e) =v} \nu_{s}(r(e)), 
\end{eqnarray*}
for all $v$ in $V$.
We say that the state is \emph{normalized}
if 
\[
\sum_{v \in V_{0}} \nu_{r}(v)\nu_{s}(v) = 1
\]
and is \emph{faithful}
 if $\nu_{r}(v), \nu_{s}(v) > 0$, 
for all $v$ in $V$.
We let $S(\mathcal{B})$ be the set of all  
states on $\mathcal{B}$ and 
$S_{1}(\mathcal{B})$ denote the set of normalized states.
\end{enumerate}
\end{defn}

\begin{lemma}
\label{BD:90}
If $\nu_{r}, \nu_{s}$ is a state on bi-infinite Bratteli diagram, 
$\mathcal{B}$, then 
\[
\sum_{v \in V_{n}} \nu_{r}(v)\nu_{s}(v)  = 
 \sum_{v \in V_{0}} \nu_{r}(v)\nu_{s}(v)
\]
for every integer $n$.
\end{lemma}

\begin{proof}
If $n$ is any integer, we have 
\begin{eqnarray*}
\sum_{v \in V_{n}} \nu_{r}(v)\nu_{s}(v) & = & 
\sum_{v \in V_{n}} \nu_{r}(v) \sum_{s(e) = v} \nu_{s}(r(e)) \\
 & = & 
\sum_{v \in V_{n}} \sum_{s(e) = v} \nu_{r}(v)  \nu_{s}(r(e)) \\
 & = & 
\sum_{e \in E_{n+1}}  \nu_{r}(s(e))  \nu_{s}(r(e)) \\
   & = & 
\sum_{v \in V_{n+1}} \sum_{r(e) = v} \nu_{r}(s(e))  \nu_{s}(v) \\
  &  =  &  \sum_{v \in V_{n+1}} \nu_{r}(v)\nu_{s}(v).
  \end{eqnarray*}
  The conclusion follows.
\end{proof}

The following result is not difficult
but quite useful in translating results from the standard case  to 
the bi-infinite case.
 The point is rather easy to state in words:
in a bi-infinite Bratteli diagram, 
for any fixed vertex $v$ in $V$, if we look at all vertices which can 
be reached from a path starting at $v$, and all the edges of such paths, 
this forms a Bratteli diagram in the usual sense. There is some re-indexing 
of vertex and edge sets. The same is true if we look at paths
\emph{ending} at $v$ instead, although the re-indexing is more complicated and 
we need to switch $s$ and $r$ maps.

\begin{prop}
\label{BD:100}
Let $\mathcal{B}$ be a bi-infinite 
Bratteli diagram, $(\nu_{r}, \nu_{s})$
a state on  $\mathcal{B}$ and 
$v$ be any vertex of $V$. 
\begin{enumerate}
\item 
Define $W_{0} = \{ v \}$ and then, inductively, 
for all $n \geq 1$,
$F_{n} = s^{-1}(W_{n-1})$, $ W_{n} = r(F_{n})$. Then 
$\mathcal{B}_{v}^{+} = (W, F, r, s)$ is a Bratteli 
diagram, the restriction of $\nu_{s}$ to $W_{n}$,  
which we denote $\nu_{s}^{v}$,  
is a state on it. 
\item 
Define $W_{0} = \{ v \}$ and then, inductively, 
for all $n \geq 1$,
$F_{n} = r^{-1}(W_{n-1})$, $ W_{n} = s(F_{n})$. Then 
$\mathcal{B}_{v}^{-} = (W, F, s, r)$ is a Bratteli 
diagram, the restriction of $\nu_{r}$ to $W_{n}$,  
which we denote $\nu_{r}^{v}$,  
 is a state on it. 
\end{enumerate}
\end{prop}


Let us remind the reader that the computation of the set of states 
for a one-sided Bratteli diagram is
a standard result, which can be easily adapted to the bi-infinite case.
 It is convenient to assume that 
$V_{n} = \{ n \} \times \{ 1, 2, \ldots, d_{n} \}$, 
for all integers $n$. Without causing 
confusion, we can interpret $E_{n}$ as a $d_{n} \times d_{n-1}$ non-negative
integer matrix whose $j,i$-entry is the number of edges in $E_{n}$ from 
$(n-1,i)$ in $V_{n-1}$ to $(n,j)$ in $V_{n}$. In the following, we let 
$\R^{+m}$ denote vectors in $\R^{m}, m \geq 1$ (written as row vectors),
 whose entries are all non-negative.

\begin{prop}
\label{BD:115}
Let $\mathcal{B}$ be a bi-infinite Bratteli diagram. 
If $\nu_{r}, \nu_{s}$ is a state
on $\mathcal{B}$, then for all integers $n$, we have 
\begin{eqnarray*}
\left( \nu_{r}(n,i)  \right)_{i=1}^{d_{n}} & \in &
 \displaystyle\bigcap_{m>n} \R^{+d_{m}} E_{m} E_{m-1} \cdots  E_{n+1},  \\
  \left( \nu_{s}(n,i) \right)_{i=1}^{d_{n}} & \in &
  \displaystyle\bigcap_{m < n} \R^{+d_{m}} E_{m+1}^{T} E_{m+2}^{T} \cdots  E_{n}^{T}.
\end{eqnarray*}
Conversely, letting $\Delta^{d-1}$ denote the standard $d-1$-simplex in $\R^{+d}$, 
the sets 
\begin{eqnarray*}
   \displaystyle\bigcap_{m>0} \R^{+d_{m}} E_{m} E_{m-1} \cdots 
    E_{1} 
   & \cap &    \Delta^{d_{0}-1} \\
   \displaystyle\bigcap_{m<0} \R^{+d_{m}} E_{m+1}^{T}  E_{m+2}^{T} 
   \cdots  E_{0}^{T}  
   & \cap &    \Delta^{d_{0}-1} 
 \end{eqnarray*}
 are non-empty. Let $x^{0}$ be in the former  and set 
 $\nu_{r}(0,i) = x_{i}^{0}, 1 \leq i \leq d_{0}$,
 $\nu_{r}(n,i) = (x E_{0} E_{-1} \cdots  E_{n+1})_{i}$, for $n < 0, 
 1 \leq i \leq d_{n}$. Finally, for each $n > 0$, inductively, there 
 is  $x^{n}$
 in $\R^{+d_{n}}$ with $x^{n-1}= x^{n}E_{n}$ and set 
  $\nu_{r}(n,i) = x_{i}^{n}, 1 \leq i \leq d_{n}$. We may also define
  $\nu_{s}$ in an analogous way and  $\nu_{s}, \nu_{r}$ is a state
  on $\mathcal{B}$.
\end{prop}

\begin{proof}
For any state $\nu_{r},\nu_{s}$ and any integer $n$, we
regard $ \left( \nu_{r}(n,i)  \right)_{i=1}^{d_{n}}$ and 
$\left( \nu_{s}(n,i)  \right)_{i=1}^{d_{n}} $ as  vectors
in $\R^{+d_{n}}$. The definition of state immediately implies
that 
\begin{eqnarray*}
\left( \nu_{r}(n,i)  \right)_{i=1}^{d_{n}} E_{n}& = &  
\left( \nu_{r}(n-1,i)  \right)_{i=1}^{d_{n-1}}, \\
\left( \nu_{s}(n-1,i)  \right)_{i=1}^{d_{n-1}} E_{n}^{T}& = &  
\left( \nu_{s}(n,i)  \right)_{i=1}^{d_{n}}.
\end{eqnarray*}
The first part of the conclusion follows immediately. (In fact, 
these equations are equivalent to the conditions on a state  given 
in part 2 of Definition  \ref{BD:75}.)

For the converse direction, it is easy to see from the fact that the 
matrices $E_{n}, n \in \Z$ are non-negative that the sets
$\R^{+d_{m}} E_{m} E_{m-1} \cdots    E_{1} , m > 1$ are closed and 
decreasing as $m $ increases. They are also invariant under
multiplication by positive scalars so their intersections with the simplex 
  $\Delta^{d_{0}-1}$ are compact, non-empty and decreasing 
  as $m$ increases. It follows that the intersection
  over all $m > 0$ is non-empty.
  
  It is a simple matter to check that, for any $n \geq 1 $, 
  the map sending $x$ in 
 $ \displaystyle\bigcap_{m>n} \R^{+d_{m}} E_{m} E_{m-1} \cdots  E_{n+1} $
 to $xE_{n}$ is a surjection to 
 $ \displaystyle\bigcap_{m>n} \R^{+d_{m}} E_{m} E_{m-1} \cdots  E_{n} $.
 It follows that the sequence $x^{n}, n \in \Z,$ 
 is well-defined and satisfies $x^{n}E_{n} = x^{n-1}$. This 
 implies that $\nu_{r}$ is a state. 
 
 The case for $\nu_{s}$ is done in  a similar way.
\end{proof}

\begin{prop}
\label{BD:110}
Let $\mathcal{B}$ be  a   Bratteli diagram 
 or a bi-infinite Bratteli diagram. 
\begin{enumerate}
\item $S(\mathcal{B})$ is non-empty.
\item If $\mathcal{B}$ is simple, then  every 
 state
 is faithful.
\end{enumerate}
\end{prop}

\begin{proof}
We prove the bi-infinite case. The other case is an 
easy consequence of that and Proposition \ref{BD:100}. 

The first part is a consequence  of  Proposition \ref{BD:115}.
For the second part, 
the following are easy consequences of the definition:

\noindent
(2a) If there is a vertex $v$ in $V_{n}$ such that 
$\nu_{r}(v) > 0$, then  there exists a vertex $w$ in $V_{n-1}$ 
such that $\nu_{r}(w) > 0$.

\noindent
(2b) If $m < n$ is such that there is a path 
from every vertex in $V_{m}$ to every vertex in $V_{n}$, and 
there is there is a vertex $v$ in $V_{m}$ such that 
$\nu_{r}(v) > 0$, then  for every  vertex $w$ in $V_{n-1}$, 
$\nu_{r}(w) > 0$.

Now let $n$ be an integer. By Lemma \ref{BD:90}, there is some 
$v$ in $V_{n}$ such that $\nu_{r}(v) > 0$. Next, choose
$m < n$ such that $E_{m,n}$ has full connections. 
It follows from the first point above that there exists 
$w$ in $V_{m}$ such that $\nu_{r}(w) > 0$. It 
then follows from the second point above that 
$\nu_{r}(v') > 0$, for all $v'$ in $V_{n}$. As 
$n$ was arbitrary, this completes the proof.
\end{proof}

We will ultimately be interested in \emph{ordered} 
bi-infinite Bratteli diagrams. We make the definition now, although
we will not make use of it until section \ref{order}.

\begin{defn}
\label{BD:40}
A \emph{bi-infinite, ordered Bratteli diagram} is a 
bi-infinite Bratteli diagram, $\mathcal{B}= (V, E, r, s)$, along 
with partial orders $\leq_{s}, \leq_{r}$ on $E$ such that, for 
any  
$e,f$ in $E$, 
they are $\leq_{s}$-comparable if and only if $s(e) = s(f)$, and 
 are $\leq_{r}$-comparable if and only if $r(e) = r(f)$.
 We write $\mathcal{B} = (V, E, r, s, \leq_{r}, \leq_{s})$.
\end{defn}

We adopt the following obvious notation: $e <_{r} f $ (respectively, 
$e <_{s} f $) if and only if 
$e \leq_{r} f$ (respectively, $e \leq_{s} f$) and $e \neq f$.

The definition of the orders
 can also be extended to $E_{m,n}$ using 
the lexicographic order carefully noting  that $\leq_{r}$ 
works right-to-left while $\leq_{s}$ 
works left-to-right. 

If $e$ is any edge in $E$, we let $S_{s}(e)$ be 
its $\leq_{s}$-successor, provided it exists. Similar, $P_{s}(e)$ 
denotes its $\leq_{s}$-predecessor. There
are analogous definitions of $S_{r}$ and $P_{r}$. 
These definitions also extend to $E_{m,n}, m < n$.

If $\mathcal{B}$
is a bi-infinite ordered Bratteli diagram, we say
an edge or finite path $e$ is $r$-maximal if it is 
maximal in the $\leq_{r}$ order. 
Analogous definitions exist for $r$-minimal, $s$-maximal and 
$s$-minimal.

\section{The path space}
\label{path}
In this section, we pass from combinatorics to
topology: to each Bratteli diagram we associate  a
 topological space, the path space along with 
 a topological equivalence relation, tail 
 equivalence. Of course, most of this is 
 well-known for standard Bratteli diagrams, so we 
 focus here
 on the bi-infinite case.

\begin{defn}
\label{path:10}
\begin{enumerate}
\item 
If $\mathcal{B}$ is a Bratteli diagram, 
we let  $X_{\mathcal{B}}$ be the space of 
infinite paths in $\mathcal{B}$: that is, an element 
of $X_{\mathcal{B}}$ is a sequence, $(x_{n})_{n \geq 1}$, where 
$x_{n}$ is in $E_{n}$ and $r(x_{n}) = s(x_{n+1})$,  for every 
positive  integer $n$.
\item 
If $\mathcal{B}$ is a bi-infinite Bratteli diagram, 
we let  $X_{\mathcal{B}}$ be the space of 
bi-infinite paths in $\mathcal{B}$: that is, an element 
of $X_{\mathcal{B}}$ is a sequence, $(x_{n})_{n \in \Z}$, where 
$x_{n}$ is in $E_{n}$ and $r(x_{n}) = s(x_{n+1})$,  for every 
integer $n$.
\end{enumerate} 
\end{defn}

We introduce some notation which is not
strictly necessary when dealing with one-sided Bratteli 
diagrams, but helps when dealing with bi-infinite 
ones.

First, if $v$ is any vertex in $V_{n}, n \in \Z$, we let 
$X^{+}_{v}$ be the set of all one-sided infinite 
paths $x= (x_{n+1}, x_{n+2}, \ldots )$ with $x_{i}$ in $E_{i}$,
for all $i  > n$,  and
 $s(x_{n+1})=v$. Observe that this coincides with the one-sided path
 space of $\mathcal{B}_{V}^{+}$, of Proposition \ref{BD:100}.
There is a similar definition for $X^{-}_{v}$  as 
one-sided infinite paths ending at $v$.

Secondly, if $x$ is any point in $X_{\mathcal{B}}$ and 
$m < n$, we let $x_{(m,n]} $ or $x_{[m+1,n]}$
denote $(x_{m+1}, \ldots, x_{n})$ which is in $E_{m,n}$. 
We also let $x_{(m, \infty)}$  or 
$x_{[m+1, \infty)}$  denote $(x_{m+1}, x_{m+2}, \ldots )$ and 
$x_{(-\infty, n]}$ or $x_{(-\infty,n+1)}$ denote 
$( \ldots, x_{n-1}, x_{n})$.
Observe that if $x$ is in $X_{\mathcal{B}}$, then $x_{[n, \infty)}$
is in $X_{s(x_{n})}^{+}$ while $x_{(- \infty, n]}$ is in 
$X_{r(x_{n})}^{-}$.

Thirdly, if $p$ is in $E_{l,m}$ and $q$ is in $E_{m,n}$ with
$r(p)=s(q)$, we let 
$pq$ denote their concatenation, which lies in $E_{l,n}$. 
In a similar way,  if $p$ is in $E_{m,n}$, $x$ is in $X_{r(p)}^{+}$
and $y$ is in $X_{s(p)}^{-}$, then $px$ is in $X_{s(p)}^{+}$, 
$yp$ is in $X_{r(p)}^{-}$ and $ypx$ is in $X_{\mathcal{B}}$.

Finally, we also use this concatenation notation
for sets of paths, rather than single elements. As an example, 
$pX_{r(p)}^{+}$ is the set of all $px$ with $x$ in $X_{r(p)}^{+}$.
Also, note that, for any vertex $v$ in $V_{n}$,  $X_{v}^{-}X_{v}^{+}$
is the set of all $x$ with $r(x_{n})= s(x_{n+1}) = v$.

We introduce the natural  topology on the path space,
 for both infinite and 
bi-infinite cases.

\begin{prop}
\label{path:20}
\begin{enumerate}
\item 
Let $\mathcal{B}$ be a Bratteli diagram. We  regard 
$X_{\mathcal{B}}$ as a subset of $\prod_{n=1}^{\infty} E_{n}$. 
Each $E_{n}$ is endowed with the discrete topology,
$\prod_{n=1}^{\infty} E_{n}$ with the product topology and 
$X_{\mathcal{B}}$ with the relative topology. In this topology, 
$X_{\mathcal{B}}$ is compact, metrizable and totally disconnected.
Moreover, if $p$ is any path in $E_{0,n}$, then the set 
\[
pX_{r(p)}^{+} = 
\{ x \in X_{\mathcal{B}} \mid x_{i} = p_{i}, 1 \leq i \leq n \}.
\]
is clopen and, as $p$ and $n$ vary, these form 
 a base for the topology of 
$X_{\mathcal{B}}$.
\item 
Let $\mathcal{B}$  be  a bi-infinite Bratteli diagram., 
 We  regard 
$X_{\mathcal{B}}$ as a subset of $\prod_{n \in \Z} E_{n}$. 
Each $E_{n}$ is endowed with the discrete topology,
$\prod_{n \in \Z} E_{n}$ with the product topology and 
$X_{\mathcal{B}}$ with the relative topology. In this topology, 
$X_{\mathcal{B}}$ is compact, metrizable and totally disconnected.
Moreover, if $p$ is any path in $E_{m,n}, m < n$,
 then the set
\[
X_{s(p)}^{-}pX_{r(p)}^{+} = \{ x \in X_{\mathcal{B}} \mid x_{i} = p_{i}, m < i \leq n \}.
\]
is  clopen and, as $m < n, p$ vary, these form a  base
 for the topology of 
$X_{\mathcal{B}}$.
\end{enumerate} 
\end{prop}

We remark that the path space $X_{\mathcal{B}}$ is a metric space 
(even an ultrametric space) with
the formula, for $x, y$ in $X_{\mathcal{B}}$,
\[
d(x,y) = \inf \{ 2^{-n} \mid n \geq 0, x_{i} = y_{i}, 1 \leq i \leq n \}
\]
in the one-sided case and 
\[
d(x,y) = \inf \{ 2^{-n} \mid n \geq 0, x_{i} = y_{i}, 1-n \leq i \leq n \}
\]
for the bi-infinite case.

Before going further, we want to look at the path spaces for simple
diagrams.
One of the difficulties of the definition of simplicity
is that it does not 
guarantee that the path space is infinite. This must be  allowed since the
$C^{*}$-algebra of $n \times n $-matrices is a simple AF-algebra, whose
associated Bratteli diagram has a finite path space. On the other hand, it
is often nice to rule out this case as not being terribly interesting.
This problem doubles for bi-infinite Bratteli diagrams. For the moment, 
we make a small useful observation.

\begin{thm}
\label{path:12}
Let $\mathcal{B}$ be a Bratteli diagram. It is simple and 
$X_{\mathcal{B}}$ is infinite if and only if,
 for every $m \geq 1$, there is $n > m $ such that 
for every vertex $v$ in $V_{m}$ and $w$ in $V_{n}$, there are at 
least two paths
$p, p'$ in $E_{m,n}$ with $s(p)=s(p')=v, r(p)=r(p')=w$.
\end{thm}

\begin{proof}
Let us first assume that $\mathcal{B}$ is simple and $X_{\mathcal{B}}$ 
is infinite. Fix $m  \geq 1$. From simplicity, we know there is $m' > m$ 
such that there is a path from every vertex in $V_{m}$ to every
 vertex in $V_{m'}$. If we consider all paths $p$ in $E_{0,m'}$, the sets
 $pX_{r(p)}$ form a finite cover of $X_{\mathcal{B}}$. As we assume this 
 space is infinite, there must exist $x \neq y$ which lie in the same 
 element. That is, there is $m'' > m'$ such that $x_{m''} \neq y_{m''}$.
 Using simplicity again, we find $n > m''$ such that there is a
  path from every
 vertex of $V_{m''}$ to $V_{n}$. It is now an easy matter to check that
 there are at least two paths from every vertex of $V_{m}$ to every vertex 
 of $V_{n}$, one that follows $x_{m'+1}, \ldots, x_{m''}$ and 
 one that follows
  $y_{m'+1}, \ldots, y_{m''}$.
  
  For the converse, the two-path condition
   obviously implies the diagram is simple. 
   It also implies that there are at least $2^{n}$ paths in $E_{0,n}$ and
    so $X_{\mathcal{B}}$ is infinite.
  \end{proof}

Let us also
note the following result for the bi-infinite case,
 which is an easy consequence
of the last result and Proposition \ref{BD:100}..

\begin{lemma}
\label{path:17}
Let $\mathcal{B}$ be a simple bi-infinite Bratteli diagram. The following are equivalent
\begin{enumerate}
\item Both  $X^{+}_{v}$ and
$X^{-}_{v}$ are infinite, for some $v$ in $V$. 
\item Both $X^{+}_{v}$ and
$X^{-}_{v}$   are  infinite, for all $v$ in $V$.
\item 
For every integer $m $, there are $l <  m < n $ such that 
for every vertex  $u$ in $V_{l}$, $v$ in $V_{m}$ and $w$ in $V_{n}$, there are 
at 
least two paths
$p, p'$ in $E_{l,m}$ with $s(p)=s(p')=u, r(p)=r(p')=v$ and 
at 
least two paths
$q, q'$ in $E_{m,n}$ with $s(q)=s(q')=v, r(q)=r(q')=w$.
\end{enumerate}
If any of these conditions hold, we say that $\mathcal{B}$ is
strongly simple.
\end{lemma}

\begin{defn}
\label{path:18}
We say that a bi-infinite Bratteli diagram $\mathcal{B}$  is 
\emph{strongly simple} if it is simple and  the conditions Lemma 
\ref{path:17} hold.
\end{defn}

We also need the notion of tail equivalence. As paths in 
the bi-infinite case have two tails, this becomes two 
equivalence relations.

\begin{defn}
\label{path:100}
\begin{enumerate}
\item
Let $\mathcal{B}$ be a Bratteli diagram. For each $x$ in 
$X_\mathcal{B}$, we let $T^{+}(x)$
be the set of paths which are right-tail equivalent to $x$.
 More precisely, for
$N \geq 0$, we define
\begin{eqnarray*}
T^{+}_{N}(x) & =  &  \{ px_{(N, \infty)} \mid p \in E_{0,N}, r(p) = r(x_{N}) \}
  \\
  & =   &
\{ y \in X_\mathcal{B} \mid y_{n} = x_{n}, \text{ for all } 
n > N \}
\end{eqnarray*}
and
$T^{+}(x) = \cup_{N \in \Z} T^{+}_{N}(x)$.

\item 
Let $\mathcal{B}$ be a bi-infinite 
Bratteli diagram. For each $x$ in $X_\mathcal{B}$,
 we define $T^{+}(x)$ ( $T^{-}(x)$ )
to be the set of all paths which are right-tail equivalent (left-tail 
equivalent, respectively) to $x$. More precisely, 
for $N$ in $\Z$, we define
\begin{eqnarray*}
T^{+}_{N}(x) & =  &  X_{r(x_{N})}^{-}x_{(N, \infty)} \\
  & =   &  \{ y \in X_\mathcal{B} \mid y_{n} = x_{n}, \text{ for all } 
n > N \} \\
 T^{+}(x) & = & \bigcup_{N \in \Z} T^{+}_{N}(x),
\end{eqnarray*}
and 
\begin{eqnarray*}
T^{-}_{N}(x) & =  & x_{(-\infty, N]}X^{+}_{r(x_{N})} \\
  &  =  &   \{ y \in X_\mathcal{B} \mid y_{n} = x_{n}, \text{ for all } 
n \leq N \} \\
T^{-}(x) & = & \bigcup_{N \in \Z} T^{-}_{N}(x).
\end{eqnarray*}
\end{enumerate}
Each set $T^{+}_{N}(x)$ is endowed with the relative topology from 
$X_{\mathcal{B}}$, while $T^{+}(x)$ is given the inductive limit topology.
We use $T^{+}(X_{\mathcal{B}})$ to denote the 
equivalence relation (or groupoid) on $X_{\mathcal{B}}$ whose equivalence 
classes are the sets $T^{+}(x), x \in X_{\mathcal{B}}$. 
There is an analogous relation
$T^{-}(X_{\mathcal{B}})$, but we will 
work mostly with $T^{+}(X_{\mathcal{B}})$.
\end{defn}

Let us recall that the inductive limit 
topology  on $T^{+}(x), x \in X_{\mathcal{B}}$, 
is the finest topology
which makes each inclusion $T^{+}_{N}(x) \subseteq T^{+}(x)$ continuous.
One can check quite easily that, for every $N$, $T^{+}_{N}(x)$ is an open 
subset of $T^{+}_{N+1}(x)$. In consequence, 
 a subset $U \subseteq T^{+}(x)$ is open in the inductive limit 
  topology if and only if 
$U \cap T^{+}_{N}(x)$ is open in $T^{+}_{N}(x)$, for every $N$.
We leave it as an instructive exercise for the reader to show
that a sequence $y_{n}, n \geq 1,$ in $T^{+}(x)$ converges to
$y$ in $T^{+}(x)$ in this topology if and only if it converges to $y$ 
in 
$X_{\mathcal{B}}$ and there exists some $N$ such that 
$y, y_{n}, n \geq 1,$ are all contained in $T^{+}_{N}(x)$.

In a standard Bratteli diagram, each tail
 equivalence class, $T_{N}^{+}(x)$, is finite
and each $T^{+}(x)$ is countable. This is not usually the case 
for bi-infinite diagrams. Instead, we must investigate the topology
on the tail equivalence classes.

\begin{prop}
\label{path:110}
Let $\mathcal{B}$ be a bi-infinite 
Bratteli diagram and  let $x$ be in $X_{\mathcal{B}}$. 
 For any path $p$ in $E_{m,n}$
with $r(p) = r(x_{n})$, the set
\[
X^{-}_{s(p)}px_{(n, \infty)} = \{ y \in X_{\mathcal{B}} \mid y_{i} = p_{i}, 
  m < i \leq n, y_{i} = x_{i}, i > n \}.
  \]
   is a compact open subset of 
  $T^{+}(x)$. Moreover, as $m, n, p$ vary, these sets form a base for
  the topology of $T^{+}(x)$. There is an analogous 
 statement for  $ x_{(-\infty, m]} p X_{r(p)}^{+}$ in 
 $T^{-}(x)$, for $p$ with
   $s(p) = s(x_{m+1})$.
\end{prop}

To this point, our discussion of the path spaces has not involved
the states in any way.
We now  see how states on the Bratteli diagram give rise
to measures on the path space. There are some subtleties
 in the bi-infinite case, but the first case is well-known. 
 We provide a sketch of the proof
for convenience.
 
\begin{prop}
\label{path:120}
Let $\mathcal{B}$ be a  Bratteli diagram and
 $\nu$ be a state on $\mathcal{B}$. 
 There is a unique  
 measure, also denoted $\nu$, on 
$X_{\mathcal{B}}$  such that 
 \[
 \nu(pX_{r(p)}^{+}) = \nu(r(p)), 
 \]
 for each $p$ in $E_{0,n}, n \geq 1$.
\end{prop} 

\begin{proof}
For each $n \geq 1$, let $C_{n}$ be the linear span of characteristic 
functions of sets $pX^{+}_{r(p)}$, where $p$ is in $E_{0,n}$, which we
denote
$\chi_{pX_{r(p)}^{+}}$.
The function $\nu_{n}:C_{n} \rightarrow \C$ defined as follows. If 
$f =  \sum_{p \in E_{0,n}} a_{p} \chi_{pX^{+}_{r(p)}}$, 
where $a_{p}$ is a complex scalar for each $p$ in $E_{0,n}$, then 
we define
\[
\nu_{n}( f) = 
\sum_{p \in E_{0,n}} a_{p}  \nu(r(p)).
\]
This is clearly a linear map and it is a simple matter to see that, 
with $f$ as above, 
\[
\vert \nu_{n}(f) \vert 
\leq \max\{ \vert a_{p} \vert  \mid p \in E_{0,n} \} 
\sum_{p \in E_{0,n}} a_{p}  \nu(r(p)) = \Vert f \Vert_{\infty} \nu(v_{0}). 
\]

Moreover, $C_{n}$ is a linear subspace of $C_{n+1}$, for all $n \geq 1$, 
and it is a consequence of the definition of a state that $\nu_{n+1}$ agrees
with $\nu_{n}$ on $C_{n}$ so the union of the $\nu_{n}$, which we 
also denote
$\nu$, defines
a linear map on the union.  The analogous norm inequality above
holds for all $f$ in the union. Hence, $\nu$ extends to a bounded linear
functional on the completion of the functions in the supremum norm.
It is a simple consequence of the Stone-Weierstrass Theorem 
(see V.8.1 of \cite{Conway:book}) that this completion 
is $C(X_{\mathcal{B}})$.
Finally, the Riesz Representation Theorem (III.5.7 of \cite{Conway:book}).  
\end{proof}

We want to establish properties of this measure. The following technical
result will be of use later.

\begin{lemma}
\label{path:122}
\begin{enumerate}
\item
Let $\mathcal{B}$ be a simple Bratteli diagram with $X_{v}$ 
infinite and let $\nu$ be a state on $\mathcal{B}$. Then we have 
\[
\lim_{n \rightarrow \infty} \max\{ \nu(v) \mid v \in V_{n} \} = 0.
\]
\item
Let $\mathcal{B}$ be a strongly simple bi-infinite 
Bratteli diagram and  $\nu_{s}, \nu_{r}$
 be a state on $\mathcal{B}$. Then we have 
\[
\lim_{n \rightarrow + \infty} \min\{ \nu_{r}(v) \mid v \in V_{n} \} = +\infty
\]
and 
\[
\lim_{n \rightarrow - \infty} \min\{ \nu_{s}(v) \mid v \in V_{n} \} = +\infty.
\]
\end{enumerate}
\end{lemma}

  \begin{proof}
  We begin with the first part.
 Let $n > 1$ and $v$ be any vertex of $V_{n}$. As we assumed the map
 $r$ is surjective, there is $e$ in $E_{n}$ with $r(e)=v$. Let $w= s(e)$
 so 
 \[
 \nu(v) = \nu(r(e)) \leq \sum_{s(f)=w}  \nu(r(f)) = \nu w) 
  \leq \max\{ \nu(v') \mid v' \in V_{n} \}.
  \]
  Taking the maximum over $v$ in $V_{n}$, we see the sequence
  we are considering is decreasing in $n$.
   Now fix an integer positive $m$. In view of  Theorem \ref{path:12},
   there is $n >m$ such that, for all $v$ in $V_{m}$ and $w$ in $V_{n}$, 
   there are at least two paths from $v$ to $w$. It follows from the 
   definition of state that $\nu(v) \geq 2 \nu(w)$, for all such
   $v,w$ and so  
   \[
   \max\{ \nu(v) \mid v \in V_{m} \} \geq 2 \max\{ \nu(v) \mid v \in V_{n} \}.
   \]
   The conclusion follows.
   
   For the second part, we first note that by Proposition \ref{BD:110},
   $\nu_{r}$ is faithful, so the minima are all strictly positive. A
    similar argument to the first case shows that 
   the sequence $\min \{ \nu_{r}(v) \mid v \in V_{n} \}$
    is increasing in $n$. Another minor variation of
    the remaining argument above shows that, for any $m \geq 1$, 
    there is $n >m$ such that    
     \[
  2 \min\{ \nu(v) \mid v \in V_{m} \} \leq  \min\{ \nu(v) \mid v \in V_{n} \}.
   \]
   The first of the two conclusions follows. The other is done in an analogous fashion.
  \end{proof}

\begin{prop}
\label{path:125}
Let $\mathcal{B}$ be a  Bratteli diagram and
 $\nu$ be a non-zero state on $\mathcal{B}$. 
 If $\mathcal{B}$ is simple, then the measure $\nu$ of Proposition
 \ref{path:120} has full support.
  If, in addition, $X_{\mathcal{B}}$ is infinite, 
 then   $\nu$ has no atoms.            
\end{prop} 

\begin{proof}
If $U$ is any non-empty open set, then there is $n \geq 1$ and a path $p$
in $E_{0,n}$ such that $p X_{r(p)}^{+} \subseteq U$ and 
$\nu( p X_{r(p)}^{+}) = \nu(r(p))$. As $\mathcal{B}$ is simple, 
$\nu$ is faithful (Proposition \ref{BD:110}), so $\nu(r(p)) > 0$.

For the second part, if $x$ in any point in $x$, for any $n \geq 1$, 
we have 
\[
\nu(\{ x \}) \leq \nu(x_{[1,n]}X_{r(x_{n})}^{+}) = \nu(r(x_{n})) 
\leq \max\{ \nu(v) \mid v \in V_{n} \}.
\]
The conclusion now follows from Lemma \ref{path:122}.
\end{proof}

If  $\mathcal{B} = (V, E, r, s)$ is a bi-infinite
Bratteli diagram, and $p$ is any finite path in 
$X_{\mathcal{B}}$, it is clear that
 $X_{s(p)}^{-} \times X_{r(p)}^{+}$ and $X_{s(p)}^{-} p X_{r(p)}^{+}$
are homeomorphic in an obvious way. We may apply 
Proposition \ref{path:120} to each 
of $ X_{\mathcal{B}_{r(p)}^{+}}$ and $X_{\mathcal{B}_{s(p)}^{-}}$
to obtain measures on $X_{s(p)}^{-}$ and $ X_{r(p)}^{+}$ and their
product can be regarded as a measure on  $X_{s(p)}^{-} p X_{r(p)}^{+}$
via the isomorphism above. It is an easy exercise to see that this collection
of measures agree where they overlap. This then proves the 
following analogue of Proposition \ref{path:120}
in the bi-infinite case. 

\begin{prop}
\label{path:130}
Let $\mathcal{B} = (V, E, r, s)$ be a bi-infinite
Bratteli diagram and 
suppose that $\nu_{s}, \nu_{r}: V \rightarrow \R$ is a state. 
There is a unique measure, which we denote by 
 $\nu_{r} \times \nu_{s}$ on $X_{\mathcal{B}}$
such that 
\[
\nu_{r} \times \nu_{s}(X_{s(p)}^{-}pX_{r(p)}^{+}) = 
\nu_{r}(s(p))\nu_{s}(r(p)), 
\]
for every $p$ in  $E_{m,n}$, with $m \leq n$.
If the state is faithful, then this measure has
 full support. If $\mathcal{B}$ is strongly simple, then
 this measure has no atoms.
\end{prop}

There remains one more class of measures to be defined
in the bi-infinite case:
on tail equivalence classes.

Let $\mathcal{B}$ be a bi-infinite Bratteli diagram and $x$ be any 
point in $X_{\mathcal{B}}$. For each $n$, we may consider the space
$T^{+}_{n}(x)= X_{r(x_{n})}^{-}x_{(n, \infty)}$ which is a compact 
open subset of $T^{+}(x)$. There is  
obvious homeomorphism  from  this space to 
$X_{s(x_{n})}^{-}$ and the measure  $\nu_{r}^{r(x_{n})}$ can be pulled
back to $T^{+}_{n}(x)$. 
It is a trivial computation to check that,
for any $m < n$,  the two measures obtained agree on 
$T^{+}_{m}(x) \subseteq T^{+}_{n}(x)$.
 The following is an immediate
  consequence of this and Proposition \ref{path:130}.

\begin{prop}
\label{path:140}
Let $\mathcal{B}$ be a bi-infinite Bratteli diagram and
 $\nu_{r}, \nu_{s}$ be a state on $\mathcal{B}$. For each 
 $x$ in $X_{\mathcal{B}}$, there is a measure $\nu^{x}_{r}$ on 
 $T^{+}(x)$ such that 
 \[
 \nu_{r}^{x}(X^{-}_{s(p)}px_{(n,\infty)}) = \nu_{r}(s(p_{m+1})), 
 \]
 for each $p$ in $E_{m,n}, m < n$ with $r(p)=r(x_{n})$.
 For $x,y$ in $\mathcal{B}$, if $T^{+}(x) = T^{+}(y)$,
 then $\nu_{r}^{x}= \nu_{r}^{y}$.
 There is also a measure $\nu^{x}_{s}$ on 
 $T^{-}(x)$ such that 
  \[
 \nu_{s}^{x}( x_{(-\infty, m]} pX^{+}_{r(p)}) = \nu_{s}(r(p_{n})), 
 \]
 for each $p$ in $E_{m,n}, m < n$ with $s(p)=s(x_{m})$.
  For $x,y$ in $\mathcal{B}$, if $T^{-}(x) = T^{-}(y)$,
 then $\nu_{s}^{x}= \nu_{s}^{y}$.
  If $\mathcal{B}$ is strongly simple, then
 these measures have
 full support and have no atoms.
\end{prop}

\section{Orders on the path space}
\label{order}

We defined orders for a bi-infinite diagram in Definition \ref{BD:40}.
We now see what effect these orders have on the infinite path space
of the last section.

The first result is a fairly standard one, adapted to the bi-infinite setting.

\begin{prop}
\label{order:5}
Every bi-infinite ordered Bratteli diagram, $\mathcal{B}$,
contains an infinite path such that every edge
is $r$-maximal ($r$-minimal, $s$-maximal or $s$-minimal).
We let $X^{r-max}_{\mathcal{B}} $
($X^{r-min}_{\mathcal{B}}, X^{s-max}_{\mathcal{B}}, X^{s-min}_{\mathcal{B}} $, 
respectively)
 denote the set of all such paths. We also let 
$X^{ext}_{\mathcal{B}} $ denote their union. Each of these sets is closed in
$X_{\mathcal{B}}$.

If $\mathcal{B}$ is finite rank
and 
 $K$ is a positive integer which bounds $\# V_{n}$, for every $n$ in $\Z$, then 
each of these sets has at most $K$ elements.
\end{prop}

\begin{proof}
The set of $r$-maximal edges in each vertex set, which we denote 
by $F_{n}$ for the moment, is a finite subset
of $E_{n}$. For a given positive integer $n$, the set of paths 
$x$ in $X_{\mathcal{B}}$ such that $x_{m}$ is $r$-maximal
 for all $-n \leq m \leq n$ is clearly closed. Intersecting these sets 
 over all values of $n$ produces $X^{r-max}_{\mathcal{B}} $, so 
 this is also closed. The same argument applies to the other sets.
 
 For the last statement, if $v$ is any vertex in $V_{n}$, 
 there is a unique $r$-maximal element $e$ of $E_{n}$ 
 with $r(e)=v$. As a consequence, if $x,y$ are in $X^{r-max}_{\mathcal{B}} $
 and $r(x_{m})=r(y_{m})$, for some $m$, then $x_{n}=y_{n}$, for all 
 $n < m$. If $x \neq y$, we may find $m(x,y)$ such that 
  $r(x_{m})=r(y_{m})$, not only for $m = m(x,y)$,
   but all $m \geq m(x,y)$ as well. If $X^{r-max}_{\mathcal{B}} $ contains
   $K+1$ distinct elements, say $x^{1}, \ldots, x^{K+1}$, then letting 
   $m $ be the minimum of $m(x^{i}, x^{j})$, over all $i \neq j$, the 
   function sending $x^{i}$ to $r(x^{i}_{m})$ is injective, contradicting
   our hypothesis. The other sets are done in a similar way.
\end{proof}

We start with some  fairly easy observations regarding 
ordinary (one-sided) Bratteli diagrams.
To motivate this, it is probably worth consider the 
standard ternary Cantor set in the real line. 

We consider the usual 
order inherited from $\R$ which is, of course, linear. 
In any linearly ordered set $X$, we say $y$ is the successor of $x$ 
if  $x < y$ and there is no $z$ with 
$x < z < y$. In this case, we also say 
that $x$ is the predecessor of $y$. 
 In the integers, every element has a successor 
 while in the real numbers, none does.
 In the Cantor ternary set, most points have neither a 
successor nor predecessor. 
The
 points  having a successor are exactly the left endpoints
of any open interval which is removed in the construction. The right endpoints
of these intervals are precisely the points with a predecessor.

In fact, these facts extend rather easily to the path space of 
an ordinary Bratteli diagram, equipped with an order, $\leq_{s}$. 
Before giving a general discussion, let us present a specific
example which fits nicely with the Cantor ternary set above. 
Consider the ordered Bratteli diagram 
\begin{eqnarray*}
V_{n} & = & \{ v_{n} \}, \\
E_{n} & = & \{ 0_{n}, 1_{n} \}, \\
0_{n} & \leq_{s} & 1_{n}
\end{eqnarray*}
for all positive  integers $n$. (We will examine the two-sided version of this in section 
\ref{Cha}.) There is standard way of encoding the Cantor
ternary set as  $0, 1$ sequences. We leave it to the reader to check the this bijection
between the path space of the Bratteli diagram and the Cantor ternary set preserves 
lexocographic order on the former
with the usual order on the latter.

We now consider the general case of an ordered Bratteli diagram.
Let $p$ be any finite path in an ordered Bratteli  diagram from $v_{0}$ 
to $V_{n-1}, n \geq 1$. Choose any edge $e_{n}$ with 
$s(e_{n}) = r(p)$ which is not maximal in the $\leq_{s}$ order. 
Let $f_{n}$ be its successor. 
Then, inductively for $i > n$, let $e_{i}$  be the greatest edge 
in the order $\leq_{s}$ with $s(e_{i}) = r(e_{i-1})$. Similarly,
inductively for $i > n$, let $f_{i}$  be the least edge 
in the order $\leq_{s}$ with $s(f_{i}) = r(f_{i-1})$.
Then the path $pf_{n}f_{n+1} \cdots$ is the successor
of $pe_{n}e_{n+1} \cdots$. In fact, all successor/predecessor 
pairs occur in this manner. We summarize the properties on the order
on the path space.

\begin{lemma}
\label{order:10}
Let $\mathcal{B}$ be a Bratteli diagram and assume that $\leq_{s}$ is 
an order on the edge set 
$E$ such that $e,f$ are comparable in $\leq_{s}$ if and only if $s(e) = s(f)$.
(Caution: the usual definition of an ordered Bratteli diagram 
uses $r(e) = r(f)$.)  We define  the (lexicographic) order
on $X_{\mathcal{B}}$ as follows: for $x, y$ in $X_{\mathcal{B}}$, we have 
$x <_{s} y$ if there is a positive integer $n$ such that
$x_{i}=y_{i}$, for all $ 1 \leq i < n$ and 
$x_{n} <_{s} y_{n}$.
\begin{enumerate}
\item 
 The relation $\leq_{s} $ on $X_{\mathcal{B}}$ is a linear order.
 \item For each $v$ in $V_{n}, n \geq 1$, there is a unique path, 
 denoted by $x_{v}^{s-max}$ in $X_{v}^{+}$ such that 
 $(x^{s-max}_{v})_{i}$ is maximal for every
 $i > n$. Moreover, if $p$ is in $E_{0n}$ with $r(p)=v$, then
 $px_{v}^{s-max}$ is the 
 greatest element of $pX_{v}^{+}$. 
 Similarly, there is a unique path, 
 denoted by $x_{v}^{s-min}$ in $X_{v}^{+}$ such that 
 $(x^{s-min}_{v})_{i}$ is minimal for every
 $i > n$. Moreover, if $p$ is in $E_{0n}$ with $r(p)=v$, then
  $px_{v}^{s-min}$ is the 
 least element of $pX_{v}^{+}$. 
 \item For $p$ in $E_{0,n}$ and $r(p)=v$, we have 
 $pX_{v}^{+} = \{ x \in X_{\mathcal{B}} \mid
  px_{v}^{s-min} \leq x \leq px_{v}^{s-max} \}$.
 \item 
 An element $x$ of $X_{\mathcal{B}}$ has a 
 successor in the order $\leq_{s}$ if and
 only if there is $n$ such that $x_{n}$  is not maximal and 
 $x_{(n,\infty)} = x_{r(x_{n})}^{s-max}$.
 Similarly, an element $x$ of $X_{\mathcal{B}}$ has a 
 predecessor in the order $\leq_{s}$ if and
 only if there is $n$ such that $x_{n}$  is not minimal and 
 $x_{(n,\infty)} = x_{r(x_{n})}^{s-min}$.
 \item The order topology from $\leq_{s}$ on $X_{\mathcal{B}}$ 
 coincides with the usual topology
 given in Proposition \ref{path:20}.
 \end{enumerate}
\end{lemma}

\begin{proof}
The proof of the first part is quite easy and we omit it except to remark that 
to see the order on the path space is linear, we need the condition
$V_{0}$ is a single vertex.

In the second part, the existence of the infinite paths
 easily follows from the fact that
for any vertex $v$, $s^{-1}\{ v \}$ is linearly ordered so it 
contains  unique $s$-maximal and $s$-minimal elements and a simple
induction argument. The properties of the paths 
$px_{r(p)}^{s-min}, px_{r(p)}^{s-max}$ are obvious from the definitions.

For the third part, if $y$ is in $pX_{v}^{+}$ and $i$ is the least integer
such that $y_{i} \neq (px_{v}^{s-min})_{i}$, then $i > n$ and so
$y \geq px_{v}^{s-min}$. Similarly $y \leq px_{v}^{s-max}$. 
Conversely, suppose $ px_{v}^{s-min} \leq y \leq px_{v}^{s-max}$.
The first
inequality implies $y_{1} \leq p_{1}$ while the second implies 
$y_{1} \leq p_{1}$. Together, these show $y_{1}=p_{1}$. 
Continuing in this way shows that
 $y_{[1,n]}=p$ which implies $y$ is in  $pX_{v}^{+}$.

 We next prove the first statement of part 4:
  suppose $x$ has the property stated, for some $n$.
 Let $y_{n}$ be the successor of $x_{n}$ in $\leq_{s}$, $y_{[1,n)}= x_{[1,n)}$
 $y_{(m,\infty)}= x_{r(y_{n})}^{s-min}$. We claim that there is no $z$ with 
 $x <_{s} z <_{s} y$, so that $y$ is the successor of $x$. First, an argument
 similar to the one of part three shows that $z_{[1,n)}= x_{[1,n)}=y_{[1,n)}$.
 Our hypothesis on $z$ then implies that $ x_{n} \leq_{s} z_{n} \leq_{s} y_{n}$
 and the choice of $y_{n}$ implies there is no $z_{n}$
  with both inequalities strict. Suppose $z_{n}=x_{n}$. 
  As $x_{i}$ is $s$-maximal for all $i > n$, it follows that $x >_{s} z$, a 
  contradiction. A similar argument shows that if $x$ has
  the other property stated, it has  a predecessor.
   
   We now prove part 5. We use the fact that the product topology
 is generated by cylinder sets, that is, sets of the form
 $pX_{r(p)}^{+}$, for some path $p$ in $E_{0,n}$, while the order
 topology is generated by open intervals 
 of the form $(y,z)$. First, if we consider
 such an open set $pX_{r(p)}^{+}$,
  let $z$ be the successor of $px_{r(p)}^{s-max}$ and $y$
 be the predecessor of $px_{r(p)}^{s-min}$. It follows from part 3 and the 
 one direction of part 4
that $pX^{+}_{r(p)} = [px_{r(p)}^{s-min},  px_{r(p)}^{s-max}]=(y,z)$.
On the other hand suppose that $(y, z)$ is a non-empty
open interval. Choose $x$ in $(y,z)$. Let $i$ be the least
positive  integer such that
$x_{i} \neq y_{i}$ and $j$ be the least positive integer such that 
$x_{j} \neq z_{j}$. Let $p=x_{1, \max\{ i, j\}]}$. It follows that 
$x \in pX_{r(p)}^{+} \subseteq (y,z)$. This completes the proof.
  
  Finally, we consider the converse direction of part 4.  
  If the condition stated fails, then there is
   a strictly increasing of positive integers $n_{i}$ such that 
   $x_{n_{i}}$ is not $s$-maximal. For each $i$, choose $y^{i}$ such that 
   $y_{[1,n_{i})} = x_{[1,n_{i})} $ and $y_{n_{i}} >_{s} x_{n_{i}}$. So 
   $y^{i} >_{s} x$, for all $i$, but converges to $x$. It follows, using 
   part 5,  that the open set $(x,y)$ is
    non-empty, for any $y >_{s} x$ so $x$ has no successor.
\end{proof}

This structure, as an ordered space, has a nice interaction
with states, as summarized below, at least in the case that 
the diagram is simple and $X_{\mathcal{B}}$ is infinite.

\begin{lemma}
\label{order:20}
Let $\mathcal{B}$ be a simple
 Bratteli diagram  with $X_{\mathcal{B}}$ infinite 
 and with an order $\leq_{s}$ as in \ref{order:10} and 
 faithful state $\nu$. 
  Define 
 $\varphi: X_{\mathcal{B}} \rightarrow [0,\nu(v_{0})]$ by 
\[
\varphi(x) = \nu \{ y \in X_{\mathcal{B}} \mid y \leq_{s} x \},
\]
for $x$ in $X_{\mathcal{B}}$, where $\nu$ is the measure defined in 
Proposition \ref{path:120}. 
The following hold.
\begin{enumerate}
\item $\varphi$  preserves order in the sense that $x \leq_{s} y$
implies $\varphi(x) \leq \varphi(y)$, for all $x, y$ 
in $X_{\mathcal{B}}$.
\item $\varphi$  is continuous.

\item For $ x \neq y$ in $X$, $\varphi(x) = \varphi(y)$ 
if and only if
$x,y$ are predecessor/successors of each other.
\item 
$\varphi$  is surjective.
\item If $\lambda$ denotes Lebesgue measure on $[0,\nu(v_{0})]$, then 
$\varphi_{*}(\nu) = \lambda$.
\end{enumerate}
\end{lemma}

\begin{proof}
The first property is clear. For the second, we observe that, for
any $x$ in $X_{\mathcal{B}}$ and $n \geq 1$, we have 
\[
\nu(x_{(0,n]}X_{r(x_{n})}^{+}) = \nu(r(x_{n}) )
\]
which tends to zero as $n$ goes to infinity by Lemma \ref{path:12}. 
It follows that $\nu(\{ x\}) = 0$, so $\nu$ has no atoms.
We also see that 
\[
\varphi(x_{(0,n]}x_{r(x_{n})}^{s-min}) \leq \varphi(x) 
\leq \varphi(x_{(0,n]}x_{r(x_{n})}^{s-max}) 
\]
and 
\begin{eqnarray*}
\varphi(x_{(0,n]}x_{r(x_{n})}^{s-max})  & = & 
\varphi(x_{(0,n]}x_{r(x_{n})}^{s-min})   \\
  &  &  + \nu( \{ y \mid 
\varphi(x_{(0,n]}x_{r(x_{n})}^{s-min} < \varphi(y) 
\leq \varphi(x_{(0,n]}x_{r(x_{n})}^{s-max} \} ) \\
  &  = & \varphi(x_{(0,n]}x_{r(x_{n})}^{s-min}) +
   \nu(x_{(0,n]}X_{r(x_{n})}^{+}) \\
    &  =  & \varphi(x_{(0,n]}x_{r(x_{n})}^{s-min})  + \nu(r(x_{n}).
    \end{eqnarray*}
    The continuity of $\varphi$ follows from these two estimates
    and the observation that $\nu(r(x_{n}))$ tends to zero
    as $n$ tends to infinity.
    
  We next suppose that $y$ is the successor of $x$ and show 
  $\varphi(x) = \varphi(y)$.  We know from the first part that 
  $\varphi(x) \leq \varphi(y)$. It follows from the 
  definitions that 
  \begin{eqnarray*}
  \varphi(y) - \varphi(x) & = & \nu \{ z \mid x < z \leq_{s} y \} \\
     &  =  &  \nu(\{y \}) \\
       & = & 0
       \end{eqnarray*}
       as  $\nu$ has no atoms. 
       Now suppose that $x \leq_{s} y$, but is not 
       the successor. There is $n \geq 1$ such that 
       $x_{i} = y_{i}, 1 \leq i  < n$ and $x_{n} < y_{n}$.
       From part 4 of Lemma \ref{order:10}, we know that
       there is some $m > n$ such that either $x_{m}$ is 
        not maximal or $y_{m}$ is not minimal. Let 
        us assume the former (the other case is similar). 
        Let $z_{m} $ be any edge with $s(z_{m}) = s(x_{m})$ and 
        $z_{m} <_{s} x_{m}$. If we let $p = x_{1} \ldots x_{m-1}z_{m}$, 
        it follows that
        \[
        x <_{s} pX_{r(p)}^{+} <_{s} y
        \]
        and so 
        \[
        \varphi(y) - \varphi(x) \geq \nu(pX_{r(p)}^{+} ) = \nu(r(z_{m})) > 0,
        \]since $\nu$ is faithful by Proposition \ref{BD:110}.

For the last part, it is clear that, for any path $p$ in $E_{0,n}$, we have
     \begin{eqnarray*}
     \nu(pX^{+}_{r(p)}) & = &  \nu(r(p))  \\
       & = &  \varphi(px^{s-max}_{r(p)}) - 
     \varphi(px^{s-min}_{r(p)}) \\
        &  =  &  \lambda(\varphi(px^{s-min}_{r(p)}), \varphi(px^{s-max}_{r(p)})) \\
          &  =  &  \lambda( \varphi( pX^{+}_{r(p)}))
          \end{eqnarray*}
          so $\nu$ and $\varphi^{*}(\lambda)$ agree on all sets 
          of the form $pX^{+}_{r(p)}$ and as these are a base for 
          the topology, they are equal.
\end{proof}

Probably it is worth noting that in the standard Cantor ternary 
set (and the correct choice of measure $\nu$), the
function $\varphi$ is the Devil's staircase, or more precisely, its
restriction to the Cantor set.

We are going to extend this notion of order to  the bi-infinite case,
as follows.

\begin{defn}
\label{order:40}
Let $\mathcal{B}$ be a strongly simple bi-infinite ordered
Bratteli diagram.  We define   orders $\leq_{s}, \leq_{r}$
on $X_{\mathcal{B}}$ as follows.
\begin{enumerate}
\item 
 for $x, y$ in $X_{\mathcal{B}}$, we have 
$x <_{r} y$ if  there is an integer $n$ such that
$x_{i}=y_{i}$, for all $i > n$ and 
$x_{n} <_{r} y_{n}$. 
For any $x,y$ in $X_{\mathcal{B}}$, we define
\[
[x,y]_{r} = \{ z \in X_{\mathcal{B}} \mid x \leq_{r} z \leq_{r} y \}
\]
and $(x,y)_{r}$ similarly.
\item 
 for $x, y$ in $X_{\mathcal{B}}$, we have 
$x <_{s} y$ if there is an integer $n$ such that
$x_{i}=y_{i}$, for all $i < n$ and 
$x_{n} <_{s} y_{n}$. 
For any $x,y$ in $X_{\mathcal{B}}$, we define
\[
[x,y]_{s} = \{ z \in X_{\mathcal{B}} \mid x \leq_{s} z \leq_{s} y \}
\]
and $(x,y)_{s}$ similarly.
\end{enumerate}
\end{defn}

\begin{lemma}
\label{order:50}
The following properties hold.
\begin{enumerate}
\item For $x,y$ in $X_{\mathcal{B}}$, they are comparable
in $\leq_{r}$ if and only if $T^{+}(x) = T^{+}(y)$. In particular, 
$\leq_{r}$ is a linear order on each  tail equivalence
 class $T^{+}(x)$.
 \item For $x,y$ in $X_{\mathcal{B}}$, they are comparable
in $\leq_{s}$ if and only if $T^{-}(x) = T^{-}(y)$. In particular, 
$\leq_{s}$ is a linear order on each  tail equivalence
 class $T^{-}(x)$.
 \item For $x$ in $X_{\mathcal{B}}$,
  $T^{+}(x) \cap ( X_{\mathcal{B}}^{r-max} \cup X_{\mathcal{B}}^{r-min})$
  is at most a single point.
   \item For $x$ in $X_{\mathcal{B}}$,
  $T^{-}(x) \cap ( X_{\mathcal{B}}^{s-max} \cup X_{\mathcal{B}}^{s-min})$
  is at most a single point.
 \item For each $v$ in $V_{n}$, there is a unique path, 
 denoted by $x_{v}^{s-max}$ (and $x_{v}^{s-min}$) in $X_{v}^{+}$ such that 
 $(x^{s-max}_{v})_{i}$ is maximal (minimal, respectively) 
 for every
 $i > n$. 
   Moreover, if $x$ is 
  in $X_{\mathcal{B}}$ and 
  $p$ is in $E_{m,n}$ with $s(p)=s(x_{m})$, then 
 \[
 x_{(-\infty, m)} pX_{r(p)}^{+}= [x_{(- \infty, m)} px_{r(p)}^{s-min}, 
  x_{(- \infty, m)} p x_{r(p)}^{s-max} ]_{s}.
  \]
  \item For each $v$ in $V_{n}$, there is a unique path, 
 denoted by $x_{v}^{r-max}$ (and $x_{v}^{r-min}$) in $X_{v}^{-}$ such that 
 $(x^{r-max}_{v})_{i}$ is maximal (minimal, respectively) 
 for every
 $i \leq  n$. 
 Moreover, if $x$ is 
  in $X_{\mathcal{B}}$ and 
  $p$ is in $E_{m,n}$ with $r(p)=r(x_{n})$, then 
 \[
 X_{s(p)}^{-}px_{(n,\infty)} = [ x_{s(p)}^{r-min}px_{(n, \infty)},
   x_{s(p)}^{r-max}px_{(n, \infty)} ]_{r}.
  \] 
  \item 
 An element $x$ of $X_{\mathcal{B}}$ has a 
 successor in the order $\leq_{r}$ if and
 only if there is $m$ such that $x_{m}$  is not $r$-maximal and 
 $x_{(-\infty, m)} = x_{s(x_{m})}^{r-max}$.
 Similarly, an element $x$ of $X_{\mathcal{B}}$ has a 
 predecessor in the order $\leq_{r}$ if and
 only if there is $m$ such that $x_{m}$  is not $r$-minimal and 
 $x_{(-\infty, m)} = x_{s(x_{m})}^{r-min}$.
 \item 
 An element $x$ of $X_{\mathcal{B}}$ has a 
 successor in the order $\leq_{s}$ if and
 only if there is $n$ such that $x_{n}$  is not $s$-maximal and 
 $x_{(n, \infty)} = x_{r(x_{n})}^{s-max}$.
 Similarly, an element $x$ of $X_{\mathcal{B}}$ has a 
 predecessor in the order $\leq_{s}$ if and
 only if there is $n$ such that $x_{n}$  is not $s$-minimal and 
 $x_{(n, \infty)} = x_{r(x_{n})}^{s-min}$.
 \end{enumerate}
\end{lemma}

\begin{proof}
The first two parts follow at once from the definitions. 

For the third, as $\leq_{r}$ is linear on $T^{+}(x)$, it can contain
at most one element of $X_{\mathcal{B}}^{r-max}$ and one element
of  $X_{\mathcal{B}}^{r-min}$. It remains to prove it cannot 
contain one from each, say $y$ and $z$ respectively.
If so, there is some $n_{0}$ such that $y_{n} = z_{n}$, for all
$n \geq n_{0}$. For each $n \geq n_{0}$ the path
 $y_{[n_{0},n]}= z_{[n_{0},n]}$
is both $r$-maximal and $r$-minimal, implying that there is only one 
path from  $s(y_{n_{0}})$ to $r(y_{n})$. As this holds
for all such $n$, it contradicts the assumption 
that $\mathcal{B}$ is strongly simple.

The remaining parts of the 
proof follows from Proposition \ref{BD:100} and Lemma  
\ref{order:10}.
\end{proof}

The last two parts of this result regarding successors and predecessors
in the two orders are important enough to warrant
the following definition.

\begin{defn}
\label{order:52}
Let $\mathcal{B}$ be a strongly simple
 bi-infinite ordered Bratteli 
diagram.
\begin{enumerate}
\item Let $\partial_{r} X_{\mathcal{B}}$ be the set of all points $x$ 
  which have either a successor or predecessor in the order $\leq_{r}$.
    Part 5 of Lemma \ref{order:50} characterizes such points and
 obviously, the  $m$ involved is unique and we denote it by $m(x)$.
 If $x$ has a successor in $\leq_{r}$,
  we denote it by $S_{r}(x)$, while its 
  predecessor is denoted by $P_{r}(x)$, if it exists.
 For such an $x$, we
  denote by $\Delta_{r}(x)$ either the $\leq_{r}$-successor or 
  $\leq_{r}$-predecessor
 of $x$, noting that it cannot have both. We regard 
 $\Delta_{r}:\partial_{r} X_{\mathcal{B}}  \rightarrow 
 \partial_{r} X_{\mathcal{B}}$ such that
  $\Delta_{r} \circ \Delta_{r}$ is the identity.
  \item 
  Let $\partial_{s} X_{\mathcal{B}}$ be the set of all points $x$
  which have either a successor or predecessor in the order $\leq_{s}$. 
 Part 6 of Lemma \ref{order:50} characterizes such points and
 obviously, the  $n$ involved is unique and we denote it by $n(x)$.
  If $x$ has a successor in $\leq_{s}$,
  we denote it by $S_{s}(x)$, while its 
  predecessor is denoted by $P_{s}(x)$, if it exists. 
 For such an $x$, we
  denote by $\Delta_{s}(x)$ either the $\leq_{s}$-successor or 
  $\leq_{s}$-predecessor
 of $x$, noting that it cannot have both. We regard 
 $\Delta_{s}:\partial_{s} X_{\mathcal{B}}  \rightarrow 
 \partial_{s} X_{\mathcal{B}}$ such that
  $\Delta_{s} \circ \Delta_{s}$ is the identity. 
  \end{enumerate}
\end{defn}

Notice that $( X_{\mathcal{B}}^{r-max} \cup X_{\mathcal{B}}^{r-min} )
 \cap \partial_{r}X_{\mathcal{B}}$ is necessarily empty, as 
 is $( X_{\mathcal{B}}^{s-max} \cup X_{\mathcal{B}}^{s-min} ) 
 \cap \partial_{s}X_{\mathcal{B}}$.

The following result is rather trivial, but probably worth observing.

\begin{lemma}
\label{order:55}
Let $\mathcal{B}$ be a bi-infinite ordered
Bratteli diagram, $(\nu_{r}, \nu_{s})$
a state on  $\mathcal{B}$ and 
$v$ be any vertex of $V$. 
On the Bratteli diagram $\mathcal{B}^{+}_{v}$ (or $\mathcal{B}^{-}_{v}$)
of Proposition \ref{BD:100}, $\leq_{s}$ ($\leq_{r}$, respectively)
is an order satisfying the 
conditions of Lemma \ref{order:10}.
\end{lemma}

Lemma \ref{order:20} considered a one-sided $\leq_{s}$-ordered 
Bratteli diagram and showed how a state, $\nu$ provided a natural
map from the path space to the real line. It had a number 
of good features,  but perhaps the nicest is part 3: it identifies
two points if and only if they are predecessor/successor
in the other.
Our next task is an analogue of this lemma
for bi-infinite ordered diagrams. In fact, there are two versions to
consider. Each defines its own function: they are closely related, but
the domains are different, so it is important to distinguish them.

\begin{defn}
\label{order:60}
Let $\mathcal{B}$ be a strongly simple bi-infinite ordered
Bratteli diagram with state $(\nu_{r}, \nu_{s})$.
 For any $v$ in $V_{n}$, 
 we  define
$\varphi^{v}_{r}: X^{-}_{v} \rightarrow [0, \nu_{r}(v)]$
by 
\[
\varphi^{v}_{r}(x) = \nu_{r}\{ y \in X_{v}^{-} \mid y \leq_{r} x  \},
\]
for $x$ in $X^{-}_{v}$. (By $\nu_{r}$ we mean the measure defined by
Proposition \ref{path:100} applied to the 
state $\nu_{r}$ of Proposition \ref{BD:100} which is the restriction
of $\nu_{r}$ to the diagram $\mathcal{B}_{v}^{-}$.)
Also, we define
$\varphi^{v}_{s}: X^{+}_{v} \rightarrow [0, \nu_{s}(v)]$
by 
\[
\varphi^{v}_{s}(x) = \nu_{s}\{ y \in X_{v}^{+} \mid y \leq_{s} x  \},
\]
for $x$ in $X^{+}_{v}$. 
\end{defn}

These two functions satisfy the conclusion of Lemma \ref{order:20}
with a few obvious adjustments. The one which is worth noting
is property 4 states that $\varphi^{v}_{s}(x) = \varphi^{v}_{s}(y)$ if 
and only if $x, y$ are predecessor/successors in the $\leq_{s}$ order
  while $\varphi^{v}_{r}(x) = \varphi^{v}_{r}(y)$ if 
and only if $x, y$ are predecessor/successors in the $\leq_{r}$ order.

It will be very useful for us to compare these functions, for
different vertices, in the following sense. 

\begin{lemma}
\label{order:70}
Let $p$ be in $E_{m,n}, m < n$.
\begin{enumerate}
\item For each $x$ in $X_{s(p)}^{-}$, we have
\[
 \varphi^{r(p)}_{r}(xp)  = \varphi^{s(p)}_{r}(x)  +
\varphi^{r(p)}_{r}(x^{r-min}_{s(p)}p)
\] 
\item For each $x$ in $X_{r(p)}^{+}$, we have
\[
 \varphi^{s(p)}_{s}(px) = \varphi^{r(p)}_{s}(x) +
\varphi^{s(p)}_{s}(px^{s-max}_{r(p)}).
\]\end{enumerate}
\end{lemma}

\begin{proof}
The first follows from the facts that $x$ in $X_{s(p)}^{-}$
$\{ y \in X^{-}_{r(p)} \mid y \leq_{r} xp \} $ is the disjoint union of
$\{ y \in X^{-}_{r(p)} \mid y \leq_{r} x_{s(p)}^{r-min}p \} $
and $\{ zp \mid z \in X^{-}_{s(p)}, z \leq_{r} x \}$ and the value of
$\nu_{r}^{r(p)}$  on the latter agrees with 
$\nu_{s(p)}^{r}\{  z \in X^{-}_{s(p)} \mid z \leq_{r} x \}$.
The second part is similar.
\end{proof}

Now we turn to the second, defining analogous
 maps to those of Lemma \ref{order:20} on entire tail-equivalence classes.
We restrict our attention to right-tail-equivalence.

\begin{lemma}
\label{order:200}
Let $\mathcal{B}$ be a strongly simple bi-infinite ordered
Bratteli diagram  with state
 let $\nu_{s}, \nu_{r}$. 
For each $x$ in $X_{\mathcal{B}}$, we define
$\varphi_{r}^{x}: T^{+}(x) \rightarrow \R$, by 
\[
\varphi_{r}^{x}(y) = \left\{ \begin{array}{cl} 
   \nu^{x}_{r}\{ z \in T^{+}(x) \mid x \leq_{r} z \leq_{r} y \}, & 
   x \leq_{r} y \\
   - \nu^{x}_{r}\{ z \in T^{+}(x) \mid y \leq_{r} z \leq_{r} x \}, & 
   y \leq_{r} x \end{array} \right.
\]
where $\nu_{r}^{x}$ is defined in Proposition \ref{path:140}.
There is an analogous definition of
$\varphi_{s}^{x}: T^{-}(x) \rightarrow \R$
The following hold.
\begin{enumerate}
\item 
  For any
   $y$ in $T^{+}(x)$, we have 
   $\varphi_{r}^{y} = \varphi_{r}^{x} - \varphi_{r}^{x}(y)$.
  \item 
  $\varphi_{r}^{x}$ preserves  order.
\item If $T^{+}(x)$ is given the topology of Definition \ref{path:100}, 
then 
$\varphi^{x}_{r}$ is continuous.
 \item For $y \neq z$ in $T^{+}(x)$, 
  $\varphi_{r}^{x}(y) = \varphi_{r}^{x}(z)$ 
  if and only if $y,z$ are predecessor/successors
  of each other in $\leq_{r}$.
\item If $T^{+}(x)$ is given the topology of Definition \ref{path:100}, 
then $\varphi_{r}^{x}$ is proper.
\item Exactly one of three possibilities hold:
\begin{enumerate}
\item $T^{+}(x) \cap X_{\mathcal{B}}^{r-max} = 
T^{+}(x) \cap X_{\mathcal{B}}^{r-min}  = \emptyset$ and in this case
$\varphi_{r}^{x}(T^{+}(x)) = \R$,

\item $T^{+}(x) \cap X_{\mathcal{B}}^{r-max} = \{ y \}, 
T^{+}(x) \cap X_{\mathcal{B}}^{r-min}  = \emptyset$ and in this case
$\varphi_{r}^{x}(T^{+}(x)) = (- \infty, \varphi_{r}^{x}(y)]$, 

\item $T^{+}(x) \cap X_{\mathcal{B}}^{r-max}  = \emptyset, 
T^{+}(x) \cap X_{\mathcal{B}}^{r-min}  =   \{ z \}$ and in this case
$\varphi_{r}^{x}(T^{+}(x)) = [ \varphi_{r}^{x}(z), \infty)$
\end{enumerate} 
   \item 
   If $\lambda$ denotes Lebesgue measure on 
   $\varphi_{r}^{x}(T^{+}(x))$, then 
   $(\varphi_{r}^{x})_{*}(\nu_{r}^{x}) = \lambda$.
\end{enumerate}
\end{lemma}

\begin{proof}
 This first property follows from  the definition.
 
 The definition of $\nu^{x}_{r}$ is given in terms 
 of its restriction to the sets $T_{N}^{+}$, 
 for various values of $N$. Furthermore,  Lemma \ref{order:50} applies
 to these restrictions, so  
 the second, third and  fourth   parts follow immediately.
 It follows from the first part and the fourth part of \ref{order:20} that 
 \[
 \varphi_{r}^{x}(T^{+}_{N}(x)) = 
 [ \varphi_{r}^{x}(x_{r(x_{N})}^{r-min}x_{(N, \infty)}), 
 \varphi_{r}^{x}(x_{r(x_{N})}^{r-max}x_{(N, \infty)})) ]
 \]
 which is an interval of length $\nu_{r}(r(x_{N})$.
 
 In part 6, the fact that these are the only 
 three possibilities follows from 
 part 3 of Lemma \ref{order:50}. We must
  prove the range of $\varphi^{x}_{r}$
 is a claimed.
 If $x_{N}$ is not r-maximal, it is an easy exercise to check that
 \[
 \varphi_{r}^{x}(x_{r(x_{N})}^{r-max}x_{(N, \infty)})
  - \varphi_{r}^{x}(x_{r(x_{N-1})}^{r-max}x_{(N-1, \infty)}) \geq 
   \nu_{r}(r(x_{N})).
   \]
   Similarly, if $x_{n}$ is not r-minimal, then
    \[
 \varphi_{r}^{x}(x_{r(x_{N-1})}^{r-min}x_{(N, \infty)})
  - \varphi_{r}^{x}(x_{r(x_{N})}^{r-min}x_{(N-1, \infty)}) \leq 
 -   \nu_{r}(r(x_{N})).
   \]
   
Our hypotheses and Proposition \ref{path:122} shows that
\[
\lim_{N \rightarrow \infty } \min \{ \nu_{r}(v) \mid v \in V_{N} \}
 = \infty.
\]
Conclusions five and six follow easily from these observations and 
results from \ref{path:20}.

The last statement follows from  the last part of
Lemma \ref{order:20}.
\end{proof}

\section{Singular points}
\label{singular}
We are now ready to begin the journey from the 
infinite path space of an ordered bi-infinite Bratteli 
diagram, $\mathcal{B}$, together with a state, $\nu_{s}, \nu_{r}$,
to the surface $S_{\mathcal{B}}$. 

The basic idea is an extremely simple one: to make  a 
quotient space from the path space
$X_{\mathcal{B}}$ by identifying $x$ with $\Delta_{s}(x)$, for
all $x$ in $\partial_{s}X_{\mathcal{B}}$ and 
$y$ with $\Delta_{r}(y)$, for
all $y$ in $\partial_{r}X_{\mathcal{B}}$. We can already see in Lemma \ref{order:60}
that this works quite well, at least locally, and that our functions
$\varphi^{v}_{s}, \varphi^{v}_{r}$ provide an explicit homeomorphism
between the quotient space and a Euclidean one.
But there are a number of subtleties to deal with.
Ultimately, it is necessary pass to 
a distinguished subset, $Y_{\mathcal{B}}$, of $X_{\mathcal{B}}$.
This can already be seen to be necessary since $X_{\mathcal{B}}$
is compact, while our surface will not be. In fact, there two types
of points which need to be removed. The first, which might be 
called \emph{extremal} with respect to the ordering are fairly obvious
and we have seen these already in Proposition \ref{order:5}.
The second type, which we call \emph{singular}, are more subtle.
The main objective of this section is to identify these points
precisely and discuss some of their properties.

For this section, we assume that 
$\mathcal{B}$ is a strongly simple
 ordered bi-infinite Bratteli diagram with 
state $\nu_{r}, \nu_{s}$.
 
Recall the definitions of $X^{ext}_{\mathcal{B}}, X^{s-max}_{\mathcal{B}}, 
 X^{s-min}_{\mathcal{B}}, X^{r-max}_{\mathcal{B}}, X^{r-min}_{\mathcal{B}}$
 given in
Proposition \ref{order:5}. These will be removed from $X_{\mathcal{B}}$ 
simply because our maps $\Delta_{s}, \Delta_{r}$ are not defined 
on them (in general).
Also recall that in Definition \ref{order:52}, 
$\partial_{s}X_{\mathcal{B}}, \partial_{r}X_{\mathcal{B}}$, respectively, 
are defined to exclude $X^{ext}_{\mathcal{B}}$.

As we are going to take a quotient by identifying points under \emph{both}
$\Delta_{s}$ and $\Delta_{r}$,  we need some compatibility between these maps.
In short, we require that they commute when both are defined.

As we have seen above, $\Delta_{s}(x)$ will be left-tail equivalent to $x$
and we have even given a name to the least integer where they differ: $n(x)$.
Similarly, the greatest integer where $x$ and $\Delta_{r}(x)$ differ is 
called $m(x)$. If we are to compute $\Delta_{r}\circ \Delta_{s}(x)$ 
(assuming for the moment it is defined), one of two 
rather distinct things happens. If $m(x)< n(x)$, the computation
of $\Delta_{s}(x)$ changes no entry, $x_{n}$, with $ n < n(x)$. It follows that 
$n(\Delta_{s}(x)) = n(x)$. Moreover,
 the computation of $\Delta_{r}(\Delta_{s}(x))$
is pretty much the same as that of $\Delta_{r}(x)$.

The following picture should prove helpful:

$$\begin{tikzpicture}
  \filldraw (3,4) circle (2pt);
    \filldraw (3,2) circle (2pt);
    \filldraw (4,3) circle (2pt);
      \filldraw (7,3) circle (2pt);
      \filldraw (8,4) circle (2pt);
      \filldraw (8,2) circle (2pt);
 
  \draw (3.5,1) node {$m(x)$};
   \draw (7.5,1) node {$n(x)$};

     \draw (1,4) -- (3,4);
       \draw (1,2) -- (3,2);
         \draw (3,4) -- (4,3);
         \draw (3,2) -- (4,3); 
          \draw (4,3) -- (7,3);
          \draw (7,3) -- (8,4);
          \draw (7,3) -- (8,2);      
                \draw (8,4) -- (10,4);
                \draw (8,2) -- (10,2);
  \end{tikzpicture}$$

One can actually see four different paths here: $x, \Delta_{s}(x),
\Delta_{r}(x)$ and $\Delta_{r}(\Delta_{s}(x))$.
  The important conclusion one draws is that 
$  \Delta_{r}\circ \Delta_{s}(x) = \Delta_{s}\circ \Delta_{r}(x)$.

Of course, there is a second possibility when $n (x) \leq m(x)$, summarized by 
the following picture:

$$\begin{tikzpicture}
  \filldraw (3,4) circle (2pt);
    \filldraw (3,2) circle (2pt);
   
      \filldraw (8,4) circle (2pt);
      \filldraw (8,2) circle (2pt);
 
  \draw (3.5,1) node {$n(x)$};
   \draw (7.5,1) node {$m(x)$};

     \draw (1,4) -- (3,4);
       \draw (1,2) -- (3,2);
         \draw (3,4) -- (8,2);
         \draw (3,4) -- (8,4); 
          \draw (3,2) -- (8,4);     
                \draw (8,4) -- (10,4);
                \draw (8,2) -- (10,2);
  \end{tikzpicture}$$
 which shows the paths 
$x, \Delta_{s}(x)$
and
$\Delta_{r}(x)$. The issue now becomes whether or not
 $\Delta_{r}\circ \Delta_{s}(x) = \Delta_{s}\circ \Delta_{r}(x)$. 
 It is possible
 but there is no reason that it must occur. At this point, 
 the reader may wish to take a 
 look at the example in section \ref{Cha}.
 
 Let us take a moment to discuss why the equation 
 $\Delta_{r}\circ \Delta_{s}(x) = \Delta_{s}\circ \Delta_{r}(x)$
 is important. If one thinks back to the example of the Cantor ternary set,
  identifying successor/predecessor pairs produces a closed interval.
 One can think of the two points which are identified as a 'left coordinate'
 and a 'right coordinate' of the point.  Passing to a bi-infinite diagram, 
 we will realize our quotient space in $\R^{2}$:  the left tail  provides the 
 $x$-coordinate and the right, the $y$-coordinate. Some points
 will have two coordinates in both $x$ and $y$ directions. What our formula is
 designed to capture is the notion that if we move
  horizontally first and then vertically
 we should get the same as moving vertically first and then horizontally. 
 If we do not (as we suggest above), then this tells us that the space is not
 'flat' at such a point. 
 
 We now develop these ideas more precisely.

\begin{defn}
\label{singular:50}
If $\mathcal{B}$ is a strongly simple bi-infinite ordered Bratteli diagram, 
we define 
$\partial X_{\mathcal{B}} = \partial_{s}X_{\mathcal{B}} \cap
 \partial_{r}X_{\mathcal{B}}$
 and 
\[
\Sigma_{\mathcal{B}} = 
\{ x \in \partial X_{\mathcal{B}} \mid \Delta_{s} \circ \Delta_{r} (x) 
\neq \Delta_{r} \circ \Delta_{s} (x)  \}.
\]
\end{defn}

\begin{prop}
\label{singular:80}
We have 
$\Delta_{s}(\partial X_{\mathcal{B}}) = 
   \partial X_{\mathcal{B}},    \,
\Delta_{r}(\partial X_{\mathcal{B}} )=
   \partial X_{\mathcal{B}} $ and 
$\Delta_{s}(\Sigma_{\mathcal{B}} ) 
 = \Sigma_{\mathcal{B}}   =
 \Delta_{r}(\Sigma_{\mathcal{B}}).$
\end{prop}

\begin{proof}
The first two equalities are already noted in 
in Definition \ref{order:52}.
We prove the second equality of the last statement. 
Assume $x$ is not in $\Sigma_{\mathcal{B}}$ so that 
$ \Delta_{s} \circ \Delta_{r} (x) 
= \Delta_{r} \circ \Delta_{s} (x)$. We have 
\begin{eqnarray*}
\Delta_{s} \circ \Delta_{r} (\Delta_{r}(x)) & = & 
\Delta_{s} \circ \Delta_{r} \circ \Delta_{r}(x) \\
    & = & 
\Delta_{s}(x) \\
 & = & 
\Delta_{r} \circ \Delta_{r} \circ \Delta_{s}(x) \\
 & = & 
\Delta_{r} \circ \Delta_{s} \circ \Delta_{r}(x) \\
 & = & 
\Delta_{r} \circ \Delta_{s}(\Delta_{r}(x)) 
\end{eqnarray*}
implying that $\Delta_{r}(x)$ is also not in $ \Sigma_{\mathcal{B}}$.
\end{proof}

We now give a proper written proof of what was shown by 
our first diagram above.

\begin{lemma}
\label{singular:90}
Let $x$ be in $\partial X_{\mathcal{B}}$.
 If $m(x) < n(x)$, then $x$ is not in $ \Sigma_{\mathcal{B}}$.
\end{lemma}

\begin{proof}
It is clear that $\Delta_{r}(x)_{i} = x_{i}$, whenever $i>m(x)$. It follows 
that $n(\Delta_{r}(x)) = n(x)$ and that $\Delta_{s} \circ \Delta_{r}(x)_{i} = \Delta_{s}(x)_{i}$
for all $ i > m(x)$. It also follows from the definition of $\Delta_{s}$ that 
$\Delta_{s} \circ \Delta_{r}(x)_{i} = \Delta_{r}(x)_{i}$, for all $ i < n(x)$.

The same argument shows that $\Delta_{s}(x)_{i} = x_{i}$, whenever $i < n(x)$
and that 
$\Delta_{r} \circ \Delta_{s}(x)_{i} = \Delta_{r}(x)_{i}$
for all $ i < n(x)$. It also follows from the definition of $\Delta_{r}$ that 
$\Delta_{r} \circ \Delta_{s}(x)_{i} = \Delta_{s}(x)_{i}$, for all $ i > m(x)$.

Combining the first fact with the fourth,
 if $i > m(x)$, we have 
\[
\Delta_{s} \circ \Delta_{r}(x)_{i} = \Delta_{s}(x)_{i} = \Delta_{r} \circ \Delta_{s}(x)_{i}.
\]
Combining the second fact with the third,
 if $i < n(x)$, we have 
\[
\Delta_{s} \circ \Delta_{r}(x)_{i} = \Delta_{r}(x)_{i} = \Delta_{r} \circ \Delta_{s}(x)_{i}.
\]
As every $i$ satisfies either $i > m(x)$ or $i < n(x)$, we conclude
that 
\[
\Delta_{s} \circ \Delta_{r}(x) =\Delta_{r} \circ \Delta_{s}(x).
\]
\end{proof}

The set $ \Sigma_{\mathcal{B}}$ plays an important part in what follows
and it will be useful to establish some simple facts about it.

\begin{lemma}
\label{singular:100}
Define  functions
 $\epsilon_{r}, \epsilon_{s}: \partial X_{\mathcal{B}} \rightarrow E $
 by 
\begin{eqnarray*}
\epsilon_{r}(x) &  =  &  x_{m(x)}, \\ 
\epsilon_{s}(x) &  =  &  x_{n(x)}.
\end{eqnarray*}
The function $\epsilon_{r} \times \epsilon_{s}:
 \partial X_{\mathcal{B}} \rightarrow E  \times E$ is finite-to-one.
  In particular, $\partial X_{\mathcal{B}}$ is a countable
subset of $X$.

The restriction of $\epsilon_{r}$ to 
  $ \Sigma_{\mathcal{B}}$ is at most four-to-one.
 The only possible limit points of $ \Sigma_{\mathcal{B}}$ are in 
$X^{ext}_{\mathcal{B}}$.
\end{lemma}

\begin{proof}
By definition, for a given $x$ in
 $\partial X_{\mathcal{B}} =  
 \partial_{s}  X_{\mathcal{B}}\cap \partial_{r} X_{\mathcal{B}}$, 
 there are exactly four possibilities.
  One of them is 
that
for all $i < m(x)$, $x_{i}$ is $\leq_{r}$-minimal and
 for all
$i > n(x)$, $x_{i}$ is $\leq_{s}$-maximal. 
The other three are obtained by replacing one, other or both
'maximal' by 'minimal'.
It follows then by a simple induction argument that $x_{m(x)}$
uniquely determines  $x_{i}$ for all $ i \leq m(x)$. Similarly, 
$x_{n(x)}$ uniquely determines $x_{i}$ for all $ i \geq n(x)$. 
Finally, there are only finitely many paths from 
$r(x_{m(x)}) $ to $s(x_{n(x)}) $, when $m(x) < n(x)$. 

If, in addition, $x$ is in $ \Sigma_{\mathcal{B}}$, 
then  we know  from 
Lemma \ref{singular:90} that  $m(x) \geq n(x)$. 
Hence, $x$ is determined uniquely by $x_{m(x)}$.

If $x^{k}, k \geq 1$ is any sequence in $\Sigma^{S}_{\mathcal{B}}$,
let us assume each term
 satisfies the first of the four possibilities above.
If, in addition, the points are all distinct, then the values of 
$m(x^{k})$ are distinct, for $k \geq 1$. We may then assume that
they are converging to $+\infty$. It is simple to check that any 
limit point of this sequence is contained
 in $X_{\mathcal{B}}^{r-min}$.
\end{proof}

We complete this section with a very useful technical result
on how the map
 $\Delta_{s}$ preserves the order $\leq_{r}$ and the measures
$\nu^{x}_{r}$.

\begin{prop}
\label{singular:150}
Let $x \leq_{r} y$ be in
 $X_{\mathcal{B}} \cap \partial_{s}X_{\mathcal{B}}$ such that 
$[x,y]_{r} $ is disjoint from  $
\Sigma_{\mathcal{B}} \cup X^{ext}_{\mathcal{B}}$.
Then 
\[
\Delta_{s}([x,y]_{r} ) = [\Delta_{s}(x), \Delta_{s}(y)]_{r}.
\]
and the restriction of $\Delta_{s}$ to $[x,y]_{r} $
preserves $\leq_{r}$. 

Moreover, we have
$\nu^{\Delta_{s}(x)}_{r}(\Delta_{s}(E)) = \nu^{x}_{r}(E)$, 
for every Borel set $E \subseteq [x,y]_{r}$.
\end{prop}

\begin{proof}
We will assume  that $x_{n}, y_{n}$ are $s$-maximal for all
sufficiently large $n$; the other case is similar.
 Choose $ n(x), n(y) < n$ such that 
$x_{(n, \infty)} = y_{(n, \infty)}$.

Let $m < n(x)  , n(y)$ and define $P_{m}$ to be all
 paths $p$ in $E_{m,n}$ such that $r(p) = r(x_{n})= r(y_{n})$,
 $x_{[m,n]} \leq_{r} p \leq_{r} y_{[m,n]}$ and $p$ is all $s$-maximal 
 edges. Observe that if $p$ is in $P_{m}$, then $p_{[m+1,n]}$ is in
 $P_{m+1}$  (if $ m+1 < n(x), n(y)$). If $P_{m}$
 is non-empty for all $m < n(x), n(y)$, a 
 standard compactness argument shows
 that we can find $z$ in $X_{r(x_{n})}^{-}$ consisting entirely 
 of $s$-maximal edges and satisfying $x \leq_{r} zx_{(n, \infty)} \leq y$.
 This contradicts the condition that $[x,y]_{r}$ is disjoint from 
 $X^{ext}_{\mathcal{B}}$. Hence, there exists $m < n(x), n(y)$ 
 such that $P_{m}$ is empty.
 
 We fix such an $m$. The elements $p$
  of $E_{m,n}$ with $r(p) = r(x_{n})$ and 
  $x_{[m,n]} \leq_{r} p \leq_{r}  y_{[m,n]}$ are linearly ordered
 by $\leq_{r}$
 and we list them as 
 \[
 x_{[m,n]} = p^{0} <_{r} p^{1} <_{r} \cdots <_{r} p^{k}= y_{[m,n]}.
 \]
 Using our choice of $m$, we 
 let $q^{i} = S_{s}(p^{i})$ for $0 \leq i \leq k$.
 
For each $0 \leq i \leq n$, the set $ X_{s(p^{i})}^{-}p^{i}x_{(n, \infty)}$ 
is linearly ordered by $\leq_{r}$. Moreover, the order is determined 
by the entries less than $m$. On the other hand, applying $\Delta_{s}$ 
affects only the entries greater than or equal to $m$. This implies
that $\Delta_{s}$ preserves order on each of these sets. 
We also note that
\[
\Delta_{s}(X_{s(p^{i})}^{-}p^{i}x_{(n, \infty)} ) =  X_{s(p^{i})}^{-}q^{i}x_{r(q^{i})}^{s-min}.
\]
 
 For $0 \leq i \leq k$, it is clear that  
 $ x^{i} = x_{s(p^{i})}^{r-max}p^{i}x_{(n, \infty)} $ is the largest 
 element of $ X_{s(p^{i})}^{-}p^{i}x_{(n, \infty)}$ in $\leq_{r}$.
  It is also
   in $[x,y]_{r}$ as well as 
 $\partial_{s}X _{\mathcal{B}}$ and  $\partial_{r}X _{\mathcal{B}}$.

Using our hypothesis that $[x,y_{r}$ is disjoint
from $\Sigma_{\mathcal{B}}$, we can now compute, for $0 \leq i < k$,
\begin{eqnarray*}
S_{s}(x^{i}) & <_{r} & S_{r} \circ S_{s}(x^{i})  \\
  & =  &   S_{s} \circ S_{r}(x^{i})  \\
    &  =  &  S_{s}( x_{s(p^{i+1})}^{r-min}p^{i+1}x_{r(p)}^{r-max} ) \\
    & = &  x_{s(p^{i+1})}^{r-min}q^{i+1}x_{r(q^{i+1})}^{r-min}  
\end{eqnarray*}
which is the least element of 
$\Delta_{s}(X_{s(p^{i+1})}^{-}p^{i+1}x_{(n, \infty)} )$.
It follows that $\Delta_{s}$ preserves order  when applied to all of 
$[x,y]_{r}$.

For the last statement, let $z$ be any point of $[x,y]_{r}$ and $k$
be any integer less than or equal to $m$. 
 We consider the set
$E= X^{-}_{s(z_{k+1})}z_{(k, \infty)}$. This is a clopen subset
of $[x,y]_{r}$ and such subsets are a base
 for its topology, so it suffices to
prove the statement for this set. By Proposition \ref{path:140}, 
we have $\nu^{x}_{r}(E) = \nu_{r}(s(z_{k+1})$.

As $k \leq m$, we have 
$\Delta_{s}(E)= X^{-}_{s(z_{k+1})}\Delta_{s}(z)_{(k, \infty)}$
$\nu^{x}_{r}(\Delta_{s}(E)) = \nu_{r}(s(z_{k+1})$ also.
\end{proof}

\section{The surface}
\label{surface}

Having identified extremal points and singular points
in the last section, the goal of this section is to pass from the 
infinite path space of a bi-infinite ordered Bratteli
diagram, $X_{\mathcal{B}}$, to its associated
surface, which we will denote by 
$S_{\mathcal{B}}$. Moreover, if we are given a state on the Bratteli 
diagram, we will construct an explicit 
system of charts for this space which shows that it 
is a translation surface.

There are a number of intermediary steps.
First, we must remove both extremal and singular  
points from $X_{\mathcal{B}}$. Then, we must identify points 
$x$ and $\Delta_{r}(x)$ and also $x$ and $\Delta_{s}(x)$. 
These two identifications commute precisely because we have 
removed the singular points. However, if we simply do the first
identifications, we obtain an intermediate space, which we
denote by $S^{r}_{\mathcal{B}}$. Doing the 
other identification first results in 
$S^{s}_{\mathcal{B}}$.

\begin{defn}
\label{surface:10}
Let $ \mathcal{B}$ be a 
bi-infinite ordered Bratteli diagram,
We define 
\[
Y_{ \mathcal{B}} = X_{\mathcal{B}} -
 X^{ext}_{\mathcal{B}} - \Sigma_{\mathcal{B}}.
\]

For $m < n$, we define $E^{Y}_{m,n}$ to be those $p$ 
in $E_{m,n}$ which are neither s-maximal, s-minimal, r-maximal nor
r-minimal and for which $X^{-}_{s(p)}pX^{+}_{r(p)}$ is contained in 
$Y_{\mathcal{B}}$.
\end{defn}

\begin{rmk}
\label{surface:11}
If $\mathcal{B}$ is finite rank, 
then the set $X_{\mathcal{B}}^{ext}$ if finite and 
$X^{ext}_{\mathcal{B}} \cup \Sigma_{\mathcal{B}}$ is countable and closed, by 
Lemma \ref{singular:100}.
Hence, $Y_{\mathcal{B}}$ is 
an open set in $X_{\mathcal{B}}$. 

The surfaces we construct will be quotient of
$Y_{\mathcal{B}}$. Of course, we cannot use $X_{\mathcal{B}}$ since
translation 
surfaces are not generally compact. As we will see later, 
it is interesting
that the finite genus case will be done
with $\Sigma_{\mathcal{B}}$ empty. Its appearance 
is essential in the infinite genus case.
\end{rmk}

Further to this, let us observe if $p$ is in $E_{m,n}^{Y}$ and  $e$ is in 
$E_{m}$ with $r(e) = s(p)$ then $ep$ is in  $E_{m-1,n}^{Y}$; if $f$ is 
in $E_{n+1}$ with $s(f)= r(p)$, then $pf$ is in  $E_{m,n+1}^{Y}$. 
Let us also show that the sets 
$X^{-}_{s(p)}pX^{+}_{r(p)}, p \in \cup_{m < n} E_{m,n}^{Y}$ form
 an open cover of $Y_{\mathcal{B}}$.
If $x$ is in $Y_{\mathcal{B}}$, then it must have edges which are 
not $s$-maximal, not $s$-minimal, not  $r$-maximal  and  not $r$-minimal. 
Select $m'< n'$ so that the path $p =x_{[m',n']}$ contains one of each.
In addition, as $x$ is in $Y_{\mathcal{B}}$ which is open, we may find 
$m < m' < n' < n$ such that 
$X^{-}_{s(x_{m})}x_{(m,n]}X^{+}_{r(x_{n})} \subseteq Y_{\mathcal{B}}$.
It follows that $x_{(m,n]}$ is in $E^{Y}_{m,n}$.

The next result is quite easy and  will be useful later on.

\begin{prop}
\label{surface:12}
Let $\mathcal{B}$ be a finite rank, strongly simple bi-infinite ordered
Bratteli diagram. There exists $m < 0$ such that 
$r: E^{Y}_{m,0} \rightarrow V_{0}$  is surjective. In fact, 
$r: E^{Y}_{m-n,n} \rightarrow V_{n}$ is surjective for all $n \geq 0$. 
\end{prop}

\begin{proof}
Let $K$ be a bound on the size of the vertex sets.
As there is a unique $s$-maximal path and a unique $s$-minimal 
from each vertex in $V_{l}$, $l < 0$, the total number of such paths in 
$E_{l,0}$ is $2K$. As our diagram is strongly simple, we may choose $l<0$
such that there are more than $2K$ paths from $V_{l}$ to each vertex 
of $V_{0}$. Now choose $m < l$ such that there are at 
least three paths between each vertex of $V_{m}$ and each vertex in $V_{l}$.

Let $v$ be in $V_{0}$. Choose $p$ in $E_{l,0}$ with $r(p)=v$
which is neither $s$-maximal nor $s$-minimal. 
Next choose $q$ in $E_{m,l}$ with $r(q) = s(p)$ 
which is not $r$-maximal nor $r$-minimal. If 
$x$ is any path in $X_{s(q)}^{-}qpX_{v}^{+}$, it is clear that it is not
in $X_{\mathcal{B}}^{ext}$. In addition, if $x$ is 
in $\partial X_{\mathcal{B}}$, then
$n(x) \geq l$ while $m(x) < l$. By Lemma \ref{singular:90}, $x$ is
not in $\Sigma_{\mathcal{B}}$. So $qp$ is in $E_{m,0}^{Y}$ with $r(qp)=v$.
\end{proof}

There is one more property which we will require of $Y_{\mathcal{B}}$: 
it should be invariant under both
$\Delta_{r}$ and $\Delta_{s}$.

This will follow from  the assumptions that  
$X^{ext}_{\mathcal{B}} \cap \partial_{s}X_{\mathcal{B}}$ and 
$ X^{ext}_{\mathcal{B}} \cap \partial_{r}X_{\mathcal{B}}$ are empty.
In fact, the set $\partial_{s}X_{\mathcal{B}}$ is defined to be 
disjoint from $X^{s-max}_{\mathcal{B}}$ and   $X^{s-min}_{\mathcal{B}}$, but 
as the $\leq_{s}$ and $\leq_{r}$ orders are essentially independent, 
there is no reason the same should be true of 
$X^{r-max}_{\mathcal{B}}$ and   $X^{r-min}_{\mathcal{B}}$.
As a concrete example, consider the bi-infinite Bratteli diagram, $\mathcal{B}$,
 of section \ref{Cha}
with each vertex set having one vertex and each edge set, $V_{n}$,
 with two edges, $0_{n}, 1_{n}$ and set $0_{n} <_{s} 1_{n}$. Take any bi-infinite path
 is $\partial_{s}X_{\mathcal{B}}$ and define the $\leq_{r}$ so that this path is $r$-maximal.

 \begin{prop}
 \label{surface:18}
 Let $\mathcal{B}$ be a bi-infinite ordered Bratteli diagram. 
 If 
 $X^{ext}_{\mathcal{B}} \cap \partial_{s}X_{\mathcal{B}}$ and 
$ X^{ext}_{\mathcal{B}} \cap \partial_{r}X_{\mathcal{B}}$ are empty, then 
\begin{eqnarray*}
\Delta_{s}(\partial_{s}X_{\mathcal{B}} \cap Y_{\mathcal{B}})  & = & 
\partial_{s}X_{\mathcal{B}} \cap Y_{\mathcal{B}} \\ 
\Delta_{r}(\partial_{r}X_{\mathcal{B}} \cap Y_{\mathcal{B}})  & = & 
\partial_{r}X_{\mathcal{B}} \cap Y_{\mathcal{B}}. 
\end{eqnarray*}
 \end{prop}
 
 \begin{proof}
 Let $x$ be in $\partial_{s}X_{\mathcal{B}} \cap Y_{\mathcal{B}}$. 
 It follows that $\Delta_{s}(x)$ is in $\partial_{s}X_{\mathcal{B}}$.
 As $x$ is in $Y_{\mathcal{B}}$ it is not in $\Sigma_{\mathcal{B}}$ so 
 $\Delta_{s}(x) $ is not in $\Sigma_{\mathcal{B}}$ by 
 Proposition \ref{singular:80}. 
 Also, as it is in $\partial_{s}X_{\mathcal{B}}$, it is not in 
 $X_{\mathcal{B}}^{ext}$. Hence it is in $Y_{\mathcal{B}}$.
 
 The reverse containment follows from the fact that $\Delta_{s}^{2}$ is the identity.
 The second part is done in a similar way.
 \end{proof}

It will be convenient to collect the hypotheses we need 
for most of the remainder of the paper.

\begin{defn}
\label{surface:20}
We say that a bi-infinite ordered Bratteli diagram, $\mathcal{B}$,
 satisfies the
\emph{standing assumptions} if
\begin{enumerate}
\item $\mathcal{B}$ is finite rank (Definition \ref{BD:29}),
\item $\mathcal{B}$ is strongly simple (Definition \ref{path:18}),
\item  $X^{ext}_{\mathcal{B}} \cap \partial_{s}X_{\mathcal{B}}$ and 
$ X^{ext}_{\mathcal{B}} \cap \partial_{r}X_{\mathcal{B}}$ are empty.
\end{enumerate}
\end{defn}

We are going to make various quotient spaces from
$Y_{\mathcal{B}}$ by making identifications
of $x$ and $\Delta_{r}(x)$ and $y$ with $\Delta_{s}(y)$, 
for appropriate $x$ and $y$. Moreover, we will have specific homeomorphisms 
between these spaces and some locally Euclidean ones.

\begin{defn}
\label{surface:30}
Let $\mathcal{B}$ be an ordered bi-infinite Bratteli diagram.
\begin{enumerate}
\item 
We define the quotient space 
\[
S_{\mathcal{B}}^{r} = Y_{\mathcal{B}}/ x \sim \Delta_{r}(x), x \in
 \partial_{r}X_{\mathcal{B}} \cap Y_{\mathcal{B}}.
\]
We let $\pi^{r}$ denote the quotient map from $ Y_{\mathcal{B}}$ 
to  $S_{\mathcal{B}}^{r}$.
\item 
We define the quotient space 
\[
S_{\mathcal{B}}^{s} = Y_{\mathcal{B}}/ y \sim \Delta_{s}(y), y \in 
\partial_{s}X_{\mathcal{B}}\cap Y_{\mathcal{B}}.
\]
We let $\pi^{s}$ denote the quotient map from $ Y_{\mathcal{B}}$ 
to  $S_{\mathcal{B}}^{s}$.
\item 
We define the quotient space 
\[
S_{\mathcal{B}} = Y_{\mathcal{B}}/ y \sim \Delta_{s}(y),
 x \sim \Delta_{r}(x), x \in \partial_{r} X_{\mathcal{B}}\cap Y_{\mathcal{B}},
 y \in \partial_{s}X_{\mathcal{B}}\cap Y_{\mathcal{B}}.
\]
As this space is obviously a quotient of both 
$S_{\mathcal{B}}^{r} $ and 
$S_{\mathcal{B}}^{s}$, we let $\rho^{s}$ be the map from the former
and $\rho^{r}$ be the map from the latter and 
\[
\pi = \rho^{s} \circ \pi^{r} =\rho^{r} \circ \pi^{s}.
\]
That is, we have a commutative diagram
$$\xymatrix{   &   Y_{\mathcal{B}} \ar_{\pi^{r}}[dl] \ar^{\pi^{s}}[dr] &  \\
    S^{r}_{\mathcal{B}}  \ar_{\rho^{s}}[dr] &   &  
     S^{s}_{\mathcal{B}} \ar^{\rho^{r}}[dl] \\
       &   S_{\mathcal{B}} &  }$$
\end{enumerate}
\end{defn}

Our next goal is to provide local descriptions of the
spaces involved. More specifically, we need charts for the surface.
Of course, this is a crucial step if we are to show
that $S_{\mathcal{B}}$ is a translation surface.
Along the way, we will also obtain local descriptions
of  $S^{r}_{\mathcal{B}},   S^{s}_{\mathcal{B}}  $, which
 are somewhat simpler. 

Our charts will actually be defined as functions on the 
space $Y_{\mathcal{B}}$ to the plane, which are constant on 
equivalence classes. There are four distinct possibilities for
 these equivalence classes.
If a point of $S_{\mathcal{B}}$ is represented by a 
single point $x$ in $Y_{\mathcal{B}}$, the collection of sets
$X^{-}_{s(x_{m})}x_{(m,n]}X^{+}_{r(x_{n})}, m < n$, form a neighbourhood
base at the point $x$. In addition, the maps $\varphi^{s(x_{m})}_{r}$ and 
$\varphi^{r(x_{n})}_{s}$ of Definition \ref{order:60}
can be used to define a map to the plane.
This is not quite suitable for a chart since the image is a closed 
rectangle, rather than an open set, but eliminating the sides of this
rectangle from the image by restriction
is a simple matter and the result will be one of
our charts.

A second possibility is that a point of $S_{\mathcal{B}}$ is represented by 
a pair, $x, y =\Delta_{r}(x)$, for some $x$ in 
$\partial_{r}X_{\mathcal{B}} \cap Y_{\mathcal{B}}$. It follows
from  
Lemma \ref{order:50} that  
$x_{(i,\infty)}= y_{(i,\infty)}$, $y_{i}$ is the 
$\leq_{r}$-successor (or predecessor) of $x_{i}$ and 
$x_{(-\infty,i)} = x_{s(x_{i})}^{s-min}, y_{(-\infty,i} = x_{s(y_{i})}^{s-max}$,
for some integer $i$. In this case, we will use the \emph{pair} of paths
$(x_{(m,n]}, y_{(m,n]})$, where $m < i < n$, to 
parameterize our neighbourhoods. 

Of course ,there is a third case where the point  is represented by 
a pair, $x, y =\Delta_{s}(x)$, for some $x$ in 
$\partial_{s}X_{\mathcal{B}} \cap Y_{\mathcal{B}}$ and a fourth case
where it is represented by four points.

We formally introduce the sets which will parameterize our charts.

\begin{defn}
\label{surface:40}
Let $\mathcal{B}$ be a bi-infinite ordered Bratteli 
diagram with faithful state $\nu_{s}, \nu_{r}$. 
\begin{enumerate}
\item
For $m < n$, we let $E^{r}_{m,n}$ 
denote the set of pairs $(p^{1}, p^{2})$ with $p^{1}, p^{2}$ 
in $E^{Y}_{m,n}$ 
such that $p^{2}= S_{r}(p^{1})$; 
that is $p^{2}$ is the successor in the 
$\leq_{r}$ order. This implies $r(p^{1})= r(p^{2})$,
 which we denote by $r(p)$.

For $p=(p^{1}, p^{2})$ in $E^{r}_{m,n}$, we define
\begin{eqnarray*}
V_{1}^{r}(p) & =   (X_{s(p^{1})}^{-} - 
\{ x^{r-min}_{s(p^{1})} \} )  p^{1} X^{+}_{r(p)}, \\
V_{2}^{r}(p) & =   (X_{s(p_{2})}^{-} - 
\{ x^{r-max}_{s(p^{2})} \} )  p^{2} X^{+}_{r(p)}
\end{eqnarray*}
and $V^{r}(p) = V_{1}^{r}(p) \cup V_{2}^{r}(p)$.
\item 
For $p$ in $E^{r}_{m,n}$, we define 
$c_{r}(p) = \varphi_{r}^{r(p)}( x_{s(p^{1})}^{r-max} p^{1}) =
\varphi_{r}^{r(p)}( x_{s(p^{2})}^{r-min} p^{2})$ and  
$\psi_{r}^{p}: V^{r}(p) \rightarrow \R$ by 
\[
\psi_{r}^{p}(x) = \varphi^{r(p)}_{r}(x_{(-\infty, n]}) - c_{r}(p),
\]
for $x$ in $V^{r}(p)$.
\item
For $m < n$, we let $E^{s}_{m,n}$ 
denote the set of pairs $(p^{1}, p^{2})$ with $p^{1}, p^{2}$ 
in $E^{Y}_{m,n}$ 
such that $p^{2}= S_{s}(p^{1})$; 
that is $p^{2}$ is the successor in the $\leq_{s}$ order.
 This implies $s(p^{1})= s(p^{2})$
which we denote by $s(p)$.

For $p=(p^{1}, p^{2})$ in $E^{s}_{m,n}$, we define
\begin{eqnarray*}
V_{1}^{s}(p) & = X^{-}_{s(p)}  p^{1} (X_{r(p^{1})}^{+} - 
\{ x^{s-min}_{r(p_{1})} \} ) , \\
V_{2}^{s}(p) & =  X^{-}_{s(p)}  p^{2} (X_{r(p^{2})}^{+} - 
\{ x^{s-max}_{r(p^{2})} \} )
\end{eqnarray*}
and $V^{s}(p) = V_{1}^{s}(p) \cup V_{2}^{s}(p)$.
\item 
For $p$ in $E^{s}_{m,n}$, we define
$c_{s}(p) = \varphi^{s(p)}(p^{1}x_{r(p^{1})}^{s-max}) = 
\varphi^{s(p)}(p^{2}x_{r(p^{2})}^{s-min})$ and 
$\psi_{s}^{p}: V^{s}(p) \rightarrow \R$ by 
\[
\psi_{s}^{p}(x) = \varphi^{s(p)}_{s}(x_{[m,\infty)}) - c_{s}(p),
\]
for $x$ in $V^{s}(p)$.
   \end{enumerate}
\end{defn}

The basic properties if these sets are summarized in the following.
For brevity, we say that a subset $A \subseteq X_{\mathcal{B}}$ is 
$\Delta_{r}$-invariant (or $\Delta_{s}$-invariant)
if $\Delta_{r}(A \cap \partial_{r}X_{\mathcal{B}}) = 
A \cap \partial_{r}X_{\mathcal{B}}$ (or  
$\Delta_{s}(A \cap \partial_{s}X_{\mathcal{B}}) = 
A \cap \partial_{s}X_{\mathcal{B}}$, respectively).

\begin{lemma}
\label{surface:50}
Let $\mathcal{B}$ be a bi-infinite ordered Bratteli 
diagram satisfying the conditions of
Definition \ref{surface:20} and with faithful state $\nu_{s}, \nu_{r}$. 
Let $m < n$
and $p= (p^{1}, p^{2})$ be in $E_{m,n}^{r}$.
\begin{enumerate}
\item The set $V^{r}(p)$  
is an open subset of $Y_{\mathcal{B}}$.
\item   If $x$ is in 
$V^{r}(p) \cap \partial_{r}X_{\mathcal{B}}$ then 
exactly one of the following holds.
\begin{enumerate}
\item $m \leq m(x) \leq n$ and if $x$ is in $V_{i}(p), i=1,2$, then 
$\Delta_{r}(x)$ is in $V_{3-i}(p)$,
\item $m(x) < m$ and if $x$ is in $V_{i}(p), i=1,2$, then 
$\Delta_{r}(x)$ is in $V_{i}(p)$.
\end{enumerate}
In either case, $\Delta_{r}(x)_{(n, \infty)} = x_{(n, \infty)}$ and 
 $V^{r}(p)$ 
 is $\Delta_{r}$-invariant. 
\item 
For $x$ in $V^{r}(p)$, we have 
\[
\psi^{p}_{r}(x)  = \left\{ \begin{array}{cl} 
\varphi^{s(p^{1})}_{r}(x_{(-\infty, m)})  - \nu_{r}(s(p^{1})
 & x \in V^{r}_{1}(p), \\
\varphi^{s(p^{2})}_{r}(x_{(-\infty, m)}) 
 & x \in V^{r}_{2}(p). \end{array} \right.
\]
\item 
\begin{eqnarray*}
\psi_{r}^{p}(V_{1}^{r}(p)) & = & 
(-\nu_{r}(s(p^{1}), 0 ],\\
 \psi_{r}^{p}(V_{2}^{r}(p)) & = &
  [ 0 ,\nu_{r}(s(p^{2}) ).
\end{eqnarray*}
\item 
The map  which sends 
$x \in V^{r}(p) \rightarrow 
 \psi^{p}_{r} (x) $
is continuous 
and identifies two distinct 
points $x,y$ if and only if 
$\Delta_{r}(x)_{(-\infty,n]} = y_{(-\infty,n]}$.
\end{enumerate}
\end{lemma}

\begin{proof}
For the first part, the fact that $p^{1}, p^{2}$ are
 in $E_{m,n}^{r} \subseteq E^{Y}_{m,n}$ implies that 
 $V_{i}^{r}(p) \subseteq X^{-}_{s(p^{i})}p^{i}X^{+}_{r(p)}, i =1,2$
 are contained in $Y_{\mathcal{B}}$. The fact that 
 $V^{r}(p)$ is open is clear.

For the second, first suppose that $m(x) \geq m$ and that $x_{(-\infty,m(x))}$
is all $r$-maximal edges. From the definition of $V_{2}(p)$, $x$ cannot be 
in $V_{2}(p)$,  so $x_{(m,n]} = p^{1}$. As $p^{1}$ is not $r$-maximal, 
$m(x) \leq n$ and $\Delta_{r}(x) = x_{(-\infty,m)}p^{2}x_{(n, \infty)}$
which is in $V_{2}(p)$. Similarly, if 
$x_{(-\infty,m(x))}$
is all $r$-minimal edges, then $x$ is not in $V_{2}(p)$,
$\Delta_{r}(x)$ is in $V_{1}(p)$ and $m(x) \leq n$.
If $m < m(x)$, then $\Delta_{r}(x)_{[m,\infty)} = x_{[m, \infty)}$
so if $x$ is in $V_{i}(p)$, so is $\Delta_{r}(x)$. 
The fact that $m \leq n$ in either case implies 
$\Delta_{r}(x)_{(n, \infty)} = x_{(n, \infty)}$.

 For the third part, consider $x$ in $V_{1}(p)$. We apply Lemma 
 \ref{order:70}  with $p=p^{1}$ in the second line below:
 \begin{eqnarray*}
 \psi^{p}_{r}(x) & = & \varphi^{r(p)}_{r}(x_{(-\infty, n]}) -c_{r}(p) \\
     & = & \varphi^{r(p)}_{r}(x_{(-\infty, m)}p^{1}) - 
     \varphi^{r(p)}_{r}(x_{s(p^{1})}^{r-max}p^{1}) \\
      & = & \varphi^{s(p^{1})}_{r}(x_{(-\infty, m)})
        + \varphi^{r(p)}_{r}(x^{r-min}_{s(p^{1})}p^{1}) 
         -  \varphi^{r(p)}_{r}(x_{s(p^{1})}^{r-max}p^{1}) \\
      & = & \varphi^{s(p^{1})}_{r}(x_{(-\infty, m)}) - \nu_{r}(s(p^{1})).     
 \end{eqnarray*}
  For
 $x$ in $V_{2}^{r}(p)$, we use the same result: 
  \begin{eqnarray*}
 \psi^{p}_{r}(x) & = & \varphi^{r(p)}_{r}(x_{(-\infty, n]}) -c_{r}(p) \\
     & = & \varphi^{r(p)}_{r}(x_{(-\infty, m)}p^{2})
      - 
     \varphi^{r(p)}_{r}(x_{s(p^{2}}^{r-min}p^{2}) \\
      & = & \varphi^{s(p^{2})}_{r}(x_{(-\infty, m)})
        + \varphi^{r(p)}_{r}(x^{r-min}_{s(p^{2})}p^{2}) 
        - 
     \varphi^{r(p)}_{r}(x_{s(p^{2}}^{r-min}p^{2}).   
        \\        
 \end{eqnarray*}

 Part 4 is an immediate consequence of part 3,
   the definitions and two applications of part 4
 of Lemma \ref{order:20}.
 
 Part 5 follows from parts 2 and 3 of Lemma \ref{order:20} applied to
 $\mathcal{B}_{s(p)}^{-}$ and the fact that 
 $\varphi^{r(p)}_{r}(x^{r-max}_{s(p^{1})}p^{1}) = 
  \varphi^{r(p)}_{r}(x^{r-min}_{s(p^{2})}p^{2})$.
\end{proof}

The next result, while slightly technical, 
essentially  shows that there are enough sets $V^{r}(p)$
to cover $Y_{\mathcal{B}}$.

\begin{lemma}
\label{surface:55}
If $y$ is in $S_{\mathcal{B}}^{r}$ and $U$ is an open set 
in $Y_{\mathcal{B}}$ such that $(\pi^{r})^{-1}\{ y \} \subseteq U$, 
then there exists $m \geq 1$ and $ p= (p^{1}, p^{2}) \in E_{-m,m}^{r}$ 
such that
neither $p^{1}, p^{2}$ are $r$-maximal nor $r$-minimal in $E^{Y}_{-m,m}$
and 
 $(\pi^{r})^{-1}\{ y \} \subseteq V^{r}(p) \subseteq U$. 
 \end{lemma}

 \begin{proof}
We first consider the case when $(\pi^{r})^{-1}\{ y \} = \{ x \}$, 
which is 
not in $\partial_{r}X_{\mathcal{B}}$. As $Y_{\mathcal{B}}$ is open, 
we may find $1 \leq k$ 
such that $X_{s(x_{-k})}^{-}x_{[-k,k]}X_{r(x_{k})}^{+}$
is contained in $U$. As $x$ is 
not in $\partial_{r}X_{\mathcal{B}}$, we may 
find $m > l > l' > k$ such that 
 $x_{-m}, x_{-l}$ are not $r$-maximal and $x_{l'}$ is not $r$-minimal.
 Let $y_{-m}$ be the  $r$-successor of $x_{-m}$.
  Let $p^{1}= x_{[-m,m]}, p^{2}= y_{-m}x_{(-m,m]}$ so $p^{2}$ is the 
  $r$-successor of $p^{1}$ in $E^{Y}_{-m,m}$. 
  As $p^{1}_{[-k,k]} = p^{2}_{[-k,k]}=x_{[-k,k]}$, $V^{r}(p)$ 
 contains $x$ and  is contained in $U$. Clearly, $p^{1}$ is not $r$-maximal
 and $p^{2}$ is not $r$-minimal in $E^{Y}_{-m,n}$. Also,  $p^{1}$ 
 is not $r$-minimal
 since $x_{-l'}$ is not $r$-minimal. Similarly, 
 $p^{2}$ 
 is not $r$-maximal
 since $x_{-l}$ is not $r$-minimal.

Next, we consider the case  that   
$(\pi^{r})^{-1}\{ y \} = \{ x, S_{r}(x) \}$
which are in $\partial_{r}X_{\mathcal{B}}$.
We may choose
 $k > \vert m(x) \vert$ such that
  $X_{s(x_{-k})}^{-}x_{[-k,k]}X_{r(x_{k})}^{+}$ and 
   $X_{s(x_{-k})}^{-}S_{r}(x)_{[-k,k]}X_{r(x_{k})}^{+}$
are contained in $U$. We now choose $m > k$ such that
there are at least two paths in $E_{-m,-k}$ with range $s(x_{-k})$ and 
at least two paths in $E_{-m,-k}$ with range $s(S_{r}(x)_{-k})$.
Let $p^{1}= x_{[-m,m]}, p^{2}= S_{r}(x)_{[-m,m]}$. It
   is straightforward to check $p=(p^{1}, p^{2})$ satisfies the conclusion.
 \end{proof}
 
 There are obvious analogues of the last two results 
 for $p$ in $E^{s}_{m,n}$.

The following follows quite easily from the technical results above.

\begin{cor}
\label{surface:60}
Let $\mathcal{B}$ be an ordered bi-infinite Bratteli 
diagram satisfying the conditions of
Definition \ref{surface:20}. 
\begin{enumerate}
\item 
The space $S_{\mathcal{B}}^{r}$ is a locally compact Hausdorff space and 
$\pi^{r}: Y_{\mathcal{B}} \rightarrow S_{\mathcal{B}}^{r}$ is a continuous, 
proper surjection.
\item
The space $S_{\mathcal{B}}^{s}$ is a locally compact Hausdorff space and 
$\pi^{s}: Y_{\mathcal{B}} \rightarrow S_{\mathcal{B}}^{s}$ is a continuous, 
proper surjection.
\end{enumerate}
\end{cor}

\begin{proof}
We prove the first part only. Continuity and surjectivity follow 
from the definition of $S_{\mathcal{B}}^{r}$. We prove that the quotient map
is proper.
As $Y_{\mathcal{B}}$ is a metric space, it suffices to show that if 
$K  \subseteq S_{\mathcal{B}}^{r} $ is limit point compact, then so is 
$(\pi^{r})^{-1}(K) \subseteq Y_{\mathcal{B}}$.

Let $A$ be an infinite subset of $ (\pi^{r})^{-1}(K)$. As $\pi^{r}$ is 
at most two-to-one, $\pi^{r}(A)$ is also infinite and since $K$ is 
compact, it has a limit point, we call $z$ in $K$.

Let us first consider the case $(\pi^{r})^{-1}\{ z \} = \{ y \}$.
We will show $y$ is a limit point of $A$.
It follows that $y$ is not in $\partial_{r}X_{\mathcal{B}}$.
Let $z \in U \subseteq Y_{\mathcal{B}}$ be open. We may $n$ sufficiently
large so that, $y_{-n}$ is not $r$-maximal, $y_{(-\infty, -n)}$ are not 
all $r$-minimal and $X_{r(y_{-n})}^{-}y_{[-n+1,n-1]}X^{+}_{s(y_{n})}$ is 
contained in $U$. Let $p^{1}=y_{[-n,n]}$ and $p^{2}$ be its $r$-successor.
This means that $p^{2}_{[-n+1,n-1]} = y_{[-n+1,n-1]}$ also 
so
$y \in V^{r}(p) \subseteq U$. It follows that
 $z = \pi^{r}(y) \in \pi^{r}(V^{r}(p))$
which is an open set in $S_{\mathcal{B}}^{r}$. It follows that
$\pi^{r}(V^{r}(p))$ contains a point of  $\pi^{r}(A)$. 
Hence $A$ contains a point of $V^{r}(p) \subseteq U$.

In the second case, we assume that 
 $(\pi^{r})^{-1}\{ z \} = \{ y^{1}, y^{2}\}$,
  where $\Delta_{r}(y^{1}) =y^{2}$
 is the $r$-successor of $y^{1}$.
 We claim that either $y^{1}$ or $y^{2}$ is a limit point of $A$.
 For sufficiently large values of $n$, the pair 
 $p= (p^{1}, p^{2})=(y^{1}_{[-n,n]},y^{2}_{[-n,n]})$
 will be in $E^{r}_{-n,n}$. We also note that $y^{i}$ is 
  in $V_{i}^{r}(p)$ for
 $i=1,2$. The set $\pi^{r}(V^{r}(p))$ is open and contains $z$, hence it 
 contains a point of $\pi^{r}(A)$. It follows
 that $A$ meets either $V_{1}^{r}(p)$ or $V_{2}^{r}(p)$. This statement holds
 for each $n$ sufficiently large. It follows that there are 
 infinitely many $n$ such that $A \cap V_{1}^{r}(p)$ is not 
 empty or $A \cap V_{2}^{r}(p)$ is not empty. In the former case, $y^{1}$ 
 is a limit point of $A$, while in the latter $y^{2}$ is.
\end{proof}

The next result is a  comparison of the different 
$V^{r}(p), p \in E^{r}_{m,n}, n > m$, at least in the 
case $m=-n$. We observe that the conclusion already hints
at the condition for the charts in a translation surface.

\begin{lemma}
\label{surface:70}
Let $\mathcal{B}$ be a bi-infinite, ordered Bratteli 
diagram.

Suppose $ 1 \leq m < n$, $p=(p^{1}, p^{2})$  in 
$E^{r}_{-m,m}$  and $q = (q^{1}, q^{2})$  in $E^{r}_{-n,n}$. 
 Suppose that neither $p^{1}$ nor $p^{2}$ are $r$-minimal or $r$-maximal
 in $E^{Y}_{-m,m}$. 
  If $V^{r}(q) \cap V^{r}(p)$ is not empty, 
 then $q^{1}_{(m,n]}= q^{2}_{(m,n]}$ and we have
\[
\psi^{q}_{r}(x) = \psi^{p}_{r}(x) +c_{r}(p)  - c_{r}(q_{[-n,m]}),
\]
for each $x$ in $V^{r}(p) \cap V^{r}(q)$. Moreover, we have 
\[
c_{r}(p)  - c_{r}(q_{[-n,m]}) = \left\{ \begin{array}{cl}
 \varphi^{s(p^{1})}_{r}(x_{s(q^{1})}^{r-max}q^{1}_{[-n,-m)}) 
  - \nu_{r}(s(p^{1})), & q^{1}_{[-m,m]}= p^{1} \\
    -\nu_{r}(s(p^{1})), & q^{1}_{[-m,m]} \neq q^{2}_{[-m,m]}= p^{1} \\
  \varphi^{s(p^{2})}_{r}(x_{s(q^{2})}^{r-min}q^{2}_{[-n,-m)}), & 
  q^{2}_{[-m,m]}= p^{2} \\
  \nu_{r}(s(p^{2})), & q^{2}_{[-m,m]} \neq q^{1}_{[-m,m]}= p^{2} \\
  0, & q^{1}_{[-m,m]}= p^{1}, q^{2}_{[-m,m]}= p^{2}.
  \end{array} \right.
\]
\end{lemma}

\begin{proof}
We first show that $m(q) \leq m$. If $m(q) > m$, then 
$q_{[-m,m]}^{1}$ consists of all $r$-maximal edges
while $q_{[-m,m]}^{2}$ consists of all $r$-minimal edges. 
with the hypothesis that $p^{1}, p^{2}$ are not $r$-maximal 
or $r$-minimal implies
$V^{r}(q) \cap V^{r}(p)$ is empty. As an immediate consequence, 
we see that $q^{1}_{(m,n]}= q^{2}_{(m,n]}$.

Let $x$ be in $V^{r}(p) \cap V^{r}(q)$.
We will to apply Lemma \ref{order:70} in the third and fourth lines:
\begin{eqnarray*}
\psi^{q}_{r}(x)   & = & \varphi^{r(q)}_{r}(x_{(-\infty, n]}) -c_{r}(q)  \\
    & =  &  \varphi^{r(q)}_{r}(x_{(-\infty, m]}q^{1}_{(m,n]}) -c_{r}(q)   \\
   & = &    \varphi^{r(p)}_{r}(x_{(-\infty, m]})
  + \varphi^{r(q)}_{r}(x_{r(p)}^{r-min}q^{1}_{(m,n]}) - 
  \varphi^{r(q)}(x_{r(q)}^{r-max}q^{1}) \\
    & = & \psi^{p}_{r}(x) + c_{r}(p)  - 
    \varphi^{r(p)}_{r}(x_{r(p)}^{r-max}q^{1}_{[-n,m]}) \\
      & = & \psi^{p}_{r}(x) + c_{r}(p)  - c_{r}(q_{[-n,m]}).
\end{eqnarray*}

If we assume that $q^{1}_{[-m,m]} =p^{1}$, then we again use
Lemma \ref{order:70}
\begin{eqnarray*}
c_{r}(p)  - c_{r}(q_{[-n,m]}) & = & 
-\varphi^{r(p)}_{r}(x^{r-max}_{s(p^{1})}p^{1}) +
 \varphi^{r(q_{[-n,m]})}_{r}(x^{r-max}_{s(q^{1})}q^{1}_{[-n,-m)}p^{1}) \\
 & = & 
-\varphi^{r(p)}_{r}(x^{r-max}_{s(p^{1})}p^{1}) +
 + \varphi^{s(p^{1})}_{r}(x^{r-max}_{s(p^{1})}q^{1}_{[-n,-m)}) 
   +  \varphi^{r(p)}_{r}(x^{r-min}_{s(q^{1})}p^{1}) \\
    & = & - \nu_{r}(s(p^{1})) + 
    \varphi^{s(p^{1})}_{r}(x^{r-max}_{s(p^{1})}q^{1}_{[-n,-m)}).
\end{eqnarray*}

If we assume that $q^{1}_{[-m,m]} \neq q^{2}_{[-m,m]} =p^{1}$, then we again use
Lemma \ref{order:70}
\begin{eqnarray*}
c_{r}(p)  - c_{r}(q_{[-n,m]}) & = & 
-\varphi^{r(p)}_{r}(x^{r-max}_{s(p^{1})}p^{1}) +
 \varphi^{r(q_{[-n,m]})}_{r}(x^{r-min}_{s(q^{2})}q^{2}_{[-n,-m)}p^{1}) \\
 & = & 
-\varphi^{r(p)}_{r}(x^{r-max}_{s(p^{1})}p^{1}) +
 + \varphi^{s(p^{1})}_{r}(x^{r-min}_{s(p^{1})}q^{2}_{[-n,-m)}) 
   +  \varphi^{r(p)}_{r}(x^{r-min}_{s(q^{1})}p^{1}) \\
    & = & - \nu_{r}(p^{1}) + 
    \varphi^{s(p^{2})}_{r}(x^{r-min}_{s(p^{2})}q^{2}_{[-n,-m)}).
\end{eqnarray*}
We  claim that $q^{2}_{[-n,-m)}$ must consist of $r$-minimal edges. If
not, the predecessor of $q^{2}$ would be unchanged in 
entries $m$ and greater. This would mean that $q^{1}_{[-m,m]} = q^{2}_{[-m,m]}$, 
which is a 
contradiction. It follows that 
 $\varphi^{s(p^{2})}_{r}(x^{r-min}_{s(p^{2})}q^{2}_{[-n,-m)}) = 0$.

The third and fourth cases are
 done in a similar way and we omit the details.

Finally, we suppose that $q^{1}_{[-m,m]} =p^{1}$ and $q^{2}_{[-m,m]} =p^{2}$.
If $q^{1}$ contained an edge which was not $r$-maximal between $-n$ and 
$-m$, then it $r$-successor would be unchanged between $-m$ and $m$. 
This is not the case so $q^{1}_{[-n,-m)}$ is $r$-maximal and 
$q^{2}_{[-n,-m)}$ is $r$-minimal. It follows  
that $\varphi^{s(p^{2})}_{r}(x_{s(q^{2})}^{r-min}q^{2}_{[-n,-m)}) =0$ and 
 $c_{r}(p)  - c_{r}(q_{[-n,m]})=0$  so from the third case.
\end{proof}

The surface $S_{\mathcal{B}}$ is  more complicated. In 
 particular, our nice open cover is rather more 
 technical than the previous ones, where a point in 
 $S_{\mathcal{B}}$ has two pre-images under both 
 $\rho^{r}$ and $\rho^{s}$, or four pre-images in $Y_{\mathcal{B}}$.
 While this takes a bit of effort, we are rewarded 
 with an immediate proof that $S_{\mathcal{B}}$ is a translation surface.

\begin{defn}
\label{surface:80}
For integers $m < n$, 
we define 
$E_{m,n}^{r/s} $ to be the set of all quadruples 
$p=(p^{1,1}, p^{1,2}, p^{2,1}, p^{2,2})$ of 
distinct paths in  $E^{Y}_{m,n}$
such that 
\begin{enumerate}
\item
\begin{enumerate}
\item $p^{1,2}$ is the $s$-successor of $p^{1,1}$ in $E_{m,n}$, 
\item $p^{2,1}$ is the $r$-successor of $p^{1,1}$ in $E_{m,n}$,
\item  $p^{2,2}$ is the $s$-successor of $p^{2,1}$ and  the 
  $r$-successor of $p^{1,2}$  in $E_{m,n}$. 
  \end{enumerate} 
\item 
For $p$ be in $E_{m,n}^{r/s}$,  
 we  define
\begin{eqnarray*}
V_{1,1}(p) & =  &  \left( X_{s(p^{1,1})}^{-} - \{ x_{s(p^{1,1})}^{r-min} \} \right)
   p^{1,1} \left( X_{r(p^{1,1})}^{+} - \{ x_{r(p^{1,1})}^{s-min} \} \right) \\
V_{2,1}(p) & =  &    \left( X_{s(p^{2,1})}^{-} - \{ x_{s(p^{2,1})}^{r-max} \} \right)
 p^{2,1} \left( X_{r(p^{2,1})}^{+} - \{ x_{r(p^{2,1})}^{s-min} \} \right) \\
 V_{1,2}(p) & =  &  \left( X_{s(p^{1,2})}^{-} - \{ x_{s(p^{1,2})}^{r-min} \} \right)
  p^{1,2} \left( X_{r(p^{1,2})}^{+} - \{ x_{r(p^{1,2})}^{s-max} \} \right) \\
V_{2,2}(p) & =  &  \left( X_{s(p^{2,2})}^{-} - \{ x_{s(p^{2,2})}^{r-max} \} \right)   
  p^{2,2} \left( X_{r(p^{2,2})}^{+} - \{ x_{r(p^{2,2})}^{s-max} \} \right) 
   \end{eqnarray*}
   and 
   \[
   V(p) = V_{1,1}(p) \cup V_{1,2}(p) \cup V_{2,1}(p) \cup V_{2,2}(p).
   \]
 \item   
We also define
$\psi^{p}: V(p) \rightarrow \R^{2}$ by 
\[
\psi^{p}(x)  =   \left\{ \begin{array}{cl} 
\left( \varphi^{s(p^{1,1})}_{r}(x_{(-\infty,m)}) - \nu_{r}(s(p^{1,1})), 
 \varphi^{r(p^{1,1})}_{s}(x_{(n,\infty)}) - \nu_{s}(r(p^{1,1})))
   \right),  &  x \in V_{1,1}(p) \\
  \left( \varphi^{s(p^{1,2})}_{r}(x_{(-\infty,m)}) - \nu_{r}(s(p^{1,2})), 
 \varphi^{r(p^{1,2})}_{s}(x_{(n,\infty)}) )
   \right),  &  x \in V_{1,2}(p) \\ 
   \left( \varphi^{s(p^{2,1})}_{r}(x_{(-\infty,m)}), 
 \varphi^{r(p^{2,1})}_{s}(x_{(n,\infty)}) - \nu_{s}(r(p^{2,1})))
   \right),  &  x \in V_{2,1}(p) \\
     \left( \varphi^{s(p^{2,2})}_{r}(x_{(-\infty,m)}), 
 \varphi^{r(p^{2,2})}_{s}(x_{(n,\infty)}) 
   \right),  &  x \in V_{2,2}(p). \\
     \end{array} \right.
\]
We let $\psi^{p}(x)_{1}$ and $\psi^{p}(x)_{2}$ denote the first and second 
entries of $\psi^{p}(x)$.
\end{enumerate}
\end{defn}

Let us make some observations relating this new definition with 
the previous ones. We will not prove the following
 as it is a simple observation from the definitions.
 
\begin{lemma}
\label{surface:90}
Let $p$ be in $E^{r/s}_{m,n}$.
\begin{enumerate}
\item 
For $j=1,2$, we have $(p^{1,j}, p^{2,j})$  is in 
$E^{r}_{m,n}$ and $V_{1,j}(p) \cup V_{2,j}(p) \subseteq
 V^{r}(p^{1,j}, p^{2,j})$ and, for $x$ in $V_{1,j}(p) \cup V_{2,j}(p)$,
\[
\psi^{p}(x)_{1}  = \psi^{(p^{1,j}, p^{2,j})}_{r}(x).
\]
\item 
For $i=1,2$, we have $(p^{i,1}, p^{i,2})$  is in 
$E^{s}_{m,n}$ and $V_{i,1}(p) \cup V_{i,2}(p) \subseteq V^{s}(p^{i,1}, p^{i,2})$ and, for $x$ in $V_{i,1}(p) \cup V_{i,2}(p)$,
\[
\psi^{p}(x)_{2}  = \psi^{(p^{i,1}, p^{i,2})}_{s}(x).
\]
\end{enumerate}
\end{lemma}

We first need a version of Lemma \ref{surface:50}. Fortunately, most
of this follows quite easily from Lemmas \ref{surface:50} 
and \ref{surface:90}.

\begin{lemma}
\label{surface:100}
\begin{enumerate}
\item If $p$ is in $E^{r/s}_{m,n}, m < n$, then 
$V(p)$ is open in $Y_{\mathcal{B}}$.
\item If $p$ is in $E^{r/s}_{m,n}, m < n$, then 
$V(p)$ is invariant under $\Delta_{r}$ and $\Delta_{s}$.
\item 
\begin{eqnarray*}
\psi^{p}(V_{1,1}(p)) & =  &  
\left( - \nu_{r}(s(p^{1,1})),  0 \right] \times 
\left( - \nu_{s}(r(p^{1,1}))  ,  0 \right] \\ 
\psi^{p}(V_{2,1}(p)) & =  &  
\left[   0,  \nu_{r}(s(p_{2,1})) \right) \times 
\left( - \nu_{s}(r(p_{2,1})),  0 \right] \\ 
\psi^{p}(V_{1,2}(p)) & =  &  
\left( - \nu_{r}(s(p_{1,2})),  0 \right] \times 
\left[  0,  \nu_{s}(r(p_{1,2}))) \right) \\ 
\psi^{p}(V_{2,2}(p)) & =  &  
\left[   0,  \nu_{r}(s(p_{2,2})) \right) \times 
\left[  0,  \nu_{s}(r(p_{2,2})) \right) \\ 
\psi^{p}(V(p)) & =  &  
\left( - \nu_{r}(s(p_{1,1})),  \nu_{r}(s(p_{2,2}))  \right) \times 
\left( - \nu_{s}(r(p_{1,1})),  \nu_{s}(r(p_{2,2})) \right) 
\end{eqnarray*}
\item For $p$ in $E^{r/s}_{m,n}, m < n$, $\psi^{p}$ is continuous
and,  for $x, y$ in $V(p)$,
 $\psi^{p}(x) = \psi^{p}(y)$ if and only if $\pi(x)  = \pi(y)$ 
 in $S_{\mathcal{B}}$.
\end{enumerate}
\end{lemma}

\begin{proof}
The first part is clear from the definition.

For the second part, if we let $q= (p^{1,1}, p^{2,1})$, then 
$q$ is in $E^{r}_{m,n}$ and so $V^{r}(q)$ is $\Delta_{r}$-invariant
by part  2 of Lemma \ref{surface:50}. This $V^{r}(q)$ contains
$V_{1,1}(p) \cup V_{2,1}(p)$, although they are not 
equal. However, part 2   in \ref{surface:50} shows that if 
$x$ is in $V^{r}(q) \cap \partial_{r}X_{\mathcal{B}}$,
 then $\Delta_{r}(x)_{(n, \infty)} = x_{(n, \infty)}$
which implies that  $V_{1,1}(p) \cup V_{2,1}(p)$ 
 is $\Delta_{r}$-invariant.  
In a similar way with $q = (p^{1,2}, p^{2,2})$,
 $V_{1,2}(p) \cup V_{2,2}(p)$ is $\Delta_{r}$-invariant so 
$V(p)$ is  $\Delta_{r}$-invariant. The proof for 
$\Delta_{s}$-invariant is similar, using
 the fact that $q= (p^{i,1}, p^{i,2})$ is in $E^{s}_{m,n}$, for 
 $i=1,2$.

 The third part of the conclusion is an 
 immediate consequence of the definition and 
 Lemma \ref{order:20} applied
 to the Bratteli diagrams $\mathcal{B}_{s(p^{i,j})}^{-}$ and 
  $\mathcal{B}_{r(p^{i,j})}^{+}$. The continuity of $\psi^{p}$
  on each of the sets $V_{i,j}(p)$ also follows
  from \ref{order:20} and the observations preceding the Lemma.
  
 Let us now prove that, for any $x$ in 
 $V(p) \cap \partial_{r}X_{\mathcal{B}}$, 
 $\psi^{p}(\Delta_{r}(x)) =\psi^{p}(x)$. It is easy to see that
 $V_{1,1}(p) \cup V_{2,1}(p)$ and $V_{1,2}(p) \cup V_{2,2}(p)$ 
 are both $\Delta_{r}$-invariant and the conclusion, for
 the first coordinates, follows  by restricting to these sets and using 
 part 5 of  Lemma \ref{surface:50}. Similar
  arguments deal with the second coordinate and show that
  $\psi^{p}(\Delta_{s}(x)) =\psi^{p}(x)$, for any $x$ in 
 $V(p) \cap \partial_{s}X_{\mathcal{B}}$.

 It remains for us to  prove the converse: suppose that $x,y$ are in $V(p)$
 and $\psi^{p}(x) = \psi^{p}(y)$, we must show they are related
 by $\Delta_{r}$ and $\Delta_{s}$. For a first case, 
 suppose that  $x,y$ both lie in the same $V_{1,1} \cup V_{2,1}$, 
If we use $q=( p^{1,1}, p^{2,1})$ as before, 
 we can appeal to the results we have above and part 5 of \ref{surface:50}.
 The equality of the first coordinates tells us that 
 $y_{(-\infty, n]}= x_{(-\infty, n]} $
 or possibly that $y_{(-\infty, n]}= \Delta_{r}(x)_{(-\infty,n]}$ 
if
 $x$ is in $\partial_{r}X_{\mathcal{B}}$. In the latter case, we also know
 that $\Delta_{r}(x)_{(n, \infty )}$ by part 2 of Lemma \ref{surface:50}.

In addition, Lemma \ref{order:20} applied to $\mathcal{B}_{r(p)}^{+}$
shows that either $y_{[m,\infty)}= x_{[m,\infty)}$ or  $y_{[m,\infty)}= 
 \Delta_{s}(x)_{[m,\infty)}$
if $y$ is in $\partial_{s}X_{\mathcal{B}}$. All together, 
the four possibilities amount to
$y=x, y = \Delta_{r}(x), y = \Delta_{s}(x)$ or
$y = \Delta_{r}  \circ \Delta_{s}(x)$.

Similar arguments deal with the cases $x,y$ lie in 
$V(p)_{1,1} \cup V_{1,2}(p), V(p)_{1,2} \cup V_{2,2}(p)$
or in $V(p)_{2,1} \cup V_{2,2}(p)$. We move on to
 the case $x$ is in $V_{1,1}(p)$ while $y$ is in $V_{2,2}(p)$.
 Part 3 of the conclusion then implies that 
 $\psi^{p}(x)= \pi^{p}(y) = (0,0)$. From this it 
 follows that $x_{(-\infty, m)}$ is all $r$-maximal edges, 
  $x_{(n,\infty)}$ is all $s$-maximal edges,
 $y_{(-\infty, m)}$ is all $r$-minimal edges and 
  $y_{(n,\infty)}$ is all $s$-minimal edges.
  From this we see that $y = \Delta_{r}  \circ \Delta_{s}(x)$.
 The case $x$ is in $V_{1,2}(p)$ while $y$ is in $V_{2,1}(p)$
 is similar. This completes the proof.
 
 A similar argument using
 $q= (p^{1,1}, p^{1,2})$ is $E^{s}_{m,n}$ shows 
  $y_{[m,\infty)}= \Delta_{s}(x)_{([m,\infty)}$ if
 $x$ is in $\partial_{s}X_{\mathcal{B}}$ and 
$y_{[m,\infty)}= x_{[m,\infty)} $ otherwise. The conclusion follows.

 In addition to showing part 2, these
 arguments and part 4 of Lemma \ref{surface:50} 
 also prove that the map $\psi^{p}$ has the following form:
 for $x$ in $V_{1,1}(p), i,j = 1,2$, 
 \[
 \psi^{p}(x) =\left(  \psi_{i,j}^{r}(x_{(-\infty,m)}), 
 \psi_{i,j}^{s}(x_{(n,\infty)}) \right),
 \]
 where the functions are provided by 
 various applications of Lemma \ref{surface:50}. Letting 
 $x= x^{r-max}_{s(p^{1,1})}p^{1,1}x^{s-max}_{r(p^{1,1})}$,  we
  also have 
  \[
  \psi^{p}(x) = \psi^{p}(\Delta_{r}(x)) = \psi^{p}(\Delta_{s}(x))
  = \psi^{p}(\Delta_{r} \circ \Delta_{s}(x)) = (0,0)
  \]
 
 Parts 3, 4 and  5 of the conclusion also follow
 from these observations.
\end{proof}

We need actually need to add a rather technical condition
on our choices for $p$. Fortunately, we still have an 
ample supply of such $p$, as follows.

\begin{lemma}
\label{surface:102}
Let  $y$ be in $S_{\mathcal{B}}$ and $U$ be an open set 
in $Y_{\mathcal{B}}$ such that 
$\pi^{-1}\{ y \} \subseteq U$.
Then there exist
 $m \geq 1$ and paths $p^{i,j}, 0 \leq i,j \leq 3$
 in $E^{Y}_{-m,m}$
 such that for $i<3$, $p^{i+1,j}$ is the $r$-successor
 of $p^{i,j}$ and for $j < 3$, $p^{i,j+1}$
 is the $s$-successor of $p^{i,j}$
 in $E^{Y}_{-m,m}$ and such that
 $\pi^{-1}\{ y \} \in V(p) \subseteq U$, where 
 $p=(p^{1,1}, p^{1,2}, p^{2,1}, p^{2,2})$.
\end{lemma}

\begin{proof}
The first case to consider is when $\pi^{-1}\{ y \} = \{ x \}$, 
which  is in
 neither $\partial_{r}X_{\mathcal{B}}$ nor 
$\partial_{s}X_{\mathcal{B}}$. Then we can find $m > 1$ such that 
$X_{s(x_{-m})}^{-}x_{[-m,m]}X_{r(x_{n})}^{+}$ is contained in $U$
and there are $-m \leq i_{1} < i_{2} < i_{3} <0$, 
$x_{i_{1}}, x_{i_{2}}$ are  not
$r$-maximal and and $x_{i_{3}}$ is not $r$-minimal and there are 
$0 < j_{3} < j_{2} < j_{1} \leq m$, $x_{j_{1}}, x_{j_{2}}$ are  not
$s$-maximal and $x_{j_{3}}$ is not $s$-minimal.
 We let $p^{1,1} = x_{[-m,m]}$. Having defined $p^{i,j}$
  for some $i,j$, we set $p^{i+1,j}$ to be its $r$-successor, 
  $p^{i,j+1}$ to be its $s$-successor, $p^{i-1,j}$ to be its $r$-predecessor 
and  $p^{i,j+1}$ to be its $s$-predecessor. This defines $p^{i,j}$
for all $0 \leq i,j \leq 3$. We note that since 
$i_{1}, i_{2}, i_{3} < 0 < j_{1}, j_{2}, j_{3}$, taking the $r$-successor
followed by taking the $s$-successor is the same as performing the 
operations in the other order.

In addition, there are values of $i < -m$ for which 
$x_{i}$ is not $r$-minimal and $i > m$ for which $x_{i}$ is not $s$-minimal,
so $x$ lies in $V_{1,1}(p)$.

The second case is that $x$ lies 
in $\partial_{r}X_{\mathcal{B}}$, but not in 
$\partial_{s}X_{\mathcal{B}}$. Without loss of generality, we
 assume that $\Delta_{r}(x)$ is the $r$-successor of $x$. 
 We choose $m > \vert m(x) \vert$ such that 
 $X_{s(x_{-m})}^{-}x_{[-m,m]}X_{r(x_{m})}^{+}$ and 
  $X_{s(x_{-m})}^{-}\Delta_{r}(x)_{[-m,m]}X_{r(x_{m})}^{+}$
  contained in $U$. In addition, $m$ is chosen so that
  there are $m(x) < j_{3}< j_{2} < j_{1} < m$ such that 
  $x_{j_{1}}, x_{j_{2}}$ are not $s$-maximal and $j_{3}$ is not $s$-minimal. We also choose
  $m$ sufficiently large so that there are at least two paths
  in $E^{Y}_{-m,m(x)}$ with range equal to $s(x_{-m})$ and 
  at least three paths with range equal to $s(\Delta(x)_{-m})$.
  We let $p^{1,1} = x_{[-m,m]}$ and define the other 
  $p^{i,j}$ as before.  Arguments similar to the last 
 case show  the conclusion  holds.

 The case when  $x$ lies 
in $\partial_{s}X_{\mathcal{B}}$, but not in 
$\partial_{r}X_{\mathcal{B}}$ is similar and we omit the details.

Finally, we consider the case $x$ lies in  
in $\partial_{r}X_{\mathcal{B}}  \cap \partial_{s}X_{\mathcal{B}}$.
Without loss of generality, assume $\Delta_{r}(x) = S_{r}(x)$ 
and $\Delta_{s}(x) = S_{s}(x)$.
We choose $k \geq \vert m(x) \vert, \vert n(x) \vert$ such that
the sets $
X^{-}_{s(x_{-k})}x_{[-k,k]}X^{+}_{r(x_{m})},  
X^{-}_{s((S_{r}(x)_{-k})}x_{[-k,k]}X^{+}_{r(S_{r}(x)_{m})},
X^{-}_{s((S_{s}(x)_{-k})}x_{[-k,k]}X^{+}_{r(S_{s}(x)_{m})}$ and \newline
$ X^{-}_{s((S_{r}\circ S_{s}(x)_{-k})}x_{[-k,k]} 
X^{+}_{r(S_{r}\circ S_{s}(x)_{m})} $ are all contained in $  U$.

We then choose $m > k$ such that there are at least 
three paths from   $V_{-m}$ to each vertex $V_{-k}$ and 
at least three paths from  each vertex
of  $V_{k}$ to $V_{m}$.  We define $p^{1,1}=x_{[-m,m]}$
and the remaining $p^{i,j}$ as before. The remaining 
 details of the proof are similar to the other cases.
\end{proof}

\begin{lemma}
\label{surface:110}
Let $\mathcal{B}$ be a  bi-infinite ordered Bratteli diagram.
Let $p$ be in $E_{-m,m}^{r/s}$, $q$ in  
in $E_{-n,n}^{r/s}$ with $1 \leq m \leq n$ such that
$V(p) \cap V(q)$ is not empty. Assume that $p$ satisfies the 
condition of Lemma \ref{surface:102}. 
 There is a constant $c(p,q)$ in $\R^{2}$ such that
\[
\psi^{p}(x) = \psi^{q}(x) - c(p,q),
\]
for all $x$ in $V(p) \cap V(q)$.
\end{lemma}

\begin{proof}
We will only prove the equation above holds when considering the 
first coordinates,  $ \psi^{p}(x)_{1}$ and 
 $ \psi^{q}(x)_{1}$, respectively. The other coordinate
 is done in a similar way (by replacing all appearances of $r$ with $s$.)

Let assume that $V_{1,1}(q)$ meets $V(p)$; the other cases are similar. 
Suppose $V_{1,1}(q)$ meets $V_{i,j}(p)$, for some $ 1 \leq i,j \leq 2$
which implies that $q^{1,1}_{[-m,m]} = p^{i,j}$. From the existence
of the $p^{k,l}, 0 \leq k,l \leq 3$, we see that $q^{1,1}$ is 
not $r$-maximal in $E^{Y}_{-n,n}$.
In particular, $q^{2,1}_{(m,n]} = q^{1,1}_{(m,n]}$. Similarly, as
  $q^{1,1}$ is 
not $s$-maximal $q^{1,2}_{[-n,-m)} = q^{1,1}_{[-n,-m)}$.

We will consider four cases separate;y depending on whether
$q^{1,1}_{(m,n]}$ is $s$-maximal or not and whether 
$q^{1,1}_{[-n,-m)}$ is  $r$-maximal or not.

Let us first suppose that $q^{1,1}_{(m,n]}$ is not $s$-maximal
 and $q^{1,1}_{[-n,-m)}$ is not $r$-maximal. It follows
that taking successors in either order leaves the entries between $_m$ and 
$m$ unchanged:
$q^{1,1}_{[-m,m]} = q^{2,1}_{[-m,m]} = q^{1,2}_{[-m,m]} = q^{2,2}_{[-m,m]} 
= p^{i,j}$. 

In this case, we have $V(q) \subseteq V_{i,j}(p)$ and 
we can apply Lemma \ref{surface:70} twice. First, to the 
pair $(p^{1,j}, p^{2,j})$ and $(q^{1,1}, q^{2,1})$ and then 
to the 
pair $(p^{1,j}, p^{2,j})$ and $(q^{1,2}, q^{2,2})$.
 In the first case, we have $q^{1,1}_{[-m,m]}=p^{1,j}$ and 
 the translation involved is
 \[
 \varphi^{s(p^{1,1})}_{r}(x_{s(q^{1,1})}^{r-max}q^{1,1}_{[-n,-m)}) 
  - \nu_{r}(p^{1,1})
  \]
  and in the second, we have $q^{1,2}_{[-m,m]}=p^{1,j}$ and it is 
  \[
 \varphi^{s(p^{1,1})}_{r}(x_{s(q^{1,2})}^{r-max}q^{1,2}_{[-n,-m)}) 
  - \nu_{r}(p^{1,1})
  \]
  Since $q^{1,2}_{[-n,-m)} = q^{1,1}_{[-n,-m)}$, these are
   equal and the desired 
  conclusion follows.

Next, we continue to suppose that 
$q^{1,1}_{(m,n]}$ is not $s$-maximal and that  $q^{1,1}_{[-n,-m)}$
is $r$-maximal. Then we have $q^{2,1}_{[-m,m]} = 
q^{2,2}_{[-m,m]} = p^{i+1,j}$. 
If $i=2$,  $V_{2,1}(q) \cup V_{2,2}(q)$ is disjoint from $V(p)$ and the conclusion follows
from an application of Lemma \ref{surface:70} to
the pair $(p^{1,j}, p^{2,j})$ and $(q^{1,1}, q^{2,1})$. 

If $1=1$, then   
we can apply Lemma \ref{surface:70} twice, first with
with $(p^{1,j}, p^{2,j})$  and $(q^{1,1}, q^{2,1})$ and then 
with the pair $(p^{1,j}, p^{2,j})$  and $(q^{1,2}, q^{2,2})$. As we have
 $p^{1,j}= q^{1,1}_{[-m,m]} = q^{1,2}_{[-m,m]} $ and 
  $p^{2,j}= q^{2,1}_{[-m,m]} = q^{2,2}_{[-m,m]} $,
  both translations are trivial.

 We next consider the case when $q^{1}_{(m,n]}$ is $s$-maximal while
 $q^{1}_{[-n,-m)}$ is not $r$-maximal.  Then it
  follows that $q^{2,1}_{[-m,m]}=p^{i,j}$ while 
 $q^{1,2}_{[-m,m]}= q^{2,2}_{[-m,m]}= p^{i,j+1}$.
 If $j=2$, $V_{1,2}(q) \cup V_{2,2}(q)$ is 
 disjoint from $V(p)$  we apply Lemma \ref{surface:70} to the pair 
 $(p^{1,j}, p^{2,j})$ and $(q^{1,1}, q^{2,1})$
 If $j=1$, then we 
  make two applications of part 2 of Lemma \ref{surface:70}. The first is 
 to the pair $(p^{1,1}, p^{2,1})$ and $(q^{1,1}, q^{2,1})$ and the second to the pair $(p^{1,2}, p^{2,2})$ and $(q^{1,2}, q^{2,2})$. Since
  $q^{1,1}_{[-n,-m)} = q^{1,2}_{[-n,-m)}$, the two translations are equal.

 We finally come to the case when 
 $q^{1}_{(m,n]}$ is $s$-maximal and
 $q^{1}_{[-n,-m)}$ is  $r$-maximal. If $i=j=1$, then we have 
 $q^{i',j'}_{[-m,m]}=p^{i',j'}$ for all $i',j'$. the conclusion
 follows from two applications of part 5 of Lemma \ref{surface:70};
 first
 to the pair $(p^{1,1}, p^{2,1})$ and $(q^{1,1}, q^{2,1})$ and then 
 to the pair 
 $(p^{1,2}, p^{2,2})$ and $(q^{1,2}, q^{2,2})$. If $i=1$ and $j=2$, 
 then $V_{1,2}(q) \cup V_{2,2}(q)$ are disjoint from $V(p)$ and the result 
 follows from an application of part 4 of Lemma \ref{surface:70}
  to the pair  $(p^{1,2}, p^{2,2})$ and $(q^{1,1}, q^{2,1})$.
  If $i=2$ and $j=1$, the result follows from  applications of
  part 2 of Lemma \ref{surface:70} to  
   $(p^{1,1}, p^{2,1})$ and $(q^{1,1}, q^{2,1})$ and to
   $(p^{1,2}, p^{2,2})$ and $(q^{1,2}, q^{2,2})$. The two translations are 
   equal since $s(p^{2,1})=s(p^{2,2})$. 
 If $i=2$ and $j=2$, then $V(q) \cap V(p) = V_{1,1}(q) \cap V_{2,2}(p)$ 
 and the result follows from
 Lemma \ref{surface:70} using $(p^{1,2}, p^{2,2})$ and $(q^{1,1}, q^{2,1})$.  
\end{proof}

\begin{thm}
\label{surface:120}
Let $\mathcal{B}$ be a  bi-infinite ordered Bratteli diagram satisfying 
the conditions of \ref{surface:20}.
For each $n \geq 1$ and $p$ in $E^{r/s}_{-n,n}$, define 
$Y(p) = \pi(V(p)) \subseteq S_{\mathcal{B}}$ and let 
$\eta^{p} : Y(p) \rightarrow \R^{2}$ be the unique map satisfying
$\eta^{p} = \psi^{p} \circ \pi$. Then each $Y(p)$ is open and 
$\eta^{p}$ is a homeomorphism to its image.
The space $S_{\mathcal{B}}$ is a surface and 
the
 collection of maps $\eta^{p}$, where $p$ ranges over
$\cup_{n \geq 1} E_{-n,n}^{r/s}$, is an 
atlas for $S_{\mathcal{B}}$ 
making it a translation surface.
\end{thm}

\section{Groupoids}
\label{groupoids}

A \emph{groupoid}, $G$, very roughly, is
 a group whose product is only defined
on a subset  $G^{2} \subseteq G \times G$. We will not
need a complete definition, but we refer the 
reader to Renault \cite{Ren:gpd} and Williams \cite{Wil:gpd}
 for details. One important class of examples
are equivalence relations. These are also called \emph{principal} groupoids
and are the only ones we consider here. We refer the reader to 
Renault \cite{Ren:gpd}.

Let $Y$ be a set and
 $R \subseteq Y \times Y$ be an equivalence
relation. It is a groupoid with operations 
\[
(x, y) (x', y') = (x, y'),  \text{ if } y = x'
\]
and 
\[
(x,y)^{-1} = (y,x)
\]
for all $(x,y), (x', y') $ in $R$. The  space of units in  the groupoid,
 $R^{0}$, 
consists of all pairs $(y,y)$, $y \in Y$ and we find it convenient to 
identify this 
with $Y$ in the obvious way. Doing this,
 our range and source maps are $r(x,y)=x, 
s(x,y)=y$.  (See \cite{Ren:gpd} and\cite{Wil:gpd}.)
Hence, for any unit $y$, we have 
\[
R^{y} = r^{-1}\{ y \} = \{ y \} \times [y]_{R},
\]
(using the notation of Renault \cite{Ren:gpd}) 
which we identify with $[y]_{R}$. 

For us, the set $Y$ will be a topological space
and our equivalence relations, as groupoids,  must come with their
own topologies. This
is almost never the relative topology from the product space $Y \times Y$.
Let us remark that, 
in general, when we speak about
the topology on equivalences classes,
 we usually mean using the identification 
of the equivalence class with the set  $\{ y \} \times [y]_{R}$ 
(using 
the identification with $[y]_{R} \times \{ y \}$ 
yields the same topology) and the 
relative topology from the equivalence relation rather
 than the topology as a subset of $Y$.
 
 In addition to having topologies, our groupoids must come with a 
 Haar system.
 As the name suggests,  a Haar system is a generalization
of the notion of Haar measure on a group appropriate to groupoids. 
For equivalence relations, 
this amounts to having a collection of measures on the equivalence relation,
 $\nu^{y}$, indexed by 
the points of the underlying space. The support of the measure $\nu^{y}$ is 
the equivalence class of $y$, or more precisely $\{ y \} \times [y]_{R}$.
There are two  important properties for a Haar system.
 The first is a left-invariance condition which, in our case, 
is simply that 
$\nu^{x}=\nu^{y}$ when $(x,y)$ is in $R$. The second condition
is that, for any 
continuous compactly-supported 
function $f$ on $R$, the map sending $y$ in $Y$ to 
$\int f(z) d\nu^{y}(z)$
is continuous. 

Initially, we considered the bi-infinite path space of a Bratteli 
diagram $\mathcal{B}$, which we denoted $X_{\mathcal{B}}$. 
The notions of right and left tail equivalence on 
 $X_{\mathcal{B}}$, $T^{+}(X_{\mathcal{B}})$ and 
  $T^{-}(X_{\mathcal{B}})$, were introduced back in 
 Definition \ref{path:100}.  
 For the rest of the paper we will focus on
 $T^{+}(X_{\mathcal{B}})$.
Definition  \ref{path:100} even included
 the definition for our topology on $T^{+}(X_{\mathcal{B}})$.
 In addition, the collection of measures in 
 Proposition \ref{path:140} provide a Haar system.

In the last section, we 
introduced four new spaces, $Y_{\mathcal{B}},
S_{\mathcal{B}}^{s},  S_{\mathcal{B}}^{r}$ 
and 
 $S_{\mathcal{B}}$, the last being a surface, 
 along with certain maps between them. In addition, a state 
 on the diagram gave us an atlas for the surface.
 Our aim in this section is to transfer the equivalence relation
 $T^{+}(X_{\mathcal{B}})$
to the other spaces by means of our given quotient maps and to 
consider the horizontal foliation on the translation surface.  
We will meet  subtleties along the way.

Our ultimate aim will be to associate $C^{*}$-algebras
with these equivalence relation via the groupoid construction. 
We will discuss this in the next section.

 \subsection{AF-equivalence relations 
 $T^{+}(X_{\mathcal{B}}),T^{+}(Y_{\mathcal{B}} )$}
 
 \text{ }
 \vspace{1cm}

  Our first result concerns the relations of right-tail 
equivalence, $T^{+}(X_{\mathcal{B}})$, and left-tail equivalence, 
$T^{-}(X_{\mathcal{B}})$. We will focus on the former.
Our first result gives some basic information, 
including a nice basis for the topology defined 
in Definition \ref{path:100}, which we repeat in the statement for convenience.
 The result is standard and we omit the proof (see Renault \cite{Ren:gpd}).
 
 \begin{prop}
 \label{gpds:10}
 Let $\mathcal{B}$ be a bi-infinite  Bratteli diagram 
 with faithful state $\nu_{s}, \nu_{r}$. 
 For each integer $N$, we define
 \[
 T^{+}_{N}(X_{\mathcal{B}}) = \{ (x,y) \in X_{\mathcal{B}}^2 \mid 
 x_{(N, \infty)} = y_{(N, \infty}) \},
 \]
 which is endowed with the relative topology from
 $X_{\mathcal{B}} \times X_{\mathcal{B}}$.
Let 
\[
T^{+}(X_{\mathcal{B}}) = \bigcup_{N \in \Z} T^{+}_{N}(X_{\mathcal{B}}) 
\]
 be endowed with the inductive limit topology and
 let $\nu_{r}^{x}, x \in X_{\mathcal{B}}$, be the measures 
 defined in \ref{path:140}. 
 \begin{enumerate}
 \item $T^{+}(X_{\mathcal{B}})$ is a locally compact, Hausdorff
 groupoid.
 \item The collection of measures  
 $\nu_{r}^{x}, x \in X_{\mathcal{B}}$ (Proposition \ref{path:140})
   is a Haar system for 
 $T^{+}(X_{\mathcal{B}})$.
 \item For $m<n$ and $p, q$ in $E_{m,n}$ with $r(p)=r(q)$, the set 
 \[
 T^{+}(p,q) = \{ (x,y) \mid x_{(m,n]} = p, y_{(m,n]}=q, 
 x_{(n, \infty)} =  y_{(n, \infty)}  \}
 \]
 is a compact, open subset of  $T^{+}(X_{\mathcal{B}})$. The map 
 sending $(x,y)$ in $T^{+}(p,q) $ to \newline
   $(x_{(-\infty, m]}, y_{(-\infty, m]}, x_{(n, \infty)})$
 is a homeomorphism from $T^{+}(p,q) $
  to 
  $X_{s(p)}^{-}  \times X_{s(q)}^{-} \times X_{r(p)}^{+}$.
  Moreover, as $m,n, p,q$ vary these sets form a base
  for the topology of $T^{+}(X_{\mathcal{B}})$. 
  \end{enumerate}
 \end{prop}
 
 It will be helpful for
 us to keep track of the individual equivalences classes 
 as we progress. The basic description of $T^{+}(x), x \in X_{\mathcal{B}}$
 is contained in Lemma \ref{order:200} which we summarize here.
 First, $T^{+}(x)$ is linearly ordered by $\leq_{r}$ and its intersection 
 with $\partial_{r}X_{\mathcal{B}}$ is $\Delta_{r}$-invariant.
 The map $\varphi^{x}_{r}: T^{+}(x) \rightarrow \R$ is continuous, 
 order preserving and identifies two points $y,z$ if and only 
 $\Delta_{r}(t) = z$.
 If $T^{+}(x) \cap X^{r-max}_{\mathcal{B}}$ is non-empty. then it 
 is a single point, $y$, then $T^{+}(x) \cap X^{r-min}_{\mathcal{B}}$ 
 is empty and $\varphi^{x}_{r}(T^{+}(x)) = ( -\infty, \varphi^{x}_{r}(y)]$
 If $T^{+}(x) \cap X^{r-min}_{\mathcal{B}}$ is non-empty. then it 
 is a single point, $z$, then $T^{+}(x) \cap X^{r-max}_{\mathcal{B}}$ 
 is empty and $\varphi^{x}_{r}(T^{+}(x)) = [\varphi^{x}_{r}(z), \infty)$.
 In all other cases, $\varphi^{x}_{r}(T^{+}(x)) = \R$.

 Of course, we need to restrict this equivalence relation 
 to the subspace $Y_{\mathcal{B}} \subseteq X_{\mathcal{B}}$.
 
 \begin{defn}
 \label{gpds:15}
  Let $\mathcal{B}$ be a  
  bi-infinite  ordered Bratteli diagram . We define 
 \[
 T^{+}(Y_{\mathcal{B}}) = T^{+}(X_{\mathcal{B}}) 
 \cap \left(Y_{\mathcal{B}} \times Y_{\mathcal{B}}\right)
 \]
 and 
 \[
 T^{-}(Y_{\mathcal{B}}) = T^{-}(X_{\mathcal{B}}) 
 \cap \left( Y_{\mathcal{B}} \times Y_{\mathcal{B}}\right).
 \]
 \end{defn}

We remark that this creates some notational confusion. If $y$ is 
in $Y_{\mathcal{B}}$, does $T^{+}(y)$ refer to its class in 
$T^{+}(X_{\mathcal{B}})$ or in $T^{+}(Y_{\mathcal{B}})$?  
To keep things clearer, we always mean the former so that latter 
is written as $T^{+}(y) \cap Y_{\mathcal{B}}$.

We observe the following general result.

\begin{thm}
\label{gpds:18}
Let $X$ be a locally compact, Hausdorff, topological space 
with an equivalence relation $R$ and a   Haar system $\nu_{x}, x \in X$.
Suppose that $S$ is an open subequivalence relation of $R$. The set 
$Y = \{ y \in X \mid (y,y) \in S \}$ is an open subset of $X$ and 
the collection of measures $\nu^{S}_{y} = \nu_{y} | [y]_{S}$, for 
$y$ in $Y$ is a Haar system of $S$.
\end{thm}

\begin{proof}
The fact that $Y$ is open is clear. As $\nu_{x}$ is a Haar system
for $R$, the support of each measure is $\{x \} \times [x]_{R}$ and 
since $S$ is open, the measure $\nu^{S}_{y} $ will have support
$\{ y \} \times [y]_{S}$, for each $y$ in $Y$. The continuity property
of the measures is immediate.
\end{proof}
 
 We note that if $\mathcal{B}$ is finite rank and strongly
 simple, then $Y_{\mathcal{B}}$ is an 
 open subset of $X_{\mathcal{B}}$, $T^{+}(Y_{\mathcal{B}})$ is an open 
 subequivalence relation of $T^{+}(X_{\mathcal{B}})$ and 
  Proposition \ref{gpds:10} also holds for 
 $T^{+}(Y_{\mathcal{B}})$, if we replace $E_{m,n}$ 
 with $E_{m,n}^{Y}$, in the last condition and use Haar system 
 provided by the last theorem. We will not introduce a new 
 notation for these measures.

  \subsection{Equivalence relation $T^{\sharp}(Y_{\mathcal{B}})$}
 
 \text{ }
 \vspace{1cm}
   
   Ultimately, we want to move our equivalence 
   relations to our quotient spaces where we identify
   $x$ with $\Delta_{r}(x)$ and $y$ with 
   $\Delta_{s}(y)$, for $x$ in 
   $\partial_{r}(X_{\mathcal{B}}) \cap Y_{\mathcal{B}}$ 
   and $y$ in $\partial_{s}(X_{\mathcal{B}}) \cap Y_{\mathcal{B}}$.
   The first poses no real problem since $(x,\Delta_{r}(x))$ 
   lies in $T^{+}(Y_{\mathcal{B}})$. The second does, however.
   This is because if $x,y$ are in 
    $\partial_{s}(X_{\mathcal{B}}) $
   and $(x,y)$ is in $T^{+}( Y_{\mathcal{B}})$,
    $(\Delta_{s}(x), \Delta_{s}(y))$ may \emph{not}  be 
    $T^{+}( Y_{\mathcal{B}})$. This actually occurs in the example
    given in section \ref{Cha}. We make the obvious adjustment.

\begin{defn}
\label{gpds:25}
Let $\mathcal{B}$ be a 
bi-infinite ordered Bratteli diagram satisfying the conditions of
Definition \ref{surface:20}.

We define $T^{\sharp}(Y_{\mathcal{B}})$ to be the subset
of 
$T^{+}(Y_{\mathcal{B}})$ consisting of all pairs 
$(x,y)$ in $T^{+}(Y_{\mathcal{B}})$ satisfying the additional
condition that $(\Delta_{s}(x), \Delta_{s}(y))$ is in 
$T^{+}(Y_{\mathcal{B}})$, if $x,y$ are in $\partial_{s}X_{\mathcal{B}}$.
We let $T^{\sharp}(y)$ denote the equivalence class of $y$ in 
$T^{\sharp}(Y_{\mathcal{B}})$.
\end{defn}

It is clear that $T^{\sharp}(Y_{\mathcal{B}})$ is a subset
of 
$T^{+}(Y_{\mathcal{B}})$. We want to show it is open. In fact, it will
be useful for us to have a local description.

\begin{prop}
\label{gpds:27}
Let $\mathcal{B}$ be a 
bi-infinite ordered Bratteli diagram satisfying the conditions of
Definition \ref{surface:20}. 

 Let  $m<n$ and $p, q$ be in $E_{m,n}^{s}$ (as in 
  Definition \ref{surface:40}) with 
  $r(p)=(r(p^{1}), r(p^{2})) = (r(q^{1}), r(q^{2})) =r(q)$. 
  We define 
\begin{eqnarray*}
T^{\sharp}(p,q) & =  & \left[ \left(V_{1}^{s}(p) \times V_{1}^{s}(q) \right)
 \cap T^{+}_{n}(X_{\mathcal{B}}) \right]
\cup \left[ \left( V_{2}^{s}(p) \times V_{2}^{s}(q) \right) 
\cap T^{+}_{n}( X_{\mathcal{B}}) \right] \\
 & = & \left( T^{+}(p^{1}, q^{1}) 
    -  X_{s(p)}^{-}p^{1}x_{r(p^{1})}^{s-min} \times 
    X_{s(q)}^{-}q^{1}x_{r(q^{1})}^{s-min} \right) \\
   &  &  \cup 
  \left(  T^{+}(p^{2}, q^{2}) -  X_{s(p)}^{-}p^{2}x_{r(p^{2})}^{s-max} \times 
    X_{s(q)}^{-}q^{2}x_{r(q^{2})}^{s-max} \right).
\end{eqnarray*}
We have 
$T^{\sharp}(Y_{\mathcal{B}})$ is an open
 subgroupoid of $T^{+}(Y_{\mathcal{B}})$.
\begin{enumerate}
\item 
If $(x,y)$ is  in   $T^{\sharp}(p,q)$ with $x,y$ in 
$\partial_{s}X_{\mathcal{B}}$, then 
$(\Delta_{s}(x),\Delta_{s}(y) ) $ is in
$T^{\sharp}(p,q)$.
\item $T^{\sharp}(p,q)$ is open in $T^{\sharp}(Y_{\mathcal{B}})$. 
\item 
As $m,n, p, q$ vary the sets $T^{\sharp}(p,q)$
 cover   $T^{\sharp}(Y_{\mathcal{B}})$.
 \end{enumerate}
\end{prop}

\begin{proof} 
  For the first part, let us assume that $(x,y)$ is in 
    $T^{+}(p^{1}, q^{1}) 
    -  X_{s(p)}^{-}p^{1}x_{r(p^{1})}^{s-min} \times 
    X_{s(q)}^{-}q^{1}x_{r(q^{1})}^{s-min}$; the other case is similar.
 If $n(x) > n$, then $n(y)=n(x)$ since $(x,y)$ is in 
 $T^{+}_{n}(X_{\mathcal{B}})$. It is then clear that computing 
 $\Delta_{s}(x)$ and  $\Delta_{s}(y)$ leaves the entries less than $n(x)$
 unchanged and their entries greater than or equal
  to $n(x)$ will be equal. The conclusion follows. 
  If $n(x), n(y) \leq n$, then as the entries 
 greater than $n$ are not $s$-minimal, they must all be $s$-maximal.
 This means $\Delta_{s} =S_{s}$ on $x,y$. Since $p^{1}, q^{1}$ are 
 not $s$-maximal, we have $m \leq n(x), n(y) \leq y$ and 
$\Delta_{s}(x) = x_{(-\infty, m)}p^{2}x_{r(p^{2})}^{s-min}$ . A similar
computation for $\Delta_{s}(y)$ and the fact that
 $r(p^{2})=r(q^{2})$ shows the conclusion.

   It is clear that 
$T^{\sharp}(p,q)$ is an open subset of $T^{+}(p^{1}, q^{1}) \cup 
T^{+}(p^{2}, q^{2})$ since   $X_{s(p)}^{-}px_{r(p^{1})}^{s-min} \times 
    X_{s(q)}^{-}px_{r(q^{1})}^{s-min} $ and 
 $  X_{s(p)}^{-}px_{r(p^{2})}^{s-max} \times 
    X_{s(q)}^{-}px_{r(q^{2})}^{s-max}$ are closed.
 The second part of the conclusion follows from this and the first part.

   The proof of the third part is similar to that of 
   Lemma \ref{surface:55} and we omit the details.
  
  The final statement follows from the first three parts.
\end{proof}
  
 We will now develop a better understanding of $T^{\sharp}(Y_{\mathcal{B}})$.
 The process  raises an interesting issue. A one-sided Bratteli diagrams with 
 a $\leq_{r}$-order is usually called \emph{properly} ordered if
 there is a unique infinite path of all maximal edges, and a unique 
 infinite  path of all minimal edges.
 The first condition is equivalent to the fact that any two infinite paths
 which are $\leq_{r}$-maximal for all but finitely many edges, must be tail
 equivalent. It turns out the the situation is rather different 
 for bi-infinite diagrams.

 Consider the following diagram
$$\begin{tikzpicture}
  \filldraw (1,1) circle (2pt);
    \filldraw (2,1) circle (2pt);
    \filldraw (3,1) circle (2pt);
      \filldraw (4,1) circle (2pt);
      \filldraw (5,1) circle (2pt);
      \filldraw (6,1) circle (2pt);
  \filldraw (7,1) circle (2pt);
   \filldraw (8,1) circle (2pt);
   
    \filldraw (1,2) circle (2pt);
    \filldraw (2,2) circle (2pt);
    \filldraw (3,2) circle (2pt);
      \filldraw (4,2) circle (2pt);
      \filldraw (5,2) circle (2pt);
      \filldraw (6,2) circle (2pt);
  \filldraw (7,2) circle (2pt);
   \filldraw (8,2) circle (2pt);

     \draw (0,2) -- (9,2);
       \draw (0,1) -- (1,2);
         \draw (1,1) -- (2,2);
         \draw (2,1) -- (3,2); 
          \draw (3,1) -- (4,2);
          \draw (4,1) -- (9,1);
          
  \end{tikzpicture}$$
where we show  only the $s$-maximal edges in some bi-infinite ordered
  Bratteli diagram. 
  Note that there is a unique infinite path of $s$-maximal edges, while there 
  are two infinite paths whose edges are all $s$-maximal, for 
  sufficiently large indices, but are not tail-equivalent.
  
    
    On the other hand if we look at the following diagram
    
$$\begin{tikzpicture}
  \filldraw (1,1) circle (2pt);
    \filldraw (2,1) circle (2pt);
    \filldraw (3,1) circle (2pt);
      \filldraw (4,1) circle (2pt);
      \filldraw (5,1) circle (2pt);
      \filldraw (6,1) circle (2pt);
  \filldraw (7,1) circle (2pt);
   \filldraw (8,1) circle (2pt);
   
    \filldraw (1,2) circle (2pt);
    \filldraw (2,2) circle (2pt);
    \filldraw (3,2) circle (2pt);
      \filldraw (4,2) circle (2pt);
      \filldraw (5,2) circle (2pt);
      \filldraw (6,2) circle (2pt);
  \filldraw (7,2) circle (2pt);
   \filldraw (8,2) circle (2pt);

     \draw (0,2) -- (9,2);
       \draw (4,1) -- (5,2);
         \draw (5,1) -- (6,2);
         \draw (6,1) -- (7,2); 
          \draw (7,1) -- (8,2);
          \draw (8,1) -- (9,2);
          \draw (0,1) -- (4,1);
  \end{tikzpicture}$$
where again we  only show the $s$-maximal edges,  
there are two infinite paths
of s-maximal edges, but these are tail equivalent. 

It turns out that the number of distinct tail-equivalence classes
is the important thing here, not the number of paths in 
$X_{\mathcal{B}}^{s-max}$
and this leads to the following proposition.
 
 \begin{prop}
 \label{gpds:30}
 Let $\mathcal{B}$ be a
 bi-infinite ordered Bratteli diagram satisfying the conditions of
Definition \ref{surface:20}. The set 
 $\partial_{s}X_{\mathcal{B}}$ is invariant under
 the equivalence relation $T^{+}(Y_{\mathcal{B}})$ and 
 if $\mathcal{B}$ is finite rank, then it is the union
 of a finite number of equivalence classes.
 More specifically, we may find 
 positive integers $I_{\mathcal{B}}, J_{\mathcal{B}}$,
  $x_{1}, \ldots, x_{I_{\mathcal{B}}}\in Y_{\mathcal{B}}$ 
 such that, for all $i$,  $(x_{i})_{n} $ is $s$-maximal, 
 for all but finitely many $n \geq 0$, and 
 $x_{I_{\mathcal{B}}+1}, \ldots,
  x_{I_{\mathcal{B}}+J_{\mathcal{B}}} \in Y_{\mathcal{B}}$ 
 such that, for all $j$, $(x_{j})_{n} $ is $s$-minimal, 
 for all but finitely many $n \geq 0$, and so that 
 \[
 \partial_{s}X_{\mathcal{B}} =  
 \bigcup_{i=1}^{I_{\mathcal{B}}+J_{\mathcal{B}}}
  T^{+}(x_{i}) 
 \]
 and the sets on the right are pairwise disjoint. 
 \end{prop}
 
 \begin{proof}
 Suppose that $x_{1}, \ldots, x_{I}$ are all eventually $s$-maximal and no 
 two are right-tail equivalent. 
 Then we can find $N$, such that $(x_{i})_{n}$ is $s$-maximal, 
 for all $1 \leq i \leq I$, $n \geq N$. If 
 $s((x_{i})_{N}) = s((x_{j})_{N})$, 
 for some $i, j$, it follows from this fact that $x_{i}$ and $x_{j}$
 are right-tail equivalent and so $i=j$. It follows that 
 $I \leq \# V_{N-1}$. As $\mathcal{B}$ is
  finite rank, we see that $I$ must be bounded
 by the same constant that bounds the size of the sets $V_{n}$.
 A similar argument deals with paths that are eventually $s$-minimal.
 \end{proof}

Looking back at the two examples given above, the first has
$I_{\mathcal{B}}= 2$, while the second has 
 $I_{\mathcal{B}}= 1$.

Our problem can now be summarized by noting  that while
\[
\Delta_{s}: \bigcup_{i=1}^{I_{\mathcal{B}}} T^{+}(x_{i}) \rightarrow 
   \bigcup_{j= I_{\mathcal{B}} +1 }^{I_{\mathcal{B}}+J_{\mathcal{B}}} T^{+}(x_{j}),
   \]
is a bijection,
   it does not respect the decomposition in the unions.
    This is easily remedied
   in the following way.
   
\begin{defn}
\label{gpds:35}
Let $\mathcal{B}$ be a bi-infinite ordered Bratteli diagram 
satisfying the conditions of
Definition \ref{surface:20}
and 
 $I_{\mathcal{B}}, 
J_{\mathcal{B}}, x_{1}, 
\ldots, x_{I_{\mathcal{B}}+J_{\mathcal{B}}}$ 
be as in Proposition \ref{gpds:30}.
We define 
\[
I_{\mathcal{B}} \star_{\Delta} J_{\mathcal{B}} = 
\{ (x_{i},x_{j}), 1 \leq i \leq I_{\mathcal{B}} 
< j \leq I_{\mathcal{B}}+J_{\mathcal{B}} \mid
\Delta_{s}(T^{+}(x_{i})) \cap T^{+}(x_{j}) \neq \emptyset \}.
\]
\end{defn}

 The following is an immediate consequence of the definitions.

\begin{prop}
\label{gpds:40}
Let $\mathcal{B}$ be a bi-infinite ordered Bratteli diagram
satisfying the conditions of
Definition \ref{surface:20}.
 The equivalences classes in 
 $T^{\sharp}(Y_{\mathcal{B}})$ can be listed
 as $T^{+}(x) \cap Y_{\mathcal{B}}$,
  where $T^{+}(x) \cap \partial_{s}X_{\mathcal{B}}$
  is empty and 
 $T^{+}(x_{i}) \cap \Delta_{s}(T^{+}(x_{j}))\cap Y_{\mathcal{B}}$, 
 $\Delta_{s}(T^{+}(x_{i})) \cap T^{+}(x_{j})\cap Y_{\mathcal{B}}$, 
 where $(i,j)$ is in $I_{\mathcal{B}}
 \star_{\Delta} J_{\mathcal{B}}$.
\end{prop}

 Most importantly, the groupoid is now invariant under $\Delta_{s}$ 
   and so we may pass it on to $S^{s}_{\mathcal{B}}$.
   
   We also note the following which follows immediately from  
  Proposition \ref{gpds:40}.
   
   \begin{prop}
   \label{gpds:42}
  Let $\mathcal{B}$ be a bi-infinite ordered Bratteli diagram 
  satisfying the conditions of
Definition \ref{surface:20}.
We have  $T^{\sharp}(Y_{\mathcal{B}}) = T^{+}(Y_{\mathcal{B}})$
if and only if $I_{\mathcal{B}}= J_{\mathcal{B}}$ and, 
for each  $ 1 \leq i \leq I_{\mathcal{B}} $ there is a 
unique  $ 1 \leq j \leq J_{\mathcal{B}} $ with $(i,j)$ in 
$I_{\mathcal{B}} \star_{\Delta} J_{\mathcal{B}}$. 
In particular, if $I_{\mathcal{B}}= J_{\mathcal{B}}=1$ then
$T^{\sharp}(Y_{\mathcal{B}}) = T^{+}(Y_{\mathcal{B}})$.
   \end{prop}

Recall that each equivalence class in $T^{+}(Y_{\mathcal{B}})$ 
is linearly ordered by $\leq_{r}$. Our final result for
this subsection relates this order with equivalence classes
of $T^{\sharp}(Y_{\mathcal{B}})$. The following is an immediate consequence of
Proposition \ref{singular:150}.

\begin{prop}
\label{gpds:45}
For $(x, y)$  in $T^{+}(X_{\mathcal{B}})$, if $[x,y]_{r}$ is 
contained in $ Y_{\mathcal{B}}$, then it
 is contained in a single
 $T^{\sharp}(Y_{\mathcal{B}})$ equivalence class.
\end{prop}

\vspace{.5cm}

\subsection{Equivalence relation $T^{\sharp}( S^{s}_{\mathcal{B}})$}
 
 \text{ }
 \vspace{.5cm}

We now take the quotient by the map $\pi^{s}: Y_{\mathcal{B}} 
\rightarrow S^{s}_{\mathcal{B}}$. By definition, the 
equivalence relation $T^{\sharp}(Y_{\mathcal{B}})$ is preserved under
this quotient map. 

\begin{defn}
\label{gpds:50}
We define $T^{\sharp}(S^{s}_{\mathcal{B}})$ to be 
$\pi^{s} \times \pi^{s}( T^{\sharp}(Y_{\mathcal{B}}))$ and endow it with 
the quotient topology. For each $z$ in $S^{s}_{\mathcal{B}}$, we denote
its class in $T^{\sharp}(S^{s}_{\mathcal{B}})$ by $T^{\sharp}(z)$.
\end{defn}

We first need a local description of
the quotient  analogous to Proposition \ref{gpds:28}. 
In fact, this is an immediate consequence of 
\ref{gpds:28} and the definitions and Lemma \ref{surface:70}.

\begin{prop}
\label{gpds:55}
 Let  $m<n$ and $p, q$ be in $E_{m,n}^{s}$ (as in 
  Definition \ref{surface:40}) with 
  $r(p)=(r(p^{1}), r(p^{2})) = (r(q^{1}), r(q^{2})) =r(q)$. 
\begin{enumerate}
\item $\pi^{s} \times \pi^{s}(T^{\sharp}(p,q))$ is open in 
$T^{\sharp}(S^{s}_{\mathcal{B}})$. 
\item 
As $m,n, p, q$ vary the sets $\pi^{s} \times \pi^{s}(T^{\sharp}(p,q))$
 cover   $T^{\sharp}(S^{s}_{\mathcal{B}})$.
 \item 
The map sending $(\pi^{s}(x), \pi^{s}(y))$ in 
$\pi^{s} \times \pi^{s}(T^{\sharp}(p,q))$ to 
$(x_{(-\infty, m]}, y_{(-\infty, m]}, \psi^{p,r}(x_{(n, \infty)})$
is a homeomorphism to 
$X_{s(p^{1})}^{-} \times X_{s(q^{1})}^{-}  
\times (-\nu_{s}(r(p^{1}), \nu_{s}(r(p^{2}))$.
  \item  
 The map 
 $\pi^{s} \times \pi^{s}: T^{\sharp}(Y_{\mathcal{B}})
  \rightarrow T^{\sharp}(S^{s}_{\mathcal{B}})$ is continuous  and proper.
 \end{enumerate}
\end{prop}

Finally, we list the equivalence classes for 
$T^{\sharp}(S^{s}_{\mathcal{B}})$ which is an immediate consequence of 
Proposition \ref{gpds:40} and 
the definition of $\pi^{s}$. The last statement is an immediate consequence
of Proposition \ref{singular:150}.

\begin{prop}
\label{gpds:60}
Let $\mathcal{B}$ be a bi-infinite ordered Bratteli diagram
satisfying the conditions of
Definition \ref{surface:20}.
 The equivalences classes in 
 $T^{\sharp}(S^{s}_{\mathcal{B}})$ can be listed
 as $\pi^{s}(T^{+}(x)\cap Y_{\mathcal{B}})$,
  where $T^{+}(x) \cap \partial_{s}X_{\mathcal{B}}$
  is empty and 
 \[
 \pi^{s}\left(T^{+}(x_{i}) \cap \Delta_{s}(T^{+}(x_{j})\cap Y_{\mathcal{B}})\right)
 = \pi^{s}\left( 
 \Delta_{s}(T^{+}(x_{i})) \cap T^{+}(x_{j})\cap Y_{\mathcal{B}}\right),
 \] 
 where $(i,j)$ is in $I_{\mathcal{B}}
 \star_{\Delta} J_{\mathcal{B}}$.
 The restriction of $\pi^{s}$ to  each equivalence class is 
 a homeomorphism to its image.
 
 If $E$ is a Borel subset of 
 $T^{+}(x_{i}) \cap \Delta_{s}(T^{+}(x_{j})\cap Y_{\mathcal{B}}$, then 
 $\nu_{r}^{x_{i}}(E) = \nu_{r}^{x_{j}}(\Delta_{s}(E))$.
  For each $y$ in $Y_{\mathcal{B}}$, we define 
 $\nu_{r}^{\pi^{s}(y)}(\pi^{s}(E)) = \nu_{r}^{y}(E)$, for each 
 Borel set $E$  in 
 $T^{\sharp}(y)$. Then this is a well-defined
 Haar system for $T^{\sharp}(S^{s}_{\mathcal{B}})$.
\end{prop}

\vspace{.5cm}

\subsection{Equivalence relation $T^{\sharp}( S_{\mathcal{B}})$}
 
 \text{ }
 \vspace{.5cm}

We now want to move the groupoid $T^{\sharp}(Y_{\mathcal{B}})$ to 
our surface, $S_{\mathcal{B}}$. Recall that we denote the quotient  map
by $\rho^{r}: S^{s}_{\mathcal{B}} \rightarrow S_{\mathcal{B}}$ 
and 
$\pi = \rho^{r} \circ \pi^{s} : 
Y_{\mathcal{B}} \rightarrow S_{\mathcal{B}}$ .

\begin{defn}
\label{gpds:70}
Let $\mathcal{B}$ be a 
 bi-infinite ordered Bratteli diagram satisfying the conditions of
Definition \ref{surface:20}.
 We define
 \[
 T^{\sharp}(S_{\mathcal{B}})   =  \rho^{r}\times \rho^{r}
  (T^{\sharp}(S^{s}_{\mathcal{B}})) \\
    = \pi \times  \pi (T^{\sharp}(Y_{\mathcal{B}}))
 \] 
 and endow it with the quotient topology. For each $z$ in 
 $S_{\mathcal{B}}$, we denote
its class in $T^{\sharp}(S_{\mathcal{B}})$ by $T^{\sharp}(z)$.
\end{defn}

\begin{prop}
\label{gpds:72}
Let $\mathcal{B}$ be a 
 bi-infinite ordered Bratteli diagram
 satisfying the conditions of
Definition \ref{surface:20}. 
 Let $p, q$ be in $E_{m,n}^{r/s}, m < n$ satisfy $r(p)=r(q)$.
 
 We define 
 \begin{eqnarray*}
 U^{\sharp}(p,q) & = & ( T^{\sharp}((p^{1,1},p^{1,2}), (q^{1,1},q^{1,2}) )
 - x_{s(p^{1,1})}^{r-min}X^{+}_{s(p^{1,1})} \times 
 x_{s(q^{1,1})}^{r-min}X^{+}_{s(q^{1,1})})  \\
   &   &  \cup (T^{\sharp}((p^{2,1},p^{2,2}), (q^{2,1},q^{2,2}) )
 - x_{s(p^{2,1})}^{r-max}X^{+}_{s(p^{2,1})} \times 
 x_{s(q^{2,1})}^{r-max}X^{+}_{s(q^{2,1})}). 
 \end{eqnarray*}

 \begin{enumerate}
\item   $U^{\sharp}(p,q)$ is an open subset of 
$T^{\sharp}(Y_{\mathcal{B}})$.
\item Let $(x,y)$ be in  $U^{\sharp}(p,q)$. If $x,y$ are 
in $\partial_{r}X_{\mathcal{B}}$ then 
$(\Delta_{r}(x), \Delta_{r}(y))$ is also in  $U^{\sharp}(p,q)$. 
 If $x,y$ are 
in $\partial_{s}X_{\mathcal{B}}$ then 
$(\Delta_{s}(x), \Delta_{s}(y))$ is also in  $U^{\sharp}(p,q)$. 
 \item For each $y$ in $Y_{\mathcal{B}}$ and Borel set
 $E$ in  $ T^{\sharp}(y)$,
 we define 
 $\nu_{r}^{\pi(y)}(\pi(E)) = \nu_{r}^{y}(E) = \nu_{r}^{\pi^{s}(y)}(\pi^{s}(E))$.
 The system of measures $\nu_{r}^{z}, z \in S_{\mathcal{B}}$ is a Haar
 system for 
 is a Haar system for 
 $T^{\sharp}(S_{\mathcal{B}})$. In addition, 
 for each $y$ in $Y_{\mathcal{B}}$, the map $\rho_{r}$ from 
 $(T^{\sharp}( \pi^{s}(y)), \nu_{r}^{\pi^{s}(y)})$ 
 to $(T^{\sharp}(\pi(y)), \nu_{r}^{\pi(y)})$ 
 is an isomorphism of measure spaces.
\item The map sending $(\pi(x), \pi(y))$, for
 $(x,y)$ in $U^{\sharp}(Y_{\mathcal{B}})$, to \newline
$(\psi^{p}_{r}(x), \psi^{q}_{r}(y), \psi^{p}_{s}(x))$ is 
a homeomorphism from  $\pi \times \pi(U^{\sharp}(p,q))$ to 
$(-\nu_{r}(s(p^{1,1}), \nu_{r}(s(p^{2,1}))) \times  
(-\nu_{r}(s(q^{1,1}), \nu_{r}(s(q^{2,1}))) \times 
( -\nu_{s}(r(p^{1,1}), \nu_{s}(r(p^{1,2}))$.
  \item The maps  
  $\rho^{r} \times \rho^{r}: T^{\sharp}( S^{s}_{\mathcal{B}})
  \rightarrow T^{\sharp}(S_{\mathcal{B}})$ and 
 $\pi\times \pi: T^{\sharp}(Y_{\mathcal{B}})
  \rightarrow T^{\sharp}(S_{\mathcal{B}})$ are continuous 
   and proper.
\end{enumerate}
\end{prop}

We now want a description of the equivalence classes in 
$T^{\sharp}(S_{\mathcal{B}})$. Let $x$ be in $Y_{\mathcal{B}}$.
First, recall that, if $x$ is in $\partial_{r}X_{\mathcal{B}}$, 
$\Delta_{r}(x)$ is in $T^{+}(x)$. We also recall 
Proposition \ref{order:200} which defines a function
$\varphi^{x}_{r}: T^{+}(x) \rightarrow \R$. it is continuous, proper
and identifies two points $x, y$ if and only if 
$x$ is in $\partial_{r}X_{\mathcal{B}}$ and 
$y =\Delta_{r}(x)$. In addition, the range is either a closed semi-infinite
interval or $\R$.  It remains for us to remove the points
of $X_{\mathcal{B}}^{ext} \cup \Sigma_{\mathcal{B}}$.

\begin{prop}
\label{gpds:22}
Let $x$ be in $Y_{\mathcal{B}}$. The set 
$ \varphi^{x}_{r}(T^{+}(x) \cap Y_{\mathcal{B}})$ is 
an open subset of the real numbers and hence consists of a countable 
collection of open intervals.
Let $W$ be a subset of $ T^{+}(x) \cap Y_{\mathcal{B}}$ whose
image under $\varphi^{x}_{r}$ is one of these open intervals.
Then $W$ is closed
in $Y_{\mathcal{B}}$ if and only if the interval is bounded and is dense
in $Y_{\mathcal{B}}$ otherwise. Moreover, 
$ \varphi^{x}_{r}(T^{+}(x) \cap Y_{\mathcal{B}})$ has a
 bounded interval  if and only if
$T^{+}(x) \cap \left( X^{ext}_{\mathcal{B}} 
\cup \Sigma_{\mathcal{B}} \right)$  has at least two points.
\end{prop}

\begin{proof}
From Proposition \ref{order:200}, if $T^{+}(x)$ meets
 $X^{r-max}_{\mathcal{B}}$ then it does so at a single point, say $z$.
 The range of $\varphi^{x}_{r}(T^{+}(x)) = (-\infty, \varphi^{x}_{r}(z)]$.
 We know from Proposition \ref{order:5} that
   $X_{\mathcal{B}}^{ext}$ is finite and contains $z$
    and from Lemma \ref{singular:100}
   that $\Sigma_{\mathcal{B}}$ is countable and its only limits
   points are in  $X_{\mathcal{B}}^{ext}$. The first part 
   of the conclusion  follows. The cases that  
$T^{+}(x)$ meets
 $X^{r-max}_{\mathcal{B}}$ and that 
 $T^{+}(x) \cap X_{\mathcal{B}}^{r-max}$
 and $T^{+}(x) \cap X_{\mathcal{B}}^{r-min}$ 
 are empty are done in a similar way. We omit the details.
 
For the second part, if $\varphi^{x}_{r}(W)$ is a bounded interval, then 
it equal $(\varphi^{x}_{r}(y), \varphi^{x}_{r}(z)$, where $y,z$
are in  $X_{\mathcal{B}}^{ext} \cup \Sigma_{\mathcal{B}}$. It is 
clear that $W = I_{r}(y,z)$ ( as in Proposition
 \ref{gpds:45})  is equal to 
$ \{ w \in X_{\mathcal{B}} \mid y \leq_{r} w \leq_{r} z\} 
\cap Y_{\mathcal{B}}$ which is clearly closed in $Y_{\mathcal{B}}$

To prove that $W$ is dense in $Y_{\mathcal{B}}$ when the image is 
unbounded, there are three cases to consider, depending on which 
of $T^{+}(x) \cap X_{\mathcal{B}}^{r-max}$
 and $T^{+}(x) \cap X_{\mathcal{B}}^{r-min}$ are empty.
 We consider the case the first is empty and leave the other
 to the reader. The hypothesis, along with Lemma \ref{order:200}, 
 implies that we have $y$ in $T^{+}(x)$ such that 
 \[
 \{ z \in T^{+}(x) \mid y \leq _{r} z \} \cap \left( 
 X_{\mathcal{B}}^{ext} \cup \Sigma_{\mathcal{B}} \right)
 \]
 is empty. We claim the set  $\{ z \in T^{+}(x) \mid y \leq _{r} z \}$ 
 is dense in $X_{\mathcal{B}}$, from 
 which the conclusion follows.
 
  Let $m$ be a positive integer and let $p$ be any path in $E_{-m,m}$.
  As $\mathcal{B}$ is strongly simple, we may find $n >m$ such that
  there is a path from $r(p_{m})$ to every vertex of $V_{n}$.
  If $y_{i}$ is $r$-maximal for every $i \geq n$, with $y'_{i}=y_{i}$, 
  for $i \geq n$ and inductively defining
  $y'_{n-j} $ to be the unique $r$-maximal 
  edge with range $s(y'_{n-j+1})$ for 
  all $j \geq 1$, we see that $y'$ is in
   $X_{\mathcal{B}}^{r-max} \cap T^{+}(x)$, 
   which we assumed to be empty. Hence, we can find
   $i \geq n$ with $y_{i}$ not $r$-maximal. Define $z$ as 
   follows: $z_{j} = y_{j}$, for $j > n$, $z_{i}$ to be the 
   $r$-successor of $y_{i}$, $z_{(n,i)}$ to be any path from 
   $r(p_{m})$ to $s(z_{i})$, $z_{[-m,m]}= p_{-m,m]}$ and
   $z_{(-\infty, m)}$ any path with range $s(p_{-m})$. Then 
   $z$ is in $T^{+}(x)$, $z >_{r} y$ and $z_{[-m,m]}= p_{[-m,m]}$.
   This establishes the claim.  
   
   The last statement is now trivial. 
\end{proof}

The following result follows immediately.

\begin{prop}
\label{gpds:28}
Each equivalence class in 
$ T^{\sharp}( S_{\mathcal{B}})$ consists of a countable 
collection of open intervals.
\end{prop}

 \vspace{.5cm}

\subsection{The foliation $\mathcal{F}^{+}( S_{\mathcal{B}})$}
 
 \text{ }
 \vspace{.5cm}

\begin{defn}
\label{gpds:590}
Let $\mathcal{B}$ be a 
bi-infinite ordered Bratteli diagram 
satisfying the conditions of
Definition \ref{surface:20}. We define 
$\mathcal{F}_{\mathcal{B}}^{+}$ to be the open subequivalence 
relation 
of $T^{\sharp}(S_{\mathcal{B}})$ whose equivalence classes
are the path
 connected components of the  equivalence classes of 
 $T^{\sharp}(S_{\mathcal{B}})$.
 For any $x$ in $S_{\mathcal{B}}$, we denote its equivalence  
 class in $\mathcal{F}_{\mathcal{B}}^{+}$
  by $\mathcal{F}^{+}_{\mathcal{B}}(x)$.
\end{defn}

We remark that, in consequence of the description of the 
equivalence classes of $T^{\sharp}(S_{\mathcal{B}})$ given in 
\ref{gpds:28}, there is no distinction between path connected and 
connected.

The next result shows that $\mathcal{F}_{\mathcal{B}}^{+}$
 is the horizontal foliation
for our surface $S_{\mathcal{B}}$, when equipped with the charts of 
Theorem \ref{surface:120}.

\begin{thm}
\label{gpds:600}
Let $p$ be in $E_{m,n}^{r/s}$ for $m < n$ and let 
$u, v$ be in $Y(p) \subseteq S_{\mathcal{B}}$ 
(see Theorem \ref{surface:120}).
Then $(u,v)$ is in  $\pi \times \pi(U^{\sharp}(p,p))$ if and only if 
$\eta^{p}(u)$ and $\eta^{p}(v)$ lie on the same horizontal line.
\end{thm}

\begin{proof}
First assume  $(u,v)$ is in  $\pi \times \pi(U^{\sharp}(p,p))$.
 By Definition \ref{gpds:50}, for $i=1$ or $i=2$, we have 
$u=\pi(x), v=\pi(y)$, with 
$(x,y)$ in $T^{\sharp}((p^{i,1},p^{i,2}), (p^{i,1},p^{i,2}) )$.
Again by definition (\ref{gpds:27}), we have $j=1$ or $j=2$ such that 
$(x,y)$ is in 
$V_{i,j}(p) \times V_{i,j}(p) \cap T^{+}_{n}(X_{\mathcal{B}})$.
For the moment, assume $j=2$ as the other case is similar.

According to the definition (\ref{surface:80}), considering only 
the 
$y$-coordinate
\[
\eta^{p}(u)_{2} = \eta^{p}(\pi(x))_{2} = \psi^{p}(x) = 
\varphi_{s}^{r(p^{i,2})}(x_{(n, \infty)}) = \varphi_{s}^{r(p^{i,2})}(y_{(n, \infty)}) = \eta^{p}(v)_{2}.
\]
since $(x,y)$ is in $T^{+}_{n}(X_{\mathcal{B}})$. Of course, this means they lie on the same horizontal line.

For the converse, let us  assume without loss of generality that
$\eta^{p}(u)$ and $\eta^{p}(v)$ lie in the upper half plane. 
So we can find $i=1$ or $i=2$ and $x$ in $V_{i,2}(p)$ with 
$\psi^{p}(x) = \eta^{p}(u)$. Similarly, there is $j$, $y$ 
in $V_{j,2}(p)$ with $\psi^{p}(x) = \eta^{p}(v)$. Then considering the 
$y$-coordinate, we have 
\[
\varphi^{r(p^{1,2})}_{s}(x_{(n, \infty)}) = 
 \eta^{p}(u)_{2} = \eta^{p}(v)_{2} =
 \varphi^{r(p^{j,2})}_{s}(y_{(n, \infty)}).
 \]
 Of course, $r(p^{1,2})= r(p^{j,2})$ and part 4 of Lemma \Ref{order:200}
 implies that either $x_{(n, \infty)} = y_{(n, \infty)}$ 
 or one is the $s$-successor of 
 the other, in which case $y$ is 
 in $\partial_{s}X_{\mathcal{B}}$ and we replace it by $\Delta_{s}(y)$.
The result is a pair $(x,y)$ in $U^{\sharp}(p,p)$ with 
$\pi \times \pi(x,y) = (u,v)$.
\end{proof}

\begin{thm}
\label{gpds:610}
If $\mathcal{B}$ is a bi-infinite ordered Bratteli diagram
satisfying the conditions of
Definition \ref{surface:20}, 
$\mathcal{F}_{\mathcal{B}}^{+}$ is an open subgroupoid of
$T^{\sharp}(S_{\mathcal{B}})$.
\end{thm}

\begin{proof}
If $n \geq 1$ and  $p$ is in $E_{-n,n}^{r/s}$, then 
$\eta^{p}(Y(P)) = \psi^{p}(V(p))$ is an open rectangle from
Lemma \ref{surface:100}.  It follows from  Theorem 
\ref{gpds:600} that a pair of points in the image 
are in the image $\pi \times \pi(U^{\sharp}(p,p))$ if and only if
their images under $\eta^{p}$
 line on the same horizontal line. Since the horizontal lines
in an open rectangle
are connected, we see that any pair $(x,y)$ in 
$\pi \times \pi(U^{\sharp}(p,p))$ also lies in 
$\mathcal{F}_{\mathcal{B}}^{+}$. 

Now suppose that $(x,y)$ is in $\mathcal{F}_{\mathcal{B}}^{+}$. 
We may find a continuous function $h$ from $[0,1]$ to the 
class of $x$ in $\mathcal{F}_{\mathcal{B}}^{+}$ with $h(0)=x, h(1)=y$.
The points in the image of $h$ may be covered by sets of the form 
$Y(p)$,  $p$ in $E_{-n,n}^{r/s}$, $n \geq 1$. We extract
a finite subcover corresponding to
$p^{1}, \ldots, p^{k}$ and order them so that there is $x^{i}$ in 
$Y(p^{i})$ for $1 \leq i \leq k$ with  $x^{1} = x, 
x^{k}=y$ and $(x^{i}, x^{i+1})$ on the same horizontal line in 
$\eta^{p}(Y(p))$. If we then look at the set of all 
 $(z^{1}, \ldots, z^{k})$ such that there exist $(z^{i}, z^{i+1})$ 
 is in $\pi \times \pi(U^{\sharp}(p^{i},p^{i}))$ for $1 \leq i < k$, 
 the set of pairs $(z^{1}, z^{k})$ is open 
 in $T^{\sharp}(S_{\mathcal{B}})$ and contained in 
 $\mathcal{F}_{\mathcal{B}}^{+}$.
\end{proof}

\begin{thm}
\label{gpds:510}
Let $\mathcal{B}$ be a  
bi-infinite ordered Bratteli diagram satisfying the conditions of
Definition \ref{surface:20}. The foliation 
$\mathcal{F}_{\mathcal{B}}^{+}$ is minimal if and only if
the equivalence relation $T^{+}(X_{\mathcal{B}})$ is trivial 
on $X_{\mathcal{B}}^{ext} \cup \Sigma_{\mathcal{B}}$.
\end{thm}

\begin{proof}
This is an immediate consequence of Proposition 
\ref{gpds:22}.
\end{proof}

\section{$C^{*}$-algebras}
\label{Cstar}
We now begin our investigations into the  various $C^{*}$-algebras
associated with the groupoids of the last section.

We begin with a general discuss of the construction 
of the $C^{*}$-algebra from an equivalence relation. 
We assume that all of our spaces are locally compact and Hausdorff.
Let $Y$ be a 
topological space, $R \subseteq Y \times Y$ be an
 equivalence relation on $Y$ such that the product map from $R \times R$ 
 sending $((x,y),(y,z))$ to $(x,z)$ is continuous. We also  
 suppose we have a Haar system; that is, a collection of measures
 $\nu_{x}, x \in Y$ on $R$ such that  $\nu_{x}$ is supported
 on $[x]_{R}$, $\nu_{x}=\nu_{y}$ whenever $(x,y)$ is in $R$ and, 
 for any continuous function of compact support on $R$, $f$, 
  the function 
 sending $x$ in $Y$ to 
$\int  f(x,z) d\nu_{y}(z)$ is continuous.

These measures are then used to turn the linear 
space of compactly-supported continuous complex-valued functions
on $R$, denoted $C_{c}(R)$,
 into an algebra with the product of two elements,
  $f, g$, given by the formula;
\[
(f \cdot g) (x,y) = \int_{z \in [x]_{R}}  f(x,z)g(z,y) d\nu_{x}(z),
\]
for $(x,y)$ in $R$. The hypotheses on the Haar system is
 needed  to see that the product $f \cdot g$ 
is again continuous and compactly-supported.

For the uninitiated reader,
 it is probably a good idea at this point
to think of the example where 
$Y =  \{ 1, 2, \ldots, n\}$ and $R = Y \times Y$.
The Haar system is  counting measure on each equivalence class 
and the product above is simply matrix multiplication.

We can also define an involution as follows: for $f$ 
in $C_{c}(R)$, 
\[
f^{*}(x,y) = \overline{f(y,x)},
\]
for $(x,y)$ in $R$. In the finite case above, this is simply
the conjugate transpose of the matrix.

To obtain a $C^{*}$-algebra, we need to define a norm on this algebra
and take then its completion. All of our equivalence relations are amenable
and so this norm is actually unique.  However, we
 do not give a proof of this here. Instead we consider only the 
 norm from the left regular representation and its completion 
 which is the reduced $C^{*}$-algebra. We explain as follows.
 
 For each $y$ in $Y$, we consider the Hilbert 
 space $L^{2}([y]_{R}, \nu_{y})$
 and we define a representation $\lambda_{y}$ of $C_{c}(R)$ as operators
 on this Hilbert space by setting
 \[
 (\lambda_{y}(f)\xi)(x) = \int f(x,z)\xi(z) d\nu_{y}(z),
 \]
 for $f$ in $C_{c}(R)$, $\xi$ in  $L^{2}([y]_{R}, \nu_{y})$ and $x$ in $Y$.
 This is a bounded operator and
 \[
 \Vert f \Vert_{red} = \sup \{ \Vert \lambda_{y}(f) \Vert \mid y \in Y \} 
 \]
 is finite. The completion of $C_{c}(R)$ in this
  norm is $C^{*}_{\lambda}(R)$.

Now, we turn to our equivalence relations of interest on our various spaces.
We can summarize the results of the last section with 
a simple schematic showing our equivalence relations:
\[
T^{+}(X_{\mathcal{B}}) \supseteq T^{+}(Y_{\mathcal{B}}) 
 \supseteq T^{\sharp}(Y_{\mathcal{B}}) 
  \stackrel{\pi^{s} \times \pi^{s}}{\longrightarrow} 
  T^{\sharp}(S^{s}_{\mathcal{B}}) 
   \stackrel{\rho^{r} \times \rho^{r}}{\longrightarrow} 
  T^{\sharp}(S_{\mathcal{B}}) 
   \supseteq \mathcal{F}_{\mathcal{B}}^{+}.
\]
Here, each containment is as an open subequivalence relation
and each map is a continuous proper surjection which maps
equivalence classes surjectively to equivalence classes.
The  most obviously important ones are the first and last:
 those associated with 
right  tail equivalence on the Bratteli diagram  and 
horizontal 
foliation of the surface.

We begin with a general result on the construction.

\begin{thm}
\label{Cstar:10}
Let $X$ be a locally compact, Hausdorff, topological space 
with an equivalence relation $R$ and a   Haar system $\nu_{x}, x \in X$.
Suppose that $S$ is an open subequivalence relation of $R$. The set 
$Y = \{ y \in X \mid (y,y) \in S \}$ is an open subset of $X$ and 
the collection of measures $\nu^{S}_{y} = \nu_{y} | [y]_{S}$, for 
$y$ in $Y$ is a Haar system of $S$.
Then the natural inclusion $C_{c}(S) \subseteq C_{c}(R)$ extends to
an inclusion $C^{*}_{\lambda}(S) \subseteq C^{*}_{\lambda}(R)$.
\end{thm}

\begin{proof}
It is a simple matter to check that that the inclusion
$C_{c}(S) \subseteq C_{c}(R)$ is not only linear but also
preserves the product and involution. Finally, for any $x$ in 
$X$, we can find a subset $Y_{x}$ such that 
\[
[x]_{R} \cap Y = \cup_{y \in Y_{x}} [y]_{S},
\]
and the sets on the right are pairwise disjoint.
It follows that 
$L^{2}([x]_{R}, \nu_{x}) = \oplus_{y \in Y_{x}}
L^{2}([y]_{S}, \nu_{y}^{S}) \oplus \mathcal{N}$,
where $\mathcal{N}$ is the orthogonal complement  of 
the direct sum. If $f$ is any function in $C_{c}(S)$, then 
$\lambda_{x}(f)$ is zero on $\mathcal{N}$ and leaves each
summand  $L^{2}([y]_{S}, \nu_{y}^{S}) $ invariant. Moreover, the 
restriction of $\lambda_{x}(f)$ to $L^{2}([y]_{S}, \nu_{y}^{S}) $
is simply $\lambda_{y}(f)$. Taking the supremum of the norms of
all $\lambda_{x}(f)$, we see that the inclusion is actually
isometric for the reduced norms.
\end{proof}

\begin{cor}
\label{Cstar:20}
Let $\mathcal{B}$ be a  bi-infinite ordered Bratteli diagram
satisfying the conditions of
Definition \ref{surface:20}.
We have 
$C^{*}_{\lambda}(T^{\sharp}(Y_{\mathcal{B}})) \subseteq 
  C^{*}_{\lambda}(T^{+}(Y_{\mathcal{B}})) \subseteq 
C^{*}_{\lambda}(T^{+}(X_{\mathcal{B}}))$
and 
$C^{*}_{\lambda}( \mathcal{F}_{\mathcal{B}}^{+} ) \subseteq 
C^{*}_{\lambda}(T^{\sharp}(S_{\mathcal{B}}))$.
\end{cor}

We next turn to the two factor maps. These are slightly 
different and we must
deal with each individually.
 
\begin{thm}
\label{Cstar:30}
Let $\mathcal{B}$ be a  bi-infinite ordered Bratteli diagram
satisfying the conditions of
Definition \ref{surface:20}. 

 The map sending $f$ in $C_{c}(T^{\sharp}(S^{s}_{\mathcal{B}}))$
 to $f \circ( \pi^{s} \times \pi^{s}) $
 in   $C_{c}(T^{\sharp}(Y_{\mathcal{B}})) $
 extends to an inclusion
$C^{*}_{\lambda}(T^{\sharp}(S^{s}_{\mathcal{B}})) \subseteq 
  C^{*}_{\lambda}(T^{\sharp}(Y_{\mathcal{B}}))$.
\end{thm}

\begin{proof}
If $\pi^{s}(y) = \pi^{s}(y')$ for some $y \neq y'$ in $Y_{\mathcal{B}}$, 
then $y$ is in $\partial_{s}X_{\mathcal{B}}$ and 
$\Delta_{s}(y) = y'$. It follows from Proposition 
\ref{singular:150} that  $\nu_{r}^{y}(E)= \nu_{r}^{y'}(\Delta_{s}(E))$
and so the Haar system is well-defined. In addition, $\pi_{s}$
is a homeomorphism from the equivalence class of $y$ 
in  $T^{\sharp}(Y_{\mathcal{B}})$ to the equivalence class
of $\pi_{s}(y)$ in  $T^{\sharp}(S^{s}_{\mathcal{B}})$.
sending $f$ in $C{c}(T^{\sharp}(S^{s}_{\mathcal{B}}))$
 to $f \circ( \pi^{s} \times \pi^{s}) $
 in   $C_{c}(T^{\sharp}(Y_{\mathcal{B}})) $ is a $*$-homomorphism.
 
It also induces a unitary equivalence between the 
representation $\lambda_{y}$ of $ C^{*}_{\lambda}(T^{\sharp}(Y_{\mathcal{B}}))$
and $\lambda_{\pi_{s}(y)}$ of 
$C^{*}_{\lambda}(T^{\sharp}(S^{s}_{\mathcal{B}}))$ and hence 
the map
is isometric.
\end{proof}

\begin{cor}
\label{Cstar:40}
Let $\mathcal{B}$ be a bi-infinite ordered Bratteli diagram
satisfying the conditions of
Definition \ref{surface:20}. 
 For each $y$ in $S^{s}_{\mathcal{B}}$, defining
 $\nu^{\rho^{r}(y)}_{r}(\rho^{r}(E)) = \nu^{y}_{r}(E)$, 
 for each Borel set $E$ with $\{y\} \times E$ in 
  $T^{\sharp}(S^{s}_{\mathcal{B}})$ is a Haar system
  for $T^{\sharp}(S_{\mathcal{B}})$.
  
 The map sending $f$ in $C_{c}(T^{\sharp}(S_{\mathcal{B}}))$
 to $f \circ( \rho^{r} \times \rho^{r}) $
 in   $C_{c}(T^{\sharp}(S^{s}_{\mathcal{B}})) $
 extends to an 
isomorphism $
C^{*}_{\lambda}(T^{\sharp}(S_{\mathcal{B}})) \cong
C^{*}_{\lambda}(T^{\sharp}(S^{s}_{\mathcal{B}}))$.
\end{cor}

\begin{proof}
By Proposition \ref{gpds:72},
 $\rho_{r}$, when restricted to a single equivalence 
class of $T^{\sharp}(S^{s}_{\mathcal{B}})$, maps surjectively to a single 
equivalence class of  $T^{\sharp}(S_{\mathcal{B}})$. Moreover, it 
is an isomorphism at the level of 
measure spaces with  $\nu^{\rho^{r}(y)}_{r}$ as defined.

It is a simple computation to see that the map sending $f$ in  
 $C_{c}(T^{\sharp}(S_{\mathcal{B}}))$
 to $f \circ( \rho^{r} \times \rho^{r}) $
 in   $C_{c}(T^{\sharp}(S^{s}_{\mathcal{B}})) $ is a $*$-homomorphism.
 An argument similar to that in the proof of Proposition \ref{Cstar:30}
 shows that it is injective. To show that the map
 is surjective on the completion, we must show the range 
 of $C_{c}(T^{\sharp}(S_{\mathcal{B}})) $  is dense 
 in  $C_{c}(T^{\sharp}(S^{s}_{\mathcal{B}})) $.
 
 Let $m < n$ and $p,q$ be in $E^{r/s}_{m,n}$ with $r(p)=r(q)$. 
 Let 
 \begin{eqnarray*}
 f: X^{-}_{s(p^{1,1})} & \rightarrow & \C, \\
 g: X^{-}_{s(q^{1,1})}  & \rightarrow & \C,\\
 h: (-\nu_{s}(p^{1,1}), \nu_{s}(p^{1,2})) & \rightarrow & \C
 \end{eqnarray*}
 be continuous and compactly supported. Identifying 
 $\pi^{s} \times \pi^{s}(T^{\sharp})((p^{1,1},p^{1,2}), (q^{1,1},q^{1,2}))$
 with $X^{-}_{s(p^{1,1})} \times X^{-}_{s(q^{1,1})} \times (-\nu_{s}(p^{1,1}), \nu_{s}(p^{1,2}))$
 as in part 3 of Proposition \ref{gpds:55}, the map we denote 
 $f \otimes g \otimes h$
 sending $(\pi^{s}(x), \pi^{s}(y))$ in the former to
 $f(x_{(-\infty, m]})g( y_{(\infty, m]})h(\psi^{p,r}(n, \infty))$ 
 is a continuous function of compact support on 
 $C_{c}(T^{\sharp}(S^{s}_{\mathcal{B}}))$. If we also include 
 analogous functions using $p^{2,i}, q^{2,i}$ instead of
  $p^{1,i}, q^{1,i}$, the linear span of such functions is dense
  in  $C^{*}_{\lambda}(T^{\sharp}(S^{s}_{\mathcal{B}}))$.
  
  For each $(y,i)$ in $X_{r(p^{1,i})}^{+} \times \{ 1,2 \}$, let
  $\mathcal{M}_{(y,i)}$ denote the space of $L^{2}$-functions supported on 
  $X^{-}_{s(p^{1,i})} p^{1,i} y $ and 
   $\mathcal{N}_{(y.i)}$ denote the space of $L^{2}$-functions supported on 
  $X^{-}_{s(q^{1,i})}  p q^{1,i} y  $.
  For any $z$ in $S^{s}_{\mathcal{B}}$, it follows
  from the definition that the operator $\lambda_{z}(f \otimes g \otimes h)$
  is zero except on the spaces  $\mathcal{M}_{(y,i)}$ which is mapped
  to    $\mathcal{N}_{(y,i)}$. Moreover, if $\xi$ is in $\mathcal{M}_{(y,i)}$, 
  we have 
  \[
   \lambda_{z}(f \otimes g \otimes h) \xi = h(y) f \langle \xi, \bar{g}\rangle.
   \]
   From this it follows that 
$ \Vert   \lambda_{z}(f \otimes g \otimes h)  \Vert = \Vert h \Vert_{\infty} 
 \Vert f \Vert_{2} \Vert g \Vert_{2}$.
   
   Let $\epsilon > 0$.
   The map $\varphi^{s(p^{1,1})}_{r} : X^{-}_{s(p^{1,1})} 
   \rightarrow [-\nu_{r}^{s(p^{1,1})}, 0] $ is one-to-one 
   except on a countable set
   and so induces an isomorphism of measure spaces.
   We may find a continuous function 
   $f':  [-\nu_{r}^{s(p^{1,1})}, \nu_{r}^{s(p^{2,1})}]
   \rightarrow \C$ which is zero at the left end-point and the interval
 $[0, \nu_{r}^{s(p^{2,1})}] $ and such that 
 $\Vert f - f' \circ \varphi^{s(p^{1,1})}_{r} \Vert_{2} < \epsilon$.
 Similarly, we may find $g' [-\nu_{r}^{s(q^{1,1})}, \nu_{r}^{s(q^{2,1})}]
   \rightarrow \C$ such that 
  $\Vert g - g' \circ \varphi^{s(q^{1,1})}_{r} \Vert_{2} < \epsilon$.
  We may now define $f' \otimes g' \otimes h$ in an analogous way as 
  $f \otimes g \otimes h$ and the result is a continuous function
  of compact support on the open set 
  $U^{\sharp}(p,q) \subseteq T^{\sharp}(S_{\mathcal{B}})$ as 
  described in Proposition \ref{gpds:72}.

  A simple computation now shows that 
  \[
  \Vert \lambda(f \otimes g \otimes h) - \lambda(f' \otimes g' \otimes h)\Vert 
  \leq \epsilon \Vert h \Vert_{\infty} \left( \Vert f \Vert_{2} + \Vert g \Vert_{2}+ \epsilon \right).
  \]
  This shows that the range is dense.
\end{proof}

In the following subsections, we give more precise descriptions 
of these $C^{*}$-algebras, particularly focusing on 
inductive limit structures. Following that, our objective is to
compute their K-theory.


It is probably worth noting that this does not need an \emph{ordered}
Bratteli diagram. Also, it uses a state, but it is independent of the choice.

\begin{prop}
\label{AF:30}
Let $m < n$ be integers and let $p, q$ be in $E_{m,n}$ with 
$r(p)=r(q)$. For $(x,y)$ in $T^{+}(X_{\mathcal{B}})$, define 
\[
a_{p,q}(x,y) = \nu_{r}(s(p))^{-1/2} \nu_{r}(s(q))^{-1/2} 
\]
if $(x_{m+1}, \ldots, x_{n}) = p, (y_{m+1}, \ldots, y_{n}) = q$
and $x_{i} = y_{i}$, for all $i > n$. Define 
$a_{p,q}(x,y) = 0$ otherwise.
Then $a_{p,q}$ is a continuous, compactly supported function
on $T^{+}(X_{\mathcal{B}})$ and hence lies in $C^{*}_{\lambda}(T^{+}(X_{\mathcal{B}}))$.
Moreover, we have 
\begin{enumerate}
\item 
If $p',q'$ is another pair in $E_{m,n}$ with $r(p')=r(q')$, then 
\[
a_{p,q} a_{p',q'} = \left\{ \begin{array}{cl} a_{p,q'} 
  &  \text{ if } q=p', \\ 0 & \text{ if } q \neq p'\\  \end{array} \right.
  \]
 In particular, if $r(p) \neq r(p')$ then this product is zero.
  \item 
$a_{p,q}^{*} = a_{q,p}$.
\item $a_{p,q}= \sum_{s(e)=r(p)} a_{pe,qe}$,
\item 
$a_{p,q}  = \sum 
  \nu_{r}(s(e))^{-1/2} \nu_{r}(s(f))^{-1/2}\nu_{r}(r(e))^{1/2} \nu_{r}(r(f))^{1/2}
a_{ep,fq}$, where the sum is over all $e,f$ in $E_{m-1}$ with $r(e)=s(p), r(f)=s(q)$. 
\end{enumerate}
\end{prop}

\begin{prop}
\label{AF:40}
For integers $m < n$, let $A_{m,n}$ denote the span 
of all elements $a_{p,q}$, where $p,q$ are in $E_{m,n}$ with 
$r(p) = r(q)$.  If $v$ is a vertex in $V_{n}$, let 
$A_{m,n,v}$ denote the span 
of all elements $a_{p,q}$, where $p,q$ are in $E_{m,n}$ with 
$r(p) = r(q) =v$.  
\begin{enumerate}
\item $A_{m,n,v}$ is isomorphic to $M_{j}(\C)$, where
$j$ is the number of paths $p$ in $E_{m,n}$ with $r(p)=v$.
\item $A_{m,n} = \oplus_{v \in V_{n}}  A_{m,n,v}$,
In particular,
each $A_{m,n}$ is a finite dimensional $C^{*}$-subalgebra
of $C^{*}_{\lambda}(T^{+}(X_{\mathcal{B}}))$.
\item For all $m,n$, $K_{0}(A_{m,n}) \cong \Z^{\# V_{n}}$.
\item For all $m,n$ we have 
$A_{m-1,n} \subseteq A_{m,n} \subseteq A_{m,n+1}$.
\item  With the identifications above, the inclusion
$A_{m-1,n} \subseteq A_{m,n}$ is the identity map on $K_{0}$
and the inclusion
$A_{m,n} \subseteq A_{m,n+1}$ is the map on $K_{0}$ given by 
the edge matrix for $E_{n+1}$.
\item 
The union of $A_{m,n}$ over all $m,n$ is dense 
in   $C^{*}_{\lambda}(T^{+}(X_{\mathcal{B}}))$.
\end{enumerate}
\end{prop}

\begin{prop}
\label{AF:42}
Let $\mathcal{B}$ be a finite rank strongly simple bi-infinite
ordered Bratteli diagram. Assume that $m < n$ are such that 
$r:E^{Y}_{m,n} \rightarrow V_{n}$ is surjective. Define 
$A^{Y}_{m,n,v}$ and $A^{Y}_{m,n}$ as in Proposition 
\ref{AF:40} using $p,q$ in $E^{Y}_{m,n}$. The conclusion of
Proposition 
\ref{AF:40} holds when replacing 
$C^{*}_{\lambda}(T^{+}(X_{\mathcal{B}}))$ by 
$C^{*}_{\lambda}(T^{+}(Y_{\mathcal{B}}))$.
\end{prop}

Let $m< 0$ be chosen as in \ref{surface:12}. We consider the sequence
of subalgebras $A_{m-n,n}, n \geq 0$ in 
$C^{*}_{\lambda}(T^{+}(X_{\mathcal{B}}))$  and 
$A_{m-n,n}^{Y}, n \geq 0$ in 
$C^{*}_{\lambda}(T^{+}(Y_{\mathcal{B}}))$.

Recall that if $B$ is a $C^{*}$-subalgebra of a $C^{*}$-algebra $A$, we 
say that $B$ is \emph{full} if every closed two-sided ideal
of $A$ has non-trivial intersection with $B$ \cite{rieffel:morita}. We also say 
that $B$ is hereditary if $a$ is in $A$ and $b$ is 
in $B$ with $0 \leq a \leq b$, then $b$ is in $B$ also \cite{rieffel:morita}. 
 These two conditions imply that $A$ and $B$ are Morita equivalent and
 the inclusion $B \subseteq A$ induces an isomorphism on $K$-theory.

\begin{thm}
\label{AF:45}
Let $\mathcal{B}$ be a finite rank strongly simple bi-infinite
ordered Bratteli diagram. The $C^{*}$-algebras 
$C^{*}_{\lambda}(T^{+}(Y_{\mathcal{B}})) 
\subseteq C^{*}_{\lambda}(T^{+}(X_{\mathcal{B}}))$ are both AF-algebras
and both have  Bratteli diagram  $(V_{n}, E_{n+1}), n \geq 0$.
Both are simple and the former is a full
hereditary subalgebra of the latter.
\end{thm}

\begin{proof}
The first statement follows from our earlier results and the choice
of inductive systems given above. The fact that our diagram
 is strongly simple implies the 
 $C^{*}$-algebras are simple was shown by Bratteli \cite{Bra:AF}.
 The shortest proof that the subalgebra is hereditary is to consider the 
 function $f(x) = dist(x, X_{\mathcal{X}} -Y_{\mathcal{X}})$, using
 any metric on $X_{\mathcal{B}}$ which yields the usual topology. This
 can be viewed as an element of the multiplier 
 algebra 
 for the larger (see \cite{Put:Kexc}) and 
 $f C^{*}_{\lambda}(T^{+}(X_{\mathcal{B}})) f 
 = C^{*}_{\lambda}(T^{+}(Y_{\mathcal{B}}))$ is an
  easy computation which implies the conclusion.
\end{proof}

We remark that $C^{*}(T^{\sharp}(Y_{B}))$ is also 
an AF-algebra, but we will not give a proof. It can be
done in a similar way to what we have above and what
follows below and in the next subsection.

It will be useful for us to identify another sequence
of approximating subalgebras, although these are not finite-dimensional.

Let us explain some notation we will use. It involves tensor
products, but for our case, no knowledge of
tensor products is needed. If $A$ is any $C^{*}$-algebra and 
$X$ is a compact Hausdorff, we can view the elements 
of $A \otimes C(X)$ as functions from $X$ to $A$ which are 
continuous in the norm topology of $A$. Specifically, for $a$ 
in $A$ and $f$ in $C(X)$, we identify $a \otimes f (x) = f(x) a, x \in X$, 
which takes values in a one-dimensional subspace of $A$. 

In our case, let $p,q$ be in $E_{m,n}$ with $m < n$ and 
let $f: X_{r(p)}^{+} \rightarrow \C$ be continuous. We denote 
$a_{p,q} \otimes f$ the function on $T^{+}(X_{\mathcal{B}})$ defined by
\[
(a_{ p,q} \otimes f )(x,y) = 
\left\{ \begin{array}{cl} f(x_{(n,\infty)}) a_{p,q}, & 
x_{[m,n]}=p, y_{[m,n]}=q, x_{(n, \infty) }= y_{(n, \infty)}  \\
  0,  & \text{ otherwise } \end{array} \right.
  \]
  It is immediate that $a_{p,q} \otimes f$ is in 
  $C_{c}(T^{+}(X_{\mathcal{B}})$ and, if $p,q$ are in $E^{Y}_{m,n}$, 
  then it is in 
  $C_{c}(T^{+}(Y_{\mathcal{B}})$.
  
  For $v$ in $V_{n}$, we
   then identify $A_{m,n,v} \otimes C(X_{v}^{+})$ as a subalgebra of
  $C_{c}(T^{+}(X_{\mathcal{B}}))$. Every element may be written uniquely as 
  a sum over $p,q$ in $E_{m,n}$ with $r(p)=r(q)=v$ of terms 
 $a_{ p,q} \otimes f_{p,q}$.  We may also identify $A_{m,n,v}$
 as a subalgebra with each $f_{p,q}$ being a constant  
 function.
 Finally, we define
  $AC_{m,n} = \oplus_{v \in V_{n}} A_{m,n,v} \otimes C(X_{v}^{+})$.

\begin{prop}
\label{AF:50}
Let $p,q$ in $E_{m,n}$ with $m < n$ and $r(p)=r(q)=v$ and 
$f$ in $C(X_{v}^{+})$. For any $n < n'$ and $p'$ in $E_{n,n'}$,
 let $e_{p'} : X_{r(p')}^{+} \rightarrow  X_{s(p')}^{+}$ be 
 defined  by $e_{p'}(x) = p'x$, for $x$ in  $X_{r(p')}^{+}$.
 
 \begin{enumerate}
 \item 
 We have 
\[
(a_{p,q} \otimes f) = \sum a_{pp', qp'} \otimes (f \circ e_{p'}),
\]
where the sum is over $p'$ in $E_{n,n'}$ with $s(p')= r(p)$.
In particular, $AC_{m,n}$ is contained in $AC_{m,n'}$.
\item Identifying $a_{p,q}$ in $A_{m,n}$  with $a_{p,q}\otimes 1$ in 
 $AC_{m,n}$, we have $A_{m,n}$ is a subalgebra of  $AC_{m,n}$.
 \item 
The union of all $AC_{m,n}$ is dense in 
$C^{*}_{\lambda}(T^{+}(X_{\mathcal{B}}))$.
\end{enumerate}
\end{prop}


Our next aim is to analyze the $C^{*}$-algebra
of the equivalence relation $T^{\sharp}(S_{\mathcal{B}})$ 
on the space $S_{\mathcal{B}}$. Our main result 
is to establish an inductive limit
structure on this algebra,

 Recall from Corollaries
\ref{Cstar:20} and \ref{Cstar:40} we have 
$C^{*}(T^{\sharp}(S_{\mathcal{B}})) 
= C^{*}(T^{\sharp}(S^{s}_{\mathcal{B}}))
= C^{*}(T^{\sharp}(Y_{\mathcal{B}}))
= C^{*}(T^{+}(Y_{\mathcal{B}}))$. We will actually study the second
algebra in this list as it is more convenient and we pass over the third.

Our main tolls are the local description of 
$T^{\sharp}(S^{s}_{\mathcal{B}})$ given in Proposition
\ref{gpds:55} and the inductive limit for 
$C^{*}(T^{+}(X_{\mathcal{B}}))$ given in Proposition 
\ref{AF:50}.

Recall that $E^{s}_{m,n}, m < n$ 
consists of pairs $p= (p^{1}, p^{2})$ such that 
$p^{1}, p^{2} $ are in $E_{m,n}^{Y}$ and $p_{2}$ 
is the $s$-successor of $p_{1}$.
For $p= (p^{1}, p^{2})$ in $E^{s}_{m,n}$,
 we define 
$r(p)= (r(p^{1}),r(p^{2}))$. We also define
\[
G_{m,n}: \{ (p,q) \mid p, q \in E_{m,n}^{s}, r(p) = r(q) \},
\]
 which 
is a finite  equivalence relation on 
$E_{m,n}^{s}$
and hence also a groupoid. For 
$i=1,2$, we define 
$\alpha_{i}:G_{m,n} \rightarrow E^{Y}_{m,n} \times E^{Y}_{m,n}$
by $\alpha_{i}((p^{1},p^{2}), (q^{1}, q^{2})) = (p^{i}, q^{i})$. 

\begin{prop}
\label{IntCstar:60}
Let $\mathcal{B}$ be an ordered bi-infinite Bratteli
 diagram satisfying the conditions 
of \ref{surface:20}. 
Let $m < n$.  Suppose
$a =
\sum_{p,q \in E^{Y}_{m,n}} a_{p,q} \otimes f_{p,q}$, with 
$f_{p,q}$ in $C(X_{r(p)}^{+})$ for each $p,q$, 
in $E^{Y}_{m,n}$. Then $a$ 
is in $ C^{*}(T^{\sharp}(S^{s}_{\mathcal{B}}))$
 if and only if the following hold: 
\begin{enumerate}
\item for any $p, q$ in $E_{m,n}^{Y}$ with $r(p)=r(q)$, 
$f_{p,q} = g_{p,q} \circ \varphi_{s}^{r(p)}$, where
$g_{p,q} : [0, \nu_{s}(r(p))] \rightarrow \C$ is continuous
and $\varphi^{r(p)}_{s}$ is as in \ref{order:60},
\item for every $(p,q)$ in $E^{s}_{m,n}$ with $r(p)=r(q)$, we have 
$f_{p^{1},q^{1}}(x_{r(p^{1})}^{s-max}) = 
f_{p^{2},q^{2}}(x_{r(p^{2})}^{s-min})$.
\item if  $(p,q)$ is not in $\alpha_{1}(G_{m,n})$, then 
$f_{p,q}(x_{r(p)}^{s-max})  =0$,
\item if  $(p,q)$ is not in $\alpha_{2}(G_{m,n})$,  then 
$f_{p,q}(x_{r(p)}^{s-min})  =0$,
\end{enumerate}
We define $B_{m,n}$ to be the set of all elements, $a$,
satisfying these
conditions.
\end{prop}

\begin{proof}
We will first show that any element satisfying the conditions lies in 
$ C^{*}(T^{\sharp}(S^{s}_{\mathcal{B}}) )$.

It suffices to show that, for any $(x,y)$ in $T^{+}(Y_{\mathcal{B}})$,
 $a(x,y)$ is zero if $(x,y)$ is not in $T^{\sharp}(Y_{\mathcal{B}})$
and that $a(x,y) =a(\Delta_{s}(x),\Delta_{s}(y))$, 
if $x,y$ are in $\partial^{s}X_{\mathcal{B}}$. It is clear that 
$T^{+}(Y_{\mathcal{B}})$ and $T^{\sharp}(Y_{\mathcal{B}})$
agree except on $\partial^{s}X_{\mathcal{B}}$ and so for 
both conditions 
 we need only
consider the cases when $x,y$ are in $\partial^{s}X_{\mathcal{B}}$.
Without loss of generality, assume that $x_{n}$ is s-maximal, for 
all $n$ sufficiently large. Hence, $y_{n}$ is also.

We first observe that if $a(x,y)$ is non-zero, then both 
$p=x_{(m,n]}$ and $q=y_{(m,n]}$ are in $E^{Y}_{m,n}$, which implies
they are both in $Y_{\mathcal{B}}$. This implies that $n(x), n(y) \geq m$.
In addition, we must have
$(x,y)$ in $T^{+}_{n}(X_{\mathcal{B}})$.

If $n(x) > n$, then $n(y)=n(x) > n$ also and 
$\Delta_{s}(x)_{(m,n]}= p, \Delta_{s}(y)_{(m,n]} =q$
and from the first hypothesis
 $f(x_{(n, \infty)}) = f(\Delta_{s}(x)_{(n, \infty)}$.

The case which remains is $m< n(x), n(y) \leq n$. 
If either $S_{s}(x_{(m,n]})$
or $S_{s}(y_{(m,n]})$ is not in $E^{Y}_{m,n}$ then $a(x,y) = 0$ 
by the third condition. It is also clear that  
$a(\Delta_{s}(x), \Delta_{s}(y))  =0$ in this case.

Next, we suppose that $S_{s}(x_{(m,n]})$ 
and $S_{s}(y_{(m,n]})$ are in $E^{Y}_{m,n}$, but  \newline
$r(S_{s}(x_{(m,n]})) \neq r(S_{s}(y_{(m,n]}))$. Again by the third
condition $a(x,y) = a(\Delta_{s}(x), \Delta_{s}(y))  =0$ 
since $(\Delta_{s}(x), \Delta_{s}(y))$
 is not in $T^{+}_{n}(Y_{\mathcal{B}})$.

We are left with the case that $S_{s}(x_{(m,n]})$ 
and $S_{s}(y_{(m,n]})$ are in $E^{Y}_{m,n}$ and 
$r(S_{s}(x_{(m,n]})) = r(S_{s}(y_{(m,n]}))$. Here, the second  condition, 
using $p^{1}=p, p^{2}= S_{s}(p), q^{1}=q, q^{2}= S_{s}(q)$ clearly implies
$a(\Delta_{s}(x), \Delta_{s}(y)) = a(x,y)$.

The converse direction is relatively simple and we omit the details.
\end{proof}

As we noted above, $G_{m,n}$ is an equivalence relation on a finite set, namely
 pairs $(p,q)$ in $E_{m,n}^{Y}$ with $r(p)=r(q)$.

It should cause
no confusion if we also define
$\alpha_{i}: C^{*}_{\lambda}(G_{m,n}) \rightarrow A^{Y}_{m,n}$ by 
$\alpha_{i}(g) = \sum_{(p,q) \in G_{m,n}} g(p,q) a_{\alpha_{i}(p,q)}$, 
for any function $g:G_{m,n} \rightarrow \C$. 
It is a simple matter to verify that $\alpha_{1}, \alpha_{2}$ are 
$*$-homomorphisms.

We now want to consider, for $m < n$, the $C^{*}$-algebra 
$\oplus_{v \in V_{n}} A_{p,q, v} \otimes C[0, \nu_{s}^{v}]$.
Following Corollary \ref{Cstar:40}, we
 can regard this as a subalgebra of $AC_{m,n}$
by mapping $a_{p,q} \otimes f$ to $a_{p,q} \otimes f \circ \varphi_{s}^{v}$.
In fact, this is exactly the subalgebra of  $AC_{m,n}$
satisfying the first condition of Proposition 
\ref{IntCstar:60}.

We have two homomorphisms $ev_{0}, ev_{1}: AC_{m,n} \rightarrow A_{m,n}$
defined by 
\begin{eqnarray*}
ev_{0}\left(\sum_{p,q} a_{p,q} \right)
 & = &  \sum_{p,q}f_{p,q}(x_{r(p)}^{-}) a_{p,q} \\
ev_{1}\left( \sum_{p,q} a_{p,q} \right) 
& = &  \sum_{p,q}f_{p,q}(x_{r(p)}^{+}) a_{p,q}.
 \end{eqnarray*}

\begin{thm}
\label{IntCstar:70}
\begin{enumerate}
\item For all $m < n$, $B_{m,n}$ is a 
$C^{*}$-subalgebra of 
$C^{*}_{\lambda}(T^{\sharp}(S^{s}_{\mathcal{B}}))$.
\item 
For all $m < n$, $B_{m,n} \subseteq B_{m-1,n+1}$.
\item $\bigcup_{n=1}^{\infty} B_{-n,n}$ is dense in 
$C^{*}_{\lambda}(T^{\sharp}(S^{s}_{\mathcal{B}}))$.
\item 
For all $m < n$, we have 
\begin{eqnarray*}
B_{m,n} & \cong & \{ \left( a , h \right) \in 
\left( \oplus_{v \in V_{n}} A_{m,n,v} \otimes C[0, \nu_{s}^{v}] \right) 
 \oplus C^{*}(G_{m,n})  \mid  \\
  &   &  
ev_{0}(a) = \alpha_{2}(h), 
 ev_{1}(a)  = \alpha_{1}(h) \}.
\end{eqnarray*}
\end{enumerate}
\end{thm}

\begin{proof}
The first three parts are immediate. For the last,
for $a_{p,q} \otimes g$ in 
$\oplus_{v \in V_{n}} A_{p,q, v} \otimes C[0, \nu_{s}^{v}]$, note that 
$ev_{0}(a_{p,q} \otimes g) = g(0)a_{p,q}$ while 
$ev_{1}(a_{p,q} \otimes g) = g(\nu_{r(p)}^{s})a_{p,q}$. Then the 
conditions $ev_{0}(a) = \alpha_{2}(g), 
 ev_{1}(a)  = \alpha_{1}(g) $ are just a restatement of the 
 last three conditions of Proposition \ref{IntCstar:60}.
\end{proof}

While our $C^{*}$-algebras $B_{m,n}$ are not unital, the reader should 
compare the result in part 4 with the definition of 
\emph{recursive subhomogeneous}
 $C^{*}$-algebras given in \cite{Phi:RSHC*}.

\begin{cor}
\label{IntCstar:90}
For $m < n$, we have a short exact sequence

$$\xymatrix{ 0 \ar[r] &
 \oplus_{v \in V_{n}}  A_{m,n,v} \otimes C_{0}(0, \nu_{s}(v))  
\ar[r]  &
B_{m,n} \ar[r]  &   C^{*}(G_{m,n}) \ar[r]  &  0.}$$
\end{cor}


We now turn our attention to the $C^{*}$-algebra 
of the horizontal foliation, $C^{*}(\mathcal{F}_{\mathcal{B}}^{+})$.
When it is convenient, we will also denote 
$C^{*}(\mathcal{F}_{\mathcal{B}}^{+})$ by $C_{\mathcal{B}}^{+}$.
We want to show that $C^{*}(\mathcal{F}_{\mathcal{B}}^{+})$ has 
an inductive limit structure analogous to that of 
 $C^{*}(T^{\sharp}(S_{\mathcal{B}}))$ appearing in 
 Theorem \ref{IntCstar:70} and Corollary 
\ref{IntCstar:90}. 

If $m < n$ and $p$ is any path in $E_{m,n}^{Y}$ and $x$ 
is in $X_{r(p)}^{+}$, the set $X_{s(p)}^{-}px$ 
lies in $Y_{\mathcal{B}}$. Moreover, it is also equal to
 $[x_{s(p)}^{r-min}px, x_{s(p)}^{r-max}px]_{r}$
and its image under $\pi^{r}$ is homeomorphic to a closed interval.
Hence, its image under $\pi$ is contained in a single equivalence 
class of $\mathcal{F}_{\mathcal{B}}^{+}$ (Definition \ref{gpds:70}).

Let $p, q$ be in $G_{m,n}$; that is, they are in 
$E_{m,n}^{s}$  such that $r(p)=r(q)$. Observe that if there
are $x,y$ in $T^{+}(Y_{ \mathcal{B}})$ such that 
\[
X_{s(p^{1})}^{-}p^{1}x_{r(p^{1})}^{s-max}, 
X_{s(q^{1})}^{-}q^{1}x_{r(q^{1})}^{s-max} \subseteq [x,y]_{r} 
\subseteq Y_{\mathcal{B}}
\]
then it follows from 
Proposition \ref{singular:150} that
 \[
X_{s(p^{2})}^{-}p^{2}x_{r(p^{2})}^{s-min}, 
X_{s(q^{2})}^{-}q^{2}x_{r(q^{2})}^{s-min} \subseteq 
[\Delta_{s}(x),\Delta_{s}(y)]_{r} 
\subseteq Y_{\mathcal{B}}
\]

We define, for each $m < n$,
$H_{m,n}$ to be the set of all   $ (p,q)$ 
 in $G_{m,n}$ satisfying this condition.
This is a subgroupoid of $G_{m,n}$.

We remark that an analogue of Proposition \ref{IntCstar:60} holds: 
 we simply 
change $C^{*}_{\lambda}(T^{\sharp}(S_{\mathcal{B}}))$ 
to $C^{*}_{\lambda}(\mathcal{F}^{+}_{\mathcal{B}})) $ and
 replace $G_{m,n}$ 
in conditions 3 and 4 by $H_{m,n}$. This is an immediate consequence of 
Proposition \ref{IntCstar:60} and  Definition \ref{gpds:70}.
We let $C_{m,n}$ be the set of all elements satisfying these conditions; that is, 
$C_{m,n} = AC^{Y}_{m,n} \cap 
C^{*}_{\lambda}(\mathcal{F}^{+}_{\mathcal{B}}))  $.

We then obtain analogues of Theorem \ref{IntCstar:70} and Corollary 
\ref{IntCstar:90} which we state precisely for the record.

\begin{thm}
\label{FolAlg:20}
\begin{enumerate}
\item For all $m < n$, $C_{m,n}$ is a 
$C^{*}$-subalgebra of $ C^{*}_{\lambda}(\mathcal{F}^{+}_{\mathcal{B}})) $.
\item 
For all $m < n$, $C_{m,n} \subseteq C_{m-1,n+1}$.
\item $\bigcup_{n=1}^{\infty} C_{-n,n}$ is dense in 
$ C^{*}_{\lambda}(\mathcal{F}^{+}_{\mathcal{B}})) $.
\item 
For all $m < n$, we have 
\begin{eqnarray*}
C_{m,n}  & \cong & \{ \left( a , h \right) \in 
\left( \oplus_{v \in V_{n}} A_{m,n,v} \otimes C[0, \nu_{s}^{v}] \right) 
 \oplus C^{*}(H_{m,n})  \mid  \\
  &   &  
ev_{0}(a) = \alpha_{2}(h), 
 ev_{1}(a)  = \alpha_{1}(h) \}.
\end{eqnarray*}
\end{enumerate}
\end{thm}

\begin{cor}
\label{FolAlg:30}
For $m < n$, we have a short exact sequence
$$\xymatrix{ 0 \ar[r] & \displaystyle\bigoplus_{v \in V_{n}}  A_{m,n,v} \otimes C_{0}(0, \nu_{s}(v))  
\ar[r]  &
C_{m,n} \ar[r]  &   C^{*}(H_{m,n}) \ar[r]  &  0.}$$
\end{cor}

\section{A Fredholm module}
\label{Fred}
The aim of this section is to produce a Fredholm module for our 
$C^{*}$-algebras. This will be crucial  in the K-theory computations
of the next section.

The books by Blackadar \cite{Bla:K}, Higson and Roe \cite{HR:KHom}
 and Connes \cite{Con:NCG} are all good references for Fredholm modules.
We remind readers that, for any $C^{*}$-algebra $A$, a Fredholm
module for $A$ consists
of a Hilbert space $\mathcal{H}$, a representation
$\pi$ of $A$ on $\mathcal{H}$ and a bounded 
operator $F$ on $\mathcal{H}$
such that $(F^{2}-1)\pi(a), (F-F^{*})\pi(a)$ and
$[\pi(a), F] = \pi(a) F - F \pi(a)$ are all compact 
operators, for each $a$ in $A$.
In our case, we will give the Hilbert space and representation
of the AF-algebra, $C^{*}_{\lambda}(T^{+}(X_{\mathcal{B}}))$.
 The operator $F$ will 
actually satisfy $F^{2}=1, F=F^{*}$, but the crucial condition
that $[\pi(a), F]$ is compact, for each $a$, holds 
if we consider $a$ in the $C^{*}$-subalgebra 
$C^{*}_{\lambda}(T^{+}(Y_{\mathcal{B}})) $.
In fact, our Hilbert space comes with a natural $\Z_{2}$-grading, 
the representation $\pi$ is by even operators, while $F$ is odd.
In other words, we will have an even Fredholm module. 

The last discussion will probably not
 be very helpful to non-operator theorists.
Let us give a simple example where these properties will be clear.
At the same time, what is happening in the example is really 
exactly what is going on in our situation to follow
and so this should provide some intuition.

Let $X \subseteq [0, 1]$ be the standard Cantor ternary set. Let us 
list the open intervals in its complement (in $[0,1]$) as 
$(x_{n}, y_{n}), n \geq 1$ (the order is not important here).
Let $\mathcal{H}$ be a Hilbert space with a canonical basis
indexed by the endpoints, $\delta_{x_{n}}, \delta_{y_{n}}$. 
(One view is to put an infinite measure on $X$ with point mass
at each $x_{n}$ and $y_{n}$ and consider the space of square-integrable 
functions. The $C^{*}$-algebra of continuous functions on $X$, 
$C(X)$ can be represented as operators on this Hilbert space
by simple evaluation of the functions: we supress the 
representation and simply write 
$f \delta_{x_{n}} = f(x_{n}) \delta_{x_{n}}, 
f \delta_{y_{n}} = f(y_{n}) \delta_{y_{n}}$, for all $n \geq 1$.

Define an operator $F$ on this space Hilbert space
by specifying 
$F\delta_{x_{n}} = \delta_{y_{n}}, F\delta_{y_{n}} = \delta_{x_{n}}$, 
for all $n \geq 1$. It is trivial to see $F^{2}=I, F^{*}=F$.
It is a simple matter to check, if $f$ is locally constant, then 
$f(x_{n}) = f(y_{n})$, for all but finitely many $n$ and 
the operator $[F, f] = Ff - fF$ is finite rank. Only slightly
more subtle is that, for any $f$ in $C(X)$, 
$[F, f] $ is compact. Finally, if $\pi: X \rightarrow [0,1]$ 
denotes the devil's staircase, then $f$ in $C(X)$ has the form 
$f = g \circ \pi$, for some $g$ in $C[0,1]$ if and only if
$[F, f] =0$.

\begin{defn}
\label{Fred:10}
Let $\mathcal{B}$ be a 
bi-infinite ordered Bratteli diagram
satisfying the conditions of
Definition \ref{surface:20}.
Let $I_{\mathcal{B}}, J_{\mathcal{B}}  $ 
and $ x_{i}, 1 \leq i 
\leq I_{\mathcal{B}} + J_{\mathcal{B}},$ be as in 
Proposition \ref{gpds:30}. 
For $ 1 \leq i \leq I_{\mathcal{B}} + J_{\mathcal{B}} $, 
we define $\mathcal{H}_{i}= L^{2}(T^{+}(x_{i}), \nu^{x_{i}}_{r}) )$,  
\[
\mathcal{H}^{max}_{\mathcal{B}} =          
\bigoplus_{1 \leq i \leq {I_{\mathcal{B}}}} \mathcal{H}_{i},\hspace{1in}
\mathcal{H}^{min}_{\mathcal{B}}  =  \bigoplus_{I_{\mathcal{B}} < i 
\leq I_{\mathcal{B}} + J_{\mathcal{B}}} \mathcal{H}_{i},
\]
 and $\mathcal{H}_{\mathcal{B}}  =  \mathcal{H}^{max}_{\mathcal{B}}
 \oplus \mathcal{H}^{min}_{\mathcal{B}}$.
We define  $\pi_{\mathcal{B}} = 
\displaystyle\bigoplus_{1 \leq i \leq I_{\mathcal{B}} + J_{\mathcal{B}}} \lambda_{x_{i}}$.

Finally, we define $F_{\mathcal{B}}:
\mathcal{H}_{\mathcal{B}} \rightarrow \mathcal{H}_{\mathcal{B}}$ 
to be the operator  $(F_{\mathcal{B}} \xi)(x) = \xi( \Delta_{s}(x))$, for 
any $\xi$ in $\mathcal{H}_{\mathcal{B}}$ and $x$ in 
$\bigcup_{i} T^{+}(x_{i})$.
 \end{defn}

 We make several observations. It would probably be 
 more accurate to replace 
 $T^{+}(x_{i})$ by $T^{+}(x_{i}) \cap Y_{\mathcal{B}}$, but
 as the difference is a set of measure zero, it has
 no effect on the $L^{2}$-space. Secondly, notice that
 $\mathcal{H}_{\mathcal{B}}$ comes with a natural $\Z_{2}$-grading. The 
 associated grading operator is the identity on 
 $\mathcal{H}^{max}_{\mathcal{B}} $ and minus the 
 identity on $\mathcal{H}^{min}_{\mathcal{B}} $.
 Finally, it is a consequence 
 of Proposition \ref{singular:150} that 
 \begin{eqnarray*}
 \Delta_{s}: \bigcup_{1 \leq i \leq I_{\mathcal{B}}^{+}} T^{+}(x_{i})
  &  \rightarrow  & \bigcup_{I_{\mathcal{B}}^{+} <  j 
  \leq I_{\mathcal{B}}^{+} + J_{\mathcal{B}}^{+}} T^{+}(x_{j}) \\ 
   \Delta_{s}: 
   \bigcup_{I_{\mathcal{B}}^{+} <  j \leq I_{\mathcal{B}}^{+}
    + J_{\mathcal{B}}^{+}} T^{+}(x_{j})
  &  \rightarrow & \bigcup_{1 \leq i \leq I_{\mathcal{B}}^{+}} T^{+}(x_{i})
  \end{eqnarray*}
  are measure preserving bijections and hence induce unitary operators on the
  associated $L^{2}$-spaces. In addition, $\Delta_{s} \circ \Delta_{s}$ is 
  the identity so $F_{\mathcal{B}}$ is odd, $F_{\mathcal{B}}^{2}=1$ and 
  $F_{\mathcal{B}}= F_{\mathcal{B}}^{*}$.

We need to set out some notation.   
  If $p$ is any element of $E_{m,n}$, we define   
 \begin{eqnarray*}
 \xi_{ p}^{max} & =   & \nu_{r}(s(p))^{-1/2} \chi_{X_{s(p)}^{-}px_{v}^{s-max}},\\
 \xi_{ p}^{min} & =  & \nu_{r}(s(p))^{-1/2} \chi_{X_{s(p)}^{-}px_{v}^{s-min}}.
 \end{eqnarray*}
 Each is a unit vector in $\mathcal{H}^{max}_{\mathcal{B}}$ and 
 $\mathcal{H}^{min}_{\mathcal{B}}$, 
 respectively.
 Observe that if $e$ is the $s$-maximal ($s$-minimal)
  edge with $s(e) = r(p)$, then 
 $\xi_{p}^{max} = \xi_{pe}^{max}$ ($\xi_{p}^{min} = \xi_{pe}^{min}$, 
 respectively).
 It is an easy exercise to check that  the linear span
 of all
 such vectors is dense in $\mathcal{H}_{\mathcal{B}}$.
 
 The following is an immediate consequence of the definitions
 and the fact that 
 \[
 \Delta_{s}(X_{s(p)}^{-}px_{r(p)}^{max}) = 
 X_{s(p)}^{-}S_{s}(p)x_{r(S_{s}(p))}^{min}
 \]
 if $p$ is not $s$-maximal.
 
 \begin{lemma}
 \label{Fred:20}
 Let $p$ be in $E_{m,n}$. If $p$ is not $s$-maximal, then
 $F_{\mathcal{B}}\xi_{p}^{max} = \xi_{S_{s}(p)}^{min}$. 
 If $p$ is not $s$-minimal, then
 $F_{\mathcal{B}}\xi_{p}^{min} = \xi_{P_{s}(p)}^{max}$. 
 \end{lemma}
 
 If $\xi, \eta$ are any vectors in a Hilbert space $\mathcal{H}$, we
 define $\xi \otimes \eta^{*}$ to be the rank one operator
 defined by $(\xi \otimes \eta^{*}) \zeta = \langle \zeta, \eta\rangle \xi$, 
 for $\zeta$ in $\mathcal{H}$. If $T$ is any other operator, we have 
 $T(\xi \otimes \eta^{*}) = (T\xi) \otimes \eta^{*}$ and 
 $(\xi \otimes \eta^{*}) T = \xi \otimes (T^{*}\eta)^{*}$.
 
 It is worth noting that it is a straightforward computation from 
 the definitions that, for any $m < n \leq n'$, $p,q$ in $E_{m,n}$
 $q'$ in $E_{m,n'}$, if we let $\xi^{max}_{q'}$ be as above and 
 $a_{p,q}$ be as in \ref{AF:30}, then  
 $\pi_{\mathcal{B}}(a_{p,q}) \xi_{q'} =0$ if
 $(q')_{(m,n]} \neq q$ and 
 $\pi_{\mathcal{B}}(a_{p,q}) \xi^{max}_{q'} = \xi^{max}_{p'}$ if
 $(q')_{(m,n]} = q$, where $p' = p(q')_{(n,n']}$. An analogous 
 statement holds for $ \xi^{min}_{q'}$. In particular, the representation
 respects the grading on $\mathcal{H}_{\mathcal{B}}$. 
 In addition, it will be useful
 to have the following which is slightly less routine.

 \begin{lemma}
 \label{Fred:30}
 Let $m < n$, $p,q$ in $E_{m,n}$ with $r(p)=r(q)=v$ be in $V_{n}$ and 
 $f: X_{v}^{+} \rightarrow \C $ be continuous. For
 any $n' > n$ and $q'$ in $E_{m,n'}$, we have
 \[
 \pi_{\mathcal{B}}(a_{p,q} \otimes f )\xi_{q'}^{max} = 
 f(q'x_{r(q')}^{s-max})\xi_{p(q')_{(n,n']}}^{max}
 \]
 if $(q')_{(m,n]} = q$ and is zero otherwise, while
  \[
 \pi_{\mathcal{B}}(a_{p,q} \otimes f)\xi_{q'}^{min} = 
 f(q'x_{r(q')}^{s-min})\xi_{p(q')_{(n,n']}}^{min}
 \]
 if $(q')_{(m,n]} = q$ and is zero otherwise.
 \end{lemma}
 
 \begin{proof}
 We prove the first statement only.
 Let us consider all paths $p''$ in $E_{n,n''}$ with $s(p'') = r(p)$.
 We may identify
  $a_{pp'', qp''} = a_{p,q} \otimes \chi_{p''X_{r(p'')}^{+}}$,
  regarding $\chi_{p''X_{r(p'')}^{+}}:X_{v}^{+} \rightarrow  \C$.
 Without loss of generality, we may assume that $n'' > n'$. 
 The  continuous function $f$ 
 may be approximated by sums of such functions
 and so it suffices for us to prove the result for these functions.
 We have $\pi(a_{pp'', qp''}) \xi_{q'}^{max}$ is zero
 unless $q= (q')_{(m,n]}, p'' = (q')_{(n,n']}$ and 
 $p''_{(n',n'']}$ is s-maximal. In this case, the result
 is  $\xi_{p(q')_{(n,n']}}^{max}$. In either case, this agrees with 
 $\chi_{p''X_{r(p'')}^{+}}(x_{r(q')}^{s-max}) \xi_{p(q')_{(n,n']}}^{max}$.
 \end{proof}
 
 \begin{prop}
 \label{Fred:40}
 Let $m < n$ and assume that $p,q$ are in 
 $E_{m,n}^{Y}$ with $r(p)=r(q)=v$. 
\begin{enumerate} 
\item 
 We have 
 \[
 [\pi(a_{p,q}), F_{\mathcal{B}}] =  
 \xi_{p}^{max} \otimes (\xi^{min}_{S_{s}(q)})^{*}
 + \xi_{p}^{min} \otimes (\xi^{max}_{P_{s}(q)})^{*} 
 - \xi_{P_{s}(p)}^{max} \otimes (\xi^{min}_{q})^{*} 
    -\xi_{S_{s}(p)}^{min} \otimes (\xi^{max}_{q})^{*}. 
\]
       \item 
     Consider the function $f(x)
      = \nu_{r}(v)^{-1} \varphi^{v}_{s}(x)$,
     for $x$ in   $X_{v}^{+}$.
     We have
\[
 [\pi_{\mathcal{B}}(a_{p,q}) \otimes f,  F_{\mathcal{B}}] = 
 \xi_{p}^{max} \otimes (\xi^{min}_{S_{s}(q)})^{*} 
    -\xi_{S_{s}(p)}^{min} \otimes (\xi^{max}_{q})^{*}.
\]
  \item 
  If $g$ is any continuous $\C$-valued function on  
  $[0, 1]$,  then
     \begin{eqnarray*}
 [\pi_{\mathcal{B}}(a_{p,q} \otimes g \circ f), F_{\mathcal{B}}] 
    &  =  &   g \circ f(x_{v}^{s-max}) 
        \left( \xi_{p}^{max} \otimes (\xi^{min}_{S_{s}(q)})^{*} 
    -\xi_{S_{s}(p)}^{min} \otimes (\xi^{max}_{q})^{*} \right) 
   \\
     &  &
       +     g \circ f(x_{v}^{s-min}) 
   \left( \xi_{p}^{min} \otimes (\xi^{max}_{P_{s}(q)})^{*}  
 - \xi_{P_{s}(p)}^{max} \otimes (\xi^{min}_{q})^{*} \right)        
       \end{eqnarray*}
       \end{enumerate}    
 \end{prop}
 
 \begin{proof}
 Let $\mathcal{H}_{m}$ denote the closed linear span 
 of all vectors $\xi^{max}_{p}, \xi^{min}_{p}$, 
 where $p$ is in $E_{m,n'}$ and $n' > m$. It is clear 
 that this space is invariant 
 under $F_{\mathcal{B}}$ and a direct computation
 shows that $\pi(a_{p,q})\vert_{\mathcal{H}_{m}} =0$. It follows
 that $[\pi(a_{p,q}), F_{\mathcal{B}}]\vert_{\mathcal{H}_{m}} =0$.
 
 Next, let us consider $n' > n$ and $q'$ in $E_{m,n'}$ such that 
 $(q')_{(n,n']}$ is not s-maximal. It follows that 
 $(S_{s}(q'))_{(m,n]} = (q')_{(m,n]}$ and, in consequence,
$ \pi_{\mathcal{B}}(a_{p,q})F_{\mathcal{B}} \xi_{q'}^{max}
 = \pi_{\mathcal{B}}(a_{p,q})\xi^{min}_{S_{s}(q')}$. 
If  $q \neq (S_{s}(q'))_{(m,n]}= (q')_{(m,n]} $, this is zero.
If  $q =(S_{s}(q'))_{(m,n]}= (q')_{(m,n]} $, this equals $\xi^{min}_{p'}$
where $p' = p (S_{s}(q')_{(n,n']})$.
On the other hand, $ F_{\mathcal{B}} 
\pi_{\mathcal{B}}(a_{p,q})\xi_{q'}^{max} $ is also zero
if $q \neq  (q')_{(m,n]} $, and if $q = (q')_{(m,n]} $, it equals
$F_{\mathcal{B}} \xi^{max}_{p''} = \xi^{min}_{S_{s}(p'')}$, where 
$p'' = p (q')_{(n,n']}$. As $(q')_{(n,n']}$ is not s-maximal, we have 
$S_{s}(p'') = p S_{s}((q')_{(n,n']}) = p'$. We conclude that
$[\pi_{\mathcal{B}}(a_{p,q}), F_{\mathcal{B}}] \xi^{max}_{q'} = 0$.
 A similar argument for
$\xi^{min}_{q'}$ shows the same conclusion.

As we noted above if $q'$ is in $E_{m,n'}, n'>n$ and 
$(q')_{(n,n']}$ is s-maximal, then 
$\xi^{max}_{q'} = \xi^{max}_{(q')_{(n,n']}}$ and so
it remains to consider the case $q'$ is in $E_{m,n}$. We need to consider
$[\pi_{\mathcal{B}}(a_{p,q}), F_{\mathcal{B}}]$ on the two 
types of vectors, $\xi^{max}_{q'}$
and $\xi^{min}_{q'}$. Using the fact that $p,q$ are in $E^{Y}_{m,n}$,
we may summarize the only situations where the result is 
non-zero as follows:
\[
\begin{array}{rlrr}
\pi_{\mathcal{B}}(a_{p,q})F_{\mathcal{B}} \xi_{q'}^{max} 
& =  & \xi^{min}_{p},  & S_{s}(q')=q \\
\pi_{\mathcal{B}}(a_{p,q})F_{\mathcal{B}} \xi_{q'}^{min}
 & =  & \xi^{max}_{p},  & P_{s}(q')=q \\
F_{\mathcal{B}}\pi_{\mathcal{B}}(a_{p,q}) \xi_{q'}^{max} 
& =  & \xi^{min}_{S_{s}(p)},  & q'=q \\
F_{\mathcal{B}}\pi_{\mathcal{B}}(a_{p,q}) \xi_{q'}^{min} 
& =  & \xi^{max}_{P_{s}(p)},  & q'=q. 
\end{array}
\]
The result follows from this, Lemmas \ref{Fred:20} and \ref{Fred:30}.

  The proof for the second part is almost the same. In view of 
  Lemma \ref{Fred:30}, the operators $\pi(a_{p,q})$ and 
  $\pi_{\mathcal{B}}(a_{p,q} \otimes f)$ are equal except that 
  \[
  \pi_{\mathcal{B}}( a_{p,q} \otimes f) \xi_{q}^{max} = \xi_{q}^{max}, 
  \pi_{\mathcal{B}}(a_{p,q} \otimes f)\xi_{q}^{min} = 0.
  \]
  We omit the remaining details.
  
  For the last part, the property is clearly linear in the function 
  $g$ and we know it is satisfied by constant functions from  part 
 1 and $g(t) = t$ by part 2. We then show it holds
   for $g(t) = t^{k}, k \geq 1$, 
   by induction on $k$ by noting that 
   \begin{eqnarray*}
   [\pi_{\mathcal{B}}(a_{p,q} \otimes f(x)^{k+1}),
    F_{\mathcal{B}}] 
     &  =  &  
    [\pi_{\mathcal{B}}(a_{p,q} \otimes f(x)^{k})
  \pi_{\mathcal{B}}( a_{q,q} \otimes f(x)),
   F_{\mathcal{B}}] \\
   &  =  &  
    [\pi_{\mathcal{B}}(a_{p,q} \otimes f(x)^{k}), F_{\mathcal{B}}]  \pi_{\mathcal{B}}(a_{q,q} \otimes f(x)) \\
      &  &  +  \pi_{\mathcal{B}}(a_{p,q} \otimes f(x)^{k})
 [  \pi_{\mathcal{B}}( a_{q,q} \otimes f(x)),
  F_{\mathcal{B}}] \\
   &  =  &    
  \left( \xi_{p}^{max} \otimes (\xi^{min}_{S_{s}(q)})^{*} 
    -\xi_{S_{s}(p)}^{min} \otimes (\xi^{max}_{q})^{*}  \right)
    \pi_{\mathcal{B}}( a_{q,q} \otimes f(x)) \\
     &  &  +  \pi_{\mathcal{B}}(a_{p,q} \otimes f(x)^{k})
     \left(  \xi_{q}^{max} \otimes (\xi^{min}_{S_{s}(q)})^{*} 
    -\xi_{S_{s}(q)}^{min} \otimes (\xi^{max}_{q})^{*}
     \right) \\
        &  =  &  0   -   \xi^{min}_{S_{s}(p)} \otimes 
        (\xi^{max}_{q})^{*}    +   
     \xi^{max}_{p} \otimes (\xi^{min}_{S_{s}(q)})^{*} 
         - 0.
\end{eqnarray*}
It follows that the result holds for all polynomial functions $g$,
 and hence
for all continuous functions by continuity.
\end{proof}
 
 \begin{cor}
 \label{Fred:45}
 The triple 
 $(\mathcal{H}_{\mathcal{B}}, \pi_{\mathcal{B}}, F_{\mathcal{B}})$ 
 is an even Fredholm module
 for $C^{*}_{\lambda}(T^{+}(Y_{\mathcal{B}}))$. 
 \end{cor}
 
\begin{lemma}
\label{Fred:55}
Let $m < n$ and for each $p,q$ in $E^{Y}_{m,n}$ with $r(p)=r(q)$, let
$\alpha_{p,q}$ be a complex number. We have 
\begin{eqnarray*}
  &  & \left\Vert \sum_{(p,q) \notin \alpha_{1}(G_{m,n})} 
    \alpha_{p,q}  \xi_{p}^{max} \otimes (\xi_{q}^{max})^{*} 
       +   \sum_{(p,q) \in G_{m,n}} 
    \frac{\alpha_{p_{1},q_{1}} - \alpha_{p_{2},q_{2}}}{2}  \xi_{p}^{max} 
    \otimes (\xi_{q}^{max})^{*} \right\Vert  \\
  &    \leq   & \frac{3}{2} \left\Vert \left[\sum_{r(p)=r(q)}  \alpha_{p,q}\xi_{p}^{max} 
  \otimes (\xi_{q}^{max})^{*}, F_{\mathcal{B}}\right] \right\Vert
    \end{eqnarray*}
    and 
    \begin{eqnarray*}
  &  & \left\Vert \sum_{(p,q) \notin \alpha_{2}(G_{m,n})} 
    \alpha_{p,q}  \xi_{p}^{min} \otimes (\xi_{q}^{min})^{*} 
       +   \sum_{(p,q) \in G_{m,n}} 
    \frac{\alpha_{p_{2},q_{2}} - \alpha_{p_{1},q_{1}}}{2}  \xi_{p}^{min} 
    \otimes (\xi_{q}^{min})^{*} \right\Vert  \\
  &    \leq   & \frac{3}{2}
   \left\Vert \left[\sum_{r(p)=r(q)}  \alpha_{p,q}\xi_{p}^{min}
    \otimes (\xi_{q}^{min})^{*}, F_{\mathcal{B}}\right] \right\Vert.
    \end{eqnarray*}
\end{lemma}

\begin{proof} We will prove the first statement only.
Let $\mathcal{F}^{max} = span
   \{ \xi_{p}^{max} \otimes(\xi^{max}_{q})^{*}
    \mid p, q, \in E_{m,n}^{Y} \}$ 
  which is a finite dimensional $C^{*}$-algebra.
  
  For each $v$ in $V_{n}$, let $P_{v} = \sum 
  \xi_{p}^{max} \otimes(\xi^{max}_{p})^{*}$,
   where the sum is taken over all $p$ 
  in $E^{Y}_{m,n}$ with $r(p)=v$. Then the map 
  $\varepsilon: \mathcal{F}^{max}  \rightarrow \mathcal{F}^{max} $ defined 
  by $\varepsilon(a) = \sum_{v \in V_{n}} P_{v}aP_{v}$ is 
  a conditional expectation from $\mathcal{F}^{max}$ onto 
  $span \{ \xi_{p}^{max} \otimes (\xi_{q}^{max})^{*} \mid r(p)=r(q) \}$.
   In particular, $\varepsilon$ is a 
  contraction. Furthermore, for each 
  $v=(v_{1}, v_{2})$ in $r(E^{s}_{m,n})$, we let  
  $Q_{v} = \sum \xi_{p_{1}}^{max} \otimes(\xi^{max}_{p_{1}})^{*}$,
  where the sum is over all $(p_{1}, p_{2})$  in $ E_{m,n}^{s}$
  with $r(p_{1}, p_{2}) = v$. 
  Then the map 
  $\varepsilon': \mathcal{F}^{max}  \rightarrow \mathcal{F}^{max} $ defined 
  by $\varepsilon'(a) = \sum_{v \in r(E^{s}_{m,n})} Q_{v}aQ_{v}$ is 
  a conditional expectation from $\mathcal{F}^{max}$ onto 
  $span \{ \xi_{p_{1}}^{max} \otimes \xi_{q_{1}}^{max*} \mid (p, q) \in G_{m,n} \}$.
   In particular, $\varepsilon'$ is a 
  contraction.

  Lemma \ref{Fred:20} shows that
  \begin{eqnarray*}
   F_{\mathcal{B}}\left[\sum_{r(p)=r(q)}  
  \alpha_{p,q}\xi_{p}^{max} \otimes
   (\xi_{q}^{max})^{*}, F_{\mathcal{B}}\right] 
  & = & 
  \sum \alpha_{p,q} \left[ - \xi_{p}^{max} \otimes (\xi_{q}^{max})^{*} 
    +  \xi_{P^{s}(p)}^{max} \otimes (\xi_{P^{s}(q)}^{max})^{*}  \right. \\
  &  &  +  \left.  \xi_{p}^{min} \otimes (\xi_{q}^{min})^{*} 
    -  \xi_{S_{s}(p)}^{min} \otimes (\xi_{S_{s}(q)}^{min})^{*}  \right]
     \end{eqnarray*}
     where the sum is over $p,q$ in $E^{Y}_{m,n}$ with $r(p)=r(q)$.
     We denote this operator by $a$.
    Next, we compute $\varepsilon(a)$. The effect on the last two terms
    in the sum is to make them zero, 
    as the vectors do not lie in $\mathcal{F}^{max}$.
    The first term is unchanged and the second becomes zero 
    if $r(P^{s}(p)) \neq r(P^{s}(q))$ and is unchanged if 
    $(p,q) = \alpha_{2}((P^{s}(p)),P^{s}(q)),  (p,q))$ where 
    $\alpha_{2}$ is as described just before 
    Proposition \ref{IntCstar:60}. Hence, by 
    simply re-indexing the terms, we have 
   \[
  \varepsilon(a) = \sum_{(p,q) \notin \alpha_{1}(G_{m,n})} 
    \alpha_{p,q}  \xi_{p}^{max} \otimes (\xi_{q}^{max})^{*} 
       +   \sum_{(p,q) \in G_{m,n}} 
    (\alpha_{p_{1},q_{1}} - \alpha_{p_{2},q_{2}})
     \xi_{p}^{max} \otimes ( \xi_{q}^{max})^{*}. 
    \]
    Applying $\varepsilon'$ simply removes the first term, so we can write
 \[
  \varepsilon(a) - 2^{-1} \varepsilon'(\varepsilon(a))
  = \sum_{(p,q) \notin \alpha_{1}(G_{m,n})} 
    \alpha_{p,q}  \xi_{p}^{max} \otimes \xi_{q}^{max*} 
       +   \sum_{(p,q) \in G_{m,n}} 
    \frac{\alpha_{p_{1},q_{1}} - \alpha_{p_{2},q_{2}}}{2}  
    \xi_{p}^{max} \otimes \xi_{q}^{max*}.   
    \]
      The conclusion follows from the 
      facts that $\varepsilon, \varepsilon'$ are 
      contractions.
    \end{proof}

 \begin{thm}
 \label{Fred:60}
 An element $a$ in $C^{*}_{\lambda}(T^{+}(Y_{\mathcal{B}}))$ is in 
 $C^{*}_{\lambda}(T^{\sharp}(S^{s}_{\mathcal{B}}))$   if and only if 
 $[\pi_{\mathcal{B}}(a), F_{\mathcal{B}}] = 0$.
 \end{thm}
 
 \begin{proof}
 Let us begin by proving that if 
 $a$ is in  $C^{*}_{\lambda}(T^{\sharp}(S^{s}_{\mathcal{B}})$ , then 
 $[\pi_{\mathcal{B}}(a), F_{\mathcal{B}}] = 0$.
 To do so, we first assume that 
 $a = \sum_{p,q} a_{p,q, v} \otimes f_{p,q}$ is in 
 $AC_{m,n} = \oplus_{v} A_{m,n, v} \otimes C(X_{r(p)}^{+})$,
  for
 some $m<n$, where the sum is over 
   $p,q$ in $E_{m,n}^{Y}$ with $r(p)=r(q)$ and 
   satisfies the conditions of
   Proposition \ref{IntCstar:60}. 
 The general case then follows from 
 part 3 of Theorem \ref{IntCstar:70} and continuity.

 From the first condition of Proposition
 \ref{IntCstar:60}, we see that each $f_{p,q} = g_{p,q} \circ f_{r(p)}$,
 where $f_{v}(x) = \nu_{s}(v)^{-1} \varphi^{v}_{s}(x)$, for
 $x$ in $X_{v}^{+}$, and $g_{p,q}:[0,1] \rightarrow \C$ is continuous.
 
  In addition, we know from 
 conditions 3 and 4 that $g_{p,q}=0$ if $(p,q)$ is 
 not
  in 
 $\alpha_{1}(G_{m,n}) \cup \alpha_{1}(G_{m,n})$. 
 Applying part 3 of Proposition \ref{Fred:40}, 
 we have 
 \begin{eqnarray*}
[\pi_{\mathcal{B}}(a), F_{\mathcal{B}}] & = \sum_{(p,q) \in E^{s}_{m,n} } &
 g_{p^{1}, q^{1}}(1)   
 \left( \xi_{p^{1}}^{max} \otimes (\xi^{min}_{S_{s}(q^{1})})^{*} 
    -\xi_{S_{s}(p^{1})}^{min} \otimes (\xi^{max}_{q^{1}})^{*} \right) \\ 
       &    & + g_{p^{2}, q^{2}}(0)  
      \left( \xi_{p^{2}}^{min} \otimes (\xi^{max}_{P_{s}(q^{2})})^{*}  
 - \xi_{P_{s}(p^{2})}^{max} \otimes (\xi^{min}_{q^{2}})^{*} \right)   
 \end{eqnarray*}
 
 From  condition 2 of Proposition 
 \ref{IntCstar:60},  We also know that 
 $ g_{p^{1}, q^{1}}(1) = g_{p^{2}, q^{2}}(0)$. The definition
 of $G_{m,n}$ implies that $S_{s}(q^{1}) = q^{2}, S_{s}(p^{1})= p^{2},
 P_{s}(q^{2})=q^{1} $ and $ P_{s}(p^{2})=p^{1}$ and so the result is
 zero, as desired.

  For the converse direction, from the facts that the union of
  the $A^{Y}_{m,n}$ are dense in 
  $C^{*}_{\lambda}(T^{+}({Y}_{\mathcal{B}}))$ and the function
  sending 
  $a$ in $C^{*}_{\lambda}(T^{+}({Y}_{\mathcal{B}}))$  to 
  $\Vert  [\pi_{\mathcal{B}}(a), F_{\mathcal{B}}] \Vert$
   is continuous, it suffices for us to prove that, for any 
   $m < n$ and $a$ in $A^{Y}_{m,n}$, there is $b$ in $AC^{Y}_{m,n}$
   with $a -b $ in $B_{m,n}$ and 
    $\Vert b \Vert \leq 
    2 \Vert  [\pi_{\mathcal{B}}(a), F_{\mathcal{B}}] \Vert$.

   Let $a = \sum_{p,q} \alpha_{p,q} a_{p,q}$ be in $A_{m,n}^{Y}$, 
   where the sum is over $p,q$ in $E_{m,n}^{Y}$ with $r(p)=r(q)$.
  For each  $p,q$ in $E_{m,n}^{Y}$ with $r(p)=r(q)$, define
 $ c_{p,q} = \alpha_{p,q} $ if $ (p,q) \notin 
    \alpha_{1}(G_{m,n}) $ and 
    \[
    c_{p^{1},q^{1}} = \alpha_{p^{1},q^{1}}  +  
     \frac{\alpha_{p^{1},q^{1}} - \alpha_{p^{2},q^{2}}}{2}
     \]
     for $(p,q)$ in $G_{m,n}$. We also define 
     $ d_{p,q} = \alpha_{p,q} $ if $ (p,q) \notin 
    \alpha_{2}(G_{m,n}) $    
   and 
    \[
    d_{p^{2},q^{2}} = \alpha_{p^{1},q^{1}}  +  
     \frac{\alpha_{p^{2},q^{2}} - \alpha_{p^{1},q^{1}}}{2}
     \]
      for $(p,q)$ in $G_{m,n}$. Let 
      $b_{p,q}(t) = c_{p,q}t +  d_{p,q}(1-t)$,
      for all $t$ in $[0,1]$.
 Finally, we define 
 \[
 b = \sum_{p,q} a_{p,q} \otimes b_{p,q} \circ f_{r(p)}, 
 \]
 where $f_{r(p)} : X_{r(p)}^{+}\rightarrow [0,1]$ is as before.
 So $b$ is in $AC^{Y}_{m,n}$.
     
     It is a simple computation, using 
     the results of Propositions \ref{Fred:40} 
     and \ref{IntCstar:60} to verify that 
     $ a - b$ is in $B_{m,n}$. It remains for us to
     prove that  $\Vert b \Vert \leq 
    2 \Vert  [\pi_{\mathcal{B}}(a), F_{\mathcal{B}}] \Vert$.
 
 The map sending $a_{p,q}$ to 
   $\xi_{p}^{max} \otimes(\xi^{max}_{q})^{*}$, for $p,q$ in 
   $E^{Y}_{m,n}$ with $r(p)=r(q)$, extends linearly to 
   an injective $*$-homomorphism from  $A^{Y}_{m,n}$ to 
   $\mathcal{F}^{max} $ which is necessarily isometric.
   The desired inequality  follows from this and an application of
   Lemma \ref{Fred:55}.
 \end{proof}

\section{$K$-theory}
\label{K}
The purpose of this section is to compute the $K$-theory
of the $C^{*}$-algebras considered in the section \ref{Cstar}.
It is probably more accurate to say that we shall investigate 
the relations between the $K$-theory of the $C^{*}$-algebras.
We remark that elements of the $K_{1}$-group of any 
$C^{*}$-algebra, $A$,
 are given by equivalence classes over matrix algebras over the
 unitization of $A$, which we denote by $A^{\sim}$.

Given a bi-infinite
Bratteli diagram $\mathcal{B}$, the $K$-theory of the
 AF-algebra $ C^{*}_{\lambda}\left(T^{+}(X_{\mathcal{B}})\right)$ is readily computable 
 from the data given and the results of Proposition \ref{AF:40}.
 It is worth noting at this point that it  does not depend
 on the order structure, nor the half of the diagram indexed by the 
 negative integers.

 \begin{thm}
 \label{K:5}
 Let $\mathcal{B}$ be a bi-infinite Bratteli diagram.
For each integer $n$, we consider $E_{n}$ to be the 
$\# V_{n} \times \# V_{n-1}$ positive integer matrix which describes
the edge set $E_{n}$. We have 
\[
K_{0}\left(C^{*}_{\lambda}(T^{+}(X_{\mathcal{B}}))\right)
\cong \lim_{n \rightarrow +\infty}
 \Z^{\# V_{0}} \stackrel{E_{1}}{\rightarrow} 
 \Z^{\# V_{1}} \stackrel{E_{2}}{\rightarrow}  \cdots
\]
and  $K_{1}(C^{*}_{\lambda}(T^{+}(X_{\mathcal{B}}) ) =0$.
 \end{thm}
 
 As we noted in  Theorem \ref{AF:45},
  $C^{*}_{\lambda}(T^{+}(Y_{\mathcal{B}}))$ is a full hereditary subalgebra
  of $C^{*}_{\lambda}(T^{+}(X_{\mathcal{B}}))$ and hence they
  are Morita equivalent \cite{Ex:MorK}. The following 
  is an immediate consequence.
  
 \begin{thm}
 \label{K:10} 
 Let $\mathcal{B}$ be an ordered  bi-infinite Bratteli diagram
 satisfying the conditions of
Definition \ref{surface:20}.
 Then 
the inclusion \newline 
$ C^{*}_{\lambda}(T^{+}(Y_{\mathcal{B}}))  
\subseteq C^{*}_{\lambda}(T^{+}(X_{\mathcal{B}})) $
induces an 
order isomorphism  \newline 
$K_{0}( C^{*}_{\lambda}(T^{+}(Y_{\mathcal{B}})) )
 \cong K_{0}(C^{*}_{\lambda}(T^{+}(X_{\mathcal{B}})) )$
and $K_{1}( C^{*}_{\lambda}(T^{+}(Y_{\mathcal{B}})) ) =0$.
 \end{thm}
 
We now turn to the $C^{*}$-algebra 
$ C^{*}_{\lambda}(T^{\sharp}(S^{s}_{\mathcal{B}})) $, first considering
its $K_{1}$-group.

\begin{prop}
\label{K:20}
Let $\mathcal{B}$ be a 
bi-infinite  ordered Bratteli diagram satisfying the conditions of
Definition \ref{surface:20}.
Let $m < n$, $v$ be any vertex in $V_{n}$, $p$ be in $E_{m,n}^{Y}$
with $r(p)=v$
 and  $f_{v}:[0, \nu_{r}(v)] \rightarrow [0,1]$
be any continuous function with $f_{v}(0)=0, f(\nu_{r}(v))=1$. Then 
$K_{1}(C^{*}_{\lambda}(T^{\sharp}(S^{s}_{\mathcal{B}}))) \cong \Z$ and 
is generated by the unitary 
$u = exp( 2 \pi i f_{v} \circ \varphi_{r}^{v}(x)) a_{p,p} + (1 - a_{p,p})$ 
considered as an
 element of 
$B_{m,n}^{\sim} \subseteq 
C^{*}_{\lambda}(T^{\sharp}(S^{s}_{\mathcal{B}}))^{\sim}  $.
\end{prop}

\begin{proof}
We use the fact that $C^{*}_{\lambda}(T^{\sharp}(S^{s}_{\mathcal{B}})) $ 
is the closure
of the  union of the $B_{m,n} , m < n$, so 
\[
K_{1} (C^{*}_{\lambda}(T^{\sharp}(S^{s}_{\mathcal{B}})) )= \lim_{n \rightarrow \infty} K_{1}(B_{-n,n} ).
\]
We will first compute $K_{1}(B_{-n,n})$ and then the 
inductive limit.

We use  with the short exact sequence found 
in  Corollary \ref{IntCstar:90}
$$\xymatrix{ 0 \ar[r] & \displaystyle \bigoplus_{v \in V_{n}}  
A_{m,n,v} \otimes C_{0}(0, \nu_{s}(v))  
\ar[r]  &
B_{m,n} \ar[r]  &   C^{*}(G_{m,n}) \ar[r]  &  0.}$$
For simplicity, we denote $A_{m,n,v} \otimes C_{0}(0, \nu_{s}(v)) $
by $\mathcal{I}_{v}$.
We have the associated six-term exact sequence 
for $K$-groups
$$\xymatrix{ K_{0}\left(\displaystyle\bigoplus_{v \in V_{n}}  \mathcal{I}_{v} \right)\ar[r] & 
K_{0}(B_{m,n})  \ar[r]  &  K_{0}( C^{*}(G_{m,n})) \ar[d]  \\
 K_{1}( C^{*}(G_{m,n})) \ar[u] & K_{1}(B_{m,n})  \ar[l] 
   & K_{1}\left(\displaystyle\bigoplus_{v \in V_{n}}  
\mathcal{I}_{v} \right) \ar[l] 
}$$

Let us start with 
$K_{*}(\oplus_{v} \mathcal{I}_{v}) \cong 
\bigoplus_{v} K_{*}(\mathcal{I}_{v})$. As $A_{m,n,v}$ is a 
full matrix algebra, we have 
$K_{0}(\mathcal{I}_{v}) \cong 
K_{0}(C_{0}(0,\nu_{v}(r))) \cong K_{1}(\C) =0$
while $K_{1}(\mathcal{I}_{v}) \cong 
K_{1}(C_{0}(0,\nu_{v}(r))) \cong K_{0}(\C) \cong \Z$. 
Moreover, if $p, f_{v}$ are as above, then 
$u_{v} = exp( 2 \pi i f \circ \varphi_{r}^{v}(x)) a_{p,p} + (1 - a_{p,p})$
 is a generator of this group.
 
 We now turn to $K_{*}(C^{*}(G_{-n,n}))$. The groupoid 
 $G_{-n,n}$ is finite and its $C^{*}$-algebra is finite-dimensional. 
 Hence its $K_{1}$-group is trivial. On the other hand, 
 it is a direct sum of full matrix algebras, indexed by the 
 elements of $r(p), p \in E^{s}_{-n,n}$. It follows that 
 $K_{*}(C^{*}(G_{-n,n})) \cong \bigoplus_{r(E^{s}_{-n,n})} \Z$, 
 with generators $[ \chi_{(p,p)} ]_{0}$, where  $p$ is 
 chosen to be any path in $E^{s}_{-n,n}$, as $r(p)$ takes all 
 possible values.
 
 Our six-term exact sequence now looks like
$$
\xymatrix{ 0 \ar[r] & 
K_{0}(B_{-n,n})  \ar[r]  &  \displaystyle\bigoplus_{r(E^{s}_{-n,n})} \Z \ar^{exp}[d]  \\
 0 \ar[u] & K_{1}(B_{-n,n})  \ar[l] 
   & \displaystyle\bigoplus_{v \in V_{n}} \Z 
 \ar[l] 
}$$

It is a fairly standard argument to check that the 
exponential map takes $[ \chi_{(p,p)} ]_{0}$
in $K_{0}(C^{*}(G_{-n,n})) $, where  $p$ is 
  any path in $E^{s}_{-n,n}$, to 
 $[u_{p_{1}}]_{1} - [u_{p_{2}}]$ in 
 $\bigoplus_{v} K_{1}(\mathcal{I}_{v})$.
 
 From this we can see that the exponential map is not surjective; indeed
 for any fixed $v$, the elements  $m [u_{v}]_{1}$ 
 are all
 distinct in $K_{1}(B_{-n,n})$.
 
 To compute the inductive limit, it suffices to show that, for $v$ in 
 $V_{n}$ and $v'$ in $V_{n'}$ with $n' > n$, we have 
 $[u_{v}]_{1} = [u_{v'}]_{1}$, as elements of 
 $K_{1}(B_{m,n'})$,   provided that there is at least one 
 path $p'$ from $v$ to $v'$. Let $p$ be any element of 
 $E_{m,n}$ with $r(p) =v$ and let $f_{v'}$ be any function as above. 
 Then define $f_{v}$ as follows 
 \[
 f_{v}(x) = \left\{ \begin{array}{cl} 
     0 &  x_{(n,n']} <_{s} p' \\
     f_{v'}(x_{(n', \infty)}) &  x_{(n,n']} = p' \\
    1 &      x_{(n,n']} <_{s} p' \end{array} \right.
    \]
 It is easy to see that $f_{v}$ satisfies
  the desired conditions and that, with these choices, 
 $u_{v}=u_{v'}$.
\end{proof}

To describe the $K$-zero group, we need to establish some notation.

For any finite set $A$, let $\Z A$ denote the free abelian group on $A$. 
Recalling the definition of $I_{\mathcal{B}}
 \star_{\Delta} J_{\mathcal{B}}$
from Definition \ref{gpds:25}, we define 
\begin{eqnarray*}
\theta_{1} : \Z (I_{\mathcal{B} } \star_{\Delta}
 J_{\mathcal{B}} ) & \rightarrow & 
\Z \{ x_{1}, \ldots, x_{I_{\mathcal{B} }} \}, \\
\theta_{2} : \Z (I_{\mathcal{B} } \star_{\Delta} J_{\mathcal{B} }
 ) & \rightarrow & 
\Z \{ x_{I_{\mathcal{B} }+1}, \ldots, x_{J_{\mathcal{B} }} \}, \\
\theta : \Z (I_{\mathcal{B} } \star_{\Delta} J_{\mathcal{B} } ) 
& \rightarrow & \Z 
\{ x_{1}, \ldots, x_{I_{\mathcal{B} } + J_{\mathcal{B} }} \} \\
\sigma : \Z \{ x_{1}, \ldots, x_{I_{\mathcal{B} }+J_{\mathcal{B} }} \}
 & \rightarrow & \Z
\end{eqnarray*}
by $\theta_{1}(x_{i}, x_{j}) = x_{i}, \theta_{2}(x_{i}, x_{j})  = x_{j}$,
$\theta(x_{i}, x_{j})  = x_{i} + x_{j}$ and $\sigma(x_{i}) = 1, 
1 \leq i \leq I_{\mathcal{B}}, \sigma(x_{j}) = -1, I_{\mathcal{B}} < 
j \leq I_{\mathcal{B}} +
J_{\mathcal{B}}$. Observe that $\sigma \circ \theta =0$, 
so $\sigma$ also defines a homomorphism from $coker( \theta)$ to $\Z$.

 \begin{thm}
 \label{K:30}
 Let $\mathcal{B}$ be a 
bi-infinite  ordered Bratteli diagram satisfying the conditions of
Definition \ref{surface:20}.
There is a short exact sequence
\[
\xymatrix{ 0  \ar[r] & ker(\theta) \ar[r] & 
K_{0}( C^{*}_{\lambda}(T^{\sharp}(S^{s}_{\mathcal{B}}))) \ar^{i_{*}}[r]  &
K_{0}(C^{*}_{\lambda}(T^{+}(Y_{\mathcal{B}})) ) \ar[r] & coker(\theta) \ar^{\sigma}[r] & \Z \ar[r] & 0}
\]
where $i:  C^{*}_{\lambda}(T^{\sharp}(S^{s}_{\mathcal{B}})) \rightarrow C^{*}_{\lambda}(T^{+}(Y_{\mathcal{B}}))$ 
is the inclusion map. 
In particular, $K_{0}(C^{*}_{\lambda}(T^{\sharp}(S^{s}_{\mathcal{B}})) )$ is finite rank and is finitely generated
if and only if $K_{0}( C^{*}_{\lambda}(T^{+}(Y_{\mathcal{B}})))$ is.
If either $I_{\mathcal{B}}=1$ or 
$J_{\mathcal{B}}=1$, then $i_{*}$ is an isomorphism. 
 \end{thm}

 \begin{proof}
 We make use of the notion of the relative $K$-theory for $C^{*}$-algebras 
 along with an excision result of the second author
 \cite{Put:Kexc}.  Relative $K$-theory was 
 introduced by Karoubi \cite{Ka:K}, but we also refer the reader to
 \cite{Put:Kexc} or Haslehurst \cite{Has:RelK} for a more extensive treatment. 
 To any $C^{*}$-algebra, $A$,
 and $C^{*}$-subalgebra, $A' \subseteq A$, there  are 
 relative $K$-groups, $K_{i}(A';A), i = 0,1$ which fit into a 
 six-term exact sequence
 $$\xymatrix{ K_{0}(A';A)   \ar[r] & K_{0}(A') \ar[r]^{i_{*}} & 
K_{0}(A) \ar[d]  \\
K_{1}(A) \ar[u] &  K_{1}(A')  \ar[l]^{i_{*}} &  K_{1}(A';A) \ar[l] }$$
where $i:A' \rightarrow A$ denote the inclusion map.

In Theorems 3.2 and 3.4 of 
\cite{Put:Kexc}, the situation is described of $C^{*}$-algebras $A, B, E$
along with a bounded $*$-derivation $\delta: A + B \rightarrow E$ 
such that there is a natural isomorphism 
$K_{*}(\ker(\delta) \cap A; A) \cong K_{*}( \ker(\delta) \cap B; B)$. 
Referring back to notation established in Definition \ref{Fred:10}, 
we use $A = \bigoplus_{i=1}^{I_{\mathcal{B}}+J_{\mathcal{B}}}
 \mathcal{K}(\mathcal{H}_{i})$, 
where $\mathcal{K}$ denotes the $C^{*}$-algebra of compact operators, 
$B= C^{*}_{\lambda}(T^{+}(Y_{\mathcal{B}}))$ or more accurately, 
$B= \pi_{\mathcal{B}}(C^{*}_{\lambda}(T^{+}(Y_{\mathcal{B}})) )$. 
As we noted earlier, the 
representation is faithful under our hypotheses, so this 
amounts to a notational difference only.
We use $E = \mathcal{B}(\mathcal{H}_{\mathcal{B}})$, 
the algebra of bounded linear operators on $\mathcal{H}_{\mathcal{B}}$ and 
$\delta(x) = i \left[ x, F_{\mathcal{B}} \right]$, for any operator $x$.
(The use of $\mathcal{B}$ for the bounded linear operators and for 
the Bratteli diagram, is unfortunate, but
 should not cause any confusion.)
 
 We  need to verify the hypotheses of \cite{Put:Kexc} hold. The first is that 
 $AB \subseteq A$ and this follows from the facts that, for all
 $i$, $\mathcal{H}_{i}$ is invariant for the representation 
 $\pi_{\mathcal{B}}$  and that $A$ consists entirely of compact 
 operators on this space. 
 
 The hypotheses of Theorem 3.4 of \cite{Put:Kexc} involve the choice
 of a dense $*$-subalgebra, $\mathcal{A} \subseteq A$. For this, we
 use the linear span of all rank one operators of the form 
 $\xi_{p}^{max} \otimes (\xi_{q}^{max})^{*}$ and 
 $\xi_{p}^{min} \otimes (\xi_{q}^{min})^{*}$, where $p,q$ vary over
 $E^{Y}_{m,n}$ with $r(p)=r(q)$ and $m < n$ vary over all integers.
 
 We now verify property C1 from Theorem 3.4 of \cite{Put:Kexc}: let 
 \[
  a =\sum \alpha^{max}_{p,q} \xi_{p}^{max} \otimes (\xi_{q}^{max})^{*} 
 + \alpha^{min}_{p,q} \xi_{p}^{min} \otimes (\xi_{q}^{min})^{*},
 \]
  where 
 the sum is over  $p,q$ in
 $E^{Y}_{m,n}$ with $r(p)=r(q)$,
 be in $\mathcal{A}$.  Let 
 \[
 a' = \sum \frac{\alpha^{max}_{p_{1}, q_{1}} +
  \alpha^{min}_{p_{2}, q_{2}}}{2}
 \left(  \xi_{p_{1}}^{max} \otimes (\xi_{q_{1}}^{max})^{*}
   + \xi_{p_{2}}^{min} \otimes (\xi_{q_{2}}^{min})^{*} \right)
   \]
   where the sum is over all 
   $((p_{1}, q_{1}), (p_{2}, q_{2}))$ in $G_{m,n}$.
   It is an easy calculation that 
   $\delta(a') = i [a', F_{\mathcal{B}}] =0$ and 
   $\Vert a - a' \Vert \leq \frac{3}{2} \Vert \delta(a) \Vert$ follows
   immediately from the first part of Proposition \ref{Fred:40} 
   and Lemma \ref{Fred:55}.
   
   Using the dense $*$-subalgebra $\bigcup_{n} A^{Y}_{-n,n}$ of  
$ C^{*}_{\lambda}(T^{+}(Y_{\mathcal{B}}))$  ,
   and Lemma 4.2 of \cite{Put:Kexc},
    we see that $\delta(B) \subseteq \delta(A)$.
   
   It remains to see that condition C2 of Theorem 3.4 of
    \cite{Put:Kexc} holds.
   For that, we can assume that the $a_{1}, \ldots, a_{I}$ all lie
   in some $span\{ \xi^{max}_{p} \otimes (\xi^{max}_{q})^{*}, 
   \xi^{min}_{p} \otimes (\xi^{min}_{q})^{*} \}$, 
   as $p,q $ range over $E^{Y}_{m,n}$, 
   for some $m,n$. We let $e$ be the unit of this algebra, 
   \[
   e= \sum_{p \in E^{Y}_{m,n}} \xi_{p}^{max} \otimes (\xi_{p}^{max})^{*} 
   +\xi_{p}^{min} \otimes (\xi_{p}^{min})^{*} 
   \]
   and for
   \[
  a_{i} =\sum \alpha^{max}_{p,q} \xi_{p}^{max} \otimes (\xi_{q}^{max})^{*} 
 + \alpha^{min}_{p,q} \xi_{p}^{min} \otimes (\xi_{q}^{min})^{*},
 \]
 we use $b_{i}$ in $AC^{Y+}_{m,n}$ defined by 
 \[
 b_{i} = \sum a_{p,q} \otimes \left( \alpha^{max}_{p,q}  f_{r(p)}
 +  \alpha^{min}_{p,q}   (1 -f_{r(p)}) \right)
 \]
 where $f_{v}(x) = (\nu_{s}^{v})^{-1}\varphi^{v}_{s}(x)$,
  for $x$ in $\bigcup_{v} X_{v}^{+}$, is as in the last section.
 The desired properties follow from Proposition
  \ref{Fred:40}; we omit the details.
 We have verified the conditions of Theorem 3.4 of \cite{Put:Kexc}. In addition, 
 Theorem \ref{Fred:60} shows that 
 $ C^{*}_{\lambda}(T^{\sharp}(S^{s}_{\mathcal{B}})) = \ker(\delta) \cap C^{*}_{\lambda}(T^{+}(Y_{\mathcal{B}})) $.
 We conclude that
 conclude that $K_{*}(\ker(\delta) \cap A; A) 
 \cong K_{*}(C^{*}_{\lambda}(T^{\sharp}(S^{s}_{\mathcal{B}}));
   C^{*}_{\lambda}(T^{\sharp}(Y_{\mathcal{B}})))$.

 We now turn to the computation of 
 $K_{*}( \ker(\delta) \cap A; A)$. Recall that
 $A = \bigoplus_{i=1}^{I_{\mathcal{B}}+J_{\mathcal{B}}}
  \mathcal{K}(\mathcal{H}_{i})$. For 
 $(i,j)$ in $I_{\mathcal{B}} \star_{\Delta} J_{\mathcal{B}}$, we define 
 \[
 \mathcal{H}_{i,j} = \mathcal{H}_{i}  \cap F_{\mathcal{B}}\mathcal{H}_{j}  = 
 L^{2}(T^{+}(x_{i}) \cap \Delta_{s}(T^{+}(x_{j})))
 \]
 and observe that 
 $F_{\mathcal{B}}H_{i,j} = 
 L^{2}(\Delta_{s}(T^{+}(x_{i})) \cap T^{+}(x_{j})).$
 It is a simple matter to check that the map sending 
 $(k_{i,j})_{(i,j)} $  to 
 $\sum_{(i,j)} k_{i,j} + F_{\mathcal{B}} k_{i,j} F_{\mathcal{B}}$ 
 is an isomorphism 
 between 
 $ \oplus_{I_{\mathcal{B}} \star_{\Delta} J_{\mathcal{B}}} 
 \mathcal{K}(H_{i,j})$
 and  $\ker(\delta) \cap A $.
 
 For any Hilbert space $\mathcal{H}$, there is a canonical isomorphism 
 from $K_{0}(\mathcal{K}(\mathcal{H}))$ to $\Z$ induced by 
 the trace. In addition, 
 we have $K_{1}(\mathcal{K}(\mathcal{H})) \cong 0$ 
 (see \cite{Ex:MorK}). Hence, we have 
 $K_{1}(A) \cong K_{1}(  \ker(\delta) \cap A) \cong 0$, 
 $K_{0}(A) \cong \Z\{ x_{1}, \ldots, x_{I}\}
  \oplus \Z\{ x_{I+1}, \ldots, x_{J}\}$
 and $K_{0}( \ker(\delta) \cap A) \cong \Z I_{\mathcal{B}}
  \star_{\Delta} J_{\mathcal{B}}$. Moreover, the 
 map induced by the inclusion  $\ker(\delta) \cap A \subseteq A$ is simply
 $\theta$. In summary, the six-term exact sequence for the relative groups
 of the inclusion becomes
 \[
 \xymatrix{ K_{0}(\ker(\delta) \cap A;A) \ar[r] & 
  \Z I_{\mathcal{B}} \star_{\Delta} J_{\mathcal{B}}
   \ar[r]^{\theta} & 
 \Z\{ x_{1}, \ldots, x_{I_{\mathcal{B}}
 +J_{\mathcal{B}}}\} \ar[d] \\
  0 \ar[u] &  0 \ar[l] & K_{1}(\ker(\delta) \cap A;A) \ar[l]}
  \]
 and so $K_{0}(\ker(\delta) \cap A;A) \cong ker(\theta)$ and 
 $K_{1}(\ker(\delta) \cap A;A) \cong coker(\theta)$. Combining this with the 
 computation of the relative groups already done above and the results
 of Theorem \ref{K:10} and Proposition \ref{K:20} completes the proof.
 
 The remaining statements are straightforward. In particular, it is a simple
 matter to check that  if 
  $I_{\mathcal{B}}=1$, then 
  $I_{\mathcal{B}} \star_{\Delta} J_{\mathcal{B}} = 
  \{ 1 \} \times \{2, \ldots , J_{\mathcal{B}} \}$ 
  and $\theta(1,j) = x_{1} + x_{j}$, 
  for $ 1 < j \leq J_{\mathcal{B}}$, which is clearly injective
  and has $coker(\theta) \cong \Z$.
 \end{proof}
 
A crucial part of K-theory (at least $K_{0}$) for a $C^{*}$-algebra is
its natural order structure. As a simple example, if $\alpha, \beta$
are any two irrational numbers, then the subgroups of the real numbers 
$\Z + \alpha \Z $ and $\Z + \beta \Z$ are isomorphic as 
abstract groups, but with the relative orders from the real numbers,
they are not isomorphic in general as ordered groups. One 
of the difficulties in operator algebra K-theory is that many 
computational tools do not respect the order structure. As an 
example here, while we may easily check in some specific situation
that the map $i_{*}$ of Theorem \ref{K:30} is an isomorphism, 
it does not follow at once that it is an isomorphism of ordered
groups. Part of that is easily dealt with: the fact that
it is induced by a $*$-homomorphism of $ C^{*}_{\lambda}(T^{\sharp}(S^{s}_{\mathcal{B} }))$ in 
$ C^{*}_{\lambda}(T^{\sharp}(Y_{\mathcal{B} }))$ means that it is a positive homomorphism 
in the sense it maps the positive cone in the former into the 
positive cone in the latter.

 \begin{thm}
 \label{K:38}
Let $\mathcal{B}$ be an ordered Bratteli diagram satisfying the conditions of
Definition \ref{surface:20}. 
If the following  sequence is exact
$$\xymatrix{ 0  \ar[r] & 
\Z I_{\mathcal{B}} \star_{\Delta} J_{\mathcal{B} } 
\ar^{\hspace{-.3in}\theta}[r] & \Z \{ x_{1}, \ldots, x_{I_{\mathcal{B} }+J_{\mathcal{B} }} \}
   \ar^{\hspace{.5in}\sigma}[r] &  \Z \ar[r] & 0}$$
and the equivalence classes of the relation $T^{\sharp}(Y_{\mathcal{B}})$
are all dense, then 
\[
 i_{*}: K_{0}(C^{*}_{\lambda}(T^{\sharp}(S^{s}_{\mathcal{B} }))) 
 \rightarrow K_{0}( C^{*}_{\lambda}(T^{\sharp}(Y_{\mathcal{B} })))
 \] 
 is an isomorphism of ordered abelian groups.
 If 
 $I_{\mathcal{B}} = 1$ or $J_{\mathcal{B}}=1$,
  then both hypotheses hold: the sequence above is exact and 
  the equivalence classes of the relation $T^{\sharp}(Y_{\mathcal{B}})$
are all dense.
 \end{thm}

\begin{proof}
We know already from the last theorem  and the hypothesis
on the exact sequence that $i_{*}$ is
an isomorphism and since it is induced by a $*$-homomorphism 
at the level of $C^{*}$-algebras, it maps positive elements
to positive elements. It remains for us to show that 
every positive element
of $K_{0}(C^{*}_{\lambda}(T^{+}(Y_{\mathcal{B}})))$ 
is the image of a positive
element of $K_{0}(C^{*}_{\lambda}(T^{+}(S^{s}_{\mathcal{B}})) )$. 
In view of Theorems \ref{K:5}
and \ref{K:10},  it suffices to 
consider a projection in 
$ C^{*}_{\lambda}(T^{+}(Y_{\mathcal{B}}))$ of the form 
$a_{p,p}$, where $p$ is in $E_{m,n}^{Y}$, for some $m < n$, and show 
it is Murray-von Neumann
equivalent to one in $C^{*}_{\lambda}(T^{+}(S^{s}_{\mathcal{B}})) $.

Consider two points $x, y$ in $X_{r(p)}$ satisfying the following:
$ x \leq_{s} y$, $X_{r(p)}^{-}x \subseteq T^{+}(x_{i}) $ and  
$X_{r(p)}^{-}y \subseteq T^{+}(x_{j}) $, for some 
$ 1 \leq i \leq I_{\mathcal{B}} < j \leq I_{\mathcal{B}}+ J_{\mathcal{B}}$.
It follows that $a_{p,p} \otimes \chi_{[x, y]}$ is in
$ A^{Y}_{m,n} \otimes C(X_{r(p)}^{+}) = AC_{m,n}$ and so it determines
a class in $K_{0}( C^{*}_{\lambda}(T^{+}(S^{s}_{\mathcal{B}})) )$. 

Observe that as $p$ is not
 $s$-maximal or $s$-minimal,
  $\Delta_{s}(X_{s(p)}^{-}px)_{(m, \infty)}$ is a single path, as is 
 $\Delta_{s}(X_{s(p)}^{-}py)_{(m, \infty)}$. 
 In particular,    $\Delta_{s}(X_{s(p)}^{-}px)$ is contained 
 in  $T^{+}(x_{j'})$, for some $j'$, while 
 $\Delta_{s}(X_{s(p)}^{-}py)$ is contained 
 in  $T^{+}(x_{i'})$, for some $i'$. 

We first consider the special case that $j'=j$. 
(The case $i'=i$ can be done in a similar way.) This 
means we can find $N > n$ such that 
$ \Delta_{s}(X_{s(p)}^{-}px)_{( N, \infty)} = y_{(N, \infty)}$. 
Let $\bar{p} = \Delta_{s}(X_{s(p)}^{-}px)_{( m, N] }$.

We define 
\begin{eqnarray*}
w  & =  &  \sum_{p'} a_{p',p'} +  \cos \left( \frac{\pi}{2} \nu_{r}(r(y_{N}))^{-1}
\varphi^{r(y_{N})}_{r}(z) \right) a_{y_{(m,N]}, y_{(m,N]}} \\
  &  & 
+ \sin \left( \frac{\pi}{2} \nu_{r}(r(y_{N}))^{-1}
\varphi^{r(y_{N})}_{r}(z) \right) a_{y_{(m,N]}, \bar{p}}
\end{eqnarray*}
where the sum is over all $p'$ in $E_{m,N}$ with 
 $ px \leq_{s} p' <_{s} py$ and the variable $z$
lies in $X_{r(y_{N})}^{+}$. It is a simple matter to check that
$w^{*}w =  a_{p,p} \otimes \chi_{[x, y]}$ while 
$ww^{*}$ lies in $AC_{m,N}$. We conclude that the class of
 $a_{p,p} \otimes \chi_{[x, y]}$  lies in the image of $i_{*}$.
 
 We now consider the general case, dropping the hypothesis that $j'=i$.
 It is clear that $x_{i'} - x_{j'}$ lies in the kernel of
 $\sigma$. It follows that we may find a finite sequence
 $(i_{l}, j_{l}), 1 \leq l \leq L$ in 
 $I_{\mathcal{B}} \star_{\Delta}J_{\mathcal{B}}$ such that
 $j_{1} = j', j_{l+1} = i_{l}, 1 \leq l < L$ and 
 $i_{L}=i'$. By the minimality of
  $T^{+}(x_{i_{1}}) \cap \Delta_{s}(T^{+}(x_{j_{1}}) )$, we may find 
  $y_{1}$ in $X^{+}_{r(p)}$ with $ x <_{s} y_{1} <_{s} y$
  with 
  \[
  X_{r(p)}^{-}y_{1} \subseteq T^{+}(x_{i_{1}})
   \cap \Delta_{s}(T^{+}(x_{j_{1}}) ).
  \]
  By application of the special case above, the class of 
  $a_{p,p} \otimes \chi_{[x, y_{1}]}$ lies in the image of $i_{*}$.
  Continuing in this way, we may construct
  $ x <_{s} y_{1} <_{s}  y_{2} <_{s} \cdots <_{s} y_{L} <_{s} y$
  such that $y_{l}$ is in $T^{+}(x_{i_{l}}) \cap \Delta_{s}(T^{+}(x_{j_{l}}) )$ and 
 the class of 
  $a_{p,p} \otimes \chi_{[y_{l}, y_{l+1}]}$ and also
 $a_{p,p} \otimes \chi_{[y_{L}, y]}$  lie in the image of $i_{*}$. 
 We conclude that the class of 
 \[
   a_{p,p} \otimes \chi_{[x, y]} = a_{p,p} \otimes \chi_{[x, y_{1}]}
   + \sum_{l=1}^{L-1} a_{p,p} \otimes \chi_{[y_{l}, y_{l+1}]}
    + a_{p,p} \otimes \chi_{[y_{L}, y]}
    \]
    also lies in the image of $i_{*}$.
    Finally, we note that if we choose $x= x_{r(p)}^{s-min} $ 
    and $y = x_{r(p)}^{s-max}$,
     then $a_{p,p} = a_{p,p} \otimes \chi_{[x,y]}$.

We finish by considering the case $I_{\mathcal{B}} =1$ ( with 
$J_{\mathcal{B}}=1$ being similar). For any
 $ 2 \leq j  \leq 1+ J_{\mathcal{B}}$,
  we know
 that $\Delta_{s}(T^{+}(x_{j}))$ must be contained in $T^{+}(x_{1})$, so
 $I_{\mathcal{B}} \star_{\Delta} J _{\mathcal{B}}= \{ 1 \} \times 
 \{ 2, \ldots, 1+ J_{\mathcal{B}} \}$
 and it is a simple matter to verify the given sequence is exact.
 Next, we also have 
 \newline $\Delta_{s}(T^{+}(x_{1})) \cap T^{+}(x_{j}) =T^{+}(x_{j}) $
 which is dense by our hypotheses on $\mathcal{B}$. It follows 
 that every equivalence class in $T^{\sharp}(Y_{\mathcal{B}})$ is dense. 
\end{proof}

 We finally turn to the K-theory of the foliation algebra 
 of $(S_{\mathcal{B}}, \mathcal{F}_{\mathcal{B}}^{+})$.

\begin{rmk}
\label{K:39}
As the foliation $\mathcal{F}_{\mathcal{B}}^{+}$ arises from an action 
of $\R$ on the space $S_{\mathcal{B}}$, Connes' analogue of the Thom 
isomorphism Theorem (see 10.2.2 of \cite{Bla:K}) asserts that 
\[
K_{i}(C^{*}( \mathcal{F}_{\mathcal{B}}^{+})) \cong K^{i+1}(S_{\mathcal{B}}).
\]
On the other hand, this is not terribly useful at the moment, since
we don't know the $K$-theory (or cohomology) of the space
$S_{\mathcal{B}}$, nor does it seem particularly likely that it can 
be computed directly, given our construction. In any event, 
Connes' result does not reveal anything about the order
structure on the $K_{0}$ group of the foliation algebra. 
Instead, we will compute its $K$-theory as it relates
to our AF-algebra. Having done this, we can then use Connes' result
to compute the $K$-theory of our surface.
\end{rmk}

  We begin 
 by recalling some notation. We let $\mathcal{I}_{\mathcal{B}}$
 be the collection of connected subsets of the union of
 $\pi(T^{+}(x_{i}) \cap \Delta_{s}(T^{+}(x_{j})))$ over all 
 $(i,j)$ in $I_{\mathcal{B}} \star_{\Delta} J_{\mathcal{B}}$.
 We also recall that each such subset is homeomorphic to $\R$.
 Now, for each $I$ in $\mathcal{I}_{\mathcal{B}}$, we define 
 $\iota(I) = (i,j)$, if 
 $I \subseteq \pi(T^{+}(x_{i}) \cap \Delta_{s}(T^{+}(x_{j})))$.
 It is clearly surjective.
We also let $\iota$ be the map induced  from 
$\Z \mathcal{I}_{\mathcal{B}}$ to 
$\Z ( I_{\mathcal{B}} \star_{\Delta} J_{\mathcal{B}} )$.
 
We are going to construct a sequence of groupoids and
$C^{*}$-algebras interpolating between  $\mathcal{F}_{\mathcal{B}}^{+}$
and
$T^{\sharp}(S_{\mathcal{B}})$. 
 Let us begin by selecting 
 $\mathcal{I}_{0} \subseteq \mathcal{I}_{\mathcal{B}}$
  which contains exactly one interval from each set
  $\pi(T^{+}(x_{i}) \cap \Delta_{s}(T^{+}(x_{j}))$. That is, 
  $\iota: \mathcal{I}_{0} \rightarrow  
  I_{\mathcal{B}} \star_{\Delta} J_{\mathcal{B}}$
   is 
  a bijection. We then enumerate the remaining intervals of 
  $\mathcal{I}_{\mathcal{B}} - \mathcal{I}_{0}$ as $I_{1}, I_{2}, \ldots $.
   Although 
  this may be finite, we will ignore that in our notation.
   Observe that, for each $l \geq 1$, there is a unique $I'_{l}$ 
  in $\mathcal{I}_{0}$ such that $\iota(I_{l}) = \iota(I'_{l})$ and  the collection
   $I_{l} - I'_{l}$, $l \geq 1$ is a set of generators for 
   $ker(\iota)$ having no relations.
   
   We define a sequence of groupoids, beginning with 
   $\mathcal{F}_{0}^{+} = \mathcal{F}_{\mathcal{B}}^{+}$. Then for 
   $l \geq 1$, set $\mathcal{F}_{l}^{+}$ to be the union 
   of $\mathcal{F}_{l-1}^{+}$ with all sets $I_{l} \times I$ and 
   $I \times I_{l}$, where $I$ is in  $\mathcal{F}_{l-1}^{+}$ and 
   satisfies 
   $\iota(I) = \iota(I_{l})$. That is, on the set $\cup_{j \leq l} I_{j}$,
   $\mathcal{F}_{l}^{+}$ agrees with $T^{\sharp}(S_{\mathcal{B}})$, while
   on $\cup_{j >l} I_{j}$, it agrees with $\mathcal{F}_{\mathcal{B}}$.
   We leave it as a simple exercise to check that $\mathcal{F}_{l}^{+}$
   is an open subgroupoid of $T^{\sharp}(S_{\mathcal{B}})$, 
   $\mathcal{F}_{l}^{+}$
   is an open subgroupoid of $\mathcal{F}_{l+1}^{+}$
    and the union
   over all $l$ is $T^{\sharp}(S_{\mathcal{B}})$.
   
   We let  $j_{l}$ to denote the inclusion of 
   $C^{*}(\mathcal{F}_{l}^{+})$ in $C^{*}(T^{\sharp}(S_{\mathcal{B}}))$ 
   and $i_{l,k}$ to denote the inclusion of 
   $C^{*}(\mathcal{F}_{k}^{+})$ in $C^{*}(\mathcal{F}_{l}^{+})$, for $k \leq l$.

 \begin{thm}
 \label{K:40}
 Let $l \geq 0$ and let $j$ denotes the inclusion of 
  $C^{*}(\mathcal{F}_{l}^{+})$ in 
 $C^{*}_{r}(T^{\sharp}(S_{\mathcal{B}}))$, then 
 \[
 (j_{l})_{*}: K_{1}( C^{*}(\mathcal{F}_{l}^{+}) ) \rightarrow
 K_{1}(C^{*}_{r}(T^{\sharp}(S_{\mathcal{B}}))) \cong \Z
 \]
 is an isomorphism. 
 \end{thm}

 \begin{proof}
For $m < n$, we define the groupoid  
$H_{l,m,n} $ to be all $(p,q)$   in $G_{m,n}$ such that 
$I(p) = I(q)$ if either equals $I_{j}$, for some $j > l$. 
Recall the short exact sequence of Proposition \ref{IntCstar:90}:
$$\xymatrix{ 0 \ar[r] & \displaystyle\bigoplus_{v \in V_{n}}  A_{m,n,v} \otimes C_{0}(0, \nu_{s}(v))  
\ar[r]  &
B_{m,n} \ar^{\hspace{-.2in}q}[r]  &   C^{*}(G_{m,n}) \ar[r]  &  0.}$$
where we have used $q$ to denote the quotient map.
As $H_{l,m,n}$ is a subgroupoid of $G_{m,n}$, there is a natural inclusion
of their $C^{*}$-algebras, which we also denote $j_{l}$. 
We  define a subalgebra $C_{l,m,n}$ of
 $B_{m,n} \cap C^{*}(\mathcal{F}_{l}^{+}) $
 as the pull-back of these two maps, $q, j_{l}$.
  The inclusion coincides with
 our definition of $j_{l}$.
 That is, we have  short exact sequences
$$\xymatrix{ 0 \ar[r] & \displaystyle\bigoplus_{v \in V_{n}} 
 A_{m,n,v} \otimes C_{0}(0, \nu_{s}(v))  
\ar[r]  &
B_{m,n} \ar^{\hspace{-.2in}q}[r]  &   C^{*}(G_{m,n}) \ar[r]  &  0 \\
0 \ar[r] & \displaystyle\bigoplus_{v \in V_{n}}  
A_{m,n,v} \otimes C_{0}(0, \nu_{s}(v))  
\ar[r] \ar^{=}[u] &
C_{l,m,n} \ar^{\hspace{-.2in}q}[r] \ar^{j_{l}}[u] &   
C^{*}(H_{l,m,n}) \ar[r] \ar^{j_{l}}[u] &  0 }$$ 
It is easy to check that 
such that $C^{*}(\mathcal{F}_{l}^{+}) $ is the closure of the union
of the $C_{l,-n,n}$, over $n \geq 1$. 
While the terms involving $C^{*}(H_{l,m,n})$ and $C^{*}(G_{m,n})$ 
are different,
these $C^{*}$-algebras are both finite 
dimensional and have trivial $K_{1}$-groups
and this is sufficient to conclude the inclusion of
$C_{l,m,n}$ in $B_{m,n}$ induces an isomorphism on $K_{1}$. The conclusion
follows as $C^{*}(\mathcal{F}_{l}^{+}) $ 
and $C^{*}_{r}(T^{\sharp}(S_{\mathcal{B}}))$
are inductive limits of these sequences.
 \end{proof}
 
 Let us continue to develop the ideas of this last proof. 
 It is clear from the definitions that for fixed $m,n,l$, 
 $H_{l,m,n}$ is a subgroupoid of $H_{l+1,m,n}$. It is also a simple
 matter to check that $(p,q)$ is in $H_{l+1,m,n}$, but not in 
 $H_{l,m,n}$ if and only if $I(p)=I_{l+1}, I(q)=I_{l'}, l'\leq l$
 or vice verse. If 
 $I(p)= l+1$, its equivalence class in $H_{l+1,m,n}$ 
  consists of $q$ with $(p,q)$ in $G_{m,n}$ and 
 $I(q) = I_{l'}, l' \leq l+1$ and in $H_{l,m,n}$ this becomes
 two equivalence  classes, those $q$ with $I(q) =I_{l+1}$ and those
 $q$ with $I(q) = I_{l'}, l' \leq l$. Provided that such a pair $(p,q)$
 exists, the map from $K_{0}(C^{*}(H_{l,m,n}))$ to 
 $K_{0}(C^{*}(H_{l+1,m,n})$  is surjective and has kernel generated by
 $[\delta_{(p,p)}]_{0} - [\delta_{(q,q)}]_{0}$. If we consider the 
 exact sequences on $K$-groups associated with the commutative diagram
 $$\xymatrix{ 0 \ar[r] & \displaystyle\bigoplus_{v \in V_{n}}  
 A_{m,n,v} \otimes C_{0}(0, \nu_{s}(v))  
\ar[r]  &
C_{l+1,m,n} \ar^{\hspace{-.2in} q}[r]  &   C^{*}(H_{l+1,m,n}) \ar[r] 
 &  0 \\
0 \ar[r] & \displaystyle\bigoplus_{v \in V_{n}}
  A_{m,n,v} \otimes C_{0}(0, \nu_{s}(v))  
\ar[r] \ar^{=}[u] &
C_{l,m,n} \ar^{\hspace{-.2in}q}[r] \ar^{i_{l+1,l}}[u] &   C^{*}(H_{l,m,n}) \ar[r] \ar^{i_{l+1,l}}[u] &  0 }$$
 the $[\delta_{(p,p)}]_{0} - [\delta_{(q,q)}]_{0}$ is also in the 
 kernel of the index map
 and hence lifts to a non-zero class we denote by $\alpha_{l}$ in 
 $K_{0}(C_{l,m,n})$. As the inductive limit over $m,n$ of 
 $C_{l,m,n}$ is $C^{*}(\mathcal{F}_{l}^{+})$, $\alpha_{l}$ also
 represents a non-zero class in $K_{0}(C^{*}(\mathcal{F}_{l}^{+}))$
 which freely generates the kernel of the map to 
 $K_{0}(C^{*}(\mathcal{F}_{l+1}^{+}))$ induces by the inclusion.
 
 As for the existence of the pair $(p,q)$, we know that 
 $i(I_{l+1}) = i(I_{l'})$, for some $l' \in \mathcal{I}_{0}$.
 We may find $(i,j)$ in 
 $I_{\mathcal{B}} \star_{\Delta} J_{\mathcal{B}}$
 and $x, y $ in $T^{+}(x_{i}) \cap \Delta_{s}(T^{+}(x_{j})$ with 
 $\pi(x)$ in $I_{l+1}$ and $\pi(y)$ in $I_{l'}$.
 The fact that $(x,y)$ is in $T^{+}$, there exists $n \geq 1$ such that 
 $x_{(n,\infty)} = y_{(n, \infty)}$, As $I_{l}$ and $I_{l'}$ are open, 
 we may find $m < 0$ such that 
 $\pi(X^{-}_{s(x_{m})}x_{(m,\infty)}) \subseteq I_{l+1}$
 while $\pi(X^{-}_{s(y_{m})}x_{(m,\infty)}) \subseteq I_{l'}$. 
 Letting $p= [x_{m}, x_{n}]$ and $q = [y_{m}, y_{n}]$, this 
 pair satisfies the hypotheses for this particular $m,n$. 
 The same argument works for all lesser $m$ and greater $n$. 
  We have proved the following.

 \begin{lemma}
 \label{K:45}
 For each $l \geq 1$, with $\alpha_{l}$ as above, 
 there is a short exact sequence
 \[
 0 \rightarrow \Z\alpha_{l} \rightarrow K_{0}(C^{*}(\mathcal{F}_{l}^{+}))
 \stackrel{(i_{l+1,l})_{*}}{\longrightarrow} K_{0}(C^{*}(\mathcal{F}_{l+1}^{+})) \rightarrow 0.
 \]
 \end{lemma}

 \begin{thm}
 \label{K:50}
 There is a short exact sequence
$$\xymatrix{ 0 \ar[r] & ker(\iota)  \ar^{\hspace{-.3in}\beta}[r] &
 K_{0}( C^{*}(\mathcal{F}_{\mathcal{B}}^{+}) ) \ar^{\hspace{-.2in}j_{*}}[r] &
 K_{0}(C^{*}_{r}(T^{\sharp}(S_{\mathcal{B}}))) \ar[r] & 0}$$
 where $j$ denotes the inclusion of 
  $C^{*}(\mathcal{F}_{\mathcal{B}}^{+})$ in 
 $C^{*}_{r}(T^{\sharp}(S_{\mathcal{B}}))$.
 \end{thm}
 
 \begin{proof}
 As we noted above, we can list a free set of generators 
 for $\ker(\iota)$ as follows.
 For each $l \geq 1$, let $(i,j) = \iota(I_{l})$.
 There is a unique  $I_{l}'$ in $\mathcal{I}_{0}$ with 
 $\iota(I_{l}) = \iota(I_{l}')$ and $I_{l} - I_{l}'$, 
 as an element of $\Z \mathcal{I}_{\mathcal{B}}$ and in 
 $\ker(\iota)$. As $l \geq 1$ varies, these form a 
 free set of generators.
 
 We define the inclusion $\beta$ of $\ker(\iota)$ in 
  $K_{0}( C^{*}(\mathcal{F}_{\mathcal{B}}^{+}) ) $ as follows.
 Since each map $(i_{l+1,l})_{*}$ is surjective, so are their
 compositions. So for each $l \geq 1$, we may find 
 $\beta_{l}$ in 
 $K_{0}(C^{*}(\mathcal{F}_{0}^+)) = 
 K_{0}(C^{*}(\mathcal{F}_{\mathcal{B}}^{+}))$ 
 such that $(i_{l,0})_{*}(\beta_{l}) = \alpha_{l}$, as in Lemma 
 \ref{K:45}. For any integers
 $k_{l}, 1 \leq l \leq L$,  define
 \[
 \beta \left( \sum_{l=1}^{L}  k_{l}(I_{l} - I_{l}')  \right)
  = \sum_{l=1}^{L} k_{l} \beta_{l}.
  \]
  
  Let us first observe that, if $m > l$, then
  \[
  (i_{m,0})_{*}(\beta_{l}) =  (i_{m,l+1})_{*} \circ 
   (i_{l+1,l})_{*} \circ  (i_{l,0})_{*}(\beta_{l}) = 
  (i_{m,l+1})_{*} \circ 
   (i_{l+1,l})_{*}(\alpha_{l}) = 0.
   \]
  
  The fact that the image of $\beta$ is precisely 
  the kernel of $j_{*}$  can be seen as follows. 
As the union of the $C^{*}(\mathcal{F}_{l}^{+}), l \geq 0,$ is dense 
in $C^{*}(T^{\sharp}_{\mathcal{B}}, S_{\mathcal{B}})$,  
$K_{0}(C^{*}(T^{\sharp}_{\mathcal{B}}, S_{\mathcal{B}}))$
the inductive limit of  
\[
K_{0}(C^{*}(\mathcal{F}_{B}^{+})) = K_{0}(C^{*}(\mathcal{F}_{0}^{+}))
\stackrel{(i_{1,0})_{*}}{\rightarrow} K_{0}(C^{*}(\mathcal{F}_{1}^{+}))
\stackrel{(i_{2,1})_{*}}{\rightarrow} \cdots 
\]
First, as each  $(i_{l+1,l})_{*}$ is surjective, so is $j_{*}$.
It also follows that $a$ in $K_{0}(C^{*}(\mathcal{F}_{B}^{+}))$ is in 
the kernel of $j_{*}$ if and only if $(i_{m,0})_{*}(a) = 0$, for 
some $m \geq 1$. If $a = \beta\left( \sum_{l=1}^{L}  k_{l}(I_{l} - I_{l}') \right)$,
 then this 
holds for any $m > L$ from the definition of $\beta$ and our observation
above.

Conversely, suppose 
$0 = (i_{m,0})_{*}(a)$,
  for some $m $.
  This means that $(i_{m-1,0})_{*}(a)$ is 
  in the kernel of $(i_{m,m-1})_{*}$ and hence there is an integer $k_{m-1}$
  such that $(i_{m-1,0})_{*}(a) = k_{m-1} \alpha_{m-1}$.
   It then follows that
  \[
 (i_{m-1,m-2})_{*} \circ  (i_{m-2, 0})_{*}(a - k_{m-1} \beta_{m-1})
  = (i_{m-1, 0})_{*}(a - k_{m-1} \beta_{m-1})  = 0,
  \] 
  so we may find 
  $k_{m-2}$ such that
   $(i_{m-2, 0})_{*}(a - k_{m-1} \beta_{m-1}) = k_{m-2}\alpha_{m-2}$.
  Continuing in this way ends by seeing that 
  $a = \sum_{l=1}^{m-1} k_{l} \beta_{l}$ as desired.

 Let us finally show that $\beta$ is injective. Suppose that 
 $\beta \left( \sum_{l=1}^{L}  k_{l}(I_{l} - I_{l}')  \right) = 0$.
 It follows that 
 \[
(i_{L,0})_{*}  \left(\beta\left( \sum_{l=1}^{L}
  k_{l}(I_{l} - I_{l}')  \right) \right )  
   = \sum_{l=1}^{L} k_{l} (i_{L,0})_{*}(\beta_{l}) = k_{L}\alpha_{L},
 \]
 using again the observation above that 
 $(i_{l,0})_{*}(\beta_{l}) = 0$ if $L > l$.
 As $\alpha_{L}$ has infinite order, it follows that $k_{L}=0$.
  Continuing
    in this way shows that $k_{l}=0$, for all $1 \leq l \leq L$.
    \end{proof}
 
As we indicated earlier, knowing the $K$-theory of the 
foliation algebra allows the computation of the $K$-theory
 of the surface
$S_{\mathcal{B}}$ as an immediate consequence of Connes' analogue of
the Thom isomorphism Theorem: 10.2.2 of \cite{Bla:K}.

\begin{thm}
\label{K:59}
 Let $\mathcal{B}$ be a 
  bi-infinite ordered Bratteli diagram satisfying the conditions of
Definition \ref{surface:20}.
We have 
$ K^{i+1}(S_{\mathcal{B}}) \cong 
K_{i}(C^{*}( \mathcal{F}_{\mathcal{B}}^{+}))$.

 \end{thm}
 
 \begin{cor}
 \label{K:60}
 Let $\mathcal{B}$ be a 
  bi-infinite ordered Bratteli diagram satisfying the conditions of
Definition \ref{surface:20}. 
 If $K_{0}(A_{\mathcal{B}}^{+})$ is not finitely generated then the surface
 $S_{\mathcal{B}}$ has infinite genus.
 \end{cor}




\section{Chamanara's surface}
\label{Cha}
\begin{figure}[t]
  \includegraphics[width =6.5in]{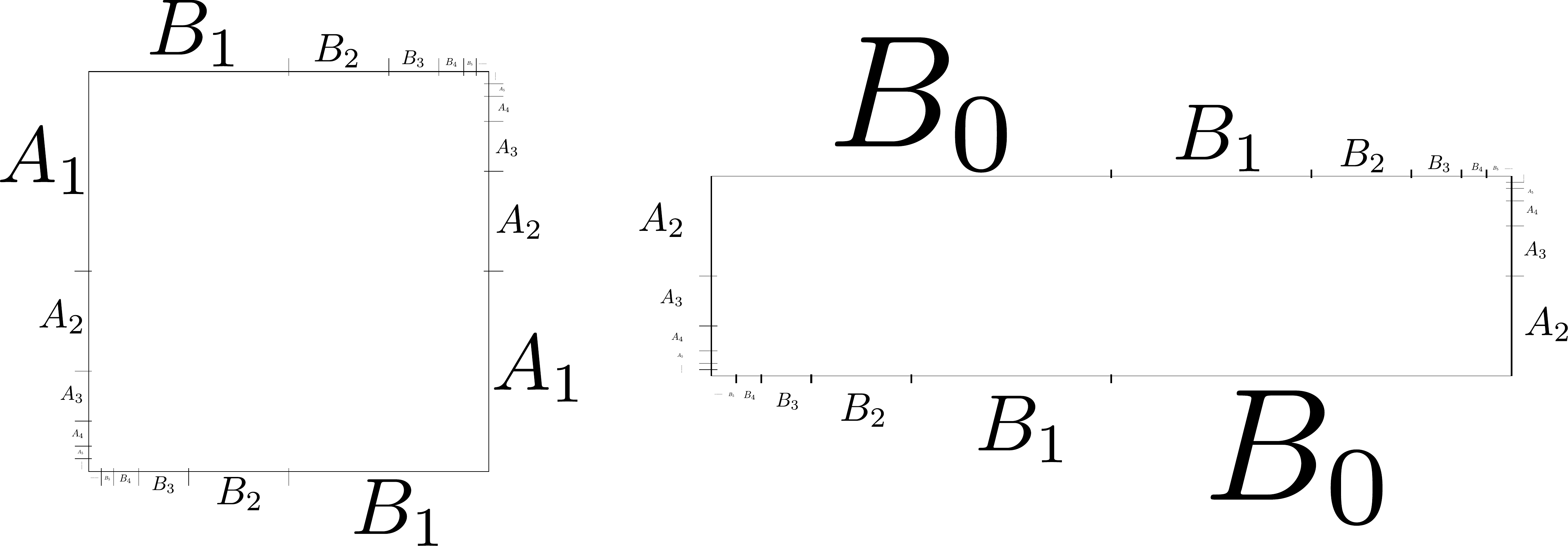}
  \caption{Two presentations of the Chamanara surface: the interiors of the edges with the same label are identified by a translation. The point at the boundary of such edges are not part of the surface and the surface has infinite genus. The presentation on the left is the standard presentation.}
  \label{fig:Chamanara}
\end{figure}
There is a family of surfaces of infinite genus introduced by Chamanara \cite{chamanara} which kicked off the study of flat geometry and dynamics of surfaces of infinite genus. The simplest of them has become known as \emph{the} Chamanara surface (see Figure \ref{fig:Chamanara}). Later, in \cite{LT}, the connection was made between this surface, the Bratteli diagram of the $2^\infty$ UHF $C^*$-algebra, and the dyadic odometer. In this section we apply our machinery to study the different algebras and their $K$-theory.

The bi-infinite, ordered Bratteli diagram which is relevant here has the properties
\begin{eqnarray*}
V_{n} & = & \{ v_{n} \}, \\
E_{n} & = & \{ 0_{n}, 1_{n} \}, \\
0_{n} & \leq_{r} & 1_{n}, \\
0_{n} & \leq_{s} & 1_{n}
\end{eqnarray*}
for all $n$ in $\Z$. 
This diagram has a state which is unique, up to scaling: 
\[
\nu_{r}(v_{n})  = 2^{n}, \nu_{s}(v_{n})  = 2^{-n}, n \in \Z.
\]
It is easy to see that 
\[
X^{ext}_{\mathcal{B}} = \{( \cdots \, 1 \, 1 \, 1 \, \cdots ), 
( \cdots \, 0 \, 0 \, 0 \, \cdots ) \}.
\]
It is also clear that in Proposition \ref{gpds:30}
that we have $I_{\mathcal{B}}= J_{\mathcal{B}}$ and we can use
$x_{1} = 1^{\infty} = ( \cdots \, 1 \, 1 \, 1 \, \cdots )$ and 
$x_{2} = 0^{\infty} = ( \cdots \, 0 \, 0 \, 0 \, \cdots )$.
It is also easy to see that $\partial^{s}X_{\mathcal{B}}$ consists
of sequences that have a last $0$, or a last $1$, while 
 $\partial^{r}X_{\mathcal{B}}$ consists
of sequences that have a first $0$, or a first $1$.
In particular, we see that the third condition of
Definition \ref{surface:20} holds.
Among these, for each integer $n$, we define  four special points:
\begin{itemize}
\item $w^{n}$: has a $1$ in entry $n$ and $0$'s elsewhere,
\item $x^{n}$: has $0$ in all entries $\leq n$ and $1$'s elsewhere, 
\item $y^{n}$: has $1$ in all entries $\leq n$ and $0$'s elsewhere, 
\item $z^{n}$: has a $0$ in entry $n$ and $1$'s elsewhere.
\end{itemize}

It is easy to check that 
\[
\begin{array}{rclrcl}
\Delta_{s}(w^{n})  & = & x^{n}, &
\Delta_{r}(w^{n})  & = & y^{n-1} \\
\Delta_{s}(z^{n})  & = & y^{n}, &
\Delta_{r}(z^{n})  & = & x^{n-1} 
\end{array}
\]
It follows that 
\[
\Delta_{r} \circ \Delta_{s}(w^{n}) = z^{n+1} \neq 
z^{n-1} = \Delta_{s} \circ \Delta_{r}(w^{n})
\]
and 
\[
\{  w^{n}, x^{n}, y^{n}, z^{n}
 \mid n \in \Z \} \subseteq  \Sigma_{\mathcal{B}}. 
 \]
The reverse containment is quite easy. 

It is fairly easy to check that the functions
 $\varphi^{1^{\infty}}_{r}, \varphi^{0^{\infty}}_{r}$
can be written quite explicitly as 
\begin{eqnarray*}
\varphi^{1^{\infty}}_{r}(x) & = & - \sum_{n \in \Z} 2^{n}(1 - x_{n}), \\
\varphi^{0^{\infty}}_{r}(y) & =  & \sum_{n \in \Z} 2^{n}y_{n},
\end{eqnarray*}
for any $x$ in $T^{+}(1^{\infty})$ and $y$ in $T^{+}(0^{\infty})$, respectively.
The ranges are
\begin{eqnarray*}
\varphi^{1^{\infty}}_{r}(T^{+}(1^{\infty}) \cap Y_{\mathcal{B}} ) & = & \cup_{n \in \Z} (-2^{n}, -2^{n+1}),\\
\varphi^{0^{\infty}}_{r}(T^{+}(0^{\infty}) \cap Y_{\mathcal{B}} ) & = & \cup_{n \in \Z} (2^{n}, 2^{n+1}).
\end{eqnarray*}
The quotient map then identifies each interval of the former with the corresponding interval
in the latter having the same length.

Moving on to K-theory, we have $K_{0}(C^{*}_{\lambda}(T^{+}(X_{\mathcal{B}}))) \cong \Z[1/2]$, 
as ordered abelian groups, with 
the latter having the usual order from the real numbers. 
In fact, the map sends the class of a projection $a_{p,p}, p \in E_{m,n}$ to
$2^{-n} =  \nu_{s}(r(p))$. 

Theorem \ref{K:38} then tells us that 
$K_{0}( C^{*}_{\lambda}(T^{\sharp}(S^{s}_{\mathcal{B}})))
 \cong K_{0}( C^{*}_{\lambda}(T^{+}(X_{\mathcal{B}}))) \cong \Z[1/2]$, 
as ordered abelian groups. The collection of connected
subsets of $T^{\sharp}(1^{\infty})$ is indexed by the integers: interval
$n$ having length $2^{n}$.That is, we have a canonical identification 
of $\mathcal{I}_{\mathcal{B}} = \{ I_{n} \mid n \in \Z \}$, where $I_{n}$ has length $2^{n}$.
 The map $\iota$ of Theorem \ref{K:50} is 
induced by sending each generator $I_{n}$ to the same thing, so $\ker(\iota)$ is the 
free abelian group with generators $I_{n}- I_{n-1}$.

We claim that $K_{0}(C^{*}(\mathcal{F}_{\mathcal{B}}))$ is the free abelian
group on a countably infinite set, which we will index by the integers. 
We will now explicitly write a set of generators.

Fix an integer $m$ and consider $p_{m}=01011$ and 
$q_{m} = 01010$ in $E_{m-5,m}^{Y}$. 
 (The presence of two $0$'s and two $1$'s guarantees that 
 we avoid $\Sigma_{\mathcal{B}}$.) Define 
 \begin{equation}
   \label{eqn:wm}
 w_{m} = a_{p_{m},p_{m}} \otimes \cos \left( \frac{\pi}{2} 2^{-m} \varphi^{v_{m}}_{s}(z) \right) 
   +  a_{q_{m},p_{m}} \otimes \sin \left( \frac{\pi}{2} 2^{-m}\varphi^{v_{m}}_{s}(z) \right),
 \end{equation}
   for $z$ in $X_{v_{m}}^{+}$, which lies in $AC_{m-5,m}$.
   A simple computation shows $w_{m}^{*} w_{m} = a_{p_{m},p_{m}}$ while 
   $w_{m}w_{m}^{*}$ equals $a_{p_{m},p_{m}}$ when evaluated at $z = x_{v_{m}}^{s-min}$ and equals
   $a_{q_{m},q_{m}}$ when evaluated at $z = x_{v_{m}}^{s-max}$. (This is the same $w_{m}$ 
   appearing in the proof of Theorem \ref{K:38}.)    Just as in Theorem \ref{K:38}, 
   this shows that $w_{m}w^{*}_{m}$ lies in $ C^{*}_{\lambda}(T^{\sharp}(S^{s}_{\mathcal{B}}))$
    and is Murray-von Neumann  equivalent to $a_{p_{m},p_{m}}$ in 
    $C^{*}_{\lambda}(T^{+}(Y_{\mathcal{B}}))) $. In particular, 
    identifying $K_{0}( C^{*}_{\lambda}(T^{+}(Y_{\mathcal{B}}))))
     \cong \Z[1/2]$, $j_{*}([w_{m}w_{m}^{*}]) = 2^{-m}$.

   There is a geometric way to visualize the functions $w_m$ in
    (\ref{eqn:wm}). First, since $|V_m| = 1$ for all $m$, we have that 
   $X_{v_m}^{-}X_{v_m}^+ = X_{\mathcal{B}}$ for all $m$. 
   What the different presentations of $X_\mathcal{B}$ as $X_{v_m}^{-}X^+_{v_m}$ 
   highlight are the special types of paths, e.g. $x_{v_m}^{s-max/min}$. 
   This type of different presentation is analogous to the different presentations
    of Chamanara's, e.g. the two presentations in Figure \ref{fig:Chamanara}.
   \begin{figure}[t]
     \includegraphics[width =6.5in]{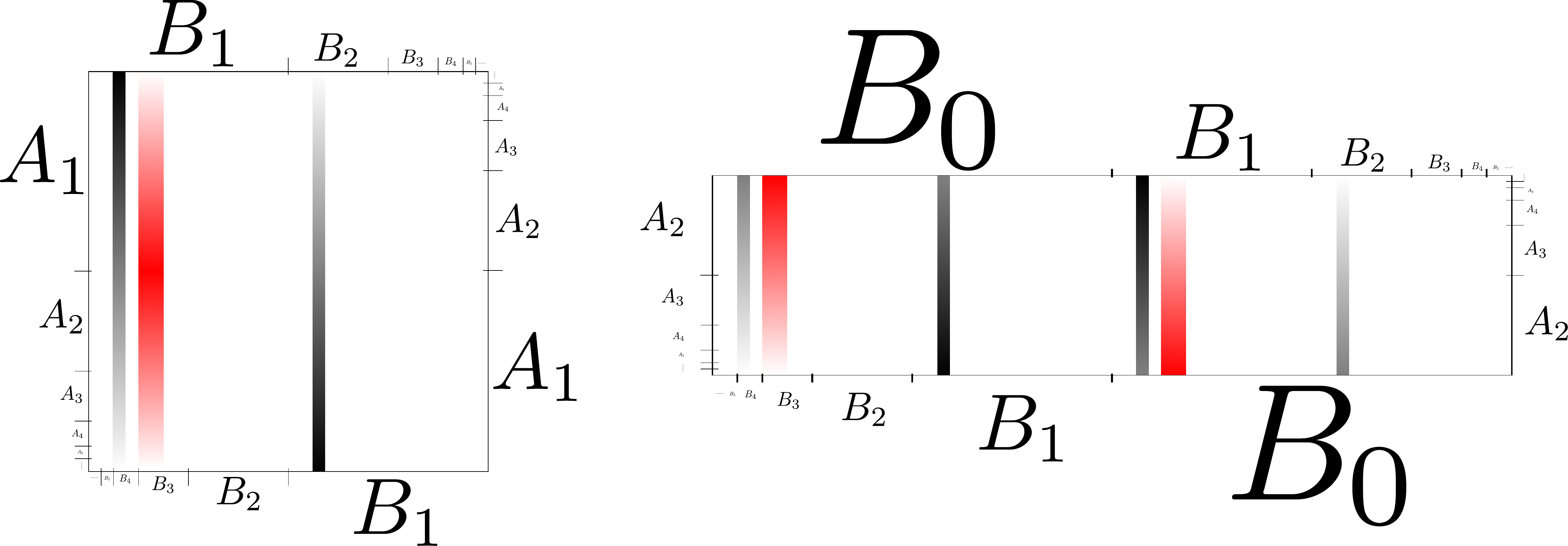}
     \caption{The functions $\bar{w}_0$ (in black) and $\bar{w}_1$ (in red) on the two presentations of Chamanara's surface from Figure \ref{fig:Chamanara}. For $\bar{w}_0$, white is 0, black is 1, and grey is in between, whereas for $\bar{w}_1$ white is 0, red is 1, and the other shades of red are in between.}
  \label{fig:wm}
   \end{figure}
   
   Let $\pi:Y_\mathcal{B}\rightarrow S_\mathcal{B}$ be the map from the 
   (nonsingular) path space to the surface. The paths $p_m,q_m\in E^Y_{m-5,m}$ 
   define cylinder sets $[p_m],[q_m]\subset Y_\mathcal{B}$ and the image 
   of these cylinder sets under $\pi$ is denoted by $U_m, V_m\subset S_\mathcal{B}$. 
   The functions $w_m$ in (\ref{eqn:wm}) are in fact the pullback of functions 
   $\bar{w}_m$ on $S_\mathcal{B}$ which are supported on $U_m\cup V_m$: $w_m = \pi^* \bar{w}_m$, see Figure \ref{fig:wm}.
   
   It follows then that $[w_{m}w_{m}^{*}] - 2 [w_{m+1}w_{m+1}^{*}]$ is in $\ker( j_{*})$. 
   Finally, one can show that under the identification of $\ker( j_{*})$ with $\ker(\iota)$
   given in Theorem \ref{K:50}, this element corresponds to $I_{m} - I_{m+1}$. 
   This computation is rather long and involves a lot of technical details
   from the main results of \cite{Put:Kexc} that we do not provide. However, given this, 
   it is a fairly simple matter to show that the  collection 
   $[w_{m}w_{m}^{*}], m \in \Z,$ generates
   all of  $K_{0}(C^{*}(\mathcal{F}_{\mathcal{B}}^{+}))$ and has no relations, completing 
   the proof of our claim above that the group is free abelian with a generating
   set indexed by the integers.

\section{Translation surfaces of finite genus}
\label{finitegenus}

The general goal of this section is to relate our constructions to 
the well-established study of translation surfaces in the finite genus case. 
More specifically, we aim to show that all  finite genus translation 
surfaces whose vertical and horizontal foliations are minimal arise
via our construction or, to be more precise, to see how 
standard techniques may be used to produce ordered Bratteli
 diagrams for finite genus surfaces.

      There are several equivalent ways to define a compact translation surface. Here we give two and refer the reader to \cite{viana:notes, zorich:intro, FM:notes} for thorough introductions to flat surfaces.
      
      Let $S$ be a compact Riemann surface of genus $g>1$ and $\alpha$ a 1-form on $S$ which is holomorphic with respect to the complex structure on $S$. The pair $(S,\alpha)$ defines a flat surface and a pair of transverse foliations, the horizontal and vertical foliations, $\mathcal{F}^\pm$. These are the foliations defined by the integrable distributions of the real and imaginary parts of $\alpha$:
      $$\mathcal{F}^+ = \langle\ker \Im \alpha\rangle\hspace{1in}\mbox{ and } \hspace{1in} \mathcal{F}^- = \langle\ker \Re \alpha\rangle.$$
The unit-time parametrization of these foliations are respectively the horizontal and vertical flows $\phi_t^+$ and $\phi_t^-$.
      
      By the Poincar\'e-Hopf index theorems, since $g>1$, these foliations (and the corresponding flows) are singular; the singular points are the zeros of $\alpha$ and these are called the singularities of $\alpha$, which are denoted by $\Sigma$. The 1-form $\alpha$ gives $S$ a flat metric on $S\setminus \Sigma$ as follows. Let $p\in S\setminus \Sigma$ and $p'$ in a neighborhood of $p$. The map $p'\mapsto \int_p^{p'}\alpha\in \mathbb{C}$ defines a chart around $p$ such that the pullback of $dz$ is $\alpha$. This gives $S\setminus \Sigma$ a flat geometry, and the reader can verify that maps which are change of coordinates between these types of charts are of the form $z\mapsto z +c$, justifying the use of the name translation surface. A saddle connection is a geodesic $\gamma\subseteq S$ with respect to the flat metric which starts and ends in $\Sigma$. More specifically, it satisfies the property that $\partial \gamma\subseteq \Sigma$. 

      The geometry fails to be flat at the singular points in $\Sigma$. At these points the local coordinate is of the form $d\left(\frac{z^{p+1}}{p+1}\right) =dz^p $ for some $p\in\mathbb{N}$, called the degree of the singularity. At a point $z\in \Sigma$ of degree $p$, the conical angle around $z$ is $2\pi (p+1)$. If $\Sigma = \{z_1,\dots, z_k\}$, and the degree at $z_i$ is $\kappa_i$, then by the Gauss-Bonnet theorem we have that $\sum_{i=1}^k \kappa_i = 2g-2$. Since the holomorphic 1-form $\alpha$ determines the geometry of the flat surface $(S,\alpha)$, it defines its area by $\mbox{Area}(S) = \frac{i}{2}\int_S \alpha\wedge \bar\alpha$.

      Another way to define a flat surface is as follows: start with a $2n$-gon $\bar{P}\subseteq\mathbb{C}$ with the property that edges come in parallel pairs of the same length. That is, $\bar{P}$ has edges $\zeta_1^+,\dots, \zeta_n^+,\zeta_1^-,\dots, \zeta_n^-$, where $\zeta_i^+$ and $\zeta_i^-$ are parallel and of the same length. Let $S = \bar{P}/\sim$ be the object obtained by the identifying pairs of edges which are parallel and of the same length: $\zeta_i^+\sim \zeta_i^-$. The holomorphic 1-form on $S$ is the unique one which pulls back as $dz$ on $\mathbb{C}$, although it may be singular at points where different edges meet. The points on $S$ where this happens is the singularity set $\Sigma$. The horizontal and vertical foliations on $S$ are now seen as the horizontal and vertical lines in $\bar{P}$. That this definition is equivalent to the one given above is left as an exercise for the reader who has not seen this before.

      Translation surfaces come in families: all translation surfaces of genus $g$ are elements of the moduli space $\mathcal{M}_g$ of translation surfaces of genus $g$. The space $\mathcal{M}_g$ is finite dimensional and it is stratified into strata $\mathcal{H}(\bar{\kappa})$, where $\bar{\kappa}$ describes how many and which types of singularities the surfaces in $\mathcal{H}(\bar{\kappa})$ are allowed to have. The stratum $\mathcal{H}(\bar{\kappa})$ is locally modeled by $H^1(S,\Sigma;\mathbb{C})$. By the remarks above, $\mathcal{H}(\kappa_1,\dots, \kappa_d)\subseteq \mathcal{M}_g$ if and only if $\bar{\kappa}$ satisfies $\sum\kappa_i = 2g-2$. The Teichm\"uller flow is the 1-parameter family of homeomorphisms of $\mathcal{M}_g$, taking $(S,\alpha)\mapsto (S,\alpha_t) =g_t(S,\alpha)$, 
      where $\Re\alpha_t = e^{-t}\Re\alpha$ and $\Im\alpha_t = e^t \Im\alpha$.

      In the rest of this section, we establish a way of defining an ordered, bi-infinite Bratteli diagram $\mathcal{B}(S,\alpha)$ for a typical choice of compact flat surface $(S,\alpha)$. 
      \subsection{Veech's zippered rectangles}
      \label{subsec:zip}
      Veech \cite{veech:gauss} introduced a way of presenting flat surfaces as the union of rectangles which are ``zippered'' on their sides. Here we review the construction. We will follow the conventions of Viana \cite{viana:notes}.

      Let $\mathcal{A}$ be an alphabet of size $d\geq 4$, whose elements
      are usually written as $\alpha$, 
      and $\pi_0,\pi_1:\mathcal{A}\rightarrow \{1,\dots, d\}$ two bijections. 
      We will consider examples with 
      $\mathcal{A} = \{ A, B, C, D \}$. We will use $\alpha$ to denote the 
      inverses of these functions, but instead of writing 
      $\alpha_{\varepsilon}(i) = \pi_{\varepsilon}^{-1}(i)$, for $\varepsilon =0,1, 
      1 \leq i \leq d$, we write $\alpha_{i}^{\varepsilon}$.
       These bijections may be written conveniently as 
      $$\pi = \left( \begin{array}{cccc} \alpha_1^0 & \alpha_2^0 &\cdots & \alpha_d^0 \\ \alpha_1^1 & \alpha_2^1 &\cdots & \alpha_d^1 \end{array}\right),$$
     the top and bottom rows being  ordered lists of the elements 
     of $\mathcal{A}$. It will always be assumed here that $\pi$ 
     defines an irreducible permutation, in the sense that  there is no $k<d$ such 
     that $\pi_1\circ \pi_0^{-1}\{1,\dots, k\} = \{1,\dots, k\}$.

      \begin{figure}[t]
        \centering
        \includegraphics[width = 3in]{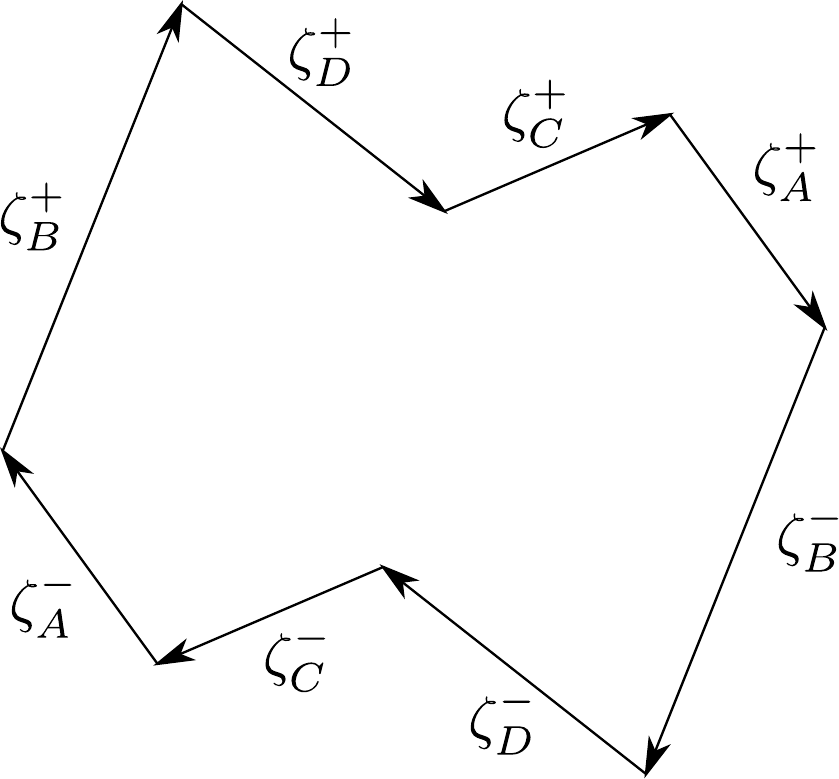}
        \caption{The flat surface defined by the vectors $\zeta_i = (\lambda,\tau)\in \mathbb{R}^\mathcal{A}_+\times T^+_\pi$.}
        \label{fig:1}
        \end{figure}
     
     We will now define vectors, $(\lambda_{\alpha}, \tau_{\alpha})$ 
     in the plane, indexed by   
     $\alpha$ in $\mathcal{A}$. Each 
     $\lambda_{\alpha}$ will be required to be positive while $\tau$ 
     satisfies 
      \begin{equation}
        \label{eqn:Tplus}
        \sum_{\pi_0(\alpha)\leq k}\tau_\alpha >0 \mbox{ and } \sum_{\pi_1(\alpha)\leq k}\tau_\alpha <0,
      \end{equation} 
      for all $ k < d$. We let $\R_{+}^{\mathcal{A}}$ to denote positive vectors
      and $T^{+}_{\pi}$ denote the set of all 
      $\tau$ in $\R^{\mathcal{A}}$ satisfying inequalities (\ref{eqn:Tplus}).
       Given $\zeta = (\lambda,\tau)$ in$ \mathbb{R}^\mathcal{A}_+\times T^+_\pi$,
        let $\Gamma = \Gamma(\pi,\lambda, \tau)\subseteq \mathbb{R}^2$ be the curve bounded by the concatenation of the vectors defined by $\zeta$:
        $$\zeta_{\alpha_1^0},\zeta_{\alpha_2^0},\cdots, \zeta_{\alpha_d^0},-\zeta_{\alpha_d^1}^1,-\zeta_{\alpha_{d-1}^1}^1,\cdots, -\zeta_{\alpha_1^1}^1.$$
        
      The constraints which define $T^+_\pi$ imply that about half of the vertices of $\Gamma$ are on the upper half plane, and the other rough half on the lower half plane. Assuming $\Gamma$ has no self-intersections\footnote{If there are self-intersections, there is a quick fix for it.}, the vector $\zeta$ defines a flat surface by first defining $\zeta_i^+$ and $\zeta^-_i$ to be the corresponding edges in the concatenation above in the upper and lower half of the plane, respectively, and then considering the interior of $\Gamma$ and making the identifications $\zeta_i^+\sim \zeta_i^-$ on the boundary edges (see Figure \ref{fig:1}).

      Given the data $(\pi,\lambda, \tau)$ as above, we now define the 
      vector $h$ in $\mathbb{R}^\mathcal{A}$ by
      $$h_\alpha = -\sum_{\pi_1(\beta)<\pi_1(\alpha)}\tau_\beta + \sum_{\pi_0(\beta)<\pi_0(\alpha)}\tau_\beta.$$
      This is more concretely expressed as $h = -\Omega_\pi(\tau)$, where $\Omega_\pi:\mathbb{R}^\mathcal{A}\rightarrow \mathbb{R}^\mathcal{A}$ is the matrix defined by
      $$\Omega_{\alpha\beta} = \left\{ \begin{array}{ll} +1 & \mbox{ if $\pi_1(\alpha)>\pi_1(\beta)$ and $\pi_0(\alpha)<\pi_0(\beta)$,} \\ \
        -1 & \mbox{ if $\pi_1(\alpha)<\pi_1(\beta)$ and $\pi_0(\alpha)>\pi_0(\beta)$,} \\
        0&\mbox{ otherwise.}\end{array}\right.$$
      Note that, the assumption that $\tau$ is in $
       T^+_\pi$ implies that $h_\alpha>0$ for all $\alpha$ in 
       $ \mathcal{A}$. We define the image of the positive cone $T^+_\pi$ under $-\Omega_\pi$ by $H^+_\pi = -\Omega_\pi(T^+_\pi)$.

      We now define rectangles $R_\alpha^\varepsilon$ of width $\lambda_\alpha$ and height $h_\alpha$ by
      \begin{equation}
      \label{eqn:rectangles}
        \begin{split}
          R_\alpha^0 & = \left(\sum_{\pi_0(\beta)<\pi_0(\alpha)}\lambda_\beta, \sum_{\pi_0(\beta)\leq \pi_0(\alpha)}\lambda_\beta \right) \times [0,h_\alpha] \\
          R_\alpha^1 & = \left(\sum_{\pi_1(\beta)<\pi_1(\alpha)}\lambda_\beta, \sum_{\pi_1(\beta)\leq \pi_1(\alpha)}\lambda_\beta \right) \times [-h_\alpha,0].
        \end{split}
      \end{equation}
      along with the ``zippers''
      \begin{equation*}
        \begin{split}
          Z_\alpha^0 & = \left\{ \sum_{\pi_0(\beta)\leq \pi_0(\alpha)}\lambda_\beta \right\} \times \left[0, \sum_{\pi_0(\beta)\leq\pi_0(\alpha)}\tau_\beta \right] \\
          Z_\alpha^1 & = \left\{ \sum_{\pi_1(\beta)\leq \pi_1(\alpha)}\lambda_\beta \right\} \times \left[ \sum_{\pi_1(\beta)\leq\pi_1(\alpha)}\tau_\beta,0 \right] \\
        \end{split}
      \end{equation*}
      \begin{figure}[t]
        \centering
        \includegraphics[width = 3in]{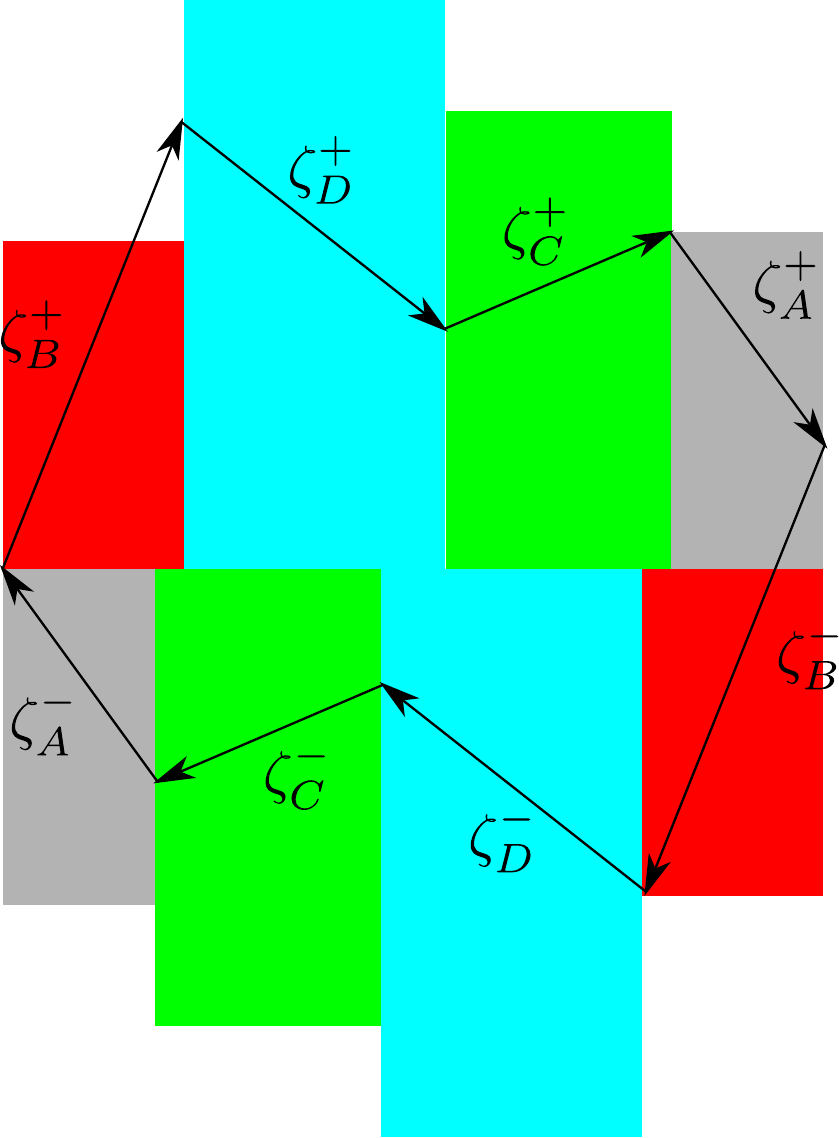}
        \caption{The zippered rectangles for the surface in Figure \ref{fig:1}.}
        \label{fig:2}
      \end{figure}
\noindent which are vertical segments ending at the points of concatenation of the curve $\Gamma$. As such, the flat surface $S(\pi, \lambda,\tau)$ can be presented as the quotient of the closure of the union of the rectangles $\{R_\alpha^0\}_{\alpha\in\mathbb{R}^\mathcal{A}}$ and zippers under a relation defined on the edges of the rectangles. The genus $g$ of this surface satisfies $2g = \mathrm{dim}\,\Omega_\pi(\mathbb{R}^\mathcal{A})$. The area of the surface is $\mathrm{Area}(S(\pi,\lambda,\tau)) = \lambda\cdot h$. Moreover, the horizontal and vertical foliations are the obvious choices. See Figure \ref{fig:2}.
      \subsection{Rauzy-Veech Induction}
      Given a triple $(\pi,\lambda,\tau)$, where $\pi = \{\pi_0,\pi_1\}$ is an irreducible permutation, 
      $\lambda$ in $\mathbb{R}^\mathcal{A}_+$ and $\tau $ in 
      $ T^+_\pi$, we will define an operation which produces a new triple $(\pi',\lambda',\tau')$ with the same properties. This procedure is known as Rauzy-Veech induction, or RV induction.

First, let us describe what this procedure is meant to do geometrically, and then we will give the details as to how it is done. Recall that from the triple $(\pi,\lambda, \tau)$ the flat surface it defines can be presented in zippered recangles form. The map $\mathcal{R}(\pi,\lambda,\tau) = (\pi',\lambda',\tau')$ gives new data from which the same surface can be presented in zippered rectangle form, except that the base one of the rectangles will be shorter and the height of one of the rectangles will be longer. This is done by cutting one of the rectangles $R_{\alpha}^\varepsilon$ into two and stacking one of the subrectangles above or below another one of the rectangles. The choices of the rectangles picked for this operation are determined by $(\pi,\lambda)$.

\begin{rmk}
\label{fg:10}
  It will be important to keep in mind one of the benefits of using Rauzy-Veech induction: it allows us to understand the behavior of the leaf of the vertical foliation on $S(\pi,\lambda,\tau)$ which emanates from the point on this surface coming from the origin in $\mathbb{R}^2$. An analogous procedure for the horizontal foliation will be described in \S \ref{subsec:H-ind}.
\end{rmk}

      First, we define $\pi'$ and $\lambda'$. Let $\alpha(\varepsilon)= \pi^{-1}_\varepsilon(d) = \alpha_d^\varepsilon$. That is, $\alpha(0)$ and 
      $\alpha(1)$ are the last entries in the top and bottom rows
      of $\pi$.
      \begin{defn}
      We say that $(\pi,\lambda)$ has
      $$\mbox{\textbf{type 0} if $\lambda_{\alpha(0)} > \lambda_{\alpha(1)}$\hspace{.7in} or\hspace{.7in} \textbf{type 1} if $\lambda_{\alpha(0)} < \lambda_{\alpha(1)}$}.$$
      If $(\pi,\lambda)$ is of type $\varepsilon\in\{0,1\}$ then the \emph{winner} is the symbol $\alpha(\varepsilon)$ and the \emph{loser} is $\alpha(1-\varepsilon)$.
      \end{defn}        
      This makes sense as long as $\lambda_{\alpha(0)}\neq \lambda_{\alpha(1)}$, so we will make the following assumption, to which we will return later.
      \begin{hyp}
      \label{fg:20}
        \label{hyp:1}
        The pair $(\pi,\lambda)$ satisfies $\lambda_{\alpha(0)}\neq \lambda_{\alpha(1)}$.
      \end{hyp}
      If $(\pi,\lambda)$ has type 0, then $\pi'$ is defined by
      \begin{equation}
      \label{eqn:type0perm}
        \pi' = \left( \begin{array}{cccccccc} \alpha_1^0 &\cdots & \alpha_{k-1}^0&\alpha_k^0&\alpha_{k+1}^0&\cdots&\cdots &\alpha(0) \\\alpha_1^1 &\cdots & \alpha_{k-1}^1&\alpha(0)&\alpha(1)&\alpha^1_{k+1}&\cdots &\alpha^1_{d-1}   \end{array}\right),
      \end{equation}
        that is,
      $$\alpha_i^{0'} = \alpha_i^0 \hspace{.4in}\mbox{ and }\hspace{.4in} \alpha^{1'}_i = \left\{\begin{array}{ll}\alpha_i^1&\mbox{ if $i\leq \pi_1(\alpha(0))$}\\ \alpha(1)&\mbox{ if $i=\pi_1(\alpha(0))+1$}\\ \alpha_{i-1}^1&\mbox{ if $i>\pi_1(\alpha(0))+1$} \end{array}\right. .$$
      The vector $\lambda'$ is now defined by
      \begin{equation}
      \label{eqn:type0lambda}
        \lambda_\alpha' = \lambda_\alpha\mbox{ if $\alpha\neq \alpha(0)$\;\;\; and \;\;\; $\lambda_{\alpha(0)}' = \lambda_{\alpha(0)}-\lambda_{\alpha(1)}$},
      \end{equation}
      whereas $\tau'$ is defined by
      \begin{equation}
      \label{eqn:type0tau}
        \tau_\alpha' = \tau_\alpha \mbox{ if $\alpha\neq \alpha(0)$\;\;\;and \;\;\; $\tau_{\alpha(0)}' = \tau_{\alpha(0)}-\tau_{\alpha(1)}$}.
      \end{equation}
      
      If $(\pi,\lambda)$ has type 1, then $\pi'$ is defined by
      \begin{equation}
      \label{eqn:type1perm}
        \pi' = \left( \begin{array}{cccccccc} \alpha_1^0 &\cdots & \alpha_{k-1}^0&\alpha(1)&\alpha(0)&\alpha^0_{k+1}&\cdots &\alpha^0_{d-1} \\ \alpha_1^1 &\cdots & \alpha_{k-1}^1&\alpha_k^1&\alpha_{k+1}^1&\cdots&\cdots &\alpha(1)    \end{array}\right),
      \end{equation}
      that is,
      $$\alpha_i^{1'} = \alpha_i^1\hspace{.4in}\mbox{ and } \hspace{.4in} \alpha^{0'}_i = \left\{\begin{array}{ll}\alpha_i^0&\mbox{ if $i\leq \pi_0(\alpha(1))$}\\ \alpha(0)&\mbox{ if $i=\pi_0(\alpha(1))+1$}\\ \alpha_{i-1}^1&\mbox{ if $i>\pi_0(\alpha(1))+1$} \end{array}\right. .$$
      The vector $\lambda'$ is now defined by
      \begin{equation}
      \label{eqn:type1lambda}
        \lambda_\alpha' = \lambda_\alpha\mbox{ if $\alpha\neq \alpha(1)$\;\;\; and \;\;\; $\lambda_{\alpha(1)}' = \lambda_{\alpha(1)}-\lambda_{\alpha(0)}$},
        \end{equation}
      whereas $\tau'$ is defined by
      \begin{equation}
      \label{type1tau}
        \tau_\alpha' = \tau_\alpha \mbox{ if $\alpha\neq \alpha(1)$\;\;\;and \;\;\; $\tau_{\alpha(1)}' = \tau_{\alpha(1)}-\tau_{\alpha(0)}$}.
      \end{equation}
      
      Let $\Theta = \Theta_{\pi,\lambda}:\mathbb{R}^\mathcal{A}\rightarrow \mathbb{R}^\mathcal{A}$ be the matrix defined, when $(\pi,\lambda)$ has type 0,as
      $$\Theta_{\alpha\gamma} = \left\{\begin{array}{ll}1&\mbox{ if }\alpha =\gamma\\1&\mbox{ if $\alpha = \alpha(1)$ and $\gamma = \alpha(0)$}\\ 0 & \mbox{ otherwise} \end{array}\right. $$
      whose inverse is
      $$\Theta_{\alpha\gamma}^{-1} = \left\{\begin{array}{ll}1&\mbox{ if }\alpha =\gamma\\-1&\mbox{ if $\alpha = \alpha(1)$ and $\gamma = \alpha(0)$}\\ 0 & \mbox{ otherwise.} \end{array}\right.$$
      When $(\pi,\lambda)$ has type 1,
      $$\Theta_{\alpha\gamma} = \left\{\begin{array}{ll}1&\mbox{ if }\alpha =\gamma\\1&\mbox{ if $\alpha = \alpha(0)$ and $\gamma = \alpha(1)$}\\ 0 & \mbox{ otherwise} \end{array}\right. $$
      whose inverse is
      $$\Theta_{\alpha\gamma}^{-1} = \left\{\begin{array}{ll}1&\mbox{ if }\alpha =\gamma\\-1&\mbox{ if $\alpha = \alpha(0)$ and $\gamma = \alpha(1)$}\\ 0 & \mbox{ otherwise.} \end{array}\right.$$
      This matrix satisfies \cite[Lemma 10.2]{viana:notes} the relation
      \begin{equation}
        \Theta\Omega_\pi\Theta^* = \Omega_{\pi'}.
      \end{equation}
      
      As such, the relations between $\lambda$ and $\lambda'$ 
      and between $\tau$ and $\tau'$, are expressed by
      $$\lambda' = \Theta^{-1*}\lambda\hspace{.2in}\mbox{ or }
      \hspace{.2in}\lambda = \Theta^*\lambda',\hspace{.3in}
      \mbox{ and }\hspace{.3in}\tau' = \Theta^{-1*}\tau,$$
      and so Rauzy-Veech induction is the map
      $$\mathcal{R}:(\pi,\lambda,\tau)\mapsto 
      (\pi', \Theta^{-1*}\lambda, \Theta^{-1*}\tau).$$
      
      In terms of zippered rectangles, RV induction has an 
      explicit expression in terms of the height vector 
      $h = -\Omega_\pi(\tau)$ in $ H^+_\pi$. Indeed, we 
      have that $\Theta\Omega_{\pi} = \Omega_{\pi'}\Theta^{-1*}$
       and so denoting $h' = -\Omega_{\pi'}(\tau')$ the corresponding 
       height vector for $\tau'$, we have that 
       $h' = \Theta h$. It is straight forward to verify that if $\tau$ in 
      $ T^+_\pi$ then $\tau' $ in $T^+_{\pi'}$.
       As such, the surface $S(\pi',\lambda',\tau')$ has area
      \begin{equation*}
          \begin{split}
            \mathrm{Area}(S(\pi',\lambda',\tau')) &
            =\lambda'\cdot (-\Omega_{\pi'}(\tau'))
             = -\Theta^{-1*}\lambda \cdot \Theta\Omega_{\pi}\Theta^*(\tau') 
             =  -\Theta^{-1*}\lambda\cdot \Theta\Omega_\pi(\tau) \\
            &= \Theta^{-1*}\lambda\cdot \Theta h
             = \lambda\cdot h = \mathrm{Area}(S(\pi,\lambda, \tau)).
          \end{split}
      \end{equation*}      

      \begin{figure}[t]
        \includegraphics[width =6.5in]{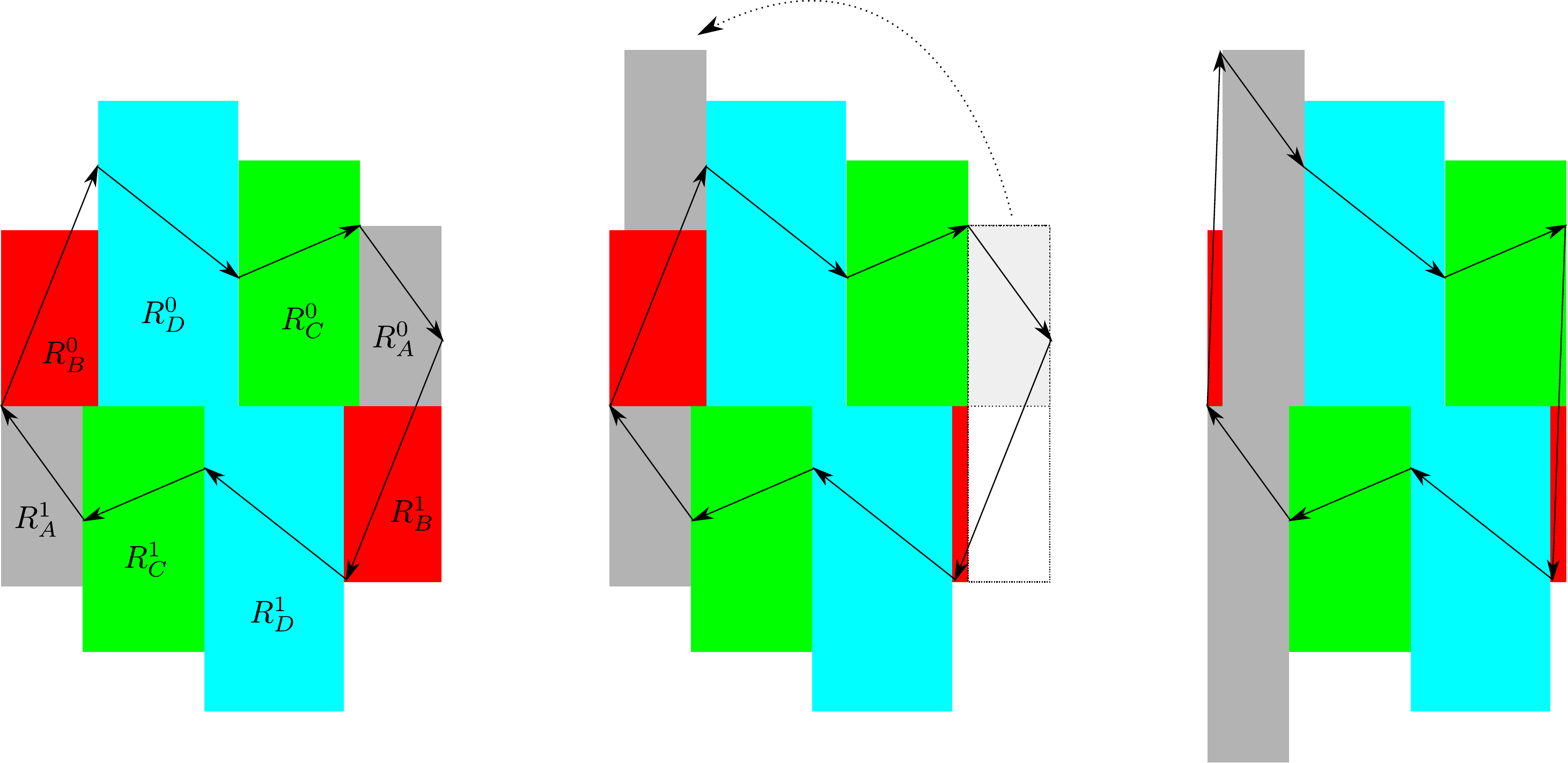}
        \caption{Geometric illustration of Rauzy-Veech induction. 
        Since $\lambda_{\alpha(1)}>\lambda_{\alpha(0)}$, this corresponds to type 1.}
        \label{fig:D3}
      \end{figure}

      Geometrically, Rauzy-Veech induction makes a vertical 
      cut through the widest rectangle at the end, takes the 
      right subrectangle, and stacks it above or below the 
      rectangle according to the rules described above. 
      Figure \ref{fig:D3} illustrates an example of what 
      Rauzy-Veech induction does to the zippered rectangles
       and the surface it represents from Figure \ref{fig:2}.
     
     Note that in the definition of RV induction, 
     whenever it was defined (Hypothesis \ref{hyp:1}),
      we may have that $\pi'\neq \pi$. Thus we can consider
       all possible permutations that can be obtained from $\pi$ under RV induction.
      \begin{defn}
      \label{fg:30}
        The \emph{Rauzy graph} $\mathcal{G}_d$ of permutations
         on $d$ elements is the directed graph which has as
          vertices equivalence classes of permutations 
          $\pi = \{\pi_0,\pi_1\}$, where $\pi\sim \pi'$ 
          whenever $\pi_1\circ\pi_0^{-1} = \pi_1'\circ \pi_0^{-1'}$,
           and there is an edge from $[\pi]$ to $[\pi']$ if 
           there are representatives $\pi$, $\pi'$ and vector
            $\lambda$ in $\mathbb{R}^\mathcal{A}_+$ such 
            that $\pi'$ is the permutation obtained from $(\pi,\lambda)$
             through Rauzy-Veech induction. A \emph{Rauzy class} is 
              by connected components of the Rauzy graph.
      \end{defn}
      There are two outgoing edges from each class $[\pi]$, 
      one for each type, as well as two incoming edges. 
      See Figures \ref{fig:hyper2} and \ref{fig:nonhyper2} for examples in genus 2.
     
      Let $(\pi,\lambda, \tau)$ in $ 
       \mathcal{C}\times \mathbb{R}^\mathcal{A}\times T^+_\pi$ 
       satisfying Hypothesis \ref{hyp:1}. Then the map
        $\mathcal{R}$ is well defined, and we obtain 
        $(\pi',\lambda',\tau') = 
        \mathcal{R}(\pi,\lambda,\tau) $ in 
        $ \mathcal{C}\times\mathbb{R}^\mathcal{A}\times T^+_\pi$. 
        We would like to once again apply $\mathcal{R}$ 
        to this new data, but we do not know a-priori 
        whether $(\pi',\lambda',\tau')$ satisfies Hypothesis \ref{hyp:1}.
      
      To establish conditions for which all iterates of RV induction are
       defined, we first need to define the 
       interval exchange transformation 
       (IET) defined by $(\pi,\lambda)$. For $\alpha$ in $\mathcal{A}$, let
      $$I_\alpha = \left[ \sum_{\pi_0(\gamma)<
      \pi_0(\alpha)}\lambda_\gamma, \sum_{\pi_0(\gamma)\leq \pi_0(\alpha)}
      \lambda_\gamma\right)$$
        and with $|I| = \|\lambda\|_1$, 
        the IET $f:[0,|I|)\rightarrow [0,|I|)$ defined by $(\pi,\lambda)$ is
            $$f(x) = x + \Omega_\pi(\lambda)_\alpha \mbox{ for $x\in I_\alpha$}.$$
      The reader is encouraged now to verify that the zippered 
      rectangle surfaces in \S \ref{subsec:zip} are 
      suspensions over the IET $f$ with roof functions 
      given by the height vector $h$. Denote by 
      $\partial I_\alpha$ the left endpoint of the interval $I_\alpha$.
            \begin{defn}
            \label{fg:40}
   A pair $(\pi,\lambda)$ satisfies the \emph{Keane condition}
    if $f^m(\partial I_\alpha)\neq \partial I_\gamma$ 
    for all $m$ in 
    $\mathbb{N}$ and $\alpha,\gamma\in\mathcal{A}$ with $\pi_0(\gamma)\neq 1$.
            \end{defn}
            \begin{rmk}
            \label{fg:50}
       \begin{enumerate}
     \item This condition guarantees that the orbits of the left 
     endpoints of the intervals are as disjoint as possible. 
     This surely guarantees Hypothesis \ref{fg:20}. 
     Below we will see that this characterizes the good data
      for which RV induction is defined for all iterates.
   \item It is known that the Keane condition implies 
   the minimality of the interval exchange transformation 
   $f$, that is, that every orbit is dense. This in turn 
   implies that the vertical foliation on $S(\pi,\lambda,\tau)$
    has no closed leaves and every leaf is dense in the surface.
              \end{enumerate}
            \end{rmk}
            \begin{thm}
            \label{fg:60}
            The following are equivalent:
            \begin{enumerate}
      \item $(\pi,\lambda)$ satisfies the Keane condition.
  \item All iterates $\mathcal{R}^n(\pi,\lambda)$
   of Rauzy-Veech induction are defined for $n>0$.
    \item For each $\alpha$ in $\mathcal{A}$, there is a
     subsequence $n_i^\alpha\rightarrow \infty$ such that $\alpha$ is the winner for $\mathcal{R}^{n_i^\alpha}(\pi,\lambda)$ for every $i$.
            \item For each $\alpha$ in 
            $\mathcal{A}$, there is a subsequence $n_i^\alpha\rightarrow \infty$ such that $\alpha$ is the loser for $\mathcal{R}^{n_i^\alpha}(\pi,\lambda)$ for every $i$.
            \end{enumerate}
            Moreover, these equivalent conditions are satisfied on a full measure subset of the space of parameters.
            \end{thm}
            For a proof, see \cite[\S 5]{viana:notes}. Thus,
             it is better to replace Hypothesis \ref{hyp:1} with the Keane condition.

            \subsection{RH Induction}
            \label{subsec:H-ind}
            The previous section reviewed a procedure which, 
            starting with some data $(\pi,\lambda,\tau)$ 
            and depending only on $\pi$ and $\lambda$,
            produced a new triple $\mathcal{R}(\pi',\lambda',\tau')$.
             Moreover, there is a precise condition that
              characterizes all data for which all iterates 
              of this procedure are defined. Denoting by 
    $(\pi^{(n)},\lambda^{(n)},\tau^{(n)})= \mathcal{R}^n(\pi,\lambda,\tau)$, 
    the surfaces $S(\pi^{(n)},\lambda^{(n)},\tau^{(n)})$ 
    are different presentations of $S(\pi,\lambda,\tau)$ 
    which allow us to keep track of longer and longer
     segments of the vertical leaf emanating from the origin.

            In this section, we define a different procedure,
             $\mathcal{P}:(\pi,\lambda,\tau)\mapsto (\pi',\lambda',\tau')$,
    with the aim of doing the same for the trajectory of the horizontal 
    foliation emanating from the origin, that is,
     we will get presentations 
     $S(\pi^{(n)},\lambda^{(n)},\tau^{(n)})$ of $S(\pi,\lambda,\tau)$ 
     through $(\pi^{(n)},\lambda^{(n)},\tau^{(n)})
     = \mathcal{P}^n(\pi,\lambda,\tau)$ which will allow us to 
     capture longer and longer segments of the horizontal 
     leaf emanating from the origin. We will then relate 
     this procedure to RV induction. This exposition is our
      own, but the recent work \cite{berk2021backward} 
      captures most of the aspects presented here.
                        
            Let us first describe and illustrate how this procedure
             is meant to work and then we will give the details. 
             Let $(\pi,\lambda,\tau)$ 
             in $\mathcal{C}\times \mathbb{R}^\mathcal{A}_+\times T^+_\pi$ and 
             consider the zippered rectangles presentation of it 
             in (\ref{eqn:rectangles}). Our goal is to extend 
             $[0,|I|)$ to $[0,|I'|)$, where
              $|I'| = \|\lambda'\|_1$. Given $\pi = (\pi_0,\pi_1)$, define
  \begin{equation}
    \beta(\varepsilon) = \pi_\varepsilon^{-1}(\pi_\varepsilon(\alpha(1-\varepsilon))+1),
                \end{equation}
                for $\varepsilon = \{0,1\}$. 
   In words, $\beta(\varepsilon)$ is the symbol immediately to 
   the right of $\alpha_d^{1-\varepsilon}$ on $\pi_\varepsilon$.            

  \begin{figure}[t]
     \includegraphics[width =4.5in]{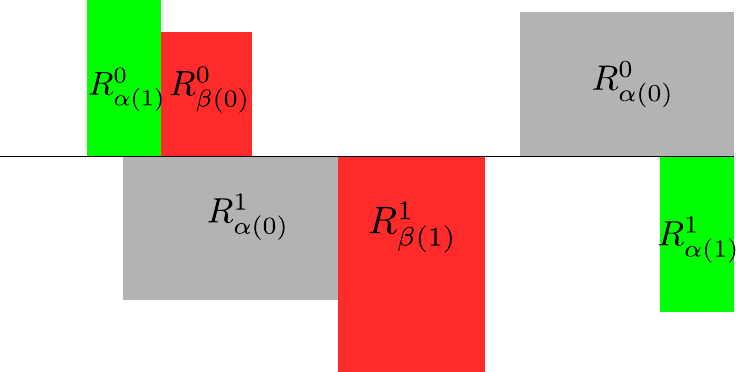}
 \caption{The case $h_{\alpha(1)}>h_{\beta(0)}$ and $h_{\alpha(0)}<h_{\beta(1)}$.}
                  \label{fig:H1}
                \end{figure}                
Suppose for the moment that $h_{\alpha(1)}>h_{\beta(0)}$ 
   and $h_{\alpha(0)}<h_{\beta(1)}$ 
     (see Figure \ref{fig:H1}). In order to extend the 
  horizontal leaf starting at the origin, 
      it must come out of the bottom edge of 
     the rectangle $R^1_{\alpha(0)}$ cut
 through $R^1_{\beta(1)}$, subdividing it into two subrectangles.
  The top rectangle will be absorbed into a larger rectangle 
  $R^1_{\alpha'(0)} = R^1_{\alpha(0)}$, while the bottom rectangle
   will be moved to the right and become $R^1_{\alpha'(1)}$. 
   Thus, we extend $[0,|I|)$ by $\lambda_{\beta(1)}$ and
    rearrange the rectangles as in Figure \ref{fig:H2}.

     If $h_{\alpha(1)}<h_{\beta(0)}$ and 
     $h_{\alpha(0)}>h_{\beta(1)}$ then an analogous procedure 
    is defined by cutting through the rectangle $R^0_{\beta(0)}$ 
  and moving the top rectangle to the right-most place on the top set of rectangles.

    It may be unclear how to proceed if $h_{\alpha(1)}<h_{\beta(0)}$
 and $h_{\alpha(0)}<h_{\beta(1)}$, as in Figure \ref{fig:2}. What really 
 determines which rectangle to cut has to do with the $\tau\in T^+_\pi$ 
 which defines $h = -\Omega_\pi(\tau)$. Indeed, in the case 
 $h_{\alpha(1)}>h_{\beta(0)}$ and $h_{\alpha(0)}<h_{\beta(1)}$
  as in Figure \ref{fig:H1} the zipper between $R^0_{\alpha(1)}$ 
  and $R^0_{\beta(0)}$ is somewhere in the interior of the right 
  edge of $R^0_{\alpha(1)}$, meaning that it is on the right edge 
  of $R^1_{\alpha(1)}$, meaning that $\sum_{\alpha}\tau_\alpha < 0$.
   Likewise, $h_{\alpha(1)}<h_{\beta(0)}$ and 
   $h_{\alpha(0)}>h_{\beta(1)}$ imply that $\sum_\alpha \tau_\alpha > 0$.

Let us remark that the case $h_{\alpha(1)}>h_{\beta(0)}$ and 
$ h_{\alpha(0)}>h_{\beta(1)}$ is impossible. Indeed, consider 
the zipper between $R^0_{\alpha(1)}$ and $R^0_{\beta(0)}$. There is a 
singularity of the flat surface somewhere between these two rectangles.
 But this singularity is to the right of $R^0_{\alpha(1)}$, which 
 means that there is a singularity on the right edge of $R^1_{\alpha(1)}$,
  which has to have height $\sum_{\alpha}\tau_\alpha < 0$. The same
   argument for the rectangles $R^1_{\alpha(0)}$ and $R^1_{\beta(1)}$ 
   implies that $\sum_\alpha\tau_\alpha >0$. Since $\tau\in T^+_\pi$,
    it satisfies one of the two conditions of (\ref{eqn:Tplus}), so
     it is impossible to have $h_{\alpha(1)}>h_{\beta(0)}$ and
      $ h_{\alpha(0)}>h_{\beta(1)}$.

 \begin{hyp}
 \label{fg:70}
    The pair $(\pi,\tau)$ with $\tau$ in $T^+_\pi$ 
    satisfies $\sum_{\alpha}\tau_\alpha \neq 0$.
                  \end{hyp}
                  
    Motivated by this discussion and following the
    terminology \cite[\S 12]{viana:notes}, we have the following definition.
   \begin{defn}
   \label{fg:80}
  If the pair $(\pi,\tau)$ satisfies Hypothesis \ref{fg:70}, 
  it will be called
  $$\mbox{\textbf{Type $0$} if $\displaystyle\sum_{\alpha}\tau_\alpha>0$\
  \hspace{.7in} or \hspace{.7in} \textbf{Type $1$} 
  if $\displaystyle\sum_{\alpha}\tau_\alpha<0$}.$$
  If $(\pi,\tau)$ is of type $\varepsilon\in\{0,1\}$ then the $\tau$-\emph{winner} is the symbol $\alpha(1-\varepsilon)$.
   \end{defn}
                  
  Thus if $(\pi,\tau)$ is type $0$, then the rectangle 
  $R^0_{\beta(0)}$ will be subdivided into two rectangles,
   the bottom part will be absorbed into $R^0_{\alpha(1)}$ 
   while the top part will be moved to the right to become $R_d^{0'}$.
    If $(\pi,\tau)$ is type $1$, the rectangle $R^1_{\beta(1)}$ will
     be subdivided into two rectangles, the top part will be absorbed 
     into $R^1_{\alpha(0)}$ while the top part will be moved to the 
     right to become $R_d^{1'}$, as depicted in Figure \ref{fig:H2}.
                  
   \begin{figure}[t]
    \includegraphics[width =6.5in]{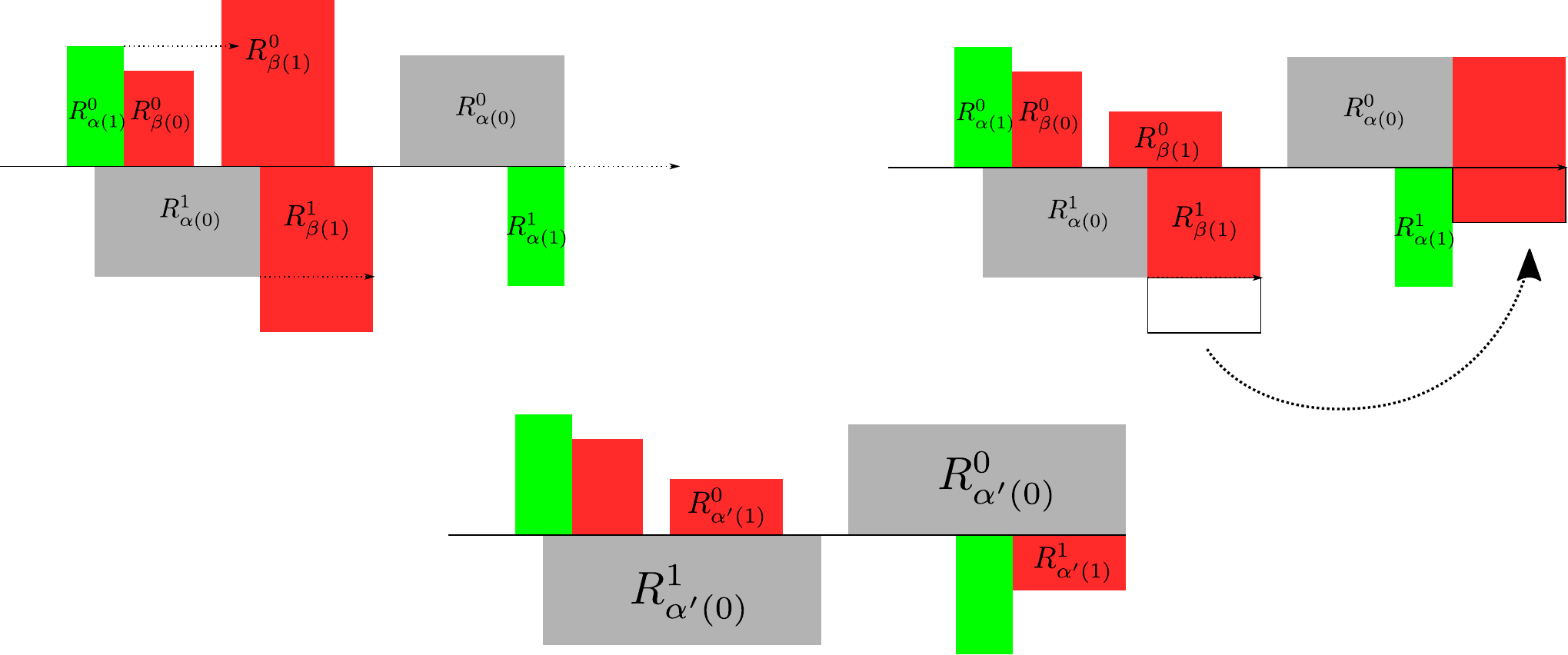}
     \caption{Starting from Figure \ref{fig:H1}, the procedure producing 
     a new zippered rectangles presentation implies that 
     $\alpha'(0)=\alpha(0)$ and $\alpha'(1) = \beta(1)$. 
     Compare with Figure \ref{fig:D3}.}
                    \label{fig:H2}
                  \end{figure}
                  
 These operations are formally defined as follows.                 
   Let $\tau$ in $T^+_\pi$ be of type $0$. Starting with $(\pi,\lambda,\tau)$
    and based on the description in the previous paragraph,
     the new data 
     $(\pi',\lambda',h') = \mathcal{P}(\pi,\lambda,h)= 
     \mathcal{P}(\pi,\lambda,-\Omega_\pi(\tau))$ is defined,
      first, by letting
   \begin{equation}
   \label{typeH0perm}
    \pi' = \left( \begin{array}{cccccccc} \alpha_1^0 &\cdots & 
    \alpha_{k-1}^0&\alpha(1)&\alpha^0_{k+2}&\cdots &\alpha(0)&\beta(0) \\
     \alpha_1^1 &\cdots & \alpha_{k-1}^1&\alpha_k^1&
     \alpha_{k+1}^1&\cdots&\cdots &\alpha(1)    \end{array}\right),
                  \end{equation}
  that is,
  $$ (\alpha^{'})^1_i = \alpha_i^1\hspace{.4in}\mbox{ and } \hspace{.4in} 
  (\alpha^{'})^0_i = \left\{\begin{array}{ll}\alpha_i^0&
  \mbox{ if $i\leq \pi_0(\alpha(1))$}\\ \alpha^0_{i+1}&
  \mbox{ if $\pi_0(\alpha(1))<i<d$}\\ \beta(0)&\mbox{ if $i=d$}
   \end{array}\right. .$$
                                    
  The vector $\lambda'$ is now defined by
 \begin{equation}
 \label{eqn:typeH0lambda}
     \lambda_\alpha' = \lambda_\alpha\mbox{ if $\alpha\neq \alpha(1)$\;\;\; 
     and \;\;\; $\lambda_{\alpha(1)}' = \lambda_{\alpha(1)}+\lambda_{\beta(0)}$},
                  \end{equation}
    whereas $h'$ is defined by
     \begin{equation}
     h_\alpha' = h_\alpha \mbox{ if 
     $\alpha\neq \beta(0)$\;\;\;and \;\;\; $h_{\beta(0)}'
      = h_{\beta(0)}-h_{\alpha(1)}$}.
                  \end{equation}
   The definition of $\tau'$ will follow from Proposition \ref{fg:80}.

     Let $\tau$ in $T^+_\pi$ be of type $1$. Starting 
     from $(\pi,\lambda,\tau)$ we now define $(\pi',\lambda',h')$ by
                  \begin{equation}
                  \label{fg:300}
      \pi' = \left( \begin{array}{cccccccc} \alpha_1^0 &\cdots
       & \alpha_{k-1}^0&\alpha_k^0&\alpha_{k+1}^0&\cdots&\cdots &\alpha(0) 
       \\\alpha_1^1 &\cdots & \alpha_{k-1}^1&\alpha(0)&\alpha^1_{k+2}&
       \cdots &\alpha(1)&\beta(1)   \end{array}\right),
                  \end{equation}
   that is,
   $$\alpha_i^{0'} = \alpha_i^0 \hspace{.4in}\mbox{ and }\hspace{.4in} \alpha^{1'}
   _i = \left\{\begin{array}{ll}\alpha_i^1&\mbox{ if 
   $i\leq \pi_1(\alpha(0))$}\\ \alpha^1_{i+1}&
   \mbox{ if $\pi_1(\alpha(0))<i<d$}\\ \beta(1)&\mbox{ if $i=d$} 
   \end{array}\right. .$$
    The vector $\lambda'$ is now defined by
    \begin{equation} 
    \label{eqn:typeH1lambda}             
         \lambda_\alpha' = \lambda_\alpha\mbox{ if 
         $\alpha\neq \alpha(0)$\;\;\; and \;\;\; 
         $\lambda_{\alpha(0)}' = \lambda_{\alpha(0)}+\lambda_{\beta(1)}$},
                  \end{equation}
     whereas $h'$ is defined by
  \begin{equation}  
  \label{eqn:typeH1h}
    h_\alpha' = h_\alpha \mbox{ if $\alpha\neq \beta(1)$\;\;\;and 
    \;\;\; $h_{\beta(1)}' = h_{\beta(1)}-h_{\alpha(0)}$}.
                  \end{equation}
   The definition of $\tau'$ will follow from Proposition \ref{fg:80}.
                  
    Let $\Psi = \Psi_{\pi,h}:\mathbb{R}^\mathcal{A}\rightarrow 
    \mathbb{R}^\mathcal{A}$ be the matrix defined, when $(\pi,h)$ has type $0$, as
                  \begin{equation}
                  \label{eqn:Psi0}
          \Psi_{\alpha\gamma} = \left\{\begin{array}{ll}1&\mbox{ if }
          \alpha =\gamma\\1&\mbox{ if $\alpha = \alpha(1)$
           and $\gamma = \beta(0)$}\\ 0 & \mbox{ otherwise} \end{array}\right.
                  \end{equation}
     whose inverse is
     $$\Psi_{\alpha\gamma}^{-1} = \left\{\begin{array}{ll}1&
     \mbox{ if }\alpha =\gamma\\-1&\mbox{ if
      $\alpha = \alpha(1)$ and $\gamma = \beta(0)$}\\
       0 & \mbox{ otherwise.} \end{array}\right.$$

  When $(\pi,h)$ has type $1$, as
                  \begin{equation}
                  \label{eqn:Psi1}
     \Psi_{\alpha\gamma} = \left\{\begin{array}{ll}1&\mbox{ if
      }\alpha =\gamma\\1&\mbox{ if $\alpha = \alpha(0)$ and
       $\gamma = \beta(1)$}\\ 0 & \mbox{ otherwise} \end{array}\right.
                  \end{equation}
                  whose inverse is
  $$\Psi_{\alpha\gamma}^{-1} = \left\{\begin{array}{ll}1&\mbox{ 
  if }\alpha =\gamma\\-1&\mbox{ if $\alpha = \alpha(0)$ and 
  $\gamma = \beta(1)$}\\ 0 & \mbox{ otherwise.} 
  \end{array}\right.$$
       Thus, the map $\mathcal{P}$ acts on data 
       as $\mathcal{P} : (\pi,\lambda,h)\mapsto (\pi',\Psi\lambda, \Psi^{-1*}h)$.

   Here we want to pick out a condition, analogous to the Keane condition
    in  Theorem \ref{fg:60}, which characterizes the data for
     which $\mathcal{P}^n(\pi,\lambda,t)$ is defined for all $n>0$.
      First observe that if we restrict ourselves to all $h$ with 
      rationally independent entries, that is, to $h\in H^+_\pi$ so that
    \begin{equation}
    \sum_{\alpha}n_\alpha h_\alpha\neq 0 \mbox{ for all }n\in\mathbb{Z}^\mathcal{A},
      \end{equation}
 then $\mathcal{P}^n(\pi,\lambda,h)$ is defined for all $n>0$. Moreover,
 the collection of all such vectors has full measure in $H^+_\pi$.

 \begin{defn}
 \label{fg:82}
   The triple $(\pi,\lambda,\tau)$ is RH-complete if for
   every $\alpha$ in $\mathcal{A}$ there is a subsequence 
   $n_i^\alpha\rightarrow \infty$ such that $\alpha$ is 
   the $\tau$-winner of $\mathcal{P}^{n_i^\alpha}(\pi,\lambda,\tau)$ 
   for all $i$.
 \end{defn}
  There is an analogous way to characterize when $\mathcal{P}^n$ is 
  defined for all $n>0$ recently proved by Berk 
  (see \cite{berk2021backward}). Compare the following with Theorem \ref{fg:60}.
  \begin{thm}[\cite{berk2021backward}]
  \label{fg:84}
    The following are equivalent:
    \begin{enumerate}
    \item All iterates $\mathcal{P}^n(\pi,\lambda,\tau)$ of RH-induction are defined.
    \item $(\pi,\lambda,\tau)$ is RH-complete.
    \item The horizontal leaf emanating from the singularity
     associated to the origin has infinite length.
    \end{enumerate}
  \end{thm}

  \subsection{Relations between RV and RH inductions}
  Here, we prove that RH induction is the inverse of RV induction.
  \begin{prop}
    \label{fg:88}
    Let $(\pi,\lambda,h) = (\pi,\lambda,-\Omega_\pi(\tau))$
     with $\tau$ in $T^+_\pi$. If $\lambda$ in $\mathbb{R}^\mathcal{A}_+$ 
     satisfies Hypothesis \ref{fg:20}, then 
     $\mathcal{P}\circ \mathcal{R}=\mathrm{Id}$. 
     If $\tau$ in $T^+_\pi$ satisfies Hypothesis 
     \ref{fg:70}, then $\mathcal{R}\circ \mathcal{P}=\mathrm{Id}$.
        \end{prop}
   \begin{proof}
   Let $(\pi',\lambda',h') = \mathcal{P}(\pi,\lambda,h)$. 
  Now suppose $\tau$ in $T^+_\pi$ is of type $0$. Then by (\ref{eqn:typeH0lambda}):
       $$\lambda'_{\alpha'(0)} = \lambda_{\beta(0)}<\lambda_{\beta(0)}
       +\lambda_{\alpha(1)} = \lambda'_{\alpha'(1)},$$
     which means that $(\pi',\lambda')$ is of type 1. Comparing 
     (\ref{eqn:type1lambda}) and (\ref{eqn:typeH0lambda}), 
     we get that $\mathcal{R}(\lambda')_\alpha = \lambda'_\alpha 
     = \lambda_\alpha $ whenever $\alpha\neq \alpha'(1) = \alpha(1)$ and
       $$\mathcal{R}(\lambda')_{\alpha'(1)} = 
       \lambda'_{\alpha'(1)}-\lambda'_{\alpha'(0)} =
        (\lambda_{\alpha(1)}+\lambda_{\beta(0)})-\lambda_{\beta(0)} 
        = \lambda_{\alpha(1)}.$$
    Finally, comparing (\ref{eqn:type1perm}) and (\ref{typeH0perm}), 
    we get that $\mathcal{R}\circ \mathcal{P}(\pi,\lambda) = 
    (\pi,\Theta^{-1*}\Psi\lambda) = (\pi,\lambda)$, so 
    $\Psi = \Theta^*$. If $\alpha'(\varepsilon)$ are the last symbols 
    of the permutation $\pi'_\varepsilon$, then $\alpha'(0) = \beta(0)$.
     So $\mathcal{R}(h')_\alpha = h_\alpha$ if $\alpha\neq \alpha'(0) = 
     \beta(0)$, and
   $$\mathcal{R}(h')_{\alpha'(0)} = h'_{\alpha'(0)}+h'_{\alpha'(1)} = 
   h'_{\beta(0)} + h'_{\alpha(1)} = (h_{\beta(0)}-h_{\alpha(1)}) + 
   h_{\alpha(1)} = h_{\beta(0)} = h_{\alpha'(0)},$$
       so $\mathcal{R}\circ \mathcal{P}(\pi,\lambda,h) = 
       (\pi, \Theta^{-1*}\Psi \lambda, \Theta \Psi^{-1*} h) =
        (\pi, \lambda,h)$. So  $\mathcal{R}\circ\mathcal{P}(\pi,\lambda,h)
         = (\pi,\lambda,h)$.

Suppose now $\tau$ in $ T^+_\pi$ is of type $1$.
 Then by (\ref{eqn:typeH1lambda}):
      $$\lambda'_{\alpha'(1)} = 
     \lambda_{\beta(1)}<\lambda_{\beta(1)}+\lambda_{\alpha(0)} 
     = \lambda'_{\alpha'(0)},$$
       which means that $(\pi',\lambda')$ is of type 0. 
       Comparing (\ref{eqn:type0lambda}) and 
        (\ref{eqn:typeH1lambda}), we get that 
    $\mathcal{R}(\lambda')_\alpha = \lambda'_\alpha 
      = \lambda_\alpha $ whenever $\alpha\neq \alpha'(0) $ and
      $$\mathcal{R}(\lambda')_{\alpha'(0)} =
      \lambda'_{\alpha'(0)}-\lambda'_{\alpha(1)} =
       (\lambda_{\alpha(0)}+\lambda_{\beta(1)})-\lambda_{\beta(1)}
        = \lambda_{\alpha(0)}.$$
     Finally, comparing (\ref{eqn:type0perm}) and
      (\ref{fg:300}), we get that 
      $\mathcal{R}\circ \mathcal{P}(\pi,\lambda)
       = (\pi,\Theta^{-1*}\Psi\lambda) = (\pi,\lambda)$, so 
       $\Psi = \Theta^*$ in this case too.

     Note that if $\alpha'(\varepsilon)$ are the last symbols of 
     the permutation $\pi'$, then $\alpha'(1) = \beta(1)$. 
     So $\mathcal{R}(h')_\alpha = h_\alpha$ if $\alpha\neq \alpha'(1) = \beta(1)$, and
     $$\mathcal{R}(h')_{\alpha'(1)} = h'_{\alpha'(1)}+h'_{\alpha'(0)} 
     = h'_{\beta(1)} + h'_{\alpha(0)} = (h_{\beta(1)}-h_{\alpha(0)}) 
     + h_{\alpha(0)} = h_{\beta(1)} = h_{\alpha'(1)},$$
       so $\mathcal{R}\circ \mathcal{P}(\pi,\lambda,h) = 
       (\pi, \Theta^{-1*}\Psi \lambda, \Theta \Psi^{-1*} h) = (\pi, \lambda,h)$.
                    

      Let $(\pi',\lambda',\tau') = \mathcal{R}(\pi,\lambda,\tau)$ and 
       $(\pi,\lambda)$ is of type $0$. Then by (\ref{eqn:type0tau})
        we have that
     $$\sum_{\alpha}\tau'_\alpha = \sum_{\alpha\neq \alpha(1)}\tau_\alpha < 0$$
    and so by (\ref{eqn:Tplus}) we have that $(\pi',\tau')$ is of type 1 
    (as in Figure \ref{fig:H2}). Let $\lambda' = \Theta^{-1*}\lambda$,
     where $\Theta$ is the type 0 matrix.  Comparing (\ref{eqn:type0perm}) 
     and (\ref{fg:300}) it also follows that $\beta'(1) = \alpha(1)$.
      In addition, comparing (\ref{eqn:type1lambda}) and (\ref{eqn:typeH1lambda}),
       we get that $\mathcal{P}(\lambda')_\alpha =
        \lambda'_\alpha = \lambda_\alpha $ whenever
         $\alpha\neq \alpha'(0)=\alpha(0)$ and
     $$\mathcal{P}(\lambda')_{\alpha'(0)} = 
     \lambda'_{\alpha'(0)}+\lambda'_{\beta'(1)} = \lambda'_{\alpha(0)}
     +\lambda'_{\alpha(1)} = (\lambda_{\alpha(0)}-\lambda_{\alpha(1)})
     +\lambda_{\alpha(1)} = \lambda_{\alpha(0)},$$
 so $\mathcal{P}\circ\mathcal{R}$ acts as the identity on $\lambda$.
                    
      It follows from (\ref{eqn:typeH1h}) that
       $\mathcal{P}(h')_\alpha = h'_\alpha = h_\alpha$ 
       if $\alpha\neq \beta'(1)=\alpha(1)$ and
    $$\mathcal{P}(h')_{\beta'(1)} = 
    h'_{\beta'(1)}-h'_{\alpha'(0)} = 
    (h_{\alpha(0)}+h_{\alpha(1)})-h_{\alpha(0)} =h_{\beta'(1)}
     = h_{\alpha(1)}= \mathcal{P}(h')_{\alpha(1)}.$$
       So it follows that $\mathcal{P}\circ \mathcal{R}(\pi,\lambda,h) 
       = (\pi,  \Psi \Theta^{-1*}\lambda , \Psi^{-1*} \Theta h) 
       = (\pi,\lambda,h)$.
  The last case is similarly proved.
              \end{proof}
              
 It follows that the map $\mathcal{P}$ changes
  the $\tau$ coordinate by $\tau\mapsto \Psi\tau$.
  
   \begin{prop}
   \label{fg:90}
   The map $\mathcal{P}$ preserves the cones $T^+_\pi$.
   \end{prop}
   \begin{proof}
   Suppose $(\pi,\tau)$ is of type 0 with $\tau$ in $T^+_\pi$ 
   and let $\mathcal{P}(\pi,\tau) = (\pi',\Psi\tau)$. Then
     $$\sum_{\pi'_0(\alpha)\leq k}\left(\Psi\tau\right)_\alpha 
     = \left\{\begin{array}{ll}\displaystyle\sum_{\pi_0(\alpha)\leq k}\tau_\alpha 
     > 0&\mbox{ if $k<\pi_0(\alpha(1)) = \pi_0'(\alpha(1))$} \\
      \displaystyle\sum_{\pi_0(\alpha)\leq k+1}
      \tau_\alpha > 0&\mbox{ if $\pi_0(\alpha(1))
       = \pi_0'(\alpha(1))\leq k < d$,} 
        \end{array}\right. $$
     where the case for $k=d-1$ follows because $(\pi,\tau)$ 
     is of type 0. We also have for any $k<d$
  $$\sum_{\pi'_1(\alpha)\leq k}\left( \Psi\tau\right)_\alpha 
  = \sum_{\pi_1(\alpha)\leq k}\tau_\alpha <0,$$
   and so it follows that $\Psi\tau\in T^+_{\pi'}$. 
   Likewise if $(\pi,\tau)$ is of type 1 then
    $$\sum_{\pi'_1(\alpha)\leq k}\left(\Psi\tau\right)_\alpha 
    = \left\{\begin{array}{ll}\displaystyle\sum_{\pi_1(\alpha)\leq k}\tau_\alpha 
    < 0&\mbox{ if $k<\pi_1(\alpha(0)) 
    = \pi_1'(\alpha(0))$} \\ \displaystyle\sum_{\pi_1(\alpha)\leq k+1}\tau_
    \alpha < 0&\mbox{ if $\pi_1(\alpha(0)) = \pi_1'(\alpha(0))\leq k < d$,}  
    \end{array}\right. $$
     where the case for $k=d-1$ follows because $(\pi,\tau)$ is of type 1. 
     We also have for any $k<d$
   $$\sum_{\pi'_0(\alpha)\leq k}\left( \Psi\tau\right)_\alpha
    = \sum_{\pi_0(\alpha)\leq k}\tau_\alpha >0,$$
  and so the defining conditions of the cones (\ref{eqn:Tplus}) are preserved. 
                  \end{proof}
Thus the map $\mathcal{P}$ is 
   the inverse of the Rauzy-Veech induction map $\mathcal{R}$ and 
   it is sometimes called ``backwards Rauzy-Veech induction''. 
   As such, the action on data triples is of the 
   form $\mathcal{P}:(\pi,\lambda,\tau)\mapsto (\pi',\Psi\lambda,\Psi\tau)$.

 \subsection{Dynamics on the space of zippered rectangles}

   Since a flat surface can be constructed from data 
   $(\pi,\lambda,\tau)$ in $\mathcal{C}\times \mathbb{R}^\mathcal{A}_+\times T^+_\pi$,
    it is natural to ask how the set of all zippered rectangles
    relates to the set of all flat surfaces. This was
     described by Veech \cite{veech:gauss}.
    \begin{defn}
    \label{fg:100}
   The \emph{space of zippered rectangles}
    corresponding to a Rauzy class $\mathcal{C}$ is the set
         $$\bar{\mathcal{V}}_\mathcal{C}=\left\{ (\pi,\lambda,\tau):
  \pi\in\mathcal{C},\lambda\in\mathbb{R}^\mathcal{A},\tau\in T^+_\pi\right\}.$$
                  \end{defn}
   There is a natural volume measure $m_\mathcal{C}$ in 
   $\bar{\mathcal{V}}_\mathcal{C}$ locally given by 
   $m_\mathcal{C} = d\pi\, d\lambda\, d\tau$, where
    $d\pi$ is the counting measure, while $d\lambda,d\tau$ 
    are restrictions of Lebesgue measure on $\mathbb{R}^\mathcal{A}$.
     The Teichm\"uller flow on $\bar{\mathcal{V}}_\mathcal{C}$
     is the one-parameter group of diffeomorphisms of 
     $\bar{\mathcal{V}}_\mathcal{C}$ defined by 
     $\Phi_t(\pi,\lambda,\tau) = (\pi, e^{-t}\lambda, e^t \tau)$.
      We emphasize here that our convention for Teichm\"uller 
      flow here is backwards Teichm\"uller flow in the general
       literature. The reason for this is that our focus here 
       is on the horizontal flow, which is renormalized by
        the Teichm\"uller flow as we have defined it.

    The Teichm\"uller flow preserves the measure $m_\mathcal{C}$. Note that
  \newline
     $\mathrm{Area}(S(\Phi_t(\pi,\lambda,\tau))) =
      \mathrm{Area}(S(\pi,\lambda,\tau))$ for any $t$ in 
      $\mathbb{R}$. Any $a>0$ defines 
      two independent global cross-sections
       $\bar{\mathcal{V}}^{\pm a}_\mathcal{C}$, defined by
   \begin{equation}
  \label{eqn:sections}
  \begin{split}
     \bar{\mathcal{V}}_\mathcal{C}^{+a}& 
     =\{(\pi,\lambda,\tau)\in\bar{\mathcal{V}}_\mathcal{C}: 
     |\Omega_\pi(\tau)|_1= |h|_1 = a\}, \\
    \bar{\mathcal{V}}_\mathcal{C}^{-a}& 
    =\{(\pi,\lambda,\tau)\in\bar{\mathcal{V}}_\mathcal{C}: |\lambda|_1 = a\}.
                    \end{split}
      \end{equation}
    The renormalization times of $(\pi,\lambda,\tau)$ are defined by
  \begin{equation}
   \label{eqn:rTime}
    t_R^+(\pi,\lambda,\tau) =
    -\log\left(1-\frac{h_{\alpha(1-\varepsilon_\tau)}}{|h|_1} \right)\hspace{.3in}
    \mbox{ and }\hspace{.3in} t_R^-(\pi,\lambda,\tau) 
    = -\log\left(1-\frac{\lambda_{\alpha(1-\varepsilon_\lambda)}}{|\lambda|_1} \right),
                  \end{equation}
     where $\varepsilon_*$ is the $*$-type of the 
     triple, for $*\in\{\lambda,\tau\}$, and it is
      immediate to check that the composition
  \begin{equation}
   \label{eqn:renormMap}
   \hat{\mathcal{P}}^{\pm} = \mathcal{P}^{\pm 1}\circ \Phi_{t_R^\pm}:
   (\pi,\lambda,\tau)\mapsto \mathcal{P}^{\pm 1}(\pi,e^{\mp t_R^\pm}
   \lambda,e^{\pm t_R^\pm}\tau)
       \end{equation}
 maps each cross section $\bar{\mathcal{V}}_\mathcal{C}^{\pm a}$ 
  to itself (assuming $\mathcal{P}^{\pm 1}$ is defined on the triple). 
  In fact, the transformation $\hat{\mathcal{P}}^\pm:
  \bar{\mathcal{V}}^{\pm a}_\mathcal{C}\rightarrow 
  \bar{\mathcal{V}}^{\pm a}_\mathcal{C}$ is an almost everywhere 
  invertible Markov map (see \cite[Corollary 20.1]{viana:notes}). 
  Let $\Pi^\pm$ be the maps defined by
   $$\Pi^+(\pi,\lambda,\tau) =  (\pi,h) = 
   (\pi,-\Omega_\pi(\tau))\hspace{.4in}\mbox{ and }\hspace{.4in} 
   \Pi^-(\pi,\lambda,\tau) =  (\pi,\lambda)$$
 for all $(\pi,\lambda, \tau)$ in $ \bar{\mathcal{V}}_\mathcal{C}$, and 
 let $m^\pm_\mathcal{C} = \Pi_*^\pm m_\mathcal{C}$ be the 
 pushforward of the volume measure and 
 $m_1^\pm$ be their restriction to the simplices
   \begin{equation*}
   \begin{split}
   \mathbb{V}^+_{\mathcal{C}}&=\Pi^+(\bar{\mathcal{V}}^1_\mathcal{C}) 
   =\left\{  (\pi,h)\in\bigsqcup_{\pi\in \mathcal{C}} \{\pi\}\times H^+_\pi: 
   |h|_1 = 1\right\}, \\
    \mathbb{V}^-_{\mathcal{C}}&=\Pi^-(\bar{\mathcal{V}}^{-1}_\mathcal{C})
     =\left\{  (\pi,\lambda)\in\bigsqcup_{\pi\in \mathcal{C}} \{\pi\}\times
      \mathbb{R}^\mathcal{A}_+: |\lambda|_1 = 1\right\}.
                    \end{split}
                  \end{equation*}
There are unique maps $\mathbb{P}^\pm: 
\mathbb{V}^\pm_\mathcal{C} \rightarrow \mathbb{V}^\pm_\mathcal{C}$ 
satisfying $\mathbb{P}^\pm \circ \Pi^\pm = \Pi^\pm \circ \hat{\mathcal{P}}^\pm$,
 for all triples $(\pi,\lambda,\tau)$ where $\mathcal{P}^{\pm 1}$ 
 is defined, which we respectively call the RH/RV renormalization maps.
  \begin{prop}
   \label{fg:110}
    The measure on $\mathbb{V}^+_\mathcal{C}$ defined by
     $$\prod_{\alpha\in\mathcal{A}}h_\alpha^{-1}\, dh$$
     is invariant under the RH renormalization map $\mathbb{P}^+$.
     \end{prop}
  This measure is a counterpart to the Gauss measures 
  $m_1^-$ on $\mathbb{V}^-_\mathcal{C}$ of Veech \cite{veech:gauss}.
   Veech proved that $m^-_1$ has an invariant density which is a 
   homogeneous rational function of $\lambda$ of degree 
   $-|\mathcal{A}|$ bounded away from zero (see \cite[\S 21]{viana:notes}). 
   The measure above is also a homogeneous rational function of 
   degree $-|\mathcal{A}|$. That a measure of this form was invariant 
   was claimed in \cite[\S 4]{putnam:iet}.
    \begin{proof}                    
   Recall that every $(\pi,h) $ in $ \mathbb{V}^+_\mathcal{C}$ has 
   two preimages $(\pi^\varepsilon,h^\varepsilon)$, 
   that is, $\mathbb{P}^+(\pi^\varepsilon,h^\varepsilon) = (\pi,h)$,
    where $\varepsilon $ in $ \{ 0,1\}$ is the $\tau$-type. 
    Let $\pi$ be represented as 
   $$\left(\begin{array}{lr}\cdots & \alpha(0) \\\cdots & \alpha(1)
   \end{array}\right).$$

 In terms of $h$, the two preimages $h^\varepsilon$ are given by
  \begin{equation}
    \label{eqn:preimages}
   h_\alpha^\varepsilon = \frac{h_\alpha}{1+h_{\alpha(1-\varepsilon)}} 
    \hspace{.2in}\mbox{ if }\hspace{.2in}\alpha \neq
     \alpha(\varepsilon)\hspace{.3in}\mbox{ and }\hspace{.3in}
      h_{\alpha(\varepsilon)}^\varepsilon = 
      \frac{h_{\alpha(\varepsilon)} + 
      h_{\alpha(1-\varepsilon)}}{1+h_{\alpha(1-\varepsilon)}}
   \end{equation}
  from which we get
  \begin{equation}
     \label{eqn:partials}
     \frac{\partial h^\varepsilon_\alpha}{\partial h_\beta} =
     \left\{ \begin{array}{cl} \displaystyle
     \frac{(1+h_{\alpha(1-\varepsilon)})\delta_{\alpha,\beta} 
     - h_\alpha\delta_{\beta,\alpha(1-\varepsilon)}}{(1+h_{\alpha(1-\varepsilon)})^2} 
     & \mbox{ if }\alpha\neq \alpha(\varepsilon) \\ 
     (1+h_{\alpha(1-\varepsilon)})^{-1} & \mbox{ if } \alpha = \beta 
     = \alpha(\varepsilon)\\ \displaystyle 
     \frac{1+h_{\alpha(\varepsilon)}}{(1+h_{\alpha(1-\varepsilon)})^2}& 
     \mbox{ if } \alpha = \alpha(\varepsilon)\mbox{ and } 
     \beta = \alpha(1-\varepsilon). 
     \end{array}\right.
     \end{equation}
  We denote by $\mathcal{F}_\varepsilon$ the map satisfying
   $\mathcal{F}_\varepsilon(h) = h^\varepsilon$ and by 
   $\mathcal{J}_\varepsilon$ its Jacobian. Note that the
    only nonzero entries of $\mathcal{J}_\varepsilon$ are
     along the diagonal, which are mostly $(1+h_{\alpha(1-\varepsilon)})^{-1}$
      except in the $\alpha(1-\varepsilon)$ entry, in which
       case it is $(1+h_{\alpha(1-\varepsilon)})^{-2}$, and in the 
       column for index $\alpha(1-\varepsilon)$, where the entry 
       for index $\alpha \neq \alpha(\varepsilon)$ is 
       $-h_\alpha/(1+h_{\alpha(1-\varepsilon)})^{-2}$.
        Thus, we can compute the determinant 
       of $\mathcal{J}_\varepsilon$ by expanding along 
       the row with index $\alpha(1-\varepsilon)$, and 
       we get that
   $$|\mathcal{J}_\varepsilon| = (1+h_{\alpha(1-\varepsilon)})^{-|\mathcal{A}|}.$$
  Let $\mathcal{D}(h) = \prod_\alpha h^{-1}_\alpha$. We would like to 
  verify that $\mathcal{D}\circ \mathcal{F}_\varepsilon|\mathcal{J}_\varepsilon|
   + \mathcal{D}\circ \mathcal{F}_{1-\varepsilon}|\mathcal{J}_{1-\varepsilon}|
    = \mathcal{D}$. First:
  \begin{equation}
    \label{eqn:comp}
 \mathcal{D}\circ \mathcal{F}_\varepsilon(h)
  = \prod_{\alpha}(h^\varepsilon_\alpha)^{-1} = \frac{(1+ h_{\alpha(1-\varepsilon)})^{|\mathcal{A}|}}{\displaystyle(h_{\alpha(\varepsilon)} + 
   h_{\alpha(1-\varepsilon)})\prod_{\alpha\neq \alpha(\varepsilon)}h_\alpha}
   = \frac{(1+h_{\alpha(1-\varepsilon)})^
   {|\mathcal{A}|}
   h_{\alpha(\varepsilon)}\mathcal{D}(h)}{h_{\alpha(\varepsilon)}
   +h_{\alpha(1-\varepsilon)}}.
                    \end{equation}
     Now putting everything together:
     $$\mathcal{D}\circ \mathcal{F}_\varepsilon|\mathcal{J}_\varepsilon| + \mathcal{D}\circ \mathcal{F}_{1-\varepsilon}|\mathcal{J}_{1-\varepsilon}| = \frac{h_{\alpha(\varepsilon)}\mathcal{D}(h)}{h_{\alpha(\varepsilon)}+h_{\alpha(1-\varepsilon)}} + \frac{h_{\alpha(1-\varepsilon)}\mathcal{D}(h)}{h_{\alpha(1-\varepsilon)}+h_{\alpha(\varepsilon)}} = \mathcal{D}(h).$$
     \end{proof}
  Let $\hat{\mathcal{V}}_\mathcal{C}\subseteq \bar{ \mathcal{V}}_\mathcal{C}$ 
  be the space of data $(\pi,\lambda,\tau)$ which satisfies the Keane
   condition and is RH-complete. It is invariant under both $\mathcal{P}$ and 
   $\Phi_t$. Define the spaces $\mathcal{V}_\mathcal{C}^\pm = 
   \hat{\mathcal{V}}_\mathcal{C}/\sim_\pm$, where $\sim_\pm$ is the relation
      \begin{equation}
  \label{eqn:preStratum}
   \mathcal{P}^{\pm 1}(\pi,\lambda,\tau)\sim \Phi_{\pm t_R^\pm}(\pi,\lambda,\tau),
     \end{equation}
  called the \emph{pre-strata} of the Rauzy class $\mathcal{C}$. The 
  Teichm\"uller flow $\Phi_t$ descends to flows $\Phi^\pm_t$ on 
  $\mathcal{V}_\mathcal{C}^\pm$, and the image of 
  $\bar{\mathcal{V}}^{^\pm 1}_\mathcal{C}\cap \hat{\mathcal{V}}_\mathcal{C}$ 
  are Poincar\'e sections for the flows. These now serve as 
  combinatorial models for the Teichm\"uller flow in the moduli 
  space of flat surfaces.

  The Teichm\"uller flows on $\mathcal{V}_\mathcal{C}^\pm$ further project
   to suspension flows over $\mathbb{P}^\pm:\mathbb{V}^\pm_\mathcal{C}\rightarrow \mathbb{V}^\pm_\mathcal{C}$ with roof
    functions $t_R^\pm$. More precisely, let
  \begin{equation}
      \label{eqn:coordinates}
    \begin{split}
    \bar{\mathbb{V}}^+_\mathcal{C}& 
 = \left\{(h,s)\in \mathbb{V}^+_\mathcal{C}\times \mathbb{R}: s\in[0,t^+_R) \right\},\\
      \bar{\mathbb{V}}^-_\mathcal{C}& = \left\{(\lambda,s)
           \in \mathbb{V}^-_\mathcal{C}\times \mathbb{R}: s\in[0,t^-_R) \right\}
                    \end{split}
                  \end{equation}
   be the set of coordinates for the suspension flows:
    $(h,s)\mapsto e^sh$ and $(\lambda,s)\mapsto e^s\lambda$.

    \begin{thm}[\cite{veech:gauss}]
    \label{fg:120}
 The RV renormalization map $\mathbb{P}^-:
 \mathbb{V}^-_\mathcal{C}\rightarrow \mathbb{V}^-_\mathcal{C}$ is ergodic
  with respect to $m_1^-$, and thus so is the Teichm\"uller flow 
  on $\bar{\mathbb{V}}^-_\mathcal{C}$ with respect to 
  $m^-_\mathcal{C} =e^{s|\mathcal{A}|}dm^-_1\, ds$. 
  Moreover, given a Rauzy class $\mathcal{C}$, there 
  exists a vector $\bar{\kappa}$ and a finite-to-one, 
  measurable map $\Pi_\mathcal{C}:
  \mathcal{V}_\mathcal{C}^-\rightarrow \mathcal{H}(\bar{\kappa}),$
   where $\mathcal{H}(\bar{\kappa})$ is stratum of flat surfaces 
   such that $\Pi_\mathcal{C}\circ \Phi_t^- = g_t\circ \Pi_\mathcal{C}$,
    and this flow is ergodic when restricted to the subset 
    of surfaces of area 1.
                  \end{thm}                  
   Using the coordinates (\ref{eqn:coordinates}), define the measure
    on $\bar{\mathbb{V}}^+_\mathcal{C}$
     $$\hat{m}^h_\mathcal{C}= 
     \sum_{\pi\in\mathcal{C}}e^{s|\mathcal{A}|}\mathcal{D}(h)d_1^\pi h\,ds,$$
    where $\mathcal{D}(h) = \prod_\alpha h_\alpha^{-1}$ and $d_1^\pi h$ is 
    the Lebesgue volume in the simplex $\Delta^+_\pi\subseteq H^+_\pi$
     of vectors $h$ with $|h|_1=1$.

   \begin{prop}
    \label{fg:130}
    The measure $\hat{m}^h_\mathcal{C}$ is $\Phi_t$-invariant.
     \end{prop}
     \begin{proof}
 Using the coordinates $(h,s)$ as above, we pick 
  a small flowbox of the form $\bar{B} = B_\delta \times [0,\epsilon]$, 
  where $B_\delta\subseteq \Delta^+_\pi$ is a small ball 
  for some $\pi\in\mathcal{C}$, where $\epsilon<\max_{h\in B_\delta} 
  t_R(\pi,h)$. For any $t$ small enough,
   \begin{equation}
    \label{eqn:measure}
    \begin{split}
     \hat{m}^h_\mathcal{C}(\Phi_t(\bar{B}))&=\int_{B_\delta}\int_
     t^{\epsilon+t}e^{s|\mathcal{A}|}\mathcal{D}(e^sh)\, ds\, d_1^\pi  
     = \int_{B_\delta}\int_t^{\epsilon+t}\frac{e^{s|\mathcal{A}|}}{e^{s|\mathcal{A}|}\prod_\alpha h_\alpha}\, 
     ds\, d_1^\pi h \\
   & = \epsilon \int_{B_\delta}\mathcal{D}(h)\, d_1^\pi h 
   = \int_{B_\delta}\int_0^{\epsilon}\frac{e^{s|\mathcal{A}|}}{e^{s|\mathcal{A}|}\prod_\alpha h_\alpha}\, 
   ds\, d_1^\pi h \\
    &=\int_{B_\delta}\int_0^{\epsilon}e^{s|\mathcal{A}|}\mathcal{D}(e^sh)\, 
    ds\, d_1^\pi = \hat{m}^h_\mathcal{C}(\bar{B}).
                      \end{split}
     \end{equation}
    This, combined with Proposition \ref{fg:110} shows the 
    $\Phi_t$-invariance of $\hat{m}^h_\mathcal{C}$.
                  \end{proof}
            
 \subsection{Bratteli diagrams for finite genus}
 \label{subsec:typicalBrat}
  Given $(\pi,\lambda,\tau)$ in $\mathcal{V}^{(1)}_\mathcal{C}$, we want 
  to produce a bi-infinite ordered  Bratteli diagram,
   $\mathcal{B}_{\pi,\lambda,\tau}$, so that the resulting 
   surface $S(\mathcal{B}_{\pi,\lambda,\tau})$ 
   is $S(\pi,\lambda,\tau)$. 
   
   We make a couple of remarks. The first is that, as we noted earlier, while
   the space 
       $S_{\mathcal{B}}$ depends only on the 
       bi-infinite ordered Bratteli diagram, the atlas for it also depends on the 
       given state $\nu_{r}, \nu_{s}$. In fact, the state here will
       be given in a rather simple fashion from $\lambda$ and $\tau$.
       
    The second comment is that we will only construct the Bratteli diagram for    
   $(\pi,\lambda,\tau)$ which are RH-complete and satisfy the 
   Keane condition. This isn't unreasonable as our foliations
   $\mathcal{F}_{\mathcal{B}}^{\pm}$ tend to be minimal 
   under rather mild restrictions.

  Let $(\pi,\lambda,\tau)$ in $\mathcal{V}_\mathcal{C}$. In order to define a
   bi-infinite  ordered Bratteli diagram $\mathcal{B}_{\pi,\lambda,\tau}$,
    it suffices to describe 
    the vertex set $V_{n}$ and the edge set $E_n$, for all integers 
    $n$, along with the partial orders $\leq_r,\leq_s$ at every vertex.
      For all $n\in\mathbb{Z}$, we define $V_{n} = \mathcal{A}$. 
      This presents a minor notational problem: if we write
      $r^{-1}(\alpha) \subseteq E_{n}$, we are considering
      $\alpha$ as an element of $V_{n}$, but this does not appear 
      explicitly in the notation. To solve this, we use
       $r_{n}:E_{n} \rightarrow \mathcal{A}$ and 
       $s_{n}:E_{n} \rightarrow \mathcal{A}$ for the range and source maps.
      Note that the set $\mathcal{A}$ is that of symbols and 
      not of their positions zippered rectangles. As such, in order
       to describe $E_n$, it suffices to provide a 
       $\mathcal{A}\times \mathcal{A}$ matrix $M_n$ which describes
        the connections between $V_{n-1}$ and $V_{n}$.

   \begin{figure}[t]
   \includegraphics[width=5.25in]{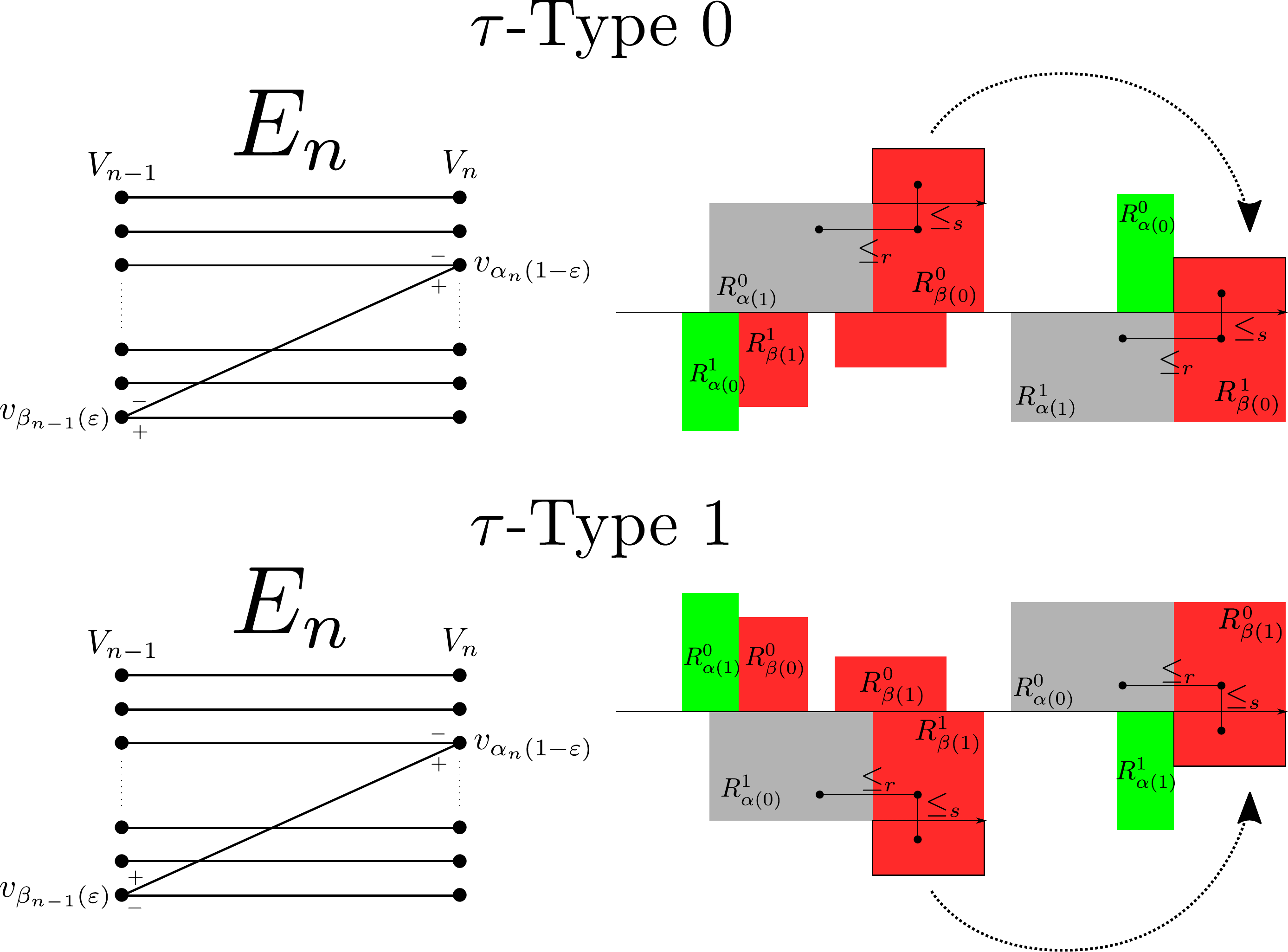}
              \caption{{\tiny The edge set $E_n$ in the case that
    $\mathcal{P}^{n-1}(\pi,\lambda,\tau)$ is of 
 $\tau$-type $\varepsilon$ in $\{0,1\}$. The $-$ and $+$ symbols indicate the orders at
the vertices where there are more than one incoming or outgoing edges. 
When $\mathcal{P}^{n-1}(\pi,\lambda,\tau)$ is of $\tau$-type 0,
 the three dots in the rectangles $R^0_{\alpha(1)}$ and $R^0_{\beta(0)}$ 
 used to define the orders $\leq_r,\leq_s$. Indeed, the two dots sharing 
 a $y$-coordinate sit on the same leaf of the horizontal leaf, making 
 them right-tail-equivalent in $\mathcal{B}$. That $e_{\alpha_n(1)} <_r e_n$
  in this case is dictated from the order on the leaf of the foliation 
  containing those two points. This same order on the horizontal 
  foliation happens in the case of $\tau$-type 0. The choice for
   the $\leq_s$ order comes from comparing two points on the same
    vertical leaf, and these take different forms depending on 
    the type. This is evident from the two figures.}}
              \label{fig:edgeset}                    
            \end{figure}
            
            For $n> 0$, let
            \begin{equation}
              \label{eqn:matrices}
              M_n = \left\{\begin{array}{ll}
               (\ref{eqn:Psi0})&\mbox{ if 
               $\mathcal{P}^{n-1} (\pi,\lambda,\tau)$ is of
                $\tau$-type 0}  \\ (\ref{eqn:Psi1})&
                \mbox{ if $\mathcal{P}^{n-1} (\pi,\lambda,\tau)$ 
                is of $\tau$-type 1,}  \end{array} \right.
            \end{equation}
  and let $E_n$ be the edge set defined by $M_n$. In other words, there is an
   edge $e_\alpha $ in $E_n$ with $s_{n}(e_\alpha)= \alpha$ in $V_{n-1}$ and 
   $ r_{n}(e_\alpha)= \alpha$ in $V_n$, for each $\alpha$ in $\mathcal{A}$. 
   We refer 
   to such edges as horizontal.  
   In addition, there is
    an edge $e_n$ in $E_n$ with $s_{n}(e_n) = \beta_{n-1}(\varepsilon)$ 
   and $r_{n}(e_n) = \alpha_n(1-\varepsilon)$, if $\mathcal{P}^n(\pi,\lambda, \tau)$ 
   is of type $\varepsilon$ in $\{0,1\}$, where $\alpha_n(\varepsilon)$ 
   and $\beta_n(\varepsilon)$ are the corresponding symbols in the permutation
    in $\mathcal{P}^n(\pi, \lambda, \tau)$. Note that 
    $\beta_n(\varepsilon) = \alpha_{n+1}(\varepsilon)$ depending
     on the type of $\mathcal{P}^n(\pi,\lambda, \tau)$.
            
    We now move to define the orders $\leq_r,\leq_s$ on $\mathcal{B}$.
     These will also depend on the $\tau$-type of 
     $\mathcal{P}^{n-1}(\pi,\lambda,\tau)$. Since $|r_{n}^{-1}(\alpha)| = 1$
      for all $\alpha \neq \alpha_n(1-\varepsilon)$ and $n>0$, it suffices 
      to define the order $\leq_r$ on 
      $\{e_n, e_{\alpha_n(1-\varepsilon)}\}=
      r^{-1}_{n}(\alpha_n(1-\varepsilon)) $, depending 
      of the type of $\mathcal{P}^{n-1}(\pi,\tau)$. We let
            \begin{equation}
              \label{eqn:orderR}
              e_{\alpha_n(1-\varepsilon)} <_r e_n
            \end{equation}
            at each $r^{-1}_{n}(\alpha_n(1-\varepsilon))$, depending on
             the type.
            Since $|s_{n}^{-1}(\alpha)| = 1$ for all 
            $\alpha\neq \beta_{n-1}(\varepsilon) = \alpha_n(\varepsilon)$ 
            and $n>0$ (here $\varepsilon$ is the type of
             $\mathcal{P}^{n-1}(\pi,\lambda, \tau)$), it suffices to define the 
             order $\leq_s$ on 
             $\{e_n, e_{\beta_{n-1}(\varepsilon)}\}=
             s_{n}^{-1}(\beta_{n-1}(\varepsilon)) $, 
             depending of the type of $\mathcal{P}^{n-1}(\pi,\lambda, \tau)$. 
             We define the orders
            \begin{equation}
              \label{eqn:orderS}
        \begin{array}{rl}
        e_n <_s e_{\alpha_n(1)}=
         e_{\beta_{n-1}(0)}&\mbox{ if $\mathcal{P}^{n-1}(\pi,\lambda,\tau)$ 
         is of $\tau$-type 0,} \\
 e_{\alpha_n(0)}= e_{\beta_{n-1}(1)} <_s e_n \hspace{.94in}& 
  \mbox{ if $\mathcal{P}^{n-1}(\pi,\lambda,\tau)$ is of $\tau$-type 1,} 
              \end{array}
            \end{equation}
  at $s_{n}^{-1}(\beta_{n-1}(\varepsilon))$. These choices define the
   positive half of $\mathcal{B}_{\pi,\lambda, \tau}$, see 
   Figure \ref{fig:edgeset} for a geometric justification 
   for these choices.
            
    The definition for the negative part will essentially be the 
    same form as (\ref{eqn:matrices}), if we use Proposition 
    \ref{fg:80}. Recall from the proof of Proposition 
    \ref{fg:80} that if $(\pi,\lambda,\tau)
     = \mathcal{P}(\pi',\lambda',\tau')$ is of $\lambda$-type 
     $\varepsilon$, then $(\pi',\lambda', \tau') = \mathcal{R}(\pi,\lambda,\tau)$ 
     is of $\tau$-type $1-\varepsilon$. Thus, going by (\ref{eqn:matrices})
      for $n=0$ we can define $M_0$ as (\ref{eqn:Psi1}) if 
      $(\pi,\lambda,\tau)$ is of $\lambda$ type $0$, and as
       (\ref{eqn:Psi0}) if $(\pi,\lambda,\tau)$ is of $\lambda$ type $1$. 
       Extending for higher powers of $\mathcal{R} = \mathcal{P}^{-1}$,
        we get, for $n\leq 0$:
            \begin{equation}
              \label{eqn:matrices2}
  M_n  = \left\{\begin{array}{ll} (\ref{eqn:Psi1})&\mbox{ 
  if $\mathcal{R}^{n} (\pi,\lambda,\tau)$ is of $\lambda$-type 0}  \\
   (\ref{eqn:Psi0})&\mbox{ if $\mathcal{R}^{n} 
   (\pi,\lambda,\tau)$ is of $\lambda$-type 1} . \end{array} \right.
            \end{equation}
            The orders are now similarly defined for the negative half: 
            we extend the definitions using (\ref{eqn:orderR}) 
            and (\ref{eqn:orderS}) depending on the $\tau$-type 
            of $\mathcal{P}^{n-1}(\pi,\lambda,\tau)$, that is,
             depending on the $\lambda$-type of $\mathcal{R}^n(\pi,\lambda,\tau)$.                  
            \begin{rmk}
              \label{fg:140}
       \begin{enumerate}
     \item Note that there are $2(|\mathcal{A}|-1)$ possible matrices 
     that can appear as $\Psi$ in (\ref{eqn:matrices}) and (\ref{eqn:matrices2}),
      all of which are invertible and of determinant 1. As such, we have 
      for the AF algebras $C^{*}_{\lambda}(T^{+}(X_{\mathcal{B}_{\pi,\lambda,\pi}}))$ 
      that the diagrams define,
   \begin{equation}
     \label{eqn:K-homology}
   K_0(C^{*}_{\lambda}(T^{+}(X_{\mathcal{B}_{\pi,\lambda,\pi}})) )\cong
   \mathbb{Z}^{|\mathcal{A}|} \cong H_1(S(\pi,\lambda,\tau),\Sigma;\mathbb{Z}),
                \end{equation}
                which had already been proved in \cite{putnam:iet}.
  This does not, however, address the subtler   issue of the natural order
  structure.           
      \item Given the definition of the Bratteli diagram
       $\mathcal{B} = \mathcal{B}_{\pi,\lambda,\tau}$ above, it is easy
        to identify some extreme elements at once: for 
        any $\alpha$ in $\mathcal{A}$ the path 
        $p_\alpha = (\dots, p_\alpha^{n-1},p_\alpha^n,p_\alpha^{n+1},\dots)$ 
        in $X_\mathcal{B}$ with $s_{n}(p_\alpha^n) = \alpha$ 
        and $r_{n}(p_\alpha^n) = \alpha$ is in $X_\mathcal{B}^{r-min}$ 
        and so $X_\mathcal{B}^{r-min}$ has exactly $|\mathcal{A}|$ elements, 
        the horizontal paths in Figure \ref{fig:edgeset}.
              \end{enumerate}
            \end{rmk}
      
      Our next task is to define a state on the Bratteli diagram which we have 
      just constructed. In fact, this is fairly simple: we
      let $(\pi_{n}, \lambda_{n}, \tau_{n})$ be $ (\pi, \lambda, \tau) $ 
      for $n=0$, $\mathcal{P}^{n}(\pi, \lambda, \tau)$ for $n >0$ and 
      $\mathcal{R}^{-n}(\pi, \lambda, \tau)$ for $n < 0$. We again let 
      $h_{n} = \Omega_{\pi_{n}}(\tau_{n})$, which lies in 
      $\R^{\mathcal{A}}$.  For $ \alpha $ in $\mathcal{A}=V_{n}$, we define
      $\nu_{r}(\alpha) = (\lambda_{n})_{\alpha}  $ and 
      $\nu_{s}(\alpha) = (h_{n})_{\alpha}  $. It is a trivial matter to see that
      this is a state on $\mathcal{B}$.
            \begin{figure}[t]
              \includegraphics[width=5.75in]{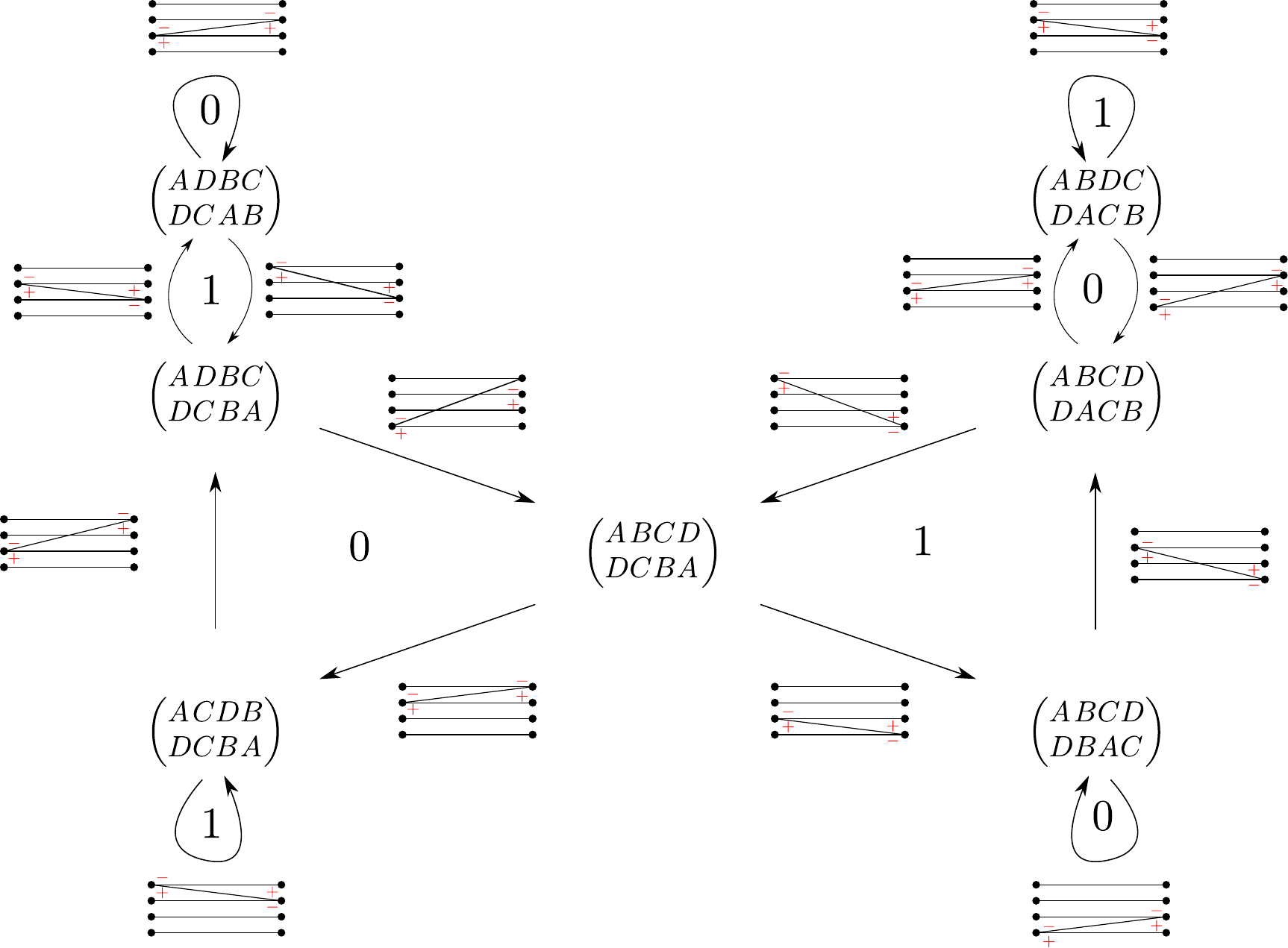}
              \caption{{\tiny The Rauzy graph for surfaces in the hyperelliptic component of $\mathcal{H}(2)$. The arrows represent a step of RH induction depending on the $\tau$-type in $\{0,1\}$. Next to every arrow is the edge set associated to the Bratteli diagram: if $(\pi,\lambda,\tau)$ is RH-complete, then it defines an infinite walk on this graph, and the edge set $E_n$ is defined by the edge set corresponding to the arrow above in the $n^{th}$ step. If $(\pi,\lambda,\tau)$ satisfies the Keane condition, then it defines an infinite backwards walk on this graph and the edge sets of the Bratteli diagram $\mathcal{B}_{\pi,\lambda,\tau}$ are defined accordingly.
              }}
              \label{fig:hyper2}
            \end{figure}

       It is a simple matter to see that these definitions mean that, 
       for any $n > 0$,  
       a symbol $\alpha$ in $\mathcal{A}$ is the $\tau$-winner 
       in RH induction, $\mathcal{P}^{n}$, if and only if the non-horizontal 
       edge of $E_{n}$ has range equal to $\alpha$. Similarly,
        a symbol $\alpha$ is the $\lambda$-winner 
       in Rauzy-Veech  induction, $\mathcal{R}^{n}$, if and only if 
       the non-horizontal 
       edge of $E_{-n}$ has range equal to $\alpha$. This proves the following.
       
            \begin{prop}
            \label{fg:144}
              The Bratteli diagram $\mathcal{B}_{\pi,\lambda,\tau}$ satisfies
               the Keane condition if, for every $\alpha$ in $\mathcal{A}$,
  $|r_{n}^{-1} \{ \alpha \}|>1$ for infinitely many negative 
  integers $n$, and is RH-complete 
                if and only if
     $|r_{n}^{-1} \{ \alpha \}|>1$ for infinitely many positive integers $n$.            
                \end{prop}

            The next fact follows from \cite[\S 1.2.4]{MMY:Roth}
             (see also \cite[Corollary 10]{berk2021backward}).
            \begin{prop}
            \label{fg:146}
              If $(\pi,\lambda,\tau)$ satisfies the Keane condition
               and is RH-complete, then $\mathcal{B}_{\pi,\lambda,\tau}$
                is strongly simple.
            \end{prop}
            
The Keane condition allows us to describe the elements of 
 $X^{ext}_{\mathcal{B}_{\pi,\lambda, \tau}}$    and even more, paths which
 are tail equivalent to these.    To do so, we introduce
 some notation.     Consider compatible representatives of the vertices 
 of the Rauzy graph of $\pi$. That is, pick a representative 
 $\pi = (\pi_0,\pi_1)$ of a vertex and consider the representatives 
 of other classes which can be reached under finitely many steps 
 of induction. Let $A_\varepsilon = \pi_{\varepsilon}(1)$, 
  the first symbol of $\pi_\varepsilon$, and note that they are 
 the first symbols in each representative in the Rauzy graph, 
 that is, they are preserved under induction. Recall that there
  is an edge $e$ defined by the $\tau$-type of $(\pi,\lambda,\tau)$,
   which satisfies $s(e) = \beta(\varepsilon)$ and
    $r(e) = \alpha(1-\varepsilon)$ whenever the $\tau$-type 
    is $\varepsilon$ (Figure \ref{fig:edgeset}).
    
    \begin{prop}
      \label{fg:150}
      Suppose that 
      $(\pi,\lambda, \tau)$ satisfies the Keane condition, is RH-complete and that 
      $x$ is an infinite path in $\mathcal{B}_{\pi,\lambda, \tau}$.   
      \begin{enumerate}
      \item Suppose  there is $n_{0}$ such that $x_{n}$ is $r$-minimal, for all 
        $n \leq n_{0}$. Then $x_{n}$ is horizontal, for all $n \leq n_{0}$.
        In particular, $X^{r-min}_{\mathcal{B}_{\pi,\lambda, \tau}}$
        consists of the $\vert \mathcal{A} \vert$ 
        infinite horizontal paths.
\item  If $x$ is in $X^{r-max}_{\mathcal{B}_{\pi,\lambda, \tau}}$, then 
  $x_{n}$ is not horizontal, 
  for infinitely many $n > 0$. 
\item Suppose that there is an integer $n_{0}$ such that
  $x_{n}$ is $s$-maximal for all $n \geq n_{0}$. Then there exists
           $m_{0} \geq n_{0}$ such that 
           $r(x_{n}) = s(x_{n}) = A_{0}$, for all $n \geq m_{0}$.  
           \item Suppose that there is an integer $n_{0}$ such that
           $x_{n}$ is $s$-minimal for all $n \geq n_{0}$. Then there exists
           $m_{0} \geq n_{0}$ such that 
           $r(x_{n}) = s(x_{n}) = A_{1}$, for all $n \geq m_{0}$.             
\end{enumerate}
            \end{prop}
 \begin{proof}
 The first part follows easily (even without the Keane condition) 
 from the fact that 
  if $x_{n}$ is $r$-minimal, then it 
 is horizontal, by the  definition of $\leq_{r}$.
 
 For the second part, suppose that $x_{n}$ is horizontal for all $n > 0$. 
 From the BK condition, there $n > 1$ with $|r^{-1}(r(x_{n})|>1$. 
 The definition of $\leq_{r}$ implies that $x_{n}$ is not $r$-maximal.
 
 The last two parts are more subtle.
   
   Observe that $A_\varepsilon \neq \beta(\varepsilon)$, since by 
   definition $\beta(\varepsilon)$ is the symbol to the right 
   of another symbol, and $A_\varepsilon$ is never to the 
   right of another symbol. Thus there is no non-horizontal 
   edge $e$ defined 
   by the graph with the property that $s(e) = A_\varepsilon$ 
   when the data is of type $\varepsilon$. This means that 
   whenever there is a non-horizontal
    edge $e$ with $s(e) = A_0$, then this 
   corresponds to type $1$, and so it is $s$-max, and likewise
    if there is a non-horizontal  edge $e$ with $s(e) = A_1$, then this corresponds
     to type $0$, and so it is $s$-min. It follows that the constant
      path $\{A_0\}$ is $s$-min while the constant path $\{A_1\}$ 
      is $s$-max, and both of these paths are also $r$-min.

     For a set $V\subseteq V_{n}$, we define 
              $$Q(V) = \{s(e)  \mid 
              e\in E_n \mbox{ is $s$-minimal and }r(e)\in V\}.$$
              For $V\subseteq V_n$ we will denote $Q^m(V)\subseteq V_{n-m}$ 
              the image of the composition of $Q$ $m$ times. We
               first observe that if $ \alpha$ is in 
               $Q^m(\{A_0 \}) \subseteq V_{n-m}$, 
           then there is an $s$-min path in $E_{n-m,n}$ with  
             $s(p)= \alpha$ and  $r(p) = A_0$.
              \begin{lemma}
              \label{fg:155}
  For any $ 1 \leq i  < d$, if 
   $V = \{ \alpha^{0}_{1}, \ldots, \alpha^{0}_{i} \} $, then 
   $Q(V)$ 
    equals one of
        $ \{ (\alpha')^{0}_{1}, \ldots, (\alpha')^{0}_{i} \} $
        or $\{ (\alpha')^{0}_{1}, \ldots, (\alpha')^{0}_{i+1}   \}$ 
    Moreover, if $(\pi, \lambda, \tau)$ is $\tau$-type 0 and 
    $\alpha(1 )$ is in 
   $V$, then 
   the latter holds.     
              \end{lemma}
              \begin{proof}
 The first case to consider is when $(\pi, \lambda, \tau)$ is $\tau$-type 
   $1$. In this case, every $s$-minimal edge in $E_{n}$ is 
   horizontal and so $Q(V) = V$, for any set $V$. On the other
    hand, $(\alpha')^{0}= \alpha^{0}$ and so the conclusion
   holds, with the first of the two cases.

   We now assume    $(\pi, \lambda, \tau)$ is $\tau$-type  $0$.
   Suppose $\alpha(1) = \alpha^{0}_{k}$, for some $1 \leq k \leq d$.
There are two
cases to consider. The first is that $\alpha(1)$ is not in $V$. In other words, 
$\alpha(1) = \alpha_{j}^{0}$, for some $j > i$. In this case, the horizontal edge to
each element  of $V$ is also $s$-minimal
so $Q(V)= V$. Moreover, the change in $\beta^{0}$ from $\alpha^{0}$ only occurs
in entries greater than $k$. In other words,  we have 
  $ \{ (\alpha')^{0}_{1}, \ldots, (\alpha')^{0}_{i} \} = 
\{ \alpha^{0}_{1}, \ldots, \alpha^{0}_{j} \} = V$
and $Q(V) = V$. 

Now, we suppose that $\alpha(1)$ is in $V$.
 In other words, 
$\alpha(1) = \alpha_{j}^{0}$, for some $j \leq i$. In this case, 
we have the non-horizontal $s$-min edge goes from $\alpha(0)$ to 
$\alpha(1)$.  Observe that 
because $i < d$, $\alpha(0)$ is not in 
$V $. It follows that $Q(V) = V \cup \{ \alpha(0) \}$. 
On the other hand, $(\alpha')^{0}$ is obtained from $\alpha^{0}$  by 
inserting 
$\alpha(0)$ to the right of $\alpha(1)$ and moving the entries 
to the right one more space to the right.
In other words, we have $\{ (\alpha')^{0}_{1}, \ldots, (\alpha')^{0}_{i+1} \}
=  V \cup \{ \alpha(0) \}$ and we are done.
\end{proof}
             
            \end{proof}

            \begin{prop}
            \label{fg:160}
    If  $(\pi, \lambda, \tau)$ is such that, 
   $|r_{n}^{-1} \{ A_{0}\} | > 1$ for 
    infinitely many $n >  0$, then any $x$ in $X_{\mathcal{B}_{\pi,\lambda, \tau}}$
    such that there is some $n_{0}$ such 
     $x_{n}$ is $s$-minimal, for all $n \geq n_{0}$,  
   there is $n_{1} $ such that $x_{n}$ is the horizontal edge
   from $A_{0}$ to itself, for all $n \geq n_{1}$.
    In particular,  $I_{\mathcal{B}_{\pi,\lambda, \tau}}^{+}=1$ with 
    $x_{1}$ being the infinite horizontal path through $A_{0}$.
   Similarly, if  $(\pi, \lambda, \tau)$ is such that, 
   $|r_{n}^{-1} \{ A_{1}\} | > 1$ for 
    infinitely many $n >  0$, then any $x$ in $X_{\mathcal{B}_{\pi,\lambda, \tau}}$
    such that there is some $n_{0}$ such 
     $x_{n}$ is $s$-maximal, for all $n \geq n_{0}$,  
   there is $n_{1} $ such that $x_{n}$ is the horizontal edge
   from $A_{1}$ to itself, for all $n \geq n_{1}$.
    In particular,  $J_{\mathcal{B}_{\pi,\lambda, \tau}}^{+}=1$ with 
    $x_{2}$ being the infinite horizontal path through $A_{0}$.
            \end{prop}
            
            \begin{proof}
   We prove the first statement only.  Choose $n_{1} > n_{0}$ so that
   $|r_{n}^{-1} \{ A_{0}\} | > 1$, for at least $| \mathcal{A} |$ values of
   $n $ between $n_{0}$ and $n_{1}$. If we then consider $A_{0}$ as 
   a vertex in $V_{n_{1}}$, and apply 
   $Q$ successively to $V = \{ A_{0} \}$, there will be
   at least $| \mathcal{A} |$ times when $Q^{m}(V)$ is strictly larger then
   $V$. For some some $n_{0} \leq n \leq n_{1}$, we have
   $Q^{n_{1}-n}(V) = V_{n}$. It follows that every $s$-minimal 
   starting in $V_{n}$ will have range equal to $A_{0}$. Also, any 
   $s$-minimal path starting at $A_{0}$ will be horizontal. As 
   $n \geq n_{0}$, $x$ satisfies both properties and the conclusion follows.
 \end{proof}

 \begin{thm}
 \label{fg:165}
  If $(\pi,\lambda, \tau)$ in $\mathcal{V}_\mathcal{C}$ 
    satisfies the Keane condition and is RH-complete, then
    $\mathcal{B}_{\pi,\lambda, \tau}$ satisfies the 
    standard conditions  of Definition \ref{surface:20}.
 \end{thm}
 
 \begin{proof}
 We have already seen that $\mathcal{B}_{\pi,\lambda, \tau}$
 is strongly simple in Proposition \ref{fg:146}. It is 
 clearly finite rank since $\#V_{n} = \# \mathcal{A}$, 
 for all integers $n$.
 
 We finally verify the third condition, starting with considering 
 $(X_{\mathcal{B}}^{s-min} \cup X_{\mathcal{B}}^{s-max})
  \cap \partial_{r}X_{\mathcal{B}}$. We know from Proposition 
  \ref{fg:160} that  $(X_{\mathcal{B}}^{s-min}$  and 
  $ X_{\mathcal{B}}^{s-max})$ consist of $x_{1}$ and $x_{2}$, the horizontal
   paths through $A_{0}$ and $A_{1}$, respectively, and hence are both in 
   $X_{\mathcal{B}}^{r-min}$ which is excluded from 
   $\partial_{r}X_{\mathcal{B}}$, by definition.
   
   We now consider  $(X_{\mathcal{B}}^{r-min} \cup X_{\mathcal{B}}^{r-max})
  \cap \partial_{s}X_{\mathcal{B}}$. Again Proposition
  \ref{fg:160} implies that $\partial_{s}X_{\mathcal{B}}$ is 
  contained in $T^{+}(x_{1})$ and $T^{+}(x_{2})$. The only paths which also
  lie in $X^{r-min}_{\mathcal{B}}$ are the horizontal paths $x_{1}$ and $x_{2}$, 
  which are excluded from 
   $\partial_{s}X_{\mathcal{B}}$, by definition. If $x$ is in 
 $T^{+}(x_{1})$ and $X^{r-max}_{\mathcal{B}}$, then $x_{n}$ is the horizontal 
 edge from $A_{0}$ to itself for all $n \geq n_{0}$. By RH-completeness, 
 there is some $n \geq n_{0}$ with $|r_{n}^{-1}\{ A_{0} \}| > 1$, which means that
 $x_{n}$ is not $r$-maximal, a contradiction. The same argument shows 
 $T^{+}(x_{2}) \cap X^{r-max}_{\mathcal{B}}$ is empty.
 \end{proof}
            \begin{figure}[t]
              \includegraphics[width=5.25in]{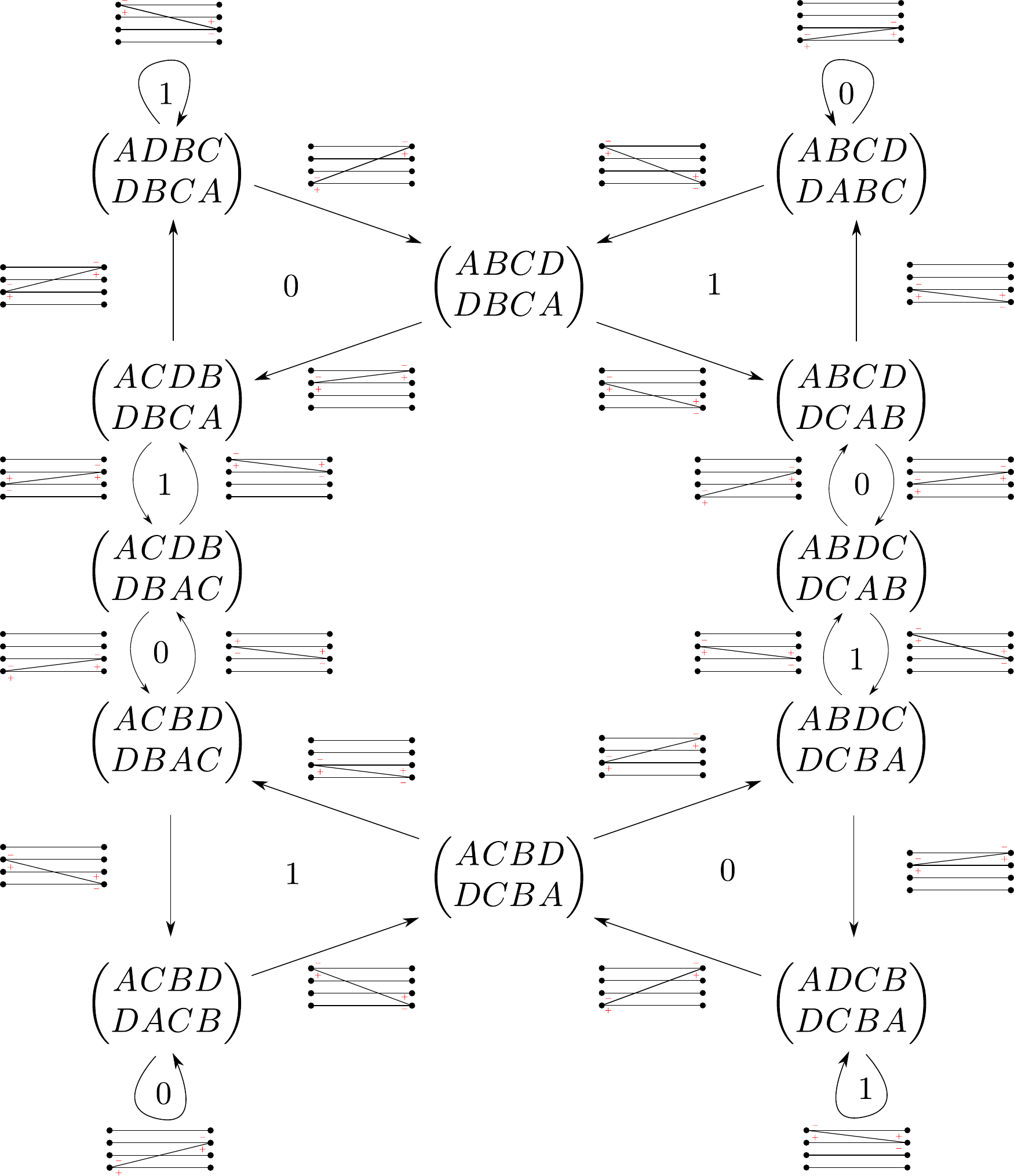}
              \caption{Rauzy graph for surfaces in the
               non-hyperelliptic component of $\mathcal{H}(2)$
                along with corresponding edge sets for the Bratteli diagrams.}
              \label{fig:nonhyper2}
            \end{figure}

     \subsection{Flatness of $\mathcal{B}_{\pi,\lambda, \tau}$}
        \label{subsec:flatness}
            In this section, we will prove the following
             flatness property of $\mathcal{B}_{\pi,\lambda, \tau}$.
            \begin{thm}
              \label{fg:200}
   If $(\pi,\lambda, \tau)$ in $\mathcal{V}_\mathcal{C}$ 
    satisfies the Keane condition and is RH-complete, then
     $\Sigma_{\mathcal{B}_{\pi,\lambda, \tau}} = \varnothing$.
            \end{thm}
      Denote by $x_\varepsilon = \{A_\varepsilon\} $ 
      the corresponding $s$-min/max paths from Proposition 
       \ref{fg:150}. Then $T^+(x_\varepsilon)$ is linearly 
       ordered by $\leq_r$ and 
       $\Delta_s:T^+(x_\varepsilon)\setminus 
       \{x_\varepsilon\}\rightarrow T^+(x_{1-\varepsilon})
       \setminus \{x_{1-\varepsilon}\}$ is a bijection.
            \begin{lemma}
              \label{fg:210}
  If $\Delta_s$ preserves $\leq_r$, then $\Sigma_\mathcal{B} = \varnothing$.
            \end{lemma}
            \begin{proof}
              Let $x\in\partial X_\mathcal{B}$ and
               suppose $\Delta_r(x)$ is the $r$-successor
                of $x$. Then $\Delta_s\circ \Delta_r(x)$
                is the $r$-successor of $\Delta_s(x)$, 
                that is $\Delta_r\circ \Delta_s(x) = \Delta_s\circ \Delta_r(x)$.
            \end{proof}
    For every $n$, $E_n$ has an edge which is not $s$-max, call it
     $y_n$ and an edge $z_n$ which is not $s$-min. These are the
      edges $\{y_n,z_n\}=s^{-1}(v_{\beta_{n-1}(\varepsilon)})$.
       Define
  \begin{equation}
 \begin{split}
   Y_n&= \left\{ x\in X_\mathcal{B} \mid
    x_n = y_n\mbox{ and }x_i\mbox{ is $s$-max for all }i>n\right\} \\
                Z_n&= \left\{ x\in X_\mathcal{B} \mid
                 x_n = z_n\mbox{ and }x_i\mbox{ is $s$-min for all }i>n\right\} .
              \end{split}
            \end{equation}
            Note that $\Delta_s:Y_n\rightarrow Z_n$ 
            is a bijection for every $n$. Moreover, by definition,
             we also have that
            $$T^+(x_0)\setminus \{x_0\} = 
            \bigsqcup_n Y_n\hspace{.4in}\mbox{
             and }\hspace{.4in} T^+(x_{1})\setminus \{x_{1}\} = \bigsqcup_n Z_n$$
            so if $Y_n\leq_r Y_{n+1}$ and $Z_n\leq_r Z_{n+1}$,
             for all $n$, then $\Delta_s: T^+(x_\varepsilon)\setminus
              \{x_\varepsilon\}\rightarrow T^+(x_{1-\varepsilon})\setminus
               \{x_{1-\varepsilon}\}$ preserves $\leq_r$.
            \begin{prop}
              \label{fg:220}
              $Y_n\leq_r Y_{n+1}$
            \end{prop}
            Let $x = \{e_i\} $ be in $Y_n$ and $x' = \{e'_i\} $ be 
            in $Y_{n+1}$. Let  $\ell>n$ be the smallest
             integer where $r(e_i) = r(e'_i)$. We will show that 
             $e_\ell\leq_r e_\ell'$. We first treat two simple cases.
    \begin{lemma}
    \label{fg:230}
    If $E_n,E_{n+1}$ are respectively of $\tau$-type 0,0
     or 1,0, then $\ell=n+1$ and $e_{n+1}\leq_r e_{n+1}'$.
            \end{lemma}
            \begin{proof}
      If they are of type 0,0, it is immediate to check that 
      $(e_n,e_{n+1})$ is the concatenation of the $s$-min path 
      from $\beta(0)$ to $\alpha(0)$ followed by the
       horizontal edge $e_{\alpha(0)}$, whereas $e_{n+1}'$ is
        the $s$-min path from some vertex to $v_{\alpha(0)}$, 
        meaning that $r(e_{n+1}) = r(e_{n+1}')$. Since horizontal 
        paths are always $r$-min, it follows that $e_{n+1}\leq_r e_{n+1}'$.
              
   If they are of type 1,0, then it is immediate to check 
   that $(e_n,e_{n+1})$ is the horizontal path associated 
   with symbol $\beta(1)$, whereas $e_{n+1}'$ is the $s$-min 
   path from some vertex to $\beta(1) = r(e_{n+1})$. Again, 
   since horizontal paths are always $r$-min, it follows 
   that $e_{n+1}\leq_r e_{n+1}'$.
            \end{proof}
            Thus we are left to inspect the cases where $E_n,E_{n+1}$ 
            are respectively of types 0,1 or 1,1.
            \begin{lemma}
            \label{fg:240}
  Suppose $E_n,E_{n+1}$ are respectively of types 0,1 or 1,1. 
  For $n+1\leq i<\ell-1$, if $r(e_i) = \gamma$ and $r(e_i') = \gamma'$, 
  then the symbol $\gamma$ is immediately to the left of $\gamma'$ 
  on the bottom row of the permutation defined by 
  $\mathcal{P}^i(\pi,\lambda,\tau)$.
            \end{lemma}
            Before proving this lemma, let us prove the proposition
             assuming the lemma.
            \begin{proof}[Proof of Proposition \ref{fg:220} 
            assuming Lemma \ref{fg:240}]
    First note that if $r(e_\ell) = r(e_\ell')$, 
    then \newline
    $\mathcal{P}^{\ell-1}(\pi,\lambda,\tau)$ is of 
    $\tau$-type 1, as this is the only way that we can have a 
    non-horizontal $s$-max edge in $E_\ell$. If $r(e_{\ell-1})= \gamma$ 
    and $r(e'_{\ell-1}) = \gamma'$ and $\gamma$ is immediately
     to the left of of $\gamma'$, then by the definition (
     \ref{eqn:matrices}) of $M_\ell$, the non-horizontal
      edge goes from $\gamma'$ in $V_{\ell-1}$ to $\gamma$ in $V_\ell$, 
      and so $e_\ell\leq_r e_\ell'$, showing that $Y_n\leq_r Y_{n+1}$.
            \end{proof}                    
      We now move to prove Lemma \ref{fg:240}. 
      To get us started, we have the following.
            \begin{lemma}
     \label{fg:250}
    If $E_n,E_{n+1}$ are respectively of $\tau$-type 0,1 or 1,1, 
    then $\ell>n+1$ and the permutation associated to
     $\mathcal{P}^{n+1}(\pi,\lambda,\tau)$ is of the form
              \begin{equation}
    \label{eqn:bottomGammas}
   \left(\begin{array}{lr} \cdots&\cdots \\
    \cdots &\gamma \gamma'\end{array} \right)
              \end{equation}
     for some $\gamma,\gamma'$ in $\mathcal{A}$,
      then $r(e_{n+1}) = \gamma$ and $r(e_{n+1}') = \gamma'$.
            \end{lemma}
            \begin{proof}
    Suppose they are respectively of type 0,1. Then the sequence of
     permutations are of the form
              $$\left(\begin{array}{r} \cdots \alpha(1)\beta(0)\cdots
               \alpha(0)\\ \cdots \alpha(0)\beta(1)\cdots \alpha(1) 
                \end{array}\right) \mapsto 
              \left(\begin{array}{r} \cdots \alpha(1)\cdots 
              \alpha(0)\beta(0)\\ \cdots \beta(0)\beta'(1)\cdots
               \alpha(1)   \end{array}\right)  \mapsto 
               \left(\begin{array}{r} \cdots \alpha(1)\cdots 
               \alpha(0)\beta(0)\\ \cdots \beta(0)\cdots \alpha(1) 
               \beta'(1) \end{array}\right),$$
              assuming $\beta(0) \neq \alpha(0)$ and $\beta'(1) \neq \alpha(1)$.
               Now, by definition, $(e_n,e_{n+1})$ is the concatenation of
    the edge from $\beta(0)$ to the vertex $\alpha(1)$ followed by
     the horizontal edge associated to the symbol $\alpha(1)$,
      and so $\gamma = \alpha(1)$, whereas $e_{n+1}'$ is the horizontal
       edge associated to the symbol $\beta'(1)$ and $\gamma' = \beta'(1)$. 
              
      Note that it cannot be the case that both $\beta(0) = \alpha(0)$
       and $\beta'(1)= \alpha(1)$ as this would make the permutation 
       irreducible. Now, if $\beta(0) = \alpha(0)$ and $\beta'(1)\neq \alpha(1)$,
        then the sequence of permutations is of the form
              $$\left(\begin{array}{r} \cdots \alpha(1) \alpha(0)\\
               \cdots \alpha(0)\beta(1)\cdots \alpha(1)  \end{array}\right) 
               \mapsto \left(\begin{array}{r} \cdots \alpha(1) \alpha(0)\\ 
               \cdots \alpha(0)\beta(1)\cdots \alpha(1)   \end{array}\right)
                 \mapsto \left(\begin{array}{r} \cdots \cdots \cdots \alpha(1) 
                 \alpha(0) \\ \cdots \alpha(0)\cdots \alpha(1) \beta(1)
                  \end{array}\right).$$
              Here, $(e_n,e_{n+1})$ is the concatenation of the 
              edge from $\alpha(0)$ to $\alpha(1)$ followed
               by the horizontal edge associated to the symbol
                $\alpha(1)$, and so $\gamma = \alpha(1)$, whereas 
                $e_{n+1}'$ is the horizontal edge associated to the 
                symbol $\beta(1)$ and $\gamma' = \beta(1)$, and the
                 result also holds here.
              
    If $\beta(0) \neq \alpha(0)$ and $\beta'(1)= \alpha(1)$,
     then the sequence of permutations is of the form
              $$\left(\begin{array}{r} \cdots \alpha(1)\beta(0)\cdots
               \alpha(0)\\ \cdots \alpha(0)\beta(1)\cdots \alpha(1) 
                \end{array}\right) \mapsto \left(\begin{array}{r}
                 \cdots \alpha(1)\cdots \alpha(0)\beta(0)\\ 
                 \cdots \beta(0) \alpha(1)   \end{array}\right)
                   \mapsto \left(\begin{array}{r} \cdots \alpha(1)\cdots 
                   \alpha(0)\beta(0)\\ \cdots \beta(0) \alpha(1) \end{array}\right),$$
     Here, $(e_n,e_{n+1})$ is the concatenation of the edge from $\beta(0)$ to 
     $\alpha(1)$ followed by the (non-horizontal) path from
      $\beta'(1) = \alpha(1)$ to $\beta(0)$, whereas $e_{n+1}'$ 
      is the horizontal edge associated to the symbol $\alpha(1)$,
       and so $(\gamma,\gamma') = (\beta(0),\alpha(1))$ and the
        case of types 0,1 is proved.
              
    Now suppose they are respectively of type 1,1. Then the sequence 
    of permutations are of the form
              $$\left(\begin{array}{r} \cdots \alpha(1)\beta(0)\cdots 
              \alpha(0)\\ \cdots \alpha(0)\beta(1)\cdots \alpha(1)  
              \end{array}\right) \mapsto \left(\begin{array}{r}  \cdots
               \beta(1)\beta'(0) \cdots \alpha(0) \\ \cdots \alpha(0)\beta'(1) 
               \cdots \alpha(1)\beta(1)   \end{array}\right)  \mapsto 
               \left(\begin{array}{r} \cdots \cdots  \cdots \cdots
               \alpha(0) \\ \cdots \alpha(0)\cdots \alpha(1) \beta(1) 
               \beta'(1) \end{array}\right),$$
   assuming $\beta(1) \neq \alpha(1)$ (note that $\beta'(1) \neq \beta(1)$ as
    equality would imply that $\alpha(0) = \alpha(1)$ making the original permutation 
   reducible). Now, by definition, $e_{n+1}'$ is the horizontal edge 
   associated to the symbol $\beta'(1)\neq \beta(1)$, whereas $(e_n,e_{n+1})$
    is the concatenation of the horizontal edge associated to the symbol
     $\beta(1)$ followed by the horizontal edge associated to the 
     same symbol, $\beta(1)$, and so $\gamma = \beta(1)$ and $\gamma' = \beta'(1)$.
              
  If $\alpha(1) = \beta(1)$, then the starting permutation is fixed under 
  $\mathcal{P}$ and $\mathcal{P}^2$ (under $\tau$-type 1) and it is
   of the form
              $$\left(\begin{array}{r} \cdots \alpha(1)\beta(0)\cdots \alpha(0)\\
               \cdots \alpha(0) \alpha(1)  \end{array}\right) .$$
              In this case, $(e_n,e_{n+1})$ is the concatenation of the 
              horizontal 
              edge with symbol $\beta(1) = \alpha(1)$ followed by
               the (non-horizontal) edge from $\beta(1)$ to $\alpha(0)$,
                whereas $e_{n+1}'$ is the horizontal edge with symbol 
                $\beta(1) = \alpha(1)$. So $(\gamma,\gamma') = 
                (\alpha(0),\alpha(1))$ in this case and the lemma is proved.
            \end{proof}
            
     \begin{proof}[Proof of Lemma \ref{fg:240}]
   Given Lemma \ref{fg:250}, we only need to prove that this 
   property does not change when applying $\mathcal{P}$. Now, if $E_{n+2}$
    is of $\tau$-type $0$ then $e_{n+2},e_{n+2}'$ are both horizontal 
    edges and the condition in the permutation in Lemma \ref{fg:250} 
    does not change. In general, going through an edge set of $\tau$-type 0 does 
    not change anything: if $e_i,e_i'$ are horizontal edges and have
     symbols $\gamma,\gamma'$, respectively, and $\gamma$ sits to the 
     left of $\gamma'$ in the bottom row, and $E_i$ is of $\tau$-type 0,
      then the new permutation will have $\gamma$ and $\gamma'$ in the 
      same relative positions in the bottom row. Thus it is only when 
      we get to an edge set $E_i$ of $\tau$-type 1 that things may change.
              
     Let $e_{i},e_{i}'\in E_{i}$ have $r(e_{i}) = v_\gamma$ and
      $r(e_{i}') = v_{\gamma'}$ and such that the permutation of 
      $\mathcal{P}^{i}(\pi,\lambda,\tau)$ is of $\tau$-type 1 and 
      has $\gamma$ immediately to the left of $\gamma'$ on the bottom row. 
      Then either
              \begin{enumerate}
              \item $\mathcal{P}^{i+1}(\pi,\lambda,\tau)$ has
               $\gamma,\gamma'$ in the same positions on the bottom row,
              \item $\mathcal{P}^{i+1}(\pi,\lambda,\tau)$ has 
              $\gamma,\gamma'$ shifted on spot to the left on 
              the bottom row, 
              \item $\mathcal{P}^{i+1}(\pi,\lambda,\tau)$ has
               $\gamma$ at the end of the bottom row, or
              \item $\mathcal{P}^{i+1}(\pi,\lambda,\tau)$ has 
              $\gamma'$ at the end of the bottom row.
              \end{enumerate}
     We now treat each case. In case (i), then the bottom row of the 
     permutation in $\mathcal{P}^{i+1}(\pi,\lambda,\tau)$ differs from 
     the bottom row of that of $\mathcal{P}^{i}(\pi,\lambda,\tau)$ on 
     some symbols to the right of $\gamma'$. This means that 
     $r(e_{i+1}) = v_{\gamma}$ and $r(e_{i+1}) = v_{\gamma'}$ and 
     so the condition is preserved. If case (ii) holds, then that 
     means that a symbol to the left of $\gamma$ got sent to the 
     end of the bottom row when going from $\mathcal{P}^{i}(\pi,\lambda,\tau)$ 
     to $\mathcal{P}^{i+1}(\pi,\lambda,\tau)$. This again implies that 
     $r(e_{i+1}) = v_{\gamma}$ and $r(e_{i+1}) = v_{\gamma'}$ and so 
     the condition is also preserved.
              
              Now suppose that case (iii) holds. Then the non-horizontal 
              edge $e_*\in E_{i+1}$ goes from $\beta_i(1) = \gamma$ 
              to $\alpha_i(0)\neq \gamma'$. This edge is $s$-max
              and so since $s(e_*) = r(e_i)$ we have that $e_* = e_{i+1}$.
               Since $\gamma$ got moved to the end of the row, we have 
               that $\gamma'$ sits immediately to the right of $\alpha_i(0)$. 
               Since $r(e_{i+1}) = \alpha_i(0)$, the condition is preserved.
              
        Finally, in case (iv), since $\gamma'$ gets moved to the end of 
        the bottom row this means that the non-horizontal edge in $E_{i+1}$
         goes from $\gamma'$ to $\gamma$ and so $r(e_{i+1}) = r(e_{i+1}')$.
          This can only happen if $i+1 = \ell$ by the definition of $\ell$, 
          but we are assuming $i<\ell-1$, so this cannot happen. We
           have proved that the condition is preserved under every
            case.
            \end{proof}
            We now move to prove that $Z_n\leq_r Z_{n+1}$. It is
             done through the same arguments used to show that 
             $Y_n\leq_r Y_{n+1}$ (Proposition \ref{fg:150}).
            \begin{prop}
              \label{fg:260}
              $Z_n\leq_r Z_{n+1}$
            \end{prop}
            Let $x = \{e_i\} $ be  in $Z_n$ and $x' = \{e'_i\} $ be in $Z_{n+1}$.
       Let $\ell>n$ be the smallest integer where
        $r(e_i) = r(e'_i)$. We will show that $e_\ell\leq_r e_\ell'$. 
            \begin{lemma}
            \label{fg:262}
   If $E_n,E_{n+1}$ are respectively of $\tau$-type 0,1 or 1,1, 
   then $\ell=n+1$ and $e_{n+1}\leq_r e_{n+1}'$.
            \end{lemma}
            \begin{proof}
              If they are of type 0,1, it is immediate to check
               that $(e_n,e_{n+1})$ is the concatenation of the 
               $s$-max path from $\beta(0)$ to $\beta(0)$
    followed by the horizontal edge $e_{\beta(0)}$, whereas $e_{n+1}'$ is 
    the (non-horizontal) $s$-max path from some vertex to 
    $\beta(0)$, meaning that $r(e_{n+1}) = r(e_{n+1}')$.
     Since horizontal paths are always $r$-min, it 
     follows that $e_{n+1}\leq_r e_{n+1}'$.
              
     If they are of type 1,1, then it is immediate to check that
      $(e_n,e_{n+1})$ is the non-horizontal horizontal path from
       $\beta(1)$ to $\alpha(0)$ followed by the horizontal path associated
        to the symbol $\alpha(0)$, whereas $e_{n+1}'$ is the $s$-max 
        path from some vertex to $v_{\alpha(0)} = r(e_{n+1})$. Again, 
        since horizontal paths are always $r$-min, it follows that
         $e_{n+1}\leq_r e_{n+1}'$.
            \end{proof}
            We now inspect the cases where $E_n,E_{n+1}$ are
             respectively of types 0,0 or 1,0.
            \begin{lemma}
              \label{fg:264}
              Suppose $E_n,E_{n+1}$ are respectively of types 0,0 or 1,0.
               For $n+1\leq i<\ell-1$, if $r(e_i) = v_\gamma$ and 
               $r(e_i') = \gamma'$, then the symbol $\gamma$ is immediately 
               to the left of $\gamma'$ on the top row of the 
               permutation defined by $\mathcal{P}^i(\pi,\lambda,\tau)$.
            \end{lemma}
            Before proving this lemma, let us prove the proposition assuming the lemma.
            \begin{proof}[Proof of Proposition \ref{fg:260} assuming
             Lemma \ref{fg:264}]
              First note that if $r(e_\ell) = r(e_\ell')$ then
               $\mathcal{P}^{\ell-1}(\pi,\lambda,\tau)$ is of 
               $\tau$-type 0, as this is the only way that we 
               can have a non-horizontal $s$-min edge in $E_\ell$. 
               If $r(e_{\ell-1})= \gamma$ and 
               $r(e'_{\ell-1}) = \gamma'$ and $\gamma$ is immediately
                to the left of of $\gamma'$, then by the definition 
                (\ref{eqn:matrices}) of $M_\ell$, the non-horizontal 
                edge goes from $\gamma'$ in $V_{\ell-1}$ to 
                $\gamma$ in $V_\ell$, and so $e_\ell\leq_r e_\ell'$, 
                showing that $Z_n\leq_r Z_{n+1}$.
            \end{proof}                    
            To prove Lemma \ref{fg:264}, we begin by proving 
            the analog of Lemma \ref{fg:250}.
            \begin{lemma}
              \label{fg:270}
              If $E_n,E_{n+1}$ are respectively of $\tau$-type 0,0 or 1,0,
               then $\ell>n+1$ and the permutation associated to
                $\mathcal{P}^{n+1}(\pi,\lambda,\tau)$ is of the form
              \begin{equation}
                \label{eqn:topGammas}
                \left(\begin{array}{lr} \cdots& \gamma \gamma' \\ 
                \cdots  &\cdots\end{array} \right),
              \end{equation}
              for some $\gamma,\gamma'$ in $\mathcal{A}$, then 
              $r(e_{n+1}) = \gamma$ and $r(e_{n+1}') = \gamma'$.
            \end{lemma}
            \begin{proof}
              Suppose they are respectively of type 0,0. Then the 
              sequence of permutations are of the form
              $$\left(\begin{array}{r} \cdots \alpha(1)\beta(0)\cdots \alpha(0)\\
               \cdots \alpha(0)\beta(1)\cdots \alpha(1)  \end{array}\right) 
               \mapsto \left(\begin{array}{r}  \cdots \alpha(1)\beta'(0) 
               \cdots \alpha(0)\beta(0)   \\ \cdots \alpha(0)\beta(1) 
               \cdots \alpha(1)   \end{array}\right)  \mapsto 
               \left(\begin{array}{r}  \cdots \alpha(1)\cdots 
               \alpha(0) \beta(0) \beta'(0)  \\ \cdots \cdots 
                \cdots \cdots \alpha(1) \end{array}\right),$$
              assuming $\beta(0) \neq \alpha(0)$ (note that $\beta'(0) \neq \beta(0)$ 
              as equality would imply that $\alpha(0) = \alpha(1)$ making the 
              original permutation reducible). Now, by definition, $e_{n+1}'$ 
              is the horizontal edge associated to the symbol 
              $\beta'(0)\neq \beta(0)$, whereas $(e_n,e_{n+1})$
               is the concatenation of the horizontal edge 
              associated to the symbol $\beta(0)$ followed by the 
              horizontal edge associated to the same symbol, $\beta(0)$, 
              and so $\gamma = \beta(0)$ and $\gamma' = \beta'(0)$.
              
              If $\alpha(0) = \beta(0)$, then the starting permutation 
              is fixed under $\mathcal{P}$ and $\mathcal{P}^2$
               (under $\tau$-type 0) and it is of the form
              $$\left(\begin{array}{r}  \cdots \alpha(1) \alpha(0) \\ 
              \cdots \alpha(0)\beta(1)\cdots \alpha(1)  \end{array}\right) .$$
              In this case $(e_n,e_{n+1})$ is the concatenation of
               horizontal edge with symbol $\beta(0) = \alpha(0)$ 
               followed by the (non-horizontal) edge from $\beta(0)$ 
               to $\alpha(1)$, whereas $e_{n+1}'$ is the horizontal 
               edge with symbol $\beta(0) = \alpha(0)$. So
                $(\gamma,\gamma') = (\alpha(1),\alpha(0))$ in this case 
                and the lemma is proved for 
               type 0,0.
              
              
              If we have type 1,0, then the sequence of permutations are of 
              the form
              $$\left(\begin{array}{r} \cdots \alpha(1)\beta(0)\cdots \alpha(0)\\ 
              \cdots \alpha(0)\beta(1)\cdots \alpha(1)  \end{array}\right)
               \mapsto \left(\begin{array}{r}  \cdots \beta(1)\beta'(0)\cdots
                \alpha(0) \\ \cdots \alpha(0)\cdots \alpha(1)\beta(1)   
                \end{array}\right)  \mapsto \left(\begin{array}{r}  \cdots
                 \beta(1)\cdots \alpha(0) \beta'(0) \\ \cdots \alpha(0)\cdots
                  \alpha(1)\beta(1) \end{array}\right),$$
              assuming $\beta(1) \neq \alpha(1)$ and $\beta'(0) \neq \alpha(0)$. 
              Now, by definition, $(e_n,e_{n+1})$ is the concatenation of the 
              edge from $\beta(1)$ to the vertex $\alpha(0)$ 
              followed by the horizontal edge associataed to the symbol 
              $\alpha(0)$, and so $\gamma = \alpha(0)$, whereas $e_{n+1}'$
               is the horizontal edge associated to the symbol 
               $\beta'(0)$ and $\gamma' = \beta'(0)$. 
              
              Note that it cannot be the case that both $\beta(1) = \alpha(1)$ 
              and $\beta'(0)= \alpha(0)$ as this would make the permutation
               irreducible. Now, if $\beta(1) = \alpha(1)$ and
                $\beta'(0)\neq \alpha(0)$, then the sequence of permutations 
                is of the form
              $$\left(\begin{array}{r}  \cdots \alpha(1)\beta(0)\cdots \alpha(0)\\ 
               \cdots \alpha(0) \alpha(1)  \end{array}\right) \mapsto 
               \left(\begin{array}{r}  \cdots \alpha(1)\beta(0)\cdots 
               \alpha(0)  \\  \cdots \alpha(0) \alpha(1)   \end{array}\right) 
                \mapsto \left(\begin{array}{r}  \cdots \alpha(1)\cdots 
                \alpha(0) \beta(0) \\ \cdots \cdots \cdots \alpha(0) 
                \alpha(1) \end{array}\right).$$
              Here ,$(e_n,e_{n+1})$ is the concatenation of the edge 
              from $\alpha(1)$ to $\alpha(0)$ followed by the horizontal 
              edge associated to the symbol $\alpha(0)$, and so 
              $\gamma = \alpha(0)$, whereas $e_{n+1}'$ is the 
              horizontal edge associated to the symbol 
              $\beta(0)$ and $\gamma' = \beta(0)$, and the result also holds here.
              
              Finally, if $\beta(1) \neq \alpha(1)$ and 
              $\beta'(0)= \alpha(0)$, then the sequence of permutations 
              is of the form
              $$\left(\begin{array}{r} \cdots \alpha(1)\beta(0)\cdots 
              \alpha(0)\\ \cdots \alpha(0)\beta(1)\cdots \alpha(1)  
              \end{array}\right) \mapsto \left(\begin{array}{r}  
              \cdots \beta(1) \alpha(0) \\ \cdots \alpha(0)\cdots 
              \alpha(1)\beta(1)   \end{array}\right)  \mapsto 
              \left(\begin{array}{r}  \cdots \beta(1) \alpha(0) \\ 
              \cdots \alpha(0)\cdots \alpha(1)\beta(1) \end{array}\right),$$
              Here $(e_n,e_{n+1})$ is the concatenation of the edge from
               $\beta(1)$ to $\alpha(0)$ followed by the (non-horizontal) 
               path from $\beta'(0) = \alpha(0)$ to $\beta(1)$, whereas 
               $e_{n+1}'$ is the horizontal edge associated to 
               the symbol $\alpha(0)$, and so 
               $(\gamma,\gamma') = (\beta(0),\alpha(1))$ and the case of 
               types 0,1 is proved.
            \end{proof}
            
            \begin{proof}[Proof of Lemma \ref{fg:264}]
              The proof of this Lemma follows the same argument as the
               proof of Lemma \ref{fg:240} except $\tau$-type 
               $\varepsilon$ has to be replaced with type $1-\varepsilon$
                and bottom rows with top rows due to Lemma \ref{fg:270}.
                 We leave the details to the reader.
            \end{proof}
            \begin{proof}[Proof of Theorem \ref{fg:200}]
              By Propositions \ref{fg:220} and \ref{fg:260},
               $\Delta_s$ preserves the $\leq_r$ order. The 
               result then follows from Lemma \ref{fg:210}.
            \end{proof}

            Now that the bi-infinite ordered Bratteli 
            diagram has been defined for a typical $(\pi,\lambda,\tau)$, 
            we move on to define the states. Define the negative and positive
             cones of $\mathcal{B}_{\pi,\lambda,\tau}$ as
            $$\mathcal{C}^-_{\pi,\lambda,\tau}= \bigcap_{n>0}
             M_{\pi,\lambda,\tau}^{(-n)}\left( \mathbb{R}^\mathcal{A}_
             +\right)\hspace{.5in}\mbox{ and }
             \hspace{.5in}\mathcal{C}^+_{\pi,\lambda,\tau}= 
             \bigcap_{n>0} M_{\pi,\lambda,\tau}^{(n)*}\left( 
             \mathbb{R}^\mathcal{A}_+\right).$$
            Recalling Proposition \ref{BD:115}, we have the following.
            \begin{lemma}
            \label{fg:280}
              The set of states for $\mathcal{B}_{\pi,\lambda,\tau}$
               is parametrized by 
              $ \mathcal{C}^-_{\pi,\lambda,\tau}\times
                \mathcal{C}^+_{\pi,\lambda,\tau}$.
            \end{lemma}
            It follows from Veech's theorem on the ergodicity of the
             Teichm\"uller flow (Theorem \ref{fg:120} above) 
             that the set of normalized states of a typical 
             triple $(\pi,\lambda,\tau)$ is in a sense unique.
            \begin{thm}
              \label{fg:290}
              For almost every $(\pi,\lambda,\tau)$, there exists a 
              normalized state $\nu = (\nu_r,\nu_s)$ for 
              $\mathcal{B}_{\pi,\lambda,\tau}$ which is unique
              in the sense that 
              any other normalized state $\nu' = (\nu_r',\nu_s')$ 
               $\mathcal{B}_{\pi,\lambda,\tau}$ satisfies 
               $(\nu_r',\nu_s') = (e^{-t}\nu_r,e^t\nu_s)$,
                for some $t$ in $\mathbb{R}$.
            \end{thm}
 \subsection{Dynamics of Bratteli diagrams}
    Since we have determined how to build a Bratteli diagram 
    $\mathcal{B}_{\pi,\lambda,\tau}$ from the
     triple $(\pi,\lambda,\tau)\in \mathcal{V}_\mathcal{C}$, 
     we point out that there is an obvious 
     relationship between the diagram for $(\pi,\lambda,\tau)$
      and that of $\mathcal{P}(\pi,\lambda,\tau)$.
            \begin{defn}
            \label{fg:190}
              Let $\mathcal{B}$ be a
               bi-infinite  ordered Bratteli diagram. The 
               \emph{shift} of $\mathcal{B}$ is the  bi-infinite ordered
    Bratteli diagram $\mathcal{B}',\leq_{r,s}'$ 
     such that $E'_n = E_{n+1}$, $V_n' = V_{n+1}$ with
      the property that $r'=r, s'=s, \leq_{r'} = \leq_{r}, \leq_{s'} = \leq_{s}$.
      We also denote the shift by $\mathcal{B}' = \sigma(\mathcal{B})$.
            \end{defn}
            In short, $ \sigma(\mathcal{B})$ 
            shifts all the indices of $\mathcal{B}$ while preserving
             the structure. It follows from the construction
              in the previous section that we have
            $$\mathcal{B}_{\mathcal{P}^n(\pi,\lambda,\tau)}
             = \sigma^n(\mathcal{B}_{\pi,\lambda,\tau}).$$
            We now make some remarks about how these ideas 
            carry over to the algebras constructed.

            First, it is straight-forward that the AF algebras
             defined by $\mathcal{B}_{\pi,\lambda,\tau}$ and 
             $\mathcal{B}_{\mathcal{P}^n(\pi,\lambda,\tau)} =
              \sigma^n(\mathcal{B}_{\pi,\lambda,\tau})$ are
               the same for every $n$. That is, they are
                independent of where one chooses the ``origin''
                 on $\mathcal{B}_{\pi,\lambda,\tau}$ to be. 
                 This is true for \emph{any} bi-infinite 
                 Bratteli diagram and not just for those 
                 $\mathcal{B}_{\pi,\lambda,\tau}$ being built
                  from zippered rectangles data.
            
            Second, if $\nu = (\nu_r,\nu_s)$ form a 
            state for $\mathcal{B}_{\pi,\lambda,\tau}$, then
             $\nu_t = (e^{-t}\nu_r,e^t\nu_s)$ is a one-parameter
              family of states for $\mathcal{B}_{\pi,\lambda,\tau}$ 
              (deforming states like this also does not depend on
               $\mathcal{B}_{\pi,\lambda,\tau}$ being built from
                zippered rectangles data). While the AF algebras 
                defined by $\mathcal{B}_{\pi,\lambda,\tau}$ do 
                not depend on the state $\nu$, the various 
                algebras associated to our foliated spaces 
                do depend on a choice of state. Thus $\nu_t$
                 gives several one-parameter families of algebras.

            In addition, given the definition of a pre-stratum 
            in (\ref{eqn:preStratum}) it is tempting to make the
             identification of the form
            $$ (\mathcal{B}_{\pi,\lambda,\tau},e^{-t_R^+}\nu_r, 
            e^{t_R^+}\nu_s)\sim 
            (\mathcal{B}_{\mathcal{P}(\pi,\lambda,\tau)},
            \sigma_* \nu_r,\sigma_* \nu_s)= 
            (\sigma(\mathcal{B}_{\pi,\lambda,\tau}),\sigma_* 
            \nu_r,\sigma_* \nu_s).$$
            Thus the Teichm\"uller flow $g_t$ (or $\Phi_t$)
             is manifested as a continuous deformation of
             the algebras by deforming the states $\nu\mapsto\nu_t$ 
             up to some time before shifting the Bratteli diagram.
 
            \subsection{The $K$-theory}
            We are at a point where we can compute the $K$-theory of the foliation algebras of the typical flat surface in any stratum $\mathcal{H}(\bar{\kappa})$. Let us summarize how we got here: through Veech's construction of zippered-rectangles, we can represent almost every flat surface $(S,\alpha)\in\mathcal{H}(\bar{\kappa})$ by a triple $(\pi,\lambda,\tau)\in\mathcal{V}_\mathcal{C}$ in the space of zippered rectangles $\mathcal{V}_\mathcal{C}$. In fact, the subset $\mathcal{V}_\mathcal{C}$ is made up exclusively of triples which satisfy the Keane condition and is RH-complete, meaning that we can assign to them a strongly simple bi-infinite Bratteli 
    diagram $\mathcal{B}_{\pi,\lambda,\tau}$. We saw in Propositions \ref{fg:150}
            and \ref{fg:160}
             that these diagrams have the property that $|X^{s-max}|=|X^{s-min}|=1$ and $X^{s-max}\cup X^{s-min}\subseteq X^{r-min}$. Moreover, in  Theorem \ref{fg:200}
              we saw that they also satisfy 
            $\Sigma_{\mathcal{B}_{\pi,\lambda,\tau}}=\varnothing$. This sets the stage to compute their $K$-theory.

            \begin{thm}
              \label{fg:400}
              For $m^-_\mathcal{C}$-almost every $(\pi,\lambda,\tau)$, we have
              \begin{align*}
                K_0(C^*(\mathcal{F}^+_{\mathcal{B}_{\pi,\lambda,\tau}})) &\cong  H_1(S(\pi,\lambda,\tau),\Sigma;\mathbb{Z}) \cong K_0( C^{*}_{\lambda}(T^{+}(X_{\mathcal{B}_{\pi,\lambda,\tau}}))) \cong \mathbb{Z}^{\mathcal{A}} \\  &\mbox{ and } \hspace{.6in}K_1( C^*(\mathcal{F}^+_{\mathcal{B}_{\pi,\lambda,\tau}}))   \cong \mathbb{Z}.
              \end{align*}
            \end{thm}
            \begin{proof}
              First, recall (\ref{eqn:K-homology}) from Remark \ref{fg:140}. Moreover, the third and fourth parts of Proposition \ref{fg:150} imply 
               that $I_{\mathcal{B}_{\pi,\lambda,\tau}} 
                J_{\mathcal{B}_{\pi,\lambda,\tau}} = 1$, 
                so by Proposition \ref{K:20} and Theorem \ref{K:30},
                 we have that
              $$K_0(C^{*}_{\lambda}(T^{\sharp}(S^{s}_{\mathcal{B}_{\pi,\lambda,\tau}})) ) 
               \cong K_0( C^{*}_{\lambda}(T^{+}(X_{\mathcal{B}_{\pi,\lambda,\tau}})))$$
              and $K_1(B_{\mathcal{B}_{\pi,\lambda,\tau}})\cong \mathbb{Z}$.
              Turning to Theorem \ref{K:50}, $\ker(\iota)$ is trivial, 
              and so by exactness we obtain that 
              $K_0(C^*(\mathcal{F}^+_{\mathcal{B}_{\pi,\lambda,\tau,\leq_{r,s}}}))
               \cong K_0( C^{*}_{\lambda}(T^{\sharp}(S^{s}_{\mathcal{B}_{\pi,\lambda,\tau}})))$.
            \end{proof}

            \subsection{Ordered $K$-theory and asymptotic cycles}
            In this subsection we connect the structure of the topological 
            invariants of the surface with that of the algebras constructed.

            First we recall the \emph{Schwartzman asymptotic cycle} \cite{schwartzman:cycle}.
             Let $\phi^+_t$ be the horizontal flow on a flat surface $S$ of finite
              genus, which we assume for the moment to be minimal and uniquely 
              ergodic, and $p\in S$ a point with an infinite trajectory. For 
              any $T$ let $\gamma_T(p)\subseteq S$ be a closed curve which 
              contains the orbit segment $\{\phi^+_t(p)\}_{t=0}^T$ and is 
              closed by a segment $\gamma^*_T(p)$ of diameter at most $\mathrm{diam}\,(S)$.
               Define $c_T(p) =[\gamma_T(p)]\in H_1(S,\Sigma;\mathbb{Z})$ to be 
               its integer homology class. This class is not uniquely defined,
                but the error is bounded independently of $T$ as the closing 
                segments $\gamma^*_T(p)$ have bounded length. The (Schwartzman) asymptotic cycle is defined as
            \begin{equation}
              \label{eqn:Schwartzman}
              c = \lim_{T\rightarrow \infty} \frac{c_T(p)}{T}\in H_1(S,\Sigma;\mathbb{R}).
            \end{equation}
            That this limit does not depend on $p$ is a consequence
             of unique ergodicity.

            Recall the map $\hat{\mathcal{P}}^+$ in (\ref{eqn:renormMap}) and consider its induced action $\hat{\mathcal{P}}^+_*:H_1(S,\Sigma;\mathbb{R})\rightarrow H_1(S,\Sigma;\mathbb{R})$. There is a natural choice of basis of $H_1(S,\Sigma;\mathbb{R})$, indexed by $\mathcal{A}$, such that $\hat{\mathcal{P}}^+_*$ is given in coordinates by $\Theta^{-1}$. This is the (backwards) \emph{Rauzy-Veech cocycle} over the space of zippered rectangles $\bar{\mathcal{V}}_\mathcal{C}$. We denote by $\hat{\mathcal{P}}_*^{(n)}= \hat{\mathcal{P}}_*^{+n}$ the linear map on homology obtained from the composition of this cocycle $n$ times. This cocycle is not integrable with respect to the measure $m^-_1$. However, Zorich \cite{zorich:gauss} found an acceleration of this cocycle, called the Zorich cocycle, which is integrable and thuse yields an Oseledets splitting of the homology space. More specifically, there exist real numbers $\nu_1>\nu_2>\dots> \nu_{k_{m_1^-}}$ (the Lyapunov spectrum) such that for $m_1^-$-almost every $(\pi,\lambda, \tau)$, there exists cycles $c_1,\dots, c_{k_{m_1^-}}\in H_1(S(\pi,\lambda,\tau),\Sigma;\mathbb{R})$ (called \emph{Zorich cycles}) and a $\hat{\mathcal{P}}^+_*$-invariant splitting of $H_1(S(\pi,\lambda,\tau),\Sigma;\mathbb{R})$
            \begin{equation}
              \label{eqn:decomp}
              H_1(S(\pi,\lambda,\tau),\Sigma;\mathbb{R}) = 
              \bigoplus_{i=1}^{k_{m_1^-}} E_i
            \end{equation}
            with $E_i = \mathrm{span}\, \{c_i\}$, such that for any
             non-zero $c\in E_i$
            $$\lim_{n\rightarrow \infty}
             \frac{\log \|\hat{\mathcal{P}}_*^{(n)}c \|}{n} = \nu_i.$$
            The Zorich cocycle preserves a symplectic form, and therefore the Lyapunov spectrum is symmetric around zero, that is, if $\nu_i$ is in the Lyapunov spectrum, then so is $-\nu_i$. Forni \cite{forni:deviation} proved that there are exactly $g$ positive and $g$ negative exponents, and Avila-Viana showed \cite{AvilaViana} that each Oseledets subspace corresponding to a non-zero exponent has dimension 1, that is, the Lyapunov spectrum is of the form $\nu_1>\nu_2>\cdots>\nu_g>0>\nu_{g+1} = -\nu_g>\cdots >-\nu_1 = \nu_g$. The top Zorich cycle, $c_1$ coincides the the Schwartzman asymptotic cycle for the horizontal flow. There is a dual cocycle to the Rauzy-Veech cocycle acting on cohomology, called the \emph{Kontsevich-Zorich cocycle}, and dual cocycles $c_1^*,\dots, c_{2g}^*\in H^1(S(\pi,\lambda,\tau),\Sigma;\mathbb{R})$ called \emph{Forni cocycles} with the same properties. In addition, $c_1^* = [i_Y \omega]\in H^1(S(\pi,\lambda,\tau),\Sigma;\mathbb{R})$, where $\omega$ is the area form on $S(\pi,\lambda,\tau)$ and $Y$ is the vector field generating the vertical foliation.

              To make the connection between the cocycles above with their Oseledets decomposition and the invariants of our algebras, we need to define the trace space of an AF algebra.
              \begin{defn}
              \label{fg:410}
        A \emph{trace} on a $*$-algebra $A$ is a linear functional 
        $\tau: A \rightarrow \C$ satisfying $\tau(ab)= \tau(ba)$, for 
        all $a, b$ in $A$. A trace $\tau$ is called 
        \emph{positive} if $\tau(a^{*}a) \geq 0$, for all $a$ in $A$.
        We let $\tr(A)$ denote the set of all traces on $A$, which is 
        a complex vector space.
              \end{defn}
              \begin{rmk}
              \label{fg:420}
                Some remarks:
                \begin{enumerate}
    \item It is a fairly easy exercise to see that, for any $n \geq1$, 
    the $*$-algebra of $n \times n$-matrices, $M_{n}(\C)$, has a 
    trace which simply sums the diagonal entries and this is 
    unique, up to a scaling factor. It follows that the set
    of traces on any finite-dimensional $C^{*}$-algebra,
    $\bigoplus_{k=1}^{K} M_{n_{k}}(\C)$, is in bijection with $\R^{K}$. 
    
    If we consider an inductive system of such $*$-algebras as we
    have in Proposition \ref{AF:40}, 
    \[
    A_{m,m} \subseteq A_{m,m+1} \subseteq \cdots
    \]
    with inclusions described by matrices $E_{m+1}, E_{m+2}, \ldots$, then the 
    set of traces on the union can be identified with 
    \[
    \lim_{\leftarrow} \R^{V_{m}} \stackrel{E_{m+1}^T}{\longleftarrow}  
    \R^{V_{m+1}} \stackrel{E_{m+2}^T}{\longleftarrow}     \cdots
    \]
   It is important to note that these traces are defined only on the union
   of the finite-dimensional algebras; most do not extend to the AF-algebra
   which is the completion.  On the other hand, it is well-known that 
   the inclusion of the  locally finite-dimensional algebra which is the union 
   in  the AF-algebra which is its completion induces an order isomorphism 
   on $K$-theory.
   
    In our situation, where   we construct these algebras from groupoids, 
    the traces correspond to finitely additive measures defined 
    on clopen transversals to the equivalence relation $T^{+}(Y_{\mathcal{B}})$.
    This idea first appeared in the work of Bowen and Franks  
    \cite{BowenFranks:homology}. This relates some of our point of 
    view with that of Bufetov's \cite{bufetov:limit, bufetov:limitVershik}.
                \item The trace space $\tr(A)$ serves as a dual to $K_0(A)$:
    if $p$ and $q$ are projections in $A$ which determine the same $K$-theory
    class, and if $\tau$ is any trace, then $\tau(p) = \tau(q)$ is a consequence
    of the trace property. Hence, there is pairing 
    $\tr(A) \times K_{0}(A)  \rightarrow \C$. 
                \end{enumerate}
              \end{rmk}
              Note that by Remark \ref{fg:140} (i), we obtain isomorphisms
              \begin{equation}
                \label{eqn:isomoprhisms}
                  \mathfrak{i}_{\pi,\lambda,\tau}:K_0(A^+_{\mathcal{B}_{\pi,\lambda,\pi}})\rightarrow  H_1(S(\pi,\lambda,\tau),\Sigma;\mathbb{Z}) \hspace{.1in} \mbox{ and }  \hspace{.1in} \mathfrak{i}^*_{\pi,\lambda,\tau}: H^1(S(\pi,\lambda,\tau),\Sigma;\mathbb{C})\rightarrow  \tr(A^+_{\mathcal{B}_{\pi,\lambda,\tau}}).
              \end{equation}
              Through these identifications, and through the 
              identifications of $K_0(A^+_{\mathcal{B}_{\pi,\lambda,\tau}})$ 
              with \newline
              $K_0(C^*(\mathcal{F}^+_{\mathcal{B}_{\pi,\lambda,\tau},\leq_{r,s}}))$ 
              from Theorem \ref{fg:400},
               the map $\hat{\mathcal{P}}^+$ also induces
                isomorphisms which we also denote as
              \begin{equation}
                \label{eqn:K0maps}
                \begin{split}
                  \hat{\mathcal{P}}^+_* &: K_0(A^+_{\mathcal{B}_{\pi,\lambda,\pi}})\longrightarrow K_0(A^+_{\sigma(\mathcal{B}_{\pi,\lambda,\pi})})\;\;\;\mbox{ and }\\
                  \hat{\mathcal{P}}^+_* &: K_0(C^*(\mathcal{F}^+_{\mathcal{B}_{\pi,\lambda,\tau},\leq_{r,s}}))\rightarrow K_0(C^*(\mathcal{F}^+_{\sigma(\mathcal{B}_{\pi,\lambda,\tau},\leq_{r,s})}))
                \end{split}
              \end{equation}
              and a maps at the level of traces. Moreover, the maps are order-preserving.
              \begin{thm}
                \label{fg:430}
                For $m_\mathcal{C}^-$-almost every $(\pi,\lambda,\tau)$, the order structure
                 on $K_0( C^{*}_{\lambda}(T^{+}(X_{\mathcal{B}_{\pi,\lambda,\tau}})))$
                  and $K_0(C^*(\mathcal{F}^+_{\mathcal{B}_{\pi,\lambda,\tau},\leq_{r,s}}))$ are determined by the first Zorich cocycle, that is, the Schwartzman asymptotic cycle, and the maps (\ref{eqn:K0maps}) are order-preserving.
              \end{thm}
              \begin{proof}
                Let $(\pi,\lambda, \tau)$ be an Oseledets-regular point for the Zorich cocycle, that is a triple so that an Oseledets decomposition of the form (\ref{eqn:decomp}) holds. We define the order structure on 
                $K_0( C^{*}_{\lambda}(T^{+}(X_{\mathcal{B}_{\pi,\lambda,\tau}})))$; the
                 structure $K_0(C^*(\mathcal{F}^+_{\mathcal{B}_{\pi,\lambda,\tau},\leq_{r,s}}))$ is obtained from the order-preserving isomorphism in Theorem \ref{K:38}.

                Define the positive cone
                $$K_0^+(  C^{*}_{\lambda}(T^{+}(X_{\mathcal{B}_{\pi,\lambda,\tau}}))) = \left\{[p]\in 
                K_0(  C^{*}_{\lambda}(T^{+}(X_{\mathcal{B}_{\pi,\lambda,\tau}}))):\mathfrak{i}^*_{\pi,\lambda,\tau}c_1^*([p])=c_1^*(\mathfrak{i}_{\pi,\lambda,\tau}([p])) >0 \right\},$$
                where $c_1^*$ is the dual of the Schwartzman cycle. 
                By the invariance of the Oseledets decomposition, 
                $\hat{\mathcal{P}}_*^+$ is an order-preserving isomorphism.
              \end{proof}
              We can now argue that there is \emph{some} appeal
               to our approach to translation flows on flat surfaces.
                It goes like this: the Schwartzman 
              asymptotic cycle is defined for flows on compact manifolds 
              or those whose homology spaces are finite dimensional. If 
              we were to pick at random, the random bi-infinite Bratteli
               $\mathcal{B}$ diagram (of finite rank, supposing for a 
               moment that there is a unique normalized state on 
               $\mathcal{B}$ in the sense of Theorem \ref{fg:290})
                and random order $\leq_{r,s}$ with a choice of normalized
                 state $\nu$ will yield a flat surface of infinite genus 
                 $S_{\mathcal{B},\leq_{r,s}}$. If we were to try to define 
                 the asymptotic cycle using (\ref{eqn:Schwartzman}) as a 
                 definition, then it is not necessarily clear it is 
                 well-defined as the topology of $H_1(S;\mathbb{R})$ is not 
                 automatically defined. However, what Theorem 
                 \ref{fg:430} suggests is that what is
                  relevant to capture the asymptotic topological 
                  information is the order structure of
                   $K_0(C^*(\mathcal{F}^+_{\mathcal{B},\leq_{r,s}}))$ 
                   especially when its inclusion into
                   $K_0(  C^{*}_{\lambda}(T^{+}(X_{\mathcal{B}_{\pi,\lambda,\tau}})))$ yields an order isomorphism, 
                   and when the shift induces an order
                    isomorphism 
                    $K_0(  C^{*}_{\lambda}(T^{+}(X_{\mathcal{B}_{\pi,\lambda,\tau}})))
                    \rightarrow K_0(  C^{*}_{\lambda}(T^{+}(X_{\sigma(\mathcal{B}_{\pi,\lambda,\tau})})))$.

\bibliographystyle{amsalpha}
\bibliography{bib.bib}
\end{document}